\documentclass[reqno]{amsart}

\usepackage{graphicx}

\usepackage[margin=1in]{geometry}
\usepackage{bbm}
\usepackage{amssymb, amsthm, mathtools}
\usepackage{amsfonts}
\usepackage{latexsym, amssymb, amsmath, amscd, amsthm, amsxtra}
\usepackage{thmtools}
\usepackage{enumerate}
\usepackage[all]{xy}
\usepackage{mathrsfs}
\usepackage{fancyhdr}
\usepackage{listings}
\usepackage{hyperref}
\usepackage{cleveref}
\usepackage{soul}
\usepackage{xcolor}
\usepackage{dsfont}
\usepackage{stmaryrd}
\SetSymbolFont{stmry}{bold}{U}{stmry}{m}{n}
\usepackage{comment}

\allowdisplaybreaks
\usepackage{array}
\usepackage{tikz}
\usetikzlibrary{arrows}
\usetikzlibrary{arrows.meta}
\usetikzlibrary{decorations.pathmorphing}
\usetikzlibrary{decorations.pathreplacing,decorations.markings}
\usetikzlibrary{calc}
\usetikzlibrary{shapes,petri, topaths, automata}
\usepackage{tikz-feynman} 
\tikzfeynmanset{compat=1.1.0}

\tikzset{
  on each segment/.style={
    decorate,
    decoration={
      show path construction,
      moveto code={},
      lineto code={
        \path [#1]
        (\tikzinputsegmentfirst) -- (\tikzinputsegmentlast);
      },
      curveto code={
        \path [#1] (\tikzinputsegmentfirst)
        .. controls
        (\tikzinputsegmentsupporta) and (\tikzinputsegmentsupportb)
        ..
        (\tikzinputsegmentlast);
      },
      closepath code={
        \path [#1]
        (\tikzinputsegmentfirst) -- (\tikzinputsegmentlast);
      },
    },
  },
  mid arrow/.style={postaction={decorate,decoration={
        markings,
        mark=at position .5 with {\arrow[#1]{Stealth[scale=1.2]}}
      }}},
}

\DeclareMathOperator{\re}{{\mathrm{Re}}}
\DeclareMathOperator{\im}{{\mathrm{Im}}}

\renewcommand{\Dot}{\mathbf {Dot}}

\hypersetup{
     colorlinks=true
}
\def\P{\mathbb{P}}

\def\E{\mathbb{E}}
\def\IE{\mathbb{IE}}

\def\Z{\mathbb{Z}}
\def\R{\mathbb{R}}
\def\N{\mathbb{N}}

\def\11{{\mathbf{1}}}

\newcommand{\be}{\begin{equation}}
\newcommand{\ee}{\end{equation}}
\newcommand{\ii}{\mathrm{i}}
\newcommand{\dd}{\mathrm{d}}

\newcommand{\e}{{\varepsilon}}
\newcommand{\cK}{\mathcal K}
\newcommand{\cL}{\mathcal L}
\newcommand{\sig}{\sigma}
\newcommand{\bsig}{{\boldsymbol{\sigma}}}

\newcommand{\ba}{{\mathbf{a}}}
\newcommand{\bal}{{\boldsymbol{\alpha}}}
\newcommand{\bbeta}{{\boldsymbol{\beta}}}
\newcommand{\bfb}{{\mathbf{b}}}
\newcommand{\cut}{\mathcal{G}}
\newcommand{\cutL}{(\mathcal{G}_L)}
\newcommand{\cutR}{(\mathcal{G}_R)}

\newcommand{\lenk}{l_{\mathcal K}}
\DeclareMathOperator{\TSP}{{\mathbf T}}

\newcommand{\BA}{\mathrm{BA}}

\newcommand{\alt}{\mathrm{alt}}

\newcommand{\fd}{{\mathfrak d}}

\newcommand{\bu}{{\bf{u}}}
\newcommand{\bv}{{\bf{v}}}
\newcommand{\bw}{{\bf{w}}}

\newcommand{\bigGamma}{\Gamma}
\newcommand{\ord}{{\mathsf{ord}}}

\newcommand{\size}{{\mathsf{size}}}

\newcommand{\Mol}{\mathcal{M}}

\newcommand{\new}{\mathrm{new}}
\newcommand{\aux}{\mathrm{aux}}
\newcommand{\ngh}{\mathrm{ng}}
\newcommand{\C}{\mathbb C}

\renewcommand{\bar}{\overline}
\newcommand{\wt}{\widetilde}
\newcommand{\wh}{\widehat}

\newcommand{\al}{\alpha}

\newcommand{\qqq}[1]{\llbracket{#1}\rrbracket}
\newcommand{\br}[1]{\left[{#1}\right]}
\newcommand{\p}[1]{\left({#1}\right)}
\newcommand{\abs}[1]{\left|{#1}\right|}

\newcommand{\sB}{{\mathsf{B}}}
\newcommand{\fP}{\mathfrak{P}}
\newcommand{\bpi}{{\boldsymbol{\pi}}}

\newcommand{\LK}{{(\mathsf{B})}}

\newcommand{\Zn}{{\mathbb Z}_L^d}
\newcommand{\ZL}{{\mathbb Z}_{WL}^d}
\newcommand{\Gc}{{\mathring G}}
\newcommand{\bD}{{\mathbf{D}}}

\newcommand{\cT}{\mathcal T}
\newcommand{\cE}{\mathcal E}
\newcommand{\wT}{\widetilde{\mathcal T}}

\newcommand{\zTheta}{\mathring \Theta}

\newcommand{\dthn}{\boldsymbol{\chi}}

\newcommand{\fa}{{\mathfrak a}}

\newcommand{\fc}{{\mathfrak c}}

\newcommand{\Err}{{\mathcal Err}}

\DeclareMathOperator{\OO}{O}
\DeclareMathOperator{\oo}{o}

\newcommand{\bpsi}{\boldsymbol{\psi}}

\newcommand{\qll}[1]{[\![{#1}]\!]}
\newcommand{\rep}[1]{({#1})}
\newcommand{\too}{\to \to}

\newcommand{\ilambda}{g}

\DeclareMathOperator{\diag}{diag}
\DeclareMathOperator{\tr}{Tr}
\DeclareMathOperator{\var}{Var}

\newcommand{\nc}{\normalcolor}

\newcommand{\cal}{\mathcal}
\newcommand{\fn}{{n}}

\newcommand{\sT}{{\mathsf T}}

\newcommand{\contract}{\text{loop-contraction}}
\newcommand{\Contract}{\text{Loop-contraction}}

\theoremstyle{plain} 
\newtheorem{theorem}{Theorem}[section]
\newtheorem*{theorem*}{Theorem}
\newtheorem{lemma}[theorem]{Lemma}

\newtheorem*{lemma*}{Lemma}

\newtheorem*{corollary*}{Corollary}

\newtheorem*{proposition*}{Proposition}
\newtheorem{claim}[theorem]{Claim}
\newtheorem*{claim*}{Claim}

\newtheorem{definition}[theorem]{Definition}  
\newtheorem*{definition*}{Definition}         
\theoremstyle{remark}
\newtheorem{example}[theorem]{Example}
\newtheorem*{example*}{Example}
\newtheorem{remark}[theorem]{Remark}

\newtheorem*{remark*}{Remark}
\newtheorem*{remarks*}{Remarks}
\newtheorem{strategy}[theorem]{Strategy}
 \numberwithin{equation}{section}

\crefname{theorem}{theorem}{theorems}
\crefname{lemma}{lemma}{lemmas}
\crefname{definition}{definition}{definitions}
\crefname{proposition}{proposition}{propositions}
\crefname{corollary}{corollary}{corollaries}
\crefname{claim}{claim}{claim}
\crefname{example}{example}{Example}
\crefname{remark}{remark}{remark}

\allowdisplaybreaks
\def\@setthanks{\vspace{-\baselineskip}\def\thanks##1{\@par##1\@addpunct.}\thankses}

\makeatletter
\let\old@setauthors\@setauthors
\def\@setauthors{%
  \begingroup
    \let\MakeUppercase\relax
    \old@setauthors
  \endgroup
}
\makeatother

\title{Delocalization of Non-Mean-Field Random Matrices \\ in Dimensions $d\ge 3$}

\author{\small Sofiia Dubova$^\star$} 
\author{Fan Yang$^\dagger$}
\author{Horng-Tzer Yau$^\ddagger$}
\author{Jun Yin$^\S$}

\thanks{$^\star$Department of Mathematics, Northwestern University, \href{mailto:sdubova@northwestern.edu}{sdubova@northwestern.edu}}
\thanks{$^\dagger$Yau Mathematical Sciences Center, Tsinghua University, \href{mailto:fyangmath@mail.tsinghua.edu.cn}{fyangmath@mail.tsinghua.edu.cn}}
\thanks{$^\ddagger$Department of Mathematics, Harvard University, \href{mailto:htyau@math.harvard.edu}{htyau@math.harvard.edu}}
\thanks{$^\S$Department of Mathematics, University of California, Los Angeles, \href{mailto:jyin@math.ucla.edu}{jyin@math.ucla.edu}}

\begin{document}

\begin{abstract}

\small 
We study $N \times N$ random band matrices $H = (H_{xy})$ with mean-zero complex Gaussian entries, where $x,y$ lie on the discrete torus \smash{$(\mathbb{Z} / \sqrt[d]{N} \mathbb{Z})^d$} in dimensions $d \ge 3$. The variance profile satisfies $\mathbb{E}|H_{xy}|^2 = S_{xy}$, with $S_{xy} = 0$ whenever the distance between $x$ and $y$ exceeds a bandwidth parameter $W$. We prove that if $W \geq N^{\mathfrak{c}}$ for some constant $\mathfrak{c} > 0$, then in the large-$N$ limit, bulk eigenvectors are delocalized, quantum unique ergodicity (QUE) holds, and the local bulk eigenvalue statistics are universal. Our proof is based on the tree approximation of the loop hierarchy \cite{YY_25} and diagrammatic techniques developed in earlier works \cite{yang2021random, yang2021delocalization, yang2022delocalization, DY, DYYY25}.

Besides random band matrices, we also study two classical non-mean-field random matrix models: the Wegner orbital and the block Anderson models. Specifically, we consider Hermitian matrices $H = V + \ilambda \Psi$ on the same discrete torus \smash{$(\mathbb{Z} / \sqrt[d]{N} \mathbb{Z})^d$},  where $V$ is a random block potential consisting of i.i.d.~complex Gaussian diagonal blocks of size $W^d \times W^d$, and $\Psi$ encodes the interactions between neighboring blocks---random in the Wegner orbital model and deterministic in the block Anderson model. The parameter $\ilambda > 0$ represents the coupling strength between blocks. Assuming again that $W \geq N^{\mathfrak{c}}$, we establish delocalization of bulk eigenvectors, QUE, and bulk universality under the condition \smash{$W^{-d/2+\varepsilon}\le \ilambda \le \varepsilon^{-1}$} for any small constant $\varepsilon>0$. 
Combined with the localization results of~\cite{PelSchShaSod} for \smash{$\ilambda \ll W^{-d/2}$}, this identifies a localization--delocalization transition at the scale $\ilambda=W^{-d/2}$ in dimensions $d \ge 3$.
\end{abstract}

\maketitle
 
{
\hypersetup{linkcolor=black}
\tableofcontents
}

\section{Introduction}

The tight-binding model introduced by Anderson \cite{Anderson} describes electron transport in disordered semiconductors. It is formulated as a discrete random Schrödinger operator on $\mathbb{Z}^{d}$ of the form
\begin{align}\label{Anderson_orig}
	H = - \Delta + \lambda V,
\end{align}
where $\Delta$ denotes the graph Laplacian on $\mathbb{Z}^{d}$, $V$ is a random potential with i.i.d.~entries, and $\lambda>0$ is a coupling parameter representing the strength of disorder. In his seminal work \cite{Anderson}, Anderson predicted a transition from delocalized to localized behavior as the disorder strength $\lambda$ increases.

Localization in the one-dimensional (1D) Anderson model, valid for all $\lambda>0$, is well established; see, for example, \cite{GMP, KunzSou, Carmona1982_Duke, Damanik2002}. In dimensions $d \geq 2$, Anderson localization was first rigorously proved by Fröhlich and Spencer \cite{FroSpen_1983} using multi-scale analysis. An alternative approach was later developed by Aizenman and Molchanov \cite{Aizenman1993} based on the fractional moment method. Since then, a substantial body of work has deepened our understanding of Anderson localization; see, for example, \cite{FroSpen_1985, Carmona1987, SimonWolff, Aizenman1994, ASFH2001, Bourgain2005, Germinet2013, DingSmart2020, LiZhang2019}. 
In contrast, the existence of delocalized states in the finite-dimensional Anderson model remains unproven—both for arbitrary disorder strengths and in all finite dimensions. In infinite dimensions, however, Aizenman and Warzel \cite{aizenman2013resonant, Bethe_PRL} rigorously established the presence of a delocalized phase for the Anderson model on the Bethe lattice, an infinite regular tree. More recently, the mobility edge phenomenon has also been proved in this setting \cite{Bethe-Anderson}. The tree geometry differs fundamentally from Euclidean lattices, as tree graphs are loop-free.

To bridge the understanding between random Schrödinger operators and random matrix theory, the random band matrix model was introduced \cite{PhysRevLett.64.1851, PhysRevLett.64.5, PhysRevLett.66.986}. A $d$-dimensional random band matrix $H = (H_{xy})$ is an $N\times N$ Hermitian random matrix defined on the discrete torus \smash{$(\mathbb{Z}/\sqrt[d]{N}\mathbb{Z})^{d}$} (assuming \smash{$\sqrt[d]{N}\in \mathbb{N}$} for simplicity). Subject to Hermitian symmetry, the entries $H_{xy}$ are independent real or complex Gaussian variables with mean zero. The variance profile $S_{xy} = \mathbb{E}|H_{xy}|^{2}$ decays to zero when the distance between $x$ and $y$ exceeds a bandwidth parameter $W$, and satisfies the normalization $\sum_y S_{xy} \equiv 1$. 
It was conjectured in \cite{PhysRevLett.64.1851, PhysRevLett.64.5, PhysRevLett.66.986,PhysRevLett.67.2405} that this model exhibits a localization–delocalization transition, accompanied by a change in spectral statistics from Poisson to GOE/GUE (Gaussian Orthogonal/Unitary Ensemble) statistics, at a critical bandwidth $W_c(N)$. For eigenvalues in the bulk of the spectrum, i.e., $|E| \le 2 - \kappa$ with some fixed $\kappa > 0$, localization and Poisson statistics are expected when $W \ll W_c$, while delocalization and GOE/GUE statistics are expected when $W \gg W_c$, where
\begin{equation}\label{eqn:transition exponents}
	W_c =
	\begin{cases}
		\sqrt N, & d=1, \\
		\sqrt{ \log N}, & d=2, \\
		\OO(1), & d\ge 3.
	\end{cases}
\end{equation}
See \cite{PB_review, Spencer2, Spencer3, Spencer1} for further discussion of these conjectures.

The density of states for random band matrices in dimensions $d \ge 1$ was established in \cite{DisPinSpe2002}.
Delocalization of 1D band matrices under the assumption $W \geq N^{a}$, for various exponents $a > 1/2$, was proved in a series of works \cite{EK_band1, ErdKno2011, erdHos2013local, BaoErd2015, HeMa2018, bourgade2017universality, bourgade2019random, bourgade2020random, yang2021random, DY}. The full conjecture was recently resolved in \cite{YY_25}, which established both delocalization and universality of bulk eigenvalue statistics under the optimal condition $W \gg \sqrt N$. These results were later extended to non-Gaussian random band matrices in \cite{erdHos2025zigzag}. 
On the other hand, localization has been proved under the condition $W \ll N^{a}$ for various exponents $a < 1/2$ in \cite{Sch2009, PelSchShaSod,Cipolloni2024, Chen2022}, and under the sharp condition $W \ll \sqrt N$ in \cite{Localization1_2}. Therefore, the localization-delocalization transition at the critical bandwidth $W = \sqrt N$ is now rigorously established for 1D random band matrices.
In two dimensions, delocalization and bulk universality under the condition $W \ge N^{\varepsilon}$ (for arbitrarily small constant $\varepsilon > 0$) was proved in \cite{DYYY25}.

In this paper, we study band matrices $H$ defined on the $d$-dimensional torus $\mathbb{Z}_{WL}^d=\mathbb{Z}^d / ((WL) \cdot \mathbb{Z}^d)$ with $N=(WL)^d$ and $d \geq 3$. For simplicity, we assume that the torus $\mathbb{Z}_{WL}^d$ is partitioned into $d$-dimensional cubes of side length $W$, indexed by the smaller torus ${\mathbb Z}_{L}^d$. For $x, y\in {\mathbb Z}_{WL}^d$ belonging to blocks labeled by $a, b \in {\mathbb Z}_L^d$, the variance profile is given by
\begin{equation}\label{eq:variance-profile}
	S_{xy} = c_{a-b} W^{-d}  \mathbf{1}(\|a-b\|_1\le 1),
\end{equation}
where $(c_{a-b}\mathbf{1}(\|a-b\|_1\le 1):a, b \in {\mathbb Z}_L^d)$ is a symmetric doubly stochastic matrix. Our main results establish delocalization of bulk eigenvectors and universality of bulk eigenvalue statistics for the band matrix $H$ in dimensions $d \ge 3$, under the assumption $W \ge N^{\varepsilon}$ for any fixed $\varepsilon>0$. In higher dimensions $d \ge 7$, weaker forms of delocalization and bulk universality were previously obtained in~\cite{yang2021delocalization, yang2022delocalization, Xu:2024aa}.

Beyond standard band matrices, the block Anderson model provides an even closer analogue to the original Anderson model~\eqref{Anderson_orig}. This model, inspired by Wegner’s orbital model~\cite{Wegner1, Wegner2, Wegner3}, describes quantum particles with multiple internal degrees of freedom (e.g., orbitals or spin) moving in a disordered medium. Concretely, we consider the matrix $H = \Psi + \lambda V$, where $V=\diag(V_1,\ldots, V_{L^d})$ and the blocks $V_i$ are i.i.d.~$W^d\times W^d$ Gaussian random matrices. The matrix $\Psi$ encodes the hopping between neighboring blocks, and in the block Anderson model, we take $\Psi$ to be the block Laplacian with the diagonal removed.

It was shown in~\cite{PelSchShaSod} that, in all dimensions $d \ge 1$, the block Anderson model exhibits localization with localization length of order $\OO(W)$ whenever $\lambda \gg W^{d/2}$. In this paper, we prove that all results obtained for random band matrices extend to the block Anderson model, provided $c \le \lambda \le W^{d/2-c}$ for a small constant $c > 0$. Combined with the localization results of~\cite{PelSchShaSod}, this establishes the localization--delocalization transition for the block Anderson model in all dimensions $d \ge 3$. In dimensions $d\in\{1,2\}$, delocalization and bulk universality were previously proved in~\cite{RBSO1D}, while partial delocalization results for dimensions $d \ge 7$ were obtained in~\cite{yang2024Del}. 
However, for the physically most relevant case $d=3$, no rigorous results on delocalization were available for either random band matrices or the block Anderson model prior to this work. One reason why $d=3$ is the final dimension in which these conjectures are resolved lies in certain critical difficulties that are intrinsic to three dimensions, as we now explain.

Denote by $G(z)=(H-z)^{-1}$ the resolvent (Green's function) of the random band and block Anderson models, for $z\in \C$. 
It is believed that the size of the resolvent entries $|G_{0x}(z)|^2$ is approximated by the resolvent of a certain random walk on \smash{$\mathbb{Z}_{WL}^d$}, see e.g., \cite{erdos2013delocalization,yang2021delocalization}. 
In dimensions $d\in\{1,2\}$, the random-walk resolvent remains roughly constant (in magnitude, up to a logarithmic correction when $d=2$) below a certain cutoff scale. Hence, an $L^\infty$-bound on $G(z)$, together with control of this cutoff scale, suffices to estimate $G(z)$. 
Starting from $d=3$, however, $|G_{0x}|^2  \asymp |x|^{-d+2}$. This forces us to track the  \emph{pointwise estimate} of $G_{0x}$. 
At the other extreme, in high dimensions, the lattice $\mathbb Z^d$ increasingly resembles a Bethe lattice. Although this analogy is difficult to exploit rigorously, we note that $(|x|^{-d+2})^\alpha$ is integrable whenever $\alpha > d/(d-2) \to 1$ as $d\to \infty$. In other words, $|x|^{-d+2}$ becomes “almost integrable” for large $d$. This observation was crucially used in the proofs for $d\ge 7$ in \cite{yang2021delocalization, yang2022delocalization, Xu:2024aa}.  
In dimension $d=3$, we need to combine the \emph{tree-approximation} method for $d\in\{1,2\}$ \cite{YY_25, DYYY25} with the \emph{diagrammatic expansion} approach for $d\ge 7$ \cite{yang2021delocalization, yang2022delocalization, Xu:2024aa}.

\medskip

To illustrate the above points, we briefly explain the main ideas developed in this paper.
For simplicity, we focus on the random band matrix model with variance profile \eqref{eq:variance-profile}; the analysis for the block Anderson model is analogous, up to some additional technical details. 
We begin by recalling the matrix Brownian motion used in \cite{DY, YY_25, DYYY25}:
\[
\dd (H_{t})_{xy} = \sqrt{S_{xy}} \, \dd (\boldsymbol{B}_{t})_{xy}, \quad \text{with}\quad H_0 = 0,
\]
where \( (\boldsymbol{B}_{t})_{xy} \) are standard independent complex Brownian motions for \( x, y \in  \Z_{WL}^d\), subject to the Hermitian symmetry condition \( (\boldsymbol{B}_{t})_{xy} = \bar{(\boldsymbol{B}_{t})_{xy}} \).  Following \cite{10.1214/19-ECP278, Sooster2019, DY}, we consider the Green's function of \( H_t \) with a carefully chosen time-dependent spectral parameter \( z_t \) (see \eqref{eq:zt} below), defined by
\[
G_t := (H_t - z_t)^{-1},
\]
whose dynamics are naturally renormalized at leading order. 
We then define the $n$-\( G \)-loop observable $\mathcal{L}^{(n)}$ as an $n$-tensor:
\be\label{eq:nGloops}
{\cal L}^{(n)}_{a_1,\ldots, a_n} \equiv {\cal L}^{(n)}_{a_1,\ldots, a_n} (t):= \mathrm{Tr} \prod_{i=1}^n \p{G_t \cdot E_{a_i}}, \quad \text{for}\quad a_i\in \Z_L^d, \ \  1\le i\le n,
\ee
where each $E_a$ is a block-averaging matrix, defined by $(E_a)_{xy} = W^{-d}\delta_{xy}$ when $x, y$ belong to the $a$-th block, and $(E_a)_{xy}=0$ otherwise. In the proof, the $G$-loops considered in this paper involve combinations of $G_t$ and its Hermitian conjugate $G^\ast_t$; however, we omit this detail in the following heuristic discussion for simplicity.

The \( G \)-loops ${\cal L}^{(n)}(t)$ defined above satisfy a system of evolution equations known as the \emph{loop hierarchy}; see equation \eqref{eq:mainStoflow} below for its precise formulation. The dynamics of \( n \)-loops depend on \((n+1)\)-loops and a martingale term, whose quadratic variation involves \( (2n+2) \)-loops. Schematically, the loop hierarchy takes the form
\be\label{eq:introloop_hier}
\dd {\cal L}^{(n)} = ( {\cal L}  \ast {\cal L})^{(n)}\dd t + {\cal E}^{\Gc,(n)}\dd t + \dd {\cal E}^{M,(n)}, 
\ee
where each term on the right-hand side has the following interpretation.
\begin{itemize}
	\item The \emph{light-weight} term \( {\cal E}^{\Gc,(n)} \) depends on \( (n+1) \)-loops, typically of the form
	\be\label{eq:introLWtermS}
	W^d \sum_{u,v\in \Zn} S^{(\sB)}_{uv} \cdot \tr[\Gc_t E_u]\cdot {\cal L}^{(n+1)}_{a_1, \ldots, a_n, v}, \quad \text{with}\quad \Gc_t: = G_t-mI_N.
	\ee
	Here, $\tr[\Gc_t E_u]$ is referred to as a \emph{light-weight}, \smash{$S^{(\sB)}_{uv} = c_{u-v} \mathbf{1}(\|u-v\|_1\le 1)$} denotes the block variance profile corresponding to \eqref{eq:variance-profile}, and $m \equiv m_{\rm sc}$ is the Stieltjes transform of the semicircle law (see \eqref{eq:defmzsc}).
	
	\item The \emph{martingale} term \( \dd {\cal E}^{M,(n)} \) has a quadratic variation depending on  \( (2n+2) \)-loops.
	
	\item The \emph{convolution} term \( ({\cal L} \ast {\cal L})^{(n)} \) consists of sums of the form \( \mathcal{L}^{(p)} \diamond \mathcal{L}^{(q)} \) with \( p+q = n+2\) and \(p,q \ge 2 \), where
	\be\label{eq:diamond}
	({\cal L}^{(p)} \diamond {\cal L}^{(q)})_{a_1, \ldots, a_n} := W^d \sum_{u,v} S^{(\sB)}_{uv} \cdot {\cal L}^{(p)}_{a_1, \ldots, a_{p-1}, u} \cdot {\cal L}^{(q)}_{a_p, \ldots, a_n, v}.
	\ee
\end{itemize}
Although the formulation above is heuristic, it captures the essential structure of the loop hierarchy.

The loop equation \eqref{eq:introloop_hier} can be generalized to random band matrices without block structure. This issue has already been addressed in earlier works; see \cite{erdHos2025zigzag} for $d=1$ and \cite{Block_reduction} for dimensions $d\in\{1,2\}$. In fact, our method can be automatically extended to the case where the variance matrix takes the following form:
\begin{align*} 
	S_{xy} = \sum_{u,v \in \ZL} \chi(x - u) \wt S_{uv} \chi(v - y),
\end{align*} 	
where $\wt S$ is a symmetric doubly stochastic matrix and \smash{$\chi : \ZL \to [0,\infty)$} is a mollifier at scale $W^{1 - \e}$. In this case, let $F_u$ denote the diagonal mollifier matrices with entries $(F_{u})_{xy}= \delta_{xy} \chi(x - u)$  for $u, x, y \in \ZL.$ Then, for \smash{$(x_1, \ldots, x_\fn) \in (\ZL)^n$}, we define the $n$-$G$-loops \smash{$\cL^{(n)}_{x_1,\ldots, x_n}(t)$} in analogy with \eqref{eq:nGloops}, replacing the matrices $E_{a_i}$ with the mollifier matrices $F_{x_i}$. The dynamics of these loops satisfy equations of the same form as \eqref{eq:introloop_hier}, with \smash{$S^{\LK}$} in \eqref{eq:introLWtermS} and \eqref{eq:diamond} replaced by \smash{$\wt S$}. Consequently, all  our arguments extend to this setting with only minor notational modifications.
To streamline the presentation, we continue to work with the simpler block structure in \eqref{eq:variance-profile}, which allows us to focus on the core technical challenges in dimensions $d\ge 3$.

A key observation from \cite{YY_25, DYYY25} is that both the light-weight and martingale terms are small errors. Consequently, \( \mathcal{L}^{(n)} \) can be approximated by \( \mathcal{K}^{(n)} \), the solution of the deterministic \emph {convolution tree  equation}: 
\be\label{eq:Knequation}
\dd \mathcal{K}^{(n)} = (\mathcal{K} \ast \mathcal{K})^{(n)}\dd t.
\ee
For any fixed \( n \), the right-hand side of equation \eqref{eq:Knequation} depends only on \( \mathcal{K}^{(k)} \) with \( k \leq n \), forming a \emph{closed system} of equations for \( (\mathcal{K}^{(2)}, \ldots, \mathcal{K}^{(n)}) \) that can be solved inductively and admits explicit solutions, as discovered in \cite{YY_25}. Since these solutions have an explicit tree representation, we refer to \( \mathcal{K}^{(n)} \) as the \emph {tree approximation} to \( \mathcal{L}^{(n)} \). 
In particular, for \( n = 2 \), equation \eqref{eq:Knequation} has the solution \( \mathcal{K}^{(2)}(t) = W^{-d} \Theta_t\), where $\Theta_t$ denotes the propagator defined in Definition \ref{def_Theta}. 
This propagator can carry different charges. For simplicity, in the discussion below, we use the notation $\Theta_t$ exclusively for the most relevant $(+,-)$-charged propagator, denoted \smash{$\Theta_t^{(+,-)}$}. It can be viewed as the resolvent of a Laplacian on the torus $\Zn$,  describing the classical diffusion of a random walk. 
This propagator serves as the basis for establishing quantum diffusion, delocalization, and the universality of eigenvalue statistics---provided the approximation \( \mathcal{L}^{(2)} \approx \mathcal{K}^{(2)} \) can be justified in a sufficiently strong sense.

To rigorously justify the tree approximation, three main ingredients are required:
\begin{enumerate}
	\item[(1)] A strategy to manage the dependence of \( {\cal E}^{\Gc,(n)} \) and \( \dd {\cal E}^{M,(n)} \) on higher-order loops \( {\cal L}^{(k)} \) with \( k > n \).
	\item[(2)]  A stability theory for the perturbed equation \( \dd {\cal L}= ({\cal L} \ast {\cal L})\dd t + \text{errors} \).
	\item[(3)] Explicit quantitative bounds showing that  \( {\cal E}^{\Gc,(n)} \) and \( \dd {\cal E}^{M,(n)} \) are indeed small.
\end{enumerate} 
In the absence of additional structure in the hierarchy, issue (1) reflects a classical obstacle in many-body dynamics, reminiscent of the well-known BBGKY hierarchy. 
In our setting, however, a key observation is that higher-order loops appear only in the error terms. This feature enables the development of a general inductive and self-improving framework, first introduced in \cite{YY_25}, based on a bootstrap argument that applies across all dimensions. It should be noted that the smallness of the martingale and light-weight terms is a highly nontrivial fact. 
Without the correct perspective, these terms may seem much larger than the leading contribution.
Indeed, in the early development of the theory of random band matrices, researchers often encountered the seemingly impenetrable issue that the error terms in the so-called \emph{$T$-equation} diverge. While we do not elaborate on the historical developments here, we will demonstrate below that the light-weight term is, in fact, a genuinely small error.

For both issues (2) and (3), we need a suitable norm to control the error terms and develop a stability theory for the loop hierarchy.  
In dimensions \( d \in \{1,2\} \), the natural choice is the \( L^\infty \)-norm, justified by the fact that \( \Theta_t \) is approximately constant up to a cutoff.  
Surprisingly, the loop hierarchy remains stable in dimensions \( d \geq 3 \) under this norm. In dimensions $d\ge 3$, the propagator satisfies
\[
\Theta_{t,ab} \lesssim (|a-b|+1)^{-d+2} e^{- c|a-b|/\ell_t }
\]
for some constant \( c>0 \), where \( \ell_t  = \min(|1-t|^{-1/2},L)\). The tree approximation \( \mathcal{K}^{(n)}_{a_1, \ldots, a_n} \) is a complicated function of the propagator entries. Hence, the \( L^\infty \)-norm is a very crude norm for these functions, and one would not a priori expect stability in this norm.   
For clarity, in the following heuristic discussion, we assume \( \im z_t \ge L^{-2} \), so that \( \ell_t^{-2} = |1-t| \asymp \im z_t \). The key new ingredient enabling \( L^\infty \)-stability of the loop hierarchy in dimensions \( d \ge 3 \) is a class of \emph{$\contract$ inequalities} (Lemma \ref{ygdhmsgq}), which extend the classical Ward's identity:
\be\label{eq:Ward_intro1}
\sum_y |G_{xy}(z)|^2=\sum_y |G_{yx}(z)|^2 =  \im G_{xx}/ \im z ,
\ee
and apply to partial sums involving the absolute values of the $n$-\(G\)-loops.

However, the \( L^\infty \)-stability of the tree approximation remains incomplete without a sufficiently precise \emph{pointwise estimate} \smash{${\cal L}^{(2)}_{ab} \approx \mathcal{K}^{(2)}_{ab}$} for the 2-$G$-loops, that is,  
\be\label{eq:n=2approx}
\left|\left({\cal L}-\cK\right)^{(2)}_{ab} \right| \ll W^{-d} \Theta_{t,ab}, \quad  \forall a,b\in \Zn.
\ee
Such a strong pointwise control is essential because of the leading term $\cal L^{(2)}\diamond \cal L^{(n)}$ in the loop hierarchy \eqref{eq:introloop_hier}. Subtracting the corresponding term $\cal K^{(2)}\diamond \cal K^{(n)}$ from the convolution tree equation \eqref{eq:Knequation} produces two terms $\cal K^{(2)}\diamond (\cal L-\cal K)^{(n)}$ and $(\cal L-\cal K)^{(2)}\diamond (\cal L-\cal K)^{(n)}$. While it is intuitively clear that the latter term is negligible compared to the former, a rigorous justification requires establishing the stability estimate \eqref{eq:n=2approx}. 
This requires a significantly more delicate analysis of the equation \eqref{eq:introloop_hier} for $n=2$.  
One main difficulty is that this equation involves \( \mathcal{L}^{(n)} \) with \( n \geq 3 \), for which only \( L^\infty \)-bounds are available. Therefore, we need to develop a method to estimate the error terms in the 2-loop equation using only the \( L^\infty \)-bounds on the higher-order loops. 
We will discuss the details of the $L^\infty$-stability and $\contract$ inequalities in \Cref{Sec:Steps34}, and the proof of the pointwise stability estimate \eqref{eq:n=2approx} in \Cref{subsec:step2_pf}.

Issue (3) involves bounding the martingale term \( \dd {\cal E}^{M,(n)} \) and the light-weight term \( {\cal E}^{\Gc,(n)} \). In dimensions \( d \ge 3 \), estimating these terms presents a serious challenge. We will discuss the difficulties related to bounding the martingale term in \Cref{subsec:pf_lem: EMn2_N}. For now, we focus on estimating the light-weight term for 
\( n=2 \): 
\[
{\cal E}^{\Gc,(2)}_{ab} = W^d \sum_{u, v} S^{(\sB)}_{uv} \cdot \tr[\Gc_t E_{u}] \cdot \mathcal{L}^{(3)}_{a,b,v}, \quad\text{with}\quad 
{\cal L}^{(3)}_{a,b,v} = W^{-3d} \sum_{x \in [a], y \in [b], z \in [v]} (G_t)_{xy} (G_t)_{yz} (G_t)_{zx},
\]
where \( x \in [a] \) means that the lattice point \( x \) belongs to the block labeled by \( a \).  
We expect an averaged local law to hold in the form \smash{\( |\tr[\Gc_t E_{u}]| \leq W^{-c} \)} for some constant \( c > 0 \).  
Next, assuming the \emph{sharp entrywise local law} for $G_t$: 
\be\label{eq:aprioriG2} 
|(G_t)_{xy}- m \delta_{xy}| \lesssim W^{-d/2}
\big({\Theta_{t,ab}}\big)^{1/2}, \quad \forall x\in [a], \ y \in[b],
\ee 
we obtain the estimate  
\begin{align} 
	|{\cal E}^{\Gc,(2)}_{ab}|  
	&\lesssim W^{-c-d/2} \sum_{|v-a|\vee |v-b| \lesssim\ell_t} \p{\Theta_{t,ab} \Theta_{t,bv} \Theta_{t,va}}^{1/2} \nonumber  \\
	& \lesssim W^{-c-d/2} \big(\Theta_{t,ab}\big)^{1/2} \ell_t^2 = W^{-c-d/2} \big(\Theta_{t,ab}\big)^{1/2}\cdot (1 - t)^{-1}. 
	\label{eq:introGc}
\end{align}
Upon integrating in time, the singularity \( (1 - t)^{-1} \) produces only a harmless logarithmic factor. However, the resulting bound contains only \smash{\( (\Theta_{t,ab})^{1/2} \)} rather than the optimal \smash{\( \Theta_{t,ab} \)}. This gap cannot be compensated by the \( W^{-c} \) factor from the averaged local law, as \( c \) is at best \( c=d \).  
Without the extra factor \smash{\( (\Theta_{t,ab})^{1/2} \)}, the light-weight term would dominate the leading term and violate the approximation in \eqref{eq:n=2approx}. In the actual proof, the situation is even more delicate, since the estimate \eqref{eq:aprioriG2} is not available a priori.

To overcome this difficulty, we employ the \emph{diagrammatic methods} developed in \cite{yang2021delocalization, yang2022delocalization, yang2021random}.  
While this method is technically involved and was previously applied only in high dimensions \( d \geq 7 \), our objective here is more modest: we aim to generate an additional “\emph{long edge}”---that is, a resolvent entry \( G_{x'y'} \) with \( |x'-y'| \gtrsim W|a-b| \)---in order to recover the missing factor \smash{\( (\Theta_{t,ab})^{1/2} \)}.  
For this purpose, the diagrammatic expansions can be extended to all dimensions \( d \geq 3 \). 

Specifically, we apply Gaussian integration by parts to express the expectation of high moments \smash{\( \E |{\cal E}_{ab}^{\Gc,(2)}|^{2p}\)} as a sum of graphs, where each graph either contains \emph{\( 2p \) additional long edges} or contributes to a sufficiently small error.   
Roughly speaking, each application of Gaussian integration by parts produces a \( W^{-d/2} \) factor, unless it generates a long edge.  
To ensure that these \( W^{-d/2} \) factors compensate for the missing \smash{\( (\Theta_{t,ab})^{1/2} \)} when \( |a-b| \asymp L \), we must expand the graphs to a sufficiently high order $k$, chosen so that \( W^{-kd/2} \le (\Theta_{t,ab})^{p} \). 
As a result, the contributing graphs are extremely complicated. To bound them, we need to sum over their internal vertices in such a way that the final contribution retains at least the \( 2p \) additional long edges. 
We will use Ward’s identity \eqref{eq:Ward_intro1} as a key tool to control the summation over internal vertices, provided the order of the summations follows a carefully designed “nested order”. This nested ordering is a crucial component in our estimates of the graphs arising from the complicated diagrammatic expansions. See \Cref{sec:graphs_ideas} for a detailed discussion. 

In addition to the above considerations, two key ingredients continue to play a central role in our analysis: the \emph{sum-zero property} \cite{yang2021delocalization, yang2022delocalization,YY_25}, and the CLT mechanism for the fluctuation cancellation of resolvents, which was previously employed in the \( d = 2 \) analysis \cite{DYYY25}. Our final objective is to incorporate all these components—diagrammatic tools, the sum-zero property, fluctuation cancellation, and the stability of the tree approximation—into the flow-based framework.  

In summary, our proof relies on a flow-based analysis of the tree approximation to the $G$-loop hierarchy, combining \emph{sharp max-norm estimates} for higher-order $(\cL-\cK)^{(n)}$-loops with $n\ge 3$ and a \emph{precise pointwise estimate} for the $(\cL-\cK)^{(2)}$-loops.
It is perhaps surprising that this seemingly ``simple" strategy succeeds without tracking the precise pointwise decay of higher-order $G$-loops. This is made possible by several key technical inputs, including \emph{$\contract$ inequalities} and the \emph{diagrammatic techniques}.
The main advantage of our approach is that it closes the analysis using only a minimal set of technical ingredients, thereby avoiding an unnecessarily complicated treatment in dimensions $d\ge 3$. A complete implementation of this strategy will be presented in the subsequent sections after we state the main results in \Cref{sec_model-results}.

\medskip
{\noindent \bf Notations.}
To facilitate the presentation, we introduce some necessary notations that will be used throughout this paper. We will use the set of natural numbers $\N=\{1,2,3,\ldots\}$ and the upper half complex plane $\C_+:=\{z\in \C:\im z>0\}$.  
In this paper, we are interested in the asymptotic regime with $N\to \infty$. When we refer to a constant, it will not depend on $N$ or $W$. Unless otherwise noted, we will use $C$, $D$ etc.~to denote large positive constants, whose values may change from line to line. Similarly, we will use $\e$, $\delta$, $\tau$, $c$, $\fc$, $\fd$ etc.~to denote small positive constants. 
For any two (possibly complex) sequences $\xi_N$ and $\zeta_N$ depending on $N$, $\xi_N = \OO(\zeta_N)$, $\zeta_N=\Omega(\xi_N)$, or $\xi_N \lesssim \zeta_N$ means that $|\xi_N| \le C|\zeta_N|$ for some constant $C>0$, whereas $\xi_N=\oo(\zeta_N)$ or $|\xi_N|\ll |\zeta_N|$ means that $|\xi_N| /|\zeta_N| \to 0$ as $N\to \infty$. We say that $\xi_N \asymp \zeta_N$ if $\xi_N = \OO(\zeta_N)$ and $\zeta_N = \OO(\xi_N)$. For any $\al,\beta\in\R$, we denote $\llbracket \al, \beta\rrbracket: = [\al,\beta]\cap {\mathbb Z}$, $\qqq{\al}:=\qqq{1,\al}$, $\al\vee \beta:=\max\{\al, \beta\}$, and $\al\wedge \beta:=\min\{\al, \beta\}$. 
Given a vector $\mathbf v$, $|\mathbf v|\equiv \|\mathbf v\|_2$ denotes the Euclidean norm and $\|\mathbf v\|_p$ denotes the $L^p$-norm. 
Given a matrix $\cal A = (\cal A_{ij})$, $\|\cal A\|$, $\|\cal A\|_{p\to p}$, and $\|\cal A\|_{\infty}\equiv \|\cal A\|_{\max}:=\max_{i,j}|\cal A_{ij}|$ denote the operator (i.e., $L^2\to L^2$) norm,  $L^p\to L^p$ norm (where we allow $p=\infty$), and maximum (i.e., $L^\infty$) norm, respectively. We will use $\cal A_{ij}$ and $ \cal A(i,j)$ interchangeably in this paper. 
We will use $I_n$ to denote an $n\times n$ identity matrix.

Given an event $\Xi$, let $\mathbf 1_\Xi$ or $\mathbf 1(\Xi)$ denote its indicator function. We will say an event $\Xi$ holds with high probability (w.h.p.) if for any constant $D>0$, $\mathbb P(\Xi)\ge 1- N^{-D}$ for large enough $N$. More generally, we say an event $\Omega$ holds $w.h.p.$ in $\Xi$ if for any constant $D>0$, $\P( \Xi\setminus \Omega)\le N^{-D}$ for large enough $N$. 
For clarity of presentation, we will use the following notion of stochastic domination introduced in \cite{Average_fluc}. 
Let \[\xi=\left(\xi^{(N)}(u):N\in\mathbb N, u\in U^{(N)}\right),\hskip 10pt \zeta=\left(\zeta^{(N)}(u):N\in\mathbb N, u\in U^{(N)}\right),\]
be two families of non-negative random variables, where $U^{(N)}$ is a possibly $N$-dependent parameter set. We say $\xi$ is stochastically dominated by $\zeta$, uniformly in $u$, if for any fixed (small) $\tau>0$ and (large) $D>0$, 
\be\label{stoch_domination}\mathbb P\bigg(\bigcup_{u\in U^{(N)}}\left\{\xi^{(N)}(u)>N^\tau\zeta^{(N)}(u)\right\}\bigg)\le N^{-D}\ee
for large enough $N\ge N_0(\tau, D)$, and we will use the notation $\xi\prec\zeta$. 
If for some complex family $\xi$ we have $|\xi|\prec\zeta$, then we will also write $\xi \prec \zeta$ or $\xi=\OO_\prec(\zeta)$. 
As a convention, for two \emph{deterministic} non-negative quantities $\xi$ and $\zeta$, we will write $\xi\prec\zeta$ if and only if $\xi\le N^\tau \zeta$ for any constant $\tau>0$.

\medskip
\noindent{\bf Acknowledgement.}   
We would like to thank Kevin Yang for fruitful discussions.

\section{The model and main results}\label{sec_model-results}

We focus in this paper on the classical block random band matrix model (i.e., the \emph{Wegner orbital model}) and the \emph{block Anderson model} with identity interactions. As noted in the introduction, these are the models studied in \cite{PelSchShaSod}, where localization of bulk eigenvectors was established under the condition $\lambda \gg W^{d/2}$. In contrast, the present work establishes a complementary result describing the delocalized regime of the Wegner orbital and block Anderson models when $\lambda \ll W^{d/2}$. 
Our approach can be readily extended to a much broader class of random band matrices and block Anderson models with general (non-identity) interactions---by combining the techniques developed here with those from \cite{RBSO1D} and the block reduction method of \cite{Block_reduction}. However, the models considered in this paper allow us to avoid inessential technical complications and to focus on the core challenges that arise in dimensions $d \ge 3$. Extensions to more general settings will be addressed in future work.
For clarity of presentation, we rescale our models by 
\be\label{def:ilambda}\ilambda=\lambda^{-1},\ee
and consider matrices of the form $H=V+\ilambda\Psi$. 
This rescaling has the advantage that, in the regime $\ilambda\ll 1$ (i.e., $\lambda\gg 1$), the density of states of $H$ may be viewed as a small perturbation of the semicircle law. 


Our models are defined on a $d$-dimensional discrete torus \(\ZL\subset\mathbb{Z}^d\) with $d\ge 3$, consisting of $N$ lattice points and side length $WL$:
\[\ZL:=\qll{ -(WL)/2+1 , (WL)/2}^d\quad \text{with}\quad N=(W\cdot L)^d.\] 
We partition $\ZL$ into $L^d$ disjoint blocks, indexed by \(\Zn:=\qll{ -L/2+1 , L/2}^d,\)
which we refer to as the \emph{block lattice}.  Each index $a=(a(1),\ldots, a(d))\in \Zn$ corresponds to a block 
\be\label{eq:blockIa}
[a] := \prod_{i=1}^d\left\llbracket (a(i)-1)W + 1, \; a(i)W \right\rrbracket, 
\ee
where $a(i)$ denotes the $i$-th coordinate of $a$ for $i\in\qqq{d}$.\footnote{In the above definitions, we implicitly assume that $L$ is even; the case of odd $L$ can be treated similarly by defining the block lattice as $ \Zn:=\qll{ -(L-1)/2 , (L-1)/2}^d$.}
We view both $\ZL$ and $\Zn$ as discrete $d$-dimension tori. We will denote the vertices of $\ZL$ by $x,y,\ldots,$ and denote those of $\Zn$ by $a,b,\ldots$ Given $x,y\in \ZL$ and $a,b\in \Zn$, we denote the periodic representatives of $x-y$ and $a-b$ by
$\rep{x-y}_{WL}$ and $\rep{a-b}_L$, respectively:  
\be\label{representativeL}\rep{x-y}_{WL}:= \left((x-y)+(WL)\Z^d\right)\cap \ZL,\quad \rep{a-b}_L:= \left((a-b)+L\Z^d\right)\cap \Z_L^d.\ee
For definiteness, we use the $L^\infty$-metric to define (periodic) distances on $\ZL$ and $\Zn$: 
\[ |x-y|\equiv \|\rep{x-y}_{WL}\|_{\infty},\quad \forall x,y\in \ZL,\quad \text{and}\quad |a-b|\equiv \|\rep{a-b}_L\|_{\infty},\quad \forall a,b\in \Zn.\]
We write $x\sim y$ if $x$ and $y$ are neighbors in $\ZL$, and similarly $a\sim b$ if $a$ and $b$ are neighbors in $\Zn$.

	\subsection{Main results for the random band matrix model}\label{subsec:main}
	Our \emph{random band matrix} (or referred to as the \emph{Wegner orbital model} following \cite{PelSchShaSod}) is defined by a complex Hermitian random block Hamiltonian \( H=(H_{xy}:x,y\in \ZL) \), where the entries $H_{xy}$ are independent (up to the Hermitian symmetry $H_{xy}=\overline{ H}_{yx}$) Gaussian random variables. More precisely, given a symmetric doubly stochastic variance matrix $S=(S_{xy}:x,y\in \ZL)$, the diagonal entries of $H$ are real Gaussian random variables, and the off-diagonal entries are complex Gaussian random variables, distributed as follows: 
	\be\label{bandcw0}
	H_{xy}\sim \mathcal{N}_{\R}(0, S_{xy}) \cdot \mathbf 1_{x=y} + \mathcal{N}_{\C}(0, S_{xy}) \cdot \mathbf 1_{x\ne y}.
	\ee
	For $x\in \cal [a]$ and $y\in \cal [b]$, the variance matrix $S\equiv S(\ilambda)$ is given by 
	\be\label{eq:variancematrix}
	S_{xy}\equiv\var (H_{xy}) := W^{-d}S^{(\sB)}_{ab}(\ilambda),\quad \text{with}\quad 
	S^{(\sB)}_{ab}(\ilambda):=\frac{1}{1+2d\ilambda^2} \mathbf 1_{a=b} + \frac{\ilambda^2}{1+2d\ilambda^2} \mathbf 1_{a\sim b},
	\ee
	where $\ilambda>0$ (recall \eqref{def:ilambda}) is a coupling parameter that quantifies the interaction strength between neighboring blocks, and \( S^{(\sB)}(\ilambda) \) denotes an \( L ^d\times L^d \) matrix. 
	Informally, the matrix $H$ consists of i.i.d.~GUE blocks on the diagonal and i.i.d.~Ginibre blocks (up to Hermitian symmetry and scaling) on the off-diagonal. 
	Note that when $\ilambda=1$, the present model reduces to the standard block random band matrices studied in \cite{YY_25, DYYY25}.

	Let \( \lambda_1 \leq \lambda_2 \leq \dots \leq \lambda_N \) denote the eigenvalues of \( H \). The corresponding normalized eigenvectors of $H$ are denote by \( \left( \boldsymbol{\psi}_k \right)_{k=1}^N \). 
	It is well-known that the empirical spectral measure $N^{-1}\sum_{k=1}^N \delta_{\lambda_k}$ converges almost surely to the Wigner semicircle law \cite{Wigner} with density 
	\(
	\rho_{\mathrm{sc}}(x) = \sqrt{(4 - x^2)_+}/{2\pi}.
	\)
	Define the Green's function (or resolvent) of the Hamiltonian $H$ as
	\be\label{def_Green}
	G(z):=(H-z)^{-1} ,  \quad z\in \C_+.
	\ee
	It is also known that as $N\to \infty$, $G(z)$ converges to the scalar matrix $m(z)I_N$ entrywise, where $m(z)$ denotes the Stieltjes transform of \( \rho_{\mathrm{sc}} \), defined by
	\be\label{eq:defmzsc}
	m(z)\equiv m_{\mathrm{sc}}(z):=\frac{-z+\sqrt{z^2-4}}{2}= \int_{\mathbb{R}} \frac{\rho_{\mathrm{sc}}(x)}{x - z} \dd x.
	\ee
	Moreover, we define the matrix\footnote{We introduce this notation to maintain consistency with the block Anderson model, where $M(z)$ is no longer a scalar matrix proportional to $m(z)$; see \eqref{self_m} and \eqref{def_G0}.}
	\be\label{eq:defMzsc} 
	M(z)\equiv M_N(z):=m(z) I_N. 
	\ee

	We now state the main results for the random band matrix model. 
	Our first main result establishes the delocalization of the bulk eigenvectors of $H$ in dimensions $d\ge 3$.  
	
	\begin{theorem}[Delocalizaiton]\label{MR:decol}
		Fix any dimension $d\ge 3$, and consider the random band matrix model defined above. Assume there exist constants $\fc,\fd>0$ such that 
		\begin{equation}\label{Main_DEL_COND}
			W \ge N^{ \mathfrak c} ,
		\end{equation}
		and that $\ilambda$ satisfies the condition 
		\be\label{eq:WO} W^{-d/2+\fd}\le \ilambda \le \fd^{-1}.\ee
		Then, for any small constants $\kappa,\tau>0$ and large constant $D>0$, the following delocalization estimate holds, provided $N$ is sufficiently large:
		\be\label{eq:psikLinfty}
		\mathbb{P}\left( \max_{k:|\lambda_k|\le 2-\kappa}\left\|\boldsymbol{\psi}_k\right\|_{\infty}^2 \le N^{-1+\tau}  \right) \ge 1-N^{-D}.
		\ee
	\end{theorem}
	
	Since the eigenvectors $\boldsymbol{\psi}_k$ are $L^2$-normalized, the $L^\infty$-bound in \eqref{eq:psikLinfty} implies that bulk eigenvectors have localization length at least $\Omega((WL)^{1-\e})$, for any small constant $\e>0$. The upper bound in condition \eqref{eq:WO} is not essential; we include it only for clarity of presentation, as our main interest is the small $\ilambda$ regime.\footnote{All results below remain valid for $\ilambda \gg 1$, provided $\ilambda$ in the relevant equations is replaced by $\ilambda \wedge 1$.} 
	The key aspect of condition \eqref{eq:WO} is the lower bound $\ilambda\ge W^{-d/2+\fd}$. Indeed, by the fractional moment method \cite{Aizenman1993}, it was shown in \cite{PelSchShaSod} that when $\ilambda\ll W^{-d/2}$, the eigenvectors of both the random band matrix model and the block Anderson model (defined in \Cref{sec:main_BA} below) are localized with localization length of order $\OO(W)$. Combined with this result, \Cref{MR:decol} demonstrates a localization–delocalization transition for the random band matrix model in the bulk spectrum as $\ilambda$ crosses the critical threshold $W^{-d/2}$.

	\Cref{MR:decol} follows immediately from the following sharp local law for the Green's function of $H$, defined as in \eqref{def_Green}. For simplicity of notation, given constants $\kappa,\e>0$, define the spectral domain 
	\be\label{eq:spectral_domain}
	\mathbf D_{\kappa,\e}:=\{z=\hat{E}+\ii\eta \in \C_+: |\hat{E}|\le 2-\kappa, N^{-1+\e}\le \eta\le 1\},
	\ee
	and the simplified notation 
	\be\label{eq:calBetaK}\cal B_{\eta,K}:= \frac{(\ilambda^{2}+\eta)^{-1}}{W^2(K+W)^{d-2}} + \frac{1}{N\eta},\quad \forall K\ge 0.\ee
	
	\begin{theorem}[Local semicircle law]\label{MR:locSC} 
		In the setting of Theorem \ref{MR:decol}, for any (small) constants $\kappa,\e, \tau>0$ and (large) constant $D>0$, the following events hold with probability $\ge 1-N^{-D}$ for large enough $N$: 
		\begin{align}\label{G_bound}
			&\bigcap_{z=\hat{E}+\ii\eta\in \mathbf D_{\kappa,\e}}\bigcap_{x,y\in \ZL} \left\{ |G_{xy}(z) - M_{xy}(z)|^2  \le W^\tau  \mathcal B_{\eta,|x-y|}\right\}\, , 
			\\
			\label{G_bound_ave}
			&\bigcap_{z=\hat{E}+\ii\eta\in \mathbf D_{\kappa,\e}}\bigg\{\max_{a\in\Zn} \Big|W^{-d}\sum_{x\in [a]}G_{xx}(z) - m(z) \Big| \le W^\tau  \mathcal B_{\eta,0}\bigg\}\, , 
		\end{align}
		where $G$ is defined in \eqref{def_Green}, $m(z)$ and $M(z)$ are defined in \eqref{eq:defmzsc} and \eqref{eq:defMzsc}, and $[a]$ is defined in \eqref{eq:blockIa}. 
	\end{theorem}
	
	\begin{proof}[\bf Proof of Theorems \ref{MR:decol}]
		Theorem \ref{MR:decol} follows directly from the entrywise local law \eqref{G_bound} via the bound 
		\be\label{eq:ukx}|\boldsymbol{\psi}_k(x)|^2 \le \eta\im G_{xx}(\lambda_k + \ii \eta),\quad \forall \eta>0\, .\ee 
		Applying the local law \eqref{G_bound} to $G_{xx}(\lambda_k+\ii\eta)$ with $\eta = N^{-1+\tau}$, we deduce that $\im G_{xx}(\lambda_k+\ii\eta)=\OO(1)$ with high probability. Combined with \eqref{eq:ukx}, this completes the proof.
	\end{proof}
	
	Under the assumptions of \Cref{MR:decol}, we can further establish a stronger \emph{quantum unique ergodicity} (QUE) estimate for the bulk eigenvectors of $H$, albeit at the cost of a slightly weaker probability bound. 
	Roughly speaking, the QUE estimates \eqref{Meq:QUE} and \eqref{Meq:QUE2} below indicate that every bulk eigenvector of $H$ is asymptotically uniformly distributed (in the sense of $L^2$-mass) across all scales larger than $W$. In particular, this implies that the localization length of every bulk eigenvector is of order $\Omega(L)$ with probability $1-\oo(1)$.
	
	\begin{theorem}[Quantum unique ergodicity]\label{MR:QUE} 
		In the setting of \Cref{MR:decol}, given $E\in [-2+\kappa,2-\kappa]$ and a constant $\e_0\in (0,\fd/2)$, define the interval  
		\be\label{tau} {\cal I}_E\equiv {\cal I}_E(\e_0):=\left\{x: |x-E|\le W^{-\e_0}\big({\ilambda W^{d/2}}/N\big)\right\}. 
		\ee
		For each $d\ge 3$ and constant $0<c<\e_0\wedge(\fd/5)$, the following estimate holds for any small constant $\tau>0$:
		\begin{equation}\label{Meq:QUE}
			\sup_{E: |E|\le 2-\kappa}  
			\max_{ a\in \Zn} \P\bigg(\max_{i,j:\lambda_i, \lambda_j \in {\cal I}_E} 
			\bigg| \sum_{x\in[a]}\overline \bpsi_i(x)\bpsi_j(x)-\frac{W^{d}}{N}\delta_{ij}  \bigg| \ge \frac{W^{d-c}}{N} \bigg) \le  W^{-(2\e_0)\wedge(2\fd/5)+2c+\tau} \, 
		\end{equation}
		provided $N$ is large enough. More generally, for any subset $A\subset \Zn$, we have
		\begin{equation}\label{Meq:QUE2}
			\sup_{E: |E|\le 2-\kappa} \mathbb{P}\bigg(\max_{k: \lambda_k\in {\cal I}_E} 
			\bigg|\sum_{a\in A}\sum_{x\in [a]}\left|\bu_k(x)\right |^2 -\frac{W^d}{N}|A|\bigg| \ge  \frac{W^{d-c}}{N}|A| \bigg) \le  W^{-(2\e_0)\wedge(2\fd/5)+2c+\tau} .
		\end{equation}  
	\end{theorem}
	
	As an important consequence of the above QUE estimates, and by employing the Green's function comparison argument developed in \cite{Xu:2024aa}, we can derive the following \emph{universality of the local bulk eigenvalue statistics} for our random band matrices. 
	Let $p_{H}(\lambda_1,\ldots, \lambda_{N})$ denote the joint symmetrized probability density of the (unordered) eigenvalues of $H$. For any $1\le n \le N$, define the $n$-point correlation function as
	$$
	p_{H}^{(n)}\left(\lambda_1, \cdots, \lambda_n\right)
	:=\int_{\R^{N-n}} p_{H}\left(\lambda_1, \cdots, \lambda_{N}\right) \mathrm{d} \lambda_{n+1} \cdots \mathrm{d} \lambda_{N}.
	$$
	Moreover, let $p_{\rm{GUE}}^{(n)}$ denote the corresponding $n$-point correlation function for an $N\times N$ GUE matrix.
	
	\begin{theorem}[Bulk universality]\label{Thm: B_Univ}
		In the setting of \Cref{MR:decol}, let $O \in C_c^{\infty}\left(\mathbb{R}^n\right)$ be an arbitrary smooth, compactly supported function. Then, for any $|E|\le 2-\kappa$ and fixed $n\in \N$, we have
		\begin{equation}\label{eq:universality}
			\lim_{N\to \infty}	\int_{\mathbb{R}^n} \mathrm{~d} \boldsymbol{\alpha}\; O(\boldsymbol{\alpha}) \left[p_{H}^{(n)}-p_{\rm{GUE}}^{(n)}\right]\left(E+\frac{\alpha_1}{N  }, \ldots, E+\frac{\alpha_n}{N}\right)   =0, 
		\end{equation} 
		where $\bal$ denotes $\boldsymbol{\alpha}=\left(\alpha_1, \ldots, \alpha_n\right)$.  
	\end{theorem}

	This bulk universality result shows that, for our random band matrices, the local eigenvalue gap statistics near any fixed bulk energy level $E$ asymptotically coincide with those of Wigner matrices. However, we note that the distribution of an individual bulk eigenvalue $\lambda_k$ of $H$ may differ significantly from that of Wigner matrices, as $\lambda_k$ can exhibit fluctuations that are much larger than $N^{-1}$.

	Similar to the one- and two-dimensional cases \cite{YY_25,DYYY25,erdHos2025zigzag,RBSO1D}, our random band matrix model in dimensions $d\ge 3$ also satisfies the \emph{quantum diffusion conjecture}, which serves as a key input for establishing the QUE estimates in \Cref{MR:QUE}. 
	To state it, we define the following $\Theta$-matrices:
	\be \label{def:Theta}
	\Theta^{(+,-)}(z)=\Theta^{(-,+)}(z): = \frac{1}{1-|m(z)|^2 S^{\LK}} ,\quad \Theta^{(+,+)}(z)=(\Theta^{(-,-)}(z))^*:= \frac{1}{1-m(z)^2 S^{\LK}} , 
	\ee
	where we recall the matrix $S^{\LK}$ defined in \eqref{eq:variancematrix} and $m(z)$ defined in \eqref{eq:defmzsc}.

	\begin{theorem}[Quantum diffusion]\label{MR:QuDiff}
		In the setting of Theorem \ref{MR:decol}, for any (small) constants $\kappa,\e, \tau>0$ and (large) constant $D>0$, the following events hold with probability $\ge 1-N^{-D}$ for all $a,b\in \Zn$ and for large enough $N$: 
		\begin{align}
			&\bigcap_{z=\hat{E}+\ii\eta\in \mathbf D_{\kappa,\e}}
			\bigg\{\bigg|\frac{1}{W^{2d}}
			\sum_{x\in[a],y\in[b]}|G_{xy}(z)|^2 - \frac{|m|^2\Theta_{ab}^{(+,-)}(z)}{W^d}\bigg| \le W^\tau \br{\p{(\mathcal B_{\eta,0})^{\frac 1 5} \mathcal B_{\eta,|x-y|}}\wedge (\mathcal B_{\eta,0})^{2}}\bigg\}\, , 
			\label{eq:diffu1}\\
			&\bigcap_{z=\hat{E}+\ii\eta\in \mathbf D_{\kappa,\e}}
			\bigg\{\bigg|\frac{1}{W^{2d}}
			\sum_{x\in[a],y\in[b]}\p{G_{xy}G_{yx}}(z) - \frac{m^2\Theta^{(+,+)}_{ab}(z)}{W^d} \bigg| \le W^\tau \br{\p{(\mathcal B_{\eta,0})^{\frac 1 5} \mathcal B_{\eta,|x-y|}}\wedge (\mathcal B_{\eta,0})^{2}}\bigg\}\, 
			\label{eq:diffu2}.
		\end{align}
		Moreover, for each $z=\hat{E}+\ii\eta\in \mathbf D_{\kappa,\e}$, the expectations $\E|G_{xy}(z)|^2$ and $\E\p{G_{xy}(z)G_{yx}(z)}$ satisfy the following bounds for any small constant $\tau>0$ and large enough $N$: 
		\begin{align}\label{Meq:QdS1}
			\max_{a,b }  \bigg|\frac{1}{W^{2d}}
			\sum_{x\in[a],y\in[b]}\E|G_{xy}(z)|^2 - \frac{|m|^2\Theta_{ab}^{(+,-)}(z)}{W^d} \bigg| \le W^\tau (\cal B_{\eta,0})^{2} \left(\p{\ilambda^2W^{d}}^{-1/ 5}+\cal B_{\eta,0}\right),   \\
			\label{Meq:QdS2}
			\max_{a,b }  \bigg|\frac{1}{W^{2d}}
			\sum_{x\in[a],y\in[b]}\E(G_{xy}G_{yx})(z) - \frac{m^2\Theta_{ab}^{(+,+)}(z)}{W^d}\bigg|\le W^\tau (\cal B_{\eta,0})^{2} \left(\p{\ilambda^2W^{d}}^{-1/ 5}+\cal B_{\eta,0}\right) .
		\end{align}
	\end{theorem}
	
	Note that when $\eta \le \ilambda^{2}/L^{d}$, we have 
	\be\label{eq:BetaK}
	\cal B_{\eta,K}\asymp (N\eta)^{-1} \ge (\ilambda^2W^d)^{-1}, \quad \forall 0\le K\lesssim L .
	\ee
	In this case, the expected estimates in \eqref{Meq:QdS1} and \eqref{Meq:QdS2} provide improvements over those in \eqref{eq:diffu1} and \eqref{eq:diffu2}, which in turn lead to the QUE estimates in \Cref{MR:QUE}.

	\begin{proof}[\bf Proof of Theorem \ref{MR:QUE}]  
		With the quantum diffusion estimates \eqref{Meq:QdS1} and \eqref{Meq:QdS2}, the proof of \Cref{MR:QUE} is similar to those for \cite[Theorem 2.5]{YY_25}, \cite[Theorem 2.4]{DYYY25}, and \cite[Theorem 2.2]{RBSO1D}. We now briefly outline the proof without giving all details. 
		
		Take $z=E+\ii\eta$ with $|E|\le 2-\kappa$ and $\eta=W^{-\e_0}({\ilambda W^{d/2}}/N)\ge W^{\fd-\e_0}/N$. With the spectral decomposition of $G(z)$ and the definition of $ {\cal I}_E$, we find that 
		\begin{align}\label{ssfa2}
			\E \sum_{i,j: \lambda_i,\lambda_j\in {\cal I}_E} \left|\bpsi_i^*\left(E_{a}-N^{-1}\right) \bpsi_j\right|^2 &\lesssim \eta^2 \E \tr \left[\operatorname{Im} G(z)\left(E_{a}-N^{-1}\right) \operatorname{Im} G(z)\left(E_{a}-N^{-1}\right)\right] \nonumber \\
			&=\frac{\eta^2}{L^{2d}}\sum_{b,b'\in \Zn}
			\mathbb{E}  
			\tr \left[\operatorname{Im} G(z)\left(E_{a}-E_{b}\right) \operatorname{Im} G(z)\left(E_{a}-E_{b'}\right)\right],
		\end{align}
		where we recall that $E_a$ is the block-averaging matrix restricted to the $a$-th block: $(E_a)_{xy} = W^{-d}\mathbf 1(x=y\in [a])$. Expanding $\im G$ as $\im G=(G-G^*)/(2\ii)$ and applying the QUE estimates \eqref{Meq:QdS1} and \eqref{Meq:QdS2}, we can bound the right-hand side of \eqref{ssfa2} by   
		\begin{align}\label{ssfa2_deter}
			W^\tau \eta^2 \p{\frac{1}{N\eta}}^2\left[\frac{1}{N\eta} + \frac{1}{(\ilambda^2W^d)^{1/5}}\right]+  \frac{C\eta^2}{W^d} \max_{\sig,\sig'\in \{+,-\}}\max_{a,b,b'}\left|\Theta^{(\sig,\sig')}_{ab}-\Theta^{(\sig,\sig')}_{ab'}\right|
		\end{align}
		for any small constant $\tau>0$ and a large constant $C>0$, where the first term comes from the application of \eqref{Meq:QdS1}--\eqref{eq:BetaK}. 
		Then, with the estimates \eqref{prop:ThfadC} and \eqref{prop:BD1} below for the $\Theta$-matrices, we get 
		$$
		\max_{\sig,\sig'\in \{+,-\}}\max_{a,b,b'}\left|\Theta^{(\sig,\sig')}_{ab}-\Theta^{(\sig,\sig')}_{ab'}\right|\prec \ilambda^{-2} .
		$$
		Plugging it into \eqref{ssfa2_deter} and further into \eqref{ssfa2} yields that 
		\begin{align*}
			\E \sum_{i,j: \lambda_i,\lambda_j\in {\cal I}_E} \left|\bpsi_i^*\left(E_{a}-N^{-1}\right) \bpsi_j\right|^2 &\lesssim {W^\tau}N^{-2}\cdot \p{W^{-\fd+\e_0}+W^{-2\fd/5}+W^{-2\e_0}}. 
		\end{align*}
		Finally, applying Markov's inequality concludes \eqref{Meq:QUE}. The proof of \eqref{Meq:QUE2} follows a similar argument, where we simply replace $E_{a}$ in \eqref{ssfa2} with $|A|^{-1}\sum_{a\in A} E_{a}$, after which all subsequent arguments remain valid.  
	\end{proof}
	
	
	\begin{proof}[\bf Proof of Theorem \ref{Thm: B_Univ}]
	Using \Cref{MR:decol}, \Cref{MR:locSC}, and the QUE estimate \eqref{Meq:QUE} as inputs, the proof of Theorem \ref{Thm: B_Univ} follows from the Green’s function comparison method developed in \cite{Xu:2024aa}. 
	The argument is essentially identical to that used for one-dimensional \cite[Theorem~2.6]{YY_25} and two-dimensional \cite[Theorem~2.6]{DYYY25} random band matrices. For instance, adapting the proof of \cite[Theorem~2.6]{DYYY25}, we only need to adjust certain parameters in the proofs of equations (2.22) and (2.23) therein.
	In that setting, given $y\in \ZL$, the ``bad'' event $\mathcal{B}\equiv \mathcal{B}(y)$ was defined by the existence of an index $\alpha$ such that $\lvert \lambda_\alpha - E\rvert \le N^{-1+\mathfrak{c}/6}$ and $\lvert M_{y,\alpha}\rvert \ge N^{-\mathfrak{c}/18}$ ($\fc$ is the constant in the assumption $W \ge N^{\mathfrak{c}}$ and $M_{y,\al}$ is defined below equation (2.24) of \cite{DYYY25}). In our setting, we modify the definition of $\cal B$ to
		\(
		\mathcal{B}(y)
		:= \bigl\{ \exists \alpha: \lvert \lambda_\alpha - E\rvert \le N^{-1} W^{\mathfrak{d}/3}, \lvert M_{y,\alpha}\rvert \ge W^{-\mathfrak{d}/6} \bigr\}.
		\)
		Applying \eqref{Meq:QUE} with $\e_0=\mathfrak{d}/3$ and $c=\mathfrak{d}/6$, we can deduce that
		\(
		\mathbb{P}(\mathcal{B}) \le W^{-\mathfrak{d}/15+\tau}.
		\)
		With these modified parameters and the new definition of $\cal B$, the remainder of the proof of \eqref{eq:universality} proceeds exactly as in the proof of \cite[Theorem~2.6]{DYYY25}. 
		Hence, we omit further details.
	\end{proof}

	\subsection{Main results for the block Anderson model}
	\label{sec:main_BA}
	
	
	We next define the block Anderson model. We begin by introducing the \emph{random block potential} $V$, an $N\times N$ complex Hermitian random block matrix whose diagonal blocks are i.i.d.~GUE matrices. In other words, $V$ can be regarded as a random band matrix whose entries are distributed according to \eqref{bandcw0}, with variance matrix $S=(S_{xy})$ given by: 
	\be\label{bandcwV}
	S_{xy}\equiv \var (H_{xy}):=W^{-d}S^{(\sB)}_{ab}(0)= W^{-d} {\bf 1}\left( a =b \right), \quad \text{for} \ \  x\in [a], \ y\in [b],
	\ee
	where \( S^{(\sB)}(0)=I_{L^d} \) denotes the matrix defined in \eqref{eq:variancematrix} with $\ilambda=0$. 
	Next, we define the \emph{block Anderson model} as a random block Schr{\"o}dinger operator of the form
	\be\label{eq:H_blocka}
	H \equiv H(\ilambda)=V+\ilambda \Psi ,\ee 
	where $\ilambda>0$ (recall \eqref{def:ilambda}) is a coupling parameter, and $\Psi$ represents the interaction Hamiltonian that introduces hopping between neighboring blocks. 
	For definiteness, we focus on the classical block Anderson model in which each block is the identity matrix: 
	\be\label{eq:Psi3D}
	\Psi|_{[a][b]}= \Psi^{\LK}\otimes
	I_{W^d},\quad \text{where}\quad  \Psi^{\LK}_{ab}=\mathbf 1(a\sim b).\ee 
	For the block Anderson model, we again denote its eigenvalues by \( \lambda_1 \leq \lambda_2 \leq \dots \leq \lambda_N \), and the corresponding normalized eigenvectors by \( \left( \boldsymbol{\psi}_k \right)_{k=1}^N \). 
	The Green's function $G(z)$ is defined analogously to \eqref{def_Green}. 
	\begin{remark}
		For simplicity, we have slightly abused notation by using the same symbols (such as $H$, $S$, $\lambda_i$, $\boldsymbol{\psi}_i$, and $G(z)$ defined above, as well as $m(z)$ and $M(z)$ defined below) for both the random band matrix model and the block Anderson model. When a distinction is needed, we will add a superscript $\BA$ to denote quantities associated with the block Anderson model.
	\end{remark}    
	
	It is well-known that as $N\to \infty$, the empirical spectral measure $N^{-1}\sum_{k=1}^N \delta_{\lambda_k}$ converges to a deterministic probability measure $\mu_{N}$, known as the free convolution of the semicircle law and the empirical measure of $\ilambda\Psi$. This measure has a continuous probability density $\rho_N(x)$ on $\R$ \cite{Biane}, with support $\mathrm{supp}(\mu_{N})=[-e_\ilambda, e_\ilambda]$, where $-e_\ilambda<0$ and $e_\ilambda>0$ denote the left and right spectral edges, respectively. 
	The Stieltjes transform $m(z)\equiv m(z,\ilambda)$ of the measure $\mu_{N}$ is defined as the unique solution to the following self-consistent equation
	\be\label{self_m}
	\frac{1}{N}\tr \frac{1}{\ilambda \Psi -z- m(z)} = m(z)
	\ee
	such that $\im m(z)>0$ for $z\in \C_+$. 
	In addition, we define the $N\times N$ matrix $M(z)\equiv M_N(z,\ilambda)$ and the $L^d\times L^d$ matrix \smash{$M^{\LK}(z)\equiv M^{\LK}_L(z,g)$} as
	\be\label{def_G0}
	M(z):= \frac{1}{\ilambda \Psi -z- m(z)} = M^{\LK}(z)\otimes I_{W^d},\quad M^{\LK}(z):=\frac{1}{\ilambda \Psi^{\LK} -z- m(z)}.
	\ee
	Note $M$ and $M^{\LK}$ are both (complex) symmetric matrices. In the random matrix theory literature (see e.g., \cite{LeeSchSteYau2015,knowles2017anisotropic,He2018,AEK_PTRF,EKS_Forum} for various settings of deformed Wigner-type matrices), the Green's function $G(z)$ is known to converge to $M(z)$ in the sense of local laws. 
	(However, existing convergence estimates in the literature are generally non-optimal in the non-mean-field setting $W^d\ll N$.)

	\begin{theorem}[Main results for the block Anderson model]\label{MR:decol_BA}
		Fix any dimension $d\ge 3$, and assume there exist constants $\fc,\fd>0$ such that \eqref{Main_DEL_COND} and \eqref{eq:WO} hold. 
		Then, for the block Anderson model defined above, the following results remain valid if, throughout their statements, we replace $2-\kappa$ with $e_\ilambda-\kappa$:
		\begin{itemize}
			\item The delocalization estimate \eqref{eq:psikLinfty} holds.
			
			\item The local laws \eqref{G_bound} and \eqref{G_bound_ave} hold for large enough $N$, where $G$ is defined in \eqref{def_Green}, and $m(z)$ and $M(z)$ are defined in \eqref{self_m} and \eqref{def_G0}, respectively. 
			
			\item The QUE estimates \eqref{Meq:QUE} and \eqref{Meq:QUE2} hold, and the bulk universality \eqref{eq:universality} holds. 
			
			\item Recall the variance matrix \(S\) from \eqref{bandcwV} and the matrix \(M\) from \eqref{def_G0}. Define the \(\Theta\)-matrices as  
			\be \label{def:Theta_BA}
			\qquad  \quad \Theta^{(+,-)}(z) := \big(1 - M^{(+,-)}(z) S^{\LK}\big)^{-1} , \ \ 
			\Theta^{(+,+)}(z) = (\Theta^{(-,-)}(z))^* := \big(1 - M^{(+,+)}(z) S^{\LK}\big)^{-1} , 
			\ee
			where \( S^{(\sB)}=I_{L^d} \), and the matrices $M^{(+,-)}=M^{(-,+)}$ and $M^{(+,+)}=(M^{(-,-)})^*$ are defined by\footnote{Note the definition \eqref{def:Theta_BA} is consistent with \eqref{def:Theta}, where \(M^{(+,-)} = |m|^2 I\) and \(M^{(+,+)} = m^2 I\).}   
			\[ M^{(+,-)}_{ab}=\frac{1}{W^d}\sum_{x\in[a],y\in[b]}|M_{xy}|^2=|M^{\LK}_{ab}|^2,\quad M^{(+,+)}_{ab}=\frac{1}{W^d}\sum_{x\in[a],y\in[b]}M_{xy}M_{yx}=(M^{\LK}_{ab})^2. \] 
			Then, the quantum diffusion estimates \eqref{eq:diffu1}--\eqref{Meq:QdS2} hold if we replace $|m|^2\Theta_{ab}^{(+,-)}$ and $m^2\Theta_{ab}^{(+,+)}$ with $ (\Theta^{(+,-)}M^{(+,-)})_{ab}$ and $(\Theta^{(+,+)}M^{(+,+)})_{ab}$, respectively.
		\end{itemize}
		
	\end{theorem}

	The lower bound $\ilambda\gg W^{-d/2}$ in \eqref{eq:WO} is sharp, as explained below \Cref{MR:decol}: when $\ilambda\ll W^{-d/2}$, the block Anderson model is localized, as proven in \cite{PelSchShaSod}. To establish \Cref{MR:decol_BA}, it is sufficient to prove the local laws \eqref{G_bound} and \eqref{G_bound_ave}, together with the quantum diffusion estimates \eqref{eq:diffu1}--\eqref{Meq:QdS2}, while the delocalization, QUE, and bulk universality can be derived as corollaries, as shown in \Cref{subsec:main}.

	\subsection{Stochastic flow and loop hierarchy}\label{sec:tools}
	
	The remainder of this paper is devoted to proving the local laws in \Cref{MR:locSC} and the quantum diffusion estimates in \Cref{MR:QuDiff} for both models. Before proceeding with the proofs, we first introduce some key tools and convenient notations for the remainder of this section, which will be utilized throughout the subsequent arguments. 
	First, we introduce the flow framework, which is the same as that employed for 1D and 2D random band matrices \cite{YY_25,DYYY25} except for a dimension-dependent scaling $W^{-d}$. Consider the following matrix Brownian motion: 
	\begin{align}\label{MBM}
		\dd (H_{t})_{xy}=\sqrt{S_{xy}}\dd (\boldsymbol{B}_{t})_{xy}, \ \ \forall x,y\in \ZL,\quad \text{where}\quad  H_{0}=\begin{cases}
			0, \ &\text{for band matrix}\\
			\ilambda\Psi, \ &\text{for block Anderson}
		\end{cases}.
	\end{align}
	Here, $(\boldsymbol{B}_{t})_{xy}$ are independent complex Brownian motions up to the Hermitian symmetry \smash{$(\boldsymbol{B}_{t})_{xy}=\overline {(\boldsymbol{B}_{t})_{yx}}$}, i.e., \smash{$t^{-1/2}\boldsymbol{B}_t$} is an $N\times N$ GUE whose entries have zero mean and unit variance; $S=(S_{xy})$ is the variance matrix defined in \eqref{eq:variancematrix} or \eqref{bandcwV}. 
	Following \cite{10.1214/19-ECP278, Sooster2019, DY}, we consider the Green's function of \( H_t \) with a carefully chosen time-dependent spectral parameter \( z_t \), whose dynamics are naturally renormalized at leading order.

	\begin{definition}[Flow framework]\label{def_flow}
		For any $E \in \mathbb R$ and $\ilambda>0$, we denote $m(E,\ilambda)\equiv m(E+\ii 0_+,\ilambda)$, as defined in \eqref{eq:defmzsc} and \eqref{self_m} for the random band matrix and block Anderson model, respectively. Based on this, we define the \emph{spectral parameter flow} $z_t$ by 
		\be\label{eq:zt}
		z_t(E,\ilambda) = E + (1-t) m(E,\ilambda),\quad \text{for}\ \ t\in [0, 1]. 
		\ee
		We refer to $E$ and $\ilambda$ as {\bf flow parameters}, which remain fixed throughout the flow. Let $z_t=E_t + \ii \eta_t$ with 
		\begin{align}\label{eta}
			E_t\equiv E_t(E,\ilambda)=E+(1-t)\re m(E,\ilambda),\quad \eta_t\equiv \eta_t(E,\ilambda) = (1-t)  \im m(E,\ilambda).
		\end{align}
		Furthermore, we define the matrix $M(E,\ilambda)$, as in \eqref{eq:defMzsc} and \eqref{def_G0} for the random band matrix and block Anderson model, respectively, with $z$ and $m(z)$ replaced by $E$ and $m(E,\ilambda)$.  
		Then, we denote Green's function of $H_{t}$ as $G_t(z,\ilambda):=(H_t(\ilambda)-z)^{-1}$, and define the resolvent flow as 
		\begin{align}\label{self_Gt}
			G_{t;E,\ilambda}\equiv G_t(z_t(E,\ilambda),\ilambda):=(H_{t}(\ilambda)-z_t{(E,\ilambda)})^{-1}.
		\end{align}
	\end{definition}
	
	\begin{remark}
		To explain the choice of the flow $z_t$ in \eqref{eq:zt}, let $M_t(z,\ilambda)\in \cal M_N(\C)$ (where $\cal M_N(\C)$ denotes the set of $N\times N$ complex matrices) be the unique solution to the matrix Dyson equation 
		\be\label{def_G0t}
		M_t(z,\ilambda):= \p{H_0 -z- t \cal S[M_t(z,\ilambda)]}^{-1},
		\ee
		such that $\im M_t$ is positive definite whenever \(z\in \C_+\). Here, the linear operator $\mathcal{S}$ is defined by
		\begin{align}\label{eq:opS}
			\mathcal{S}[X]_{xy}:=\delta_{xy}\sum_{y=1}^{N}S_{xy}X_{yy},\quad \text{for}\quad X\in \cal M_{N}(\C),
		\end{align}
		where $S_{xy}$ is given by \eqref{eq:variancematrix} and \eqref{bandcwV} for the respective models. 
		It is known that $M_t$ describes the deterministic limit of the Green’s function $(H_t-z)^{-1}$. 
		Moreover, since $\cal S[M(E,\ilambda)]=m(E,\ilambda)I_N$, one can check that \(M_t(z_t(E,\ilambda),\ilambda)\equiv M(E,\ilambda)\).  
		In other words, the deterministic limit of $G_t$ remains invariant under the evolution.  
	\end{remark}

	Given any target spectral parameter $z$, we are interested in the original resolvent $G(z)=(H-z)^{-1}$. For random band matrices, this can be achieved through the stochastic flow by carefully choosing the spectral parameter $E$. The corresponding flow for the block Anderson model will be presented later in \Cref{sec:ext-to-BA}.

	\begin{lemma}[Lemma 2.8 of \cite{YY_25}]\label{zztE}
		Fix any $z\in \mathbb C_+$ with $\im z\in (0, 1]$ and $|\re z|\le 2-\kappa$. We choose 
		\be\label{eq:t0E0}t_0\equiv t_0(z)=|m(z)|^2=\frac{\im m(z)}{\im m(z)+ \im z},\quad E\equiv E(z)=-2\frac{\re m(z)}{|m(z)|}\, .\ee
		Then, for the random band matrix model, we have  
		\begin{equation}\label{eq:zztE}
			\sqrt{t_0}m(E)=m(z), \quad     z_{t_0}(E) =\sqrt{t_0} z  , \quad G(z) \stackrel{d}{=} \sqrt{t_0} G_{t_0,E} ,
		\end{equation} 
		where ``$\stackrel{d}{=}$" means equality in distribution.
	\end{lemma}

	In the main proofs for random band matrices, we will fix a target spectral parameter $z=\hat{E}+\ii \eta\in \mathbf D_{\kappa,\e}$ for an arbitrarily small constant $\e>0$ (recall \eqref{eq:spectral_domain}). Accordingly, we choose the parameters $t_0$ and $E$ as specified in \eqref{eq:t0E0}. The second identity in \eqref{eta} implies that $1-t\asymp \eta_t$ uniformly in $t\in [0,t_0]$, i.e., during the flow from $t=0$ to $t_0$, the imaginary part $\eta_t$ decreases from $\eta_0\asymp 1$ to $\eta_{t_0}\asymp 1-t_0 \asymp \eta \ge N^{-1+\e}$.  
	For clarity, unless we want to emphasize their dependence on $E$ (or $g$), we will often omit this variable from various notations, such as $z_t(E,\ilambda)$, $E_t(E,\ilambda)$, $\eta_t(E,\ilambda)$, $m(E,\ilambda)$, $M(E,\ilambda)$, and most importantly, $G_{t;E,\ilambda}\equiv G_t$. Our focus will be on the dynamics of $G_{t}$ and the corresponding $G$-loops defined below.

	\begin{definition}[$G$-loop]\label{Def:G_loop}
		For $\sigma\in \{+,-\}$, we denote 
		\begin{equation}\nonumber
			G_{t}(\sigma):=\begin{cases}
				(H_t-z_t)^{-1}, \ \ \text{if} \ \  \sigma=+,\\
				(H_t-\bar z_t)^{-1}, \ \ \text{if} \ \ \sigma=-.
			\end{cases}
		\end{equation}
		In other words, we let $G_{t}(+)\equiv G_{t}$ and $G_{t}(-)\equiv G_{t}^*$. 
		For any $\fn\in \N$, fix indices $\bsig=(\sigma_1, \cdots \sigma_\fn)\in \{+,-\}^\fn$ and $\ba=(a_1, \ldots, a_\fn)\in (\Zn)^\fn$. We define the corresponding {\bf $\fn$-$G$-loop} by 
		\begin{equation}\label{Eq:defGLoop}
			{\cal L}^{(\fn)}_{t, \boldsymbol{\sigma}, \ba}= \tr \p{\prod_{i=1}^\fn \left(G_{t}(\sigma_i) E_{a_i}\right) }, \quad \text{where} \quad (E_{a_i})_{xy} = W^{-d}\mathbf 1(x=y\in [a_i]).
		\end{equation}
		Sometimes, we will also call $G$-loops as $\cL$-loops. Furthermore, we denote 
		\begin{equation}\label{def_mtzk}
			m (\sigma ):= \begin{cases}
				m(E,\ilambda),  &\text{if} \ \ \sigma  =+ \\
				\bar m(E,\ilambda),  &\text{if} \ \ \sigma = -
			\end{cases}, \quad M (\sigma ):= \begin{cases}
				M(E,\ilambda) ,  &\text{if} \ \ \sigma  =+ \\
				M(E,\ilambda)^*,  &\text{if} \ \ \sigma = -
			\end{cases}.
		\end{equation}
		Finally, we defined the \emph{centered resolvent} $\Gc$ as
		\begin{equation}\label{Eq:defwtG}
			\Gc_t(\sigma) := G_t(\sigma) - M(\sigma),\quad \forall t\in[0,1], \ \sig\in \{+,-\} .
		\end{equation}
		(Note that $\Gc_t(\sigma)$ was denoted by $\widetilde{G}_t(\sigma)$ in the preceding papers \cite{YY_25, DYYY25}; here we adopt a different notation for clarity.) 
		For any $a\in \Zn$, we refer to $\tr[\Gc_t(\sig) E_a]$ as a \emph{light-weight}.
	\end{definition}

	To describe the loop hierarchy for the $G$-loops, we introduce the following operations, following \cite{YY_25}. (For a graphical illustration of these operations, see also the diagrams in \cite[Definition 2.10]{YY_25}.)

	\begin{definition}\label{Def:oper_loop}
		For any fixed $n\in \N$, take an $n$-loop of the form \eqref{Eq:defGLoop}.

		\medskip
		
		\noindent 
		\emph{1.} For $k \in \qqq{\fn}$ and $a\in \Zn$, we define a ``cut-and-glue" operator ${\cut}^{(a)}_{k}$ as follows: ${\cut}^{(a)}_{k} \circ {\cal L}^{(n)}_{t, \boldsymbol{\sigma}, \mathbf{a}}$ is defined to be the loop obtained by replacing $G_t(\sigma_k)$ with $G_{t}(\sigma_k) E_a G_{t}(\sigma_k)$. In other words, it cuts the $k$-th edge $G_t(\sigma_k)$ and glues the two new ends with $E_{a}$ to get a new loop that is one unit longer. This operator can also be considered as an operator on $(\boldsymbol{\sigma},\ba)$, that is, 
		\begin{align*}{\cut}^{(a)}_{k} (\boldsymbol{\sigma}, \ba) =  \big( & (\sigma_1,\ldots, \sigma_{k-1}, \sigma_k,\sigma_k ,\sigma_{k+1},\ldots, \sigma_\fn ),\  ( a_1 ,\ldots,  a_{k-1} ,  a , a_k , a_{k+1} ,\ldots,  a_\fn  )\big).\end{align*}
		Hence, we will sometimes write ${\cut}^{(a)}_{k} \circ {\cal L}^{(\fn)}_{t, \boldsymbol{\sigma}, \ba}\equiv {\cal L}^{(\fn+1)}_{t, \;  {\cut}^{(a)}_{k} (\boldsymbol{\sigma}, \ba)}.$

		\noindent 
		\emph{2.}  For $k < l \in \qqq{\fn}$, we define another two types ``cut-and-glue" operators---${\cutL}^{(a)}_{k,l}$ from the left (``L") of $k$, and \smash{${\cutR}^{(a)}_{k,l}$} from the right (``R") of $k$---as follows. 
		In other words, these operators cut the $k$-th and $l$-th edges $G_t(\sigma_k)$ and $G_t(\sigma_l)$, and creates two chains: the left chain to the vertex $a_k$ is of length $(\fn+k-l+1)$ and contains the vertex $a_\fn$, while the right chain to the vertex $a_k$ is of length $(l-k+1)$ and does not contain the vertex $a_\fn$.  
		Then, \smash{${\cutL}^{(a)}_{k,l}$ (resp.~${\cutR}^{(a)}_{k,l}$)} gives an $(\fn+k-l+1)$-loop (resp.~$(l-k+1)$-loop) obtained by gluing the left chain (resp.~right chain) at the new vertex $a$.
		Again, we can also consider the two operators to be defined on the indices $(\boldsymbol{\sigma},\ba)$: 
		\begin{align*}
			&{\cutL}^{(a)}_{k,l} (\boldsymbol{\sigma}, \ba) = \left((\sigma_1,\ldots, \sigma_k,\sigma_l ,\ldots, \sigma_\fn ),(a_1,\ldots, a_{k-1}, a,a_l,\ldots, a_\fn )\right),\\
			&{\cutR}^{(a)}_{k,l} (\boldsymbol{\sigma}, \ba) = \left((\sigma_k,\ldots, \sigma_l),(a_k,\ldots, a_{l-1}, a)\right).
		\end{align*}
		Hence, we will sometimes write 
		${\cutL}^{(a)}_{k,l} \circ {\cal L}^{(\fn)}_{t, \boldsymbol{\sigma}, \ba}\equiv 
		{\cal L}^{(\fn+k-l+1)}_{t,   {\cutL}^{(a)}_{k,l} (\boldsymbol{\sigma}, \ba)}$ and $ {\cutR}^{(a)}_{k} \circ {\cal L}^{(\fn)}_{t, \boldsymbol{\sigma}, \ba}\equiv 
		{\cal L}^{(l-k+1)}_{t,  {\cutR}^{(a)}_{k,l} (\boldsymbol{\sigma}, \ba)}.$
	\end{definition}

	For $x,y\in \ZL$, we abbreviate $\partial_{xy}:=\partial_{(H_t)_{xy}}$. By It\^o's formula, we can derive the following SDE satisfied by the $G$-loops, called \emph{loop hierarchy}; see Lemma 2.11 of \cite{YY_25}.

	\begin{lemma}[Loop hierarchy] \label{lem:SE_basic}
		An $\fn$-$G$-loop satisfies the following SDE, called the \emph{loop hierarchy}:
		\begin{align}\label{eq:mainStoflow}
			\dd\mathcal{L}^{(n)}_{t, \boldsymbol{\sigma}, \ba} = \dd\mathcal{E}^{M,(n)}_{t, \boldsymbol{\sigma}, \mathbf{a}} + \mathcal{E}^{\Gc,(n)}_{t, \boldsymbol{\sigma}, \ba}\dd t +
			W^{d} \sum_{1 \le k < l \le n} \sum_{a, b} \left( {\cutL}^{(a)}_{k, l} \circ \mathcal{L}^{(n)}_{t, \boldsymbol{\sigma}, \ba} \right) S^{(\sB)}_{ab} \left( {\cutR}^{(b)}_{k, l} \circ \mathcal{L}^{(n)}_{t, \boldsymbol{\sigma}, \ba} \right) \dd t, 
		\end{align}
		where $S^{\LK}$ denotes $S^{\LK}(\ilambda)$ in \eqref{eq:variancematrix} (for the random band matrix model) or $S^{\LK}(0)=I_{L^d}$ (for the block Anderson model). Moreover, the martingale term \smash{$\dd\mathcal{E}^{M,(\fn)}_{t, \boldsymbol{\sigma}, \ba}$} and the light-weight term 
		\smash{$\mathcal{E}^{\Gc,(\fn)}_{t, \boldsymbol{\sigma}, \ba}$} are defined by 
		\begin{align} \label{def_Edif}
			\dd\mathcal{E}^{M,(\fn)}_{t, \boldsymbol{\sigma}, \ba} :  = & \sum_{x,y\in \ZL} 
			\left( \partial_{xy}  {\cal L}^{(\fn)}_{t, \boldsymbol{\sigma}, \ba}  \right)
			\cdot \sqrt{S _{xy}}
			\left(\dd \boldsymbol{B}_t\right)_{xy}, \\\label{def_EwtG}
			\mathcal{E}^{\Gc,(\fn)}_{t, \boldsymbol{\sigma}, \ba}: = &  {W}^d \sum_{k=1}^\fn \sum_{a, b\in \Zn} \; 
			\tr\p{ \Gc_t(\sigma_k) E_{a} }
			S^{\LK}_{ab} 
			\left( {\cut}^{(b)}_{k} \circ {\cal L}^{(\fn)}_{t, \boldsymbol{\sigma}, \ba} \right) . 
		\end{align}
		We emphasize that the superscript $(n)$ indicates the length of the $G$-loop on the left-hand side of the equation. For clarity and conciseness, we may omit this superscript when its value is clear from the context. 
	\end{lemma}

	Note the right-hand side (RHS) of equation \eqref{eq:mainStoflow} involves $G$-loops of length larger than $\fn$, and hence represents a ``hierarchy" rather than a ``self-consistent equation" for the $G$-loops. This loop hierarchy is well-approximated by the $\cK$-loops, defined as follows.

	\begin{definition}[Tree approximation]\label{Def_Ktza}
		We define the $\cK$-loop of length 1 as: 
		$${\cal K}^{(1)}_{t,\sigma,a}=m(\sigma),\quad \forall t\in [0,1],\; \;\sigma\in \{+,-\},\;\; a\in \Zn.$$ 
		For $\fn\ge 2$, we define the function \smash{${\cal K}^{(\fn)}_{t, \boldsymbol{\sigma}, \ba}$} (of $t\in[0,1]$, $\bsig\in \{+,-\}^n$, and \smash{$\ba\in (\Zn)^\fn$}) to be the unique solution to the following system of equations, referred to as the \emph{\bf convolution tree equations}: 
		\begin{align}\label{pro_dyncalK}
			\partial_t{\cal K}^{(\fn)}_{t, \boldsymbol{\sigma}, \ba} 
			=  
			W^d \sum_{1\le k < l \le \fn} \sum_{a, b} \left( \cutL^{(a)}_{k, l} \circ \mathcal{K}^{(\fn)}_{t, \boldsymbol{\sigma}, \ba} \right) S^{\LK}_{ab} \left( \cutR^{(b)}_{k, l} \circ \mathcal{K}^{(\fn)}_{t, \boldsymbol{\sigma}, \ba} \right) ,
		\end{align}
		where the operators $\cutL$ and $\cutR$ act on ${\cal K}^{(\fn)}_{t, \boldsymbol{\sigma}, \ba}$ through the actions on indices: 
		\be\label{calGonIND}
		{\cutL}^{(a)}_{k,l}  \circ {\cal K}^{(\fn)}_{t, \boldsymbol{\sigma}, \ba} := {\cal K}^{(\fn+k-l+1)}_{t,  {\cutL}^{(a)}_{k,l}  (\boldsymbol{\sigma}, \ba)} , \ \  {\cutR}^{(b)}_{k,l}  \circ {\cal K}^{(\fn)}_{t, \boldsymbol{\sigma}, \ba} := {\cal K}^{(l-k+1)}_{t,  {\cutR}^{(b)}_{k,l}  (\boldsymbol{\sigma}, \ba)}.  
		\ee
		We impose the following initial condition at $t=0$: 
		\be\label{eq:initial_K}
		{\cal K}^{(k)}_{0, \boldsymbol{\sigma}, \ba} =  {\cal M}^{(k)}_{\boldsymbol{\sigma}, \ba} ,\quad \forall k\in \N, \ \bsig\in \{+,-\}^k,\ \ba\in (\Zn)^k,\ee
		where ${\cal M}^{(k)}_{\boldsymbol{\sigma}, \ba}$ is a $k$-$M$-loop defined as
		\be\label{eq:KMloop} {\cal M}^{(k)}_{\boldsymbol{\sigma}, \ba}:=\tr\p{ \prod_{i=1}^k \left(M(\sigma_i) E_{a_i}\right) } .\ee
		(Note that this $M$-loop can be simplified as \smash{${\cal M}^{(k)}_{\boldsymbol{\sigma}, \ba}= W^{-(k-1)d}\prod_{i=1}^k m(\sig_i) \mathbf 1(a_1=\cdots=a_k)$} for the random band matrix model.) 
		We call \smash{${\cal K}^{(\fn)}_{t, \boldsymbol{\sigma}, \ba}$} an $\fn$-$\cK$-loop. Moreover, we refer to \smash{${\cal K}^{(\fn)}_{t, \boldsymbol{\sigma}, \ba}$} as providing a \emph{\bf tree approximation} of the $G$-loop \smash{$\mathcal{L}^{(n)}_{t, \boldsymbol{\sigma}, \ba}$}, since it admits an explicit tree representation, as shown in \cite{YY_25,RBSO1D} (see also \Cref{Sec:CalK} for the detailed construction).
	\end{definition}

	Given any Hermitian matrix $\cal A$, define its resolvent as $R(z):=(\cal A-z)^{-1}$ for $z= E+ \ii \eta\in \C_+$. Then, with the algebraic identity $R-R^*=2\ii \eta RR^*=2\ii \eta R^*R$, we get the well-known Ward's identity:
	\be\label{eq_Ward0}
	\begin{split}
		\sum_x \overline {R_{xy'}}  R_{xy} = \frac{1}{2\ii \eta}\p{R_{y'y}-\overline{R_{yy'}}},\quad
		\sum_x \overline {R_{y'x}}  R_{yx} = \frac{1}{2\ii \eta}\p{R_{yy'}-\overline{R_{y'y}}}.
	\end{split}
	\ee
	As a special case, if $y=y'$, we have
	\be\label{eq_Ward}
	\sum_x |R_{xy}( z)|^2 =\sum_x |R_{yx}( z)|^2 = {\im R_{yy}(z) }/{ \eta}.
	\ee
Applying \eqref{eq_Ward0} to $G_t$, we can show that the $G$-loops satisfy the following identity \eqref{WI_calL}, which we also refer to as a ``Ward's identity". In \cite{YY_25,RBSO1D}, it shows that a similar Ward's identity \eqref{WI_calK} holds for the $\cal K$-loops. 

\begin{lemma}[Ward's identity for $\cL$-loops and $\cK$-loops]\label{lem_WI_K} 
	Given $\bsig\in\{+,-\}^n$ with $\fn\ge 2$ and $\sigma_1=-\sig_{\fn}$, we have the following identities, which are called Ward's identities at the vertex $a_\fn$: 
	\begin{align}\label{WI_calL}
		&\sum_{a_\fn}{\cal L}^{(\fn)}_{t, \boldsymbol{\sigma}, \ba}=
		\frac{1}{2\ii W^d\eta_t}\left( {\cal L}^{(\fn-1)}_{t,  \wh\bsig^{(+,\fn)}, \wh\ba^{(\fn)}}- {\cal L}^{(\fn-1)}_{t,  \wh\bsig^{(-,\fn)} , \wh\ba^{(\fn)}}\right),  \\
		\label{WI_calK}
		& \sum_{a_\fn}{\cal K}^{(\fn)}_{t, \boldsymbol{\sigma}, \ba}=
		\frac{1}{2\ii W^d\eta_t}\left( {\cal K}^{(\fn-1)}_{t,  \wh\bsig^{(+,\fn)}, \wh\ba^{(\fn)}}- {\cal K}^{(\fn-1)}_{t,  \wh\bsig^{(-,\fn)}, \wh\ba^{(\fn)}}\right) , 
	\end{align}
	where $\eta_t$ is defined in \eqref{eta}, $\wh\bsig^{(\pm,\fn)}$ is obtained by removing $\sigma_\fn$ from $\boldsymbol{\sigma}$ and replacing $\sigma_1$ with $\pm$, i.e., \smash{$\wh\bsig^{(\pm,\fn)}:=(\pm, \sigma_2, \cdots \sigma_{\fn-1})$}, and \smash{$\wh\ba^{(\fn)}$} is obtained by removing $a_\fn$ from $\ba$, i.e., \smash{$\wh\ba^{(\fn)}:=(a_1, a_2,\cdots, a_{\fn-1})$}. 
\end{lemma}
\begin{proof}
	For the random band matrix, this corresponds to Lemma 3.6 in \cite{YY_25}, while for the block Anderson model, it corresponds to Lemma 3.17 in \cite{RBSO1D}. 
\end{proof}

In \Cref{Sec:CalK}, we will present a tree representation formula for the $\cK$-loops, originally discovered in \cite{YY_25} for the random band matrix model and in \cite{RBSO1D} for the block Anderson model. Using this tree representation, we can establish the following upper bound \eqref{eq:bcal_k} for $\cK$-loops. The proof, being similar to those in \cite{YY_25,RBSO1D}, is deferred to \Cref{Sec:CalK}.

\begin{lemma}[Upper bounds on $\cK$-loops]\label{ML:Kbound}
	Fix any dimension $d\ge 3$.	Then, for every $\fn\in \N$ and $t\in [0,1)$, the $\cK$-loops satisfy the upper bound 
	\begin{equation}\label{eq:bcal_k}
		\max_{\bsig\in\{+,-\}^\fn}\max_{\ba\in(\Zn)^\fn} \big| {\cal K}^{(\fn)}_{t, \boldsymbol{\sigma}, \ba} \big| \prec \left(W^{-d}B_{t,0}\right)^{n-1},      
	\end{equation}
	where $B_{t,0}=(\ilambda^2+|1-t|)^{-1}+(L^d|1-t|)^{-1}$ is defined precisely in \Cref{defi:ofB} below. 
\end{lemma}

\subsection{Propagators} 

The quantum diffusion behavior of the resolvents is governed by the $\Theta$-propagators, defined as follows.

\begin{definition}\label{def_Theta}
Given $\sigma_1,\sigma_2\in \{+,-\}$, define an $L^d\times L^d$ matrix $M^{(\sigma_1,\sigma_2)}$ as a rescaled 2-$M$-loop:
\be\label{eq:Msig} M^{(\sigma_1,\sigma_2)}_{ab}:= W^d\tr\p{ M(\sigma_1) E_{a}M(\sigma_2) E_{b}},\quad \forall a,b\in \Zn.\ee
For the random band matrix model, we have $M^{(\sigma_1,\sigma_2)}=m(\sigma_1)m(\sigma_2)I_{L^d}$; for the block Anderson model, we have \smash{$M^{(\sigma_1,\sigma_2)}_{ab}=M^{\LK}_{ba}(\sigma_1)M^{\LK}_{ab}(\sigma_2)$}, where $M^{\LK}$ is defined in \eqref{def_G0}, with $M^{\LK}(+)\equiv M^{\LK}$ and $M^{\LK}(-)\equiv (M^{\LK})^*$. 
For $t\in[0,1]$ and $\sigma_1,\sigma_2\in \{+,-\}$, the $\Theta$-propagator \smash{$\Theta_t^{(\sigma_1,\sigma_2)}$} is an $L^d\times L^d$ matrix defined as:
\begin{equation}\label{def_Thxi}
	\Theta_{t}^{(\sigma_1,\sigma_2)} := \p{1 - t M^{(\sigma_1,\sigma_2)}S^{\LK}}^{-1}, 
\end{equation}
where $S^{\LK}$ is given by $S^{\LK}(\ilambda)$ in \eqref{eq:variancematrix} or $S^{\LK}(0)=I_{L^d}$, depending on whether we consider the random band matrix model or the block Anderson model.  
We denote the entries of \smash{$\Theta_{t}^{(\sigma_1,\sigma_2)}$} as \smash{$\Theta_{t}^{(\sig_1,\sig_2)}(a,b)$ or $\Theta_{t,ab}^{(\sig_1,\sig_2)}$}. We further define the zero-mode-removed propagator \smash{$\zTheta^{(\sig_1,\sig_2)}_{t}$} by 
\begin{equation}\label{def_Thxi0}
	\zTheta^{(\sig_1,\sig_2)}_{t}(a,b):=\Theta^{(\sig_1,\sig_2)}_{t}(a,b) -{L^{-2d}}\sum_{a',b'}\Theta^{(\sig_1,\sig_2)}_{t}(a',b').
\end{equation}
\end{definition}

Note that for the random band matrix model, the $\Theta$-propagators $\Theta^{(\sig,\sig')}_{t_0}$ for $\sig,\sig'\in\{+,-\}$ reduce to the $\Theta$-matrices defined in \eqref{def:Theta} under the relation \eqref{eq:zztE}. 
Similarly, for the block Anderson model, there is a corresponding reduction to the definition in \eqref{def:Theta_BA} at $t=t_0$ under the relation \eqref{eq:zztE_BA} below.

To capture the decay profile of the $\Theta$-propagators, we introduce the following control parameter $B_{t,K}$ in \eqref{eq_B_param}. It is an order parameter that will appear in many estimates throughout the proof. 

\begin{definition}[Definition of $B$]\label{defi:ofB}
For any $t\in [0,1)$, define 
\begin{equation}\label{eq_B_param}
	B_{t, K}:=\frac{(\ilambda^{2}+|1-t|)^{-1}}{(K+1)^{d-2}}+\frac{1}{L^d|1-t|},\quad \forall K\ge 0.
\end{equation}
Since $\eta_t\asymp 1-t$ by \eqref{eta}, the parameter $\cal B_{\eta_t,K}$ defined in \eqref{eq:calBetaK} satisfies 
\be\label{eq:BtBt}
\cal B_{\eta_t,K} \asymp W^{-d}B_{t,(K/W)}.
\ee
\end{definition}

We now summarize some fundamental properties of the $\Theta$-propagators that will be used extensively in the main proof. These properties have essentially been established in previous works \cite{DYYY25, yang2024Del, RBSO1D}. For the reader's convenience, we briefly outline the proof of \Cref{lem_propTH} in \Cref{sec:pf_propTH}.

\begin{lemma}\label{lem_propTH}
For any $t\in [0,1)$, define 
\be\label{eq:ellt}
\ell_t:= \min\p{ \max\big(\ilambda |1-t|^{-\frac 1 2},1\big)\; ,\; L} \, .
\ee
For any $\sig_1,\sig_2\in\{+,-\}$, $\Theta^{(\sig_1,\sig_2)}_t$ defined in \Cref{def_Theta} satisfies the following properties:
\begin{enumerate}
	\item {\bf Symmetry}: $\Theta^{(\sig_2,\sig_1)}_t=(\Theta^{(\sig_1,\sig_2)}_t)^\top= \Theta^{(\sig_1,\sig_2)}_t$.
	
	\item {\bf Translation invariance}: For any $a,b,r\in \Zn$, we have $\Theta^{(\sig_1,\sig_2)}_t(a+r,b+r)= \Theta^{(\sig_1,\sig_2)}_t(a,b).$
	
	\item {\bf Commutativity}: We have $\big[S^{(\sB)}, \Theta^{(\sig_1,\sig_2)}_t\big]=\big[\Theta^{(\sig_1,\sig_2)}_t,\Theta^{(\sig_1,\sig_2)}_{t'}\big]=0$ for all $t\ne t'$.
	
	\item {\bf $(\infty\to\infty)$-norm}: For any $a,b\in \Zn$, we have $|\Theta_{t,ab}^{(\sig_1,\sig_2)}|\le \Theta_{t,ab}^{(+,-)}$. Moreover, 
	\be\label{eq:THETAinftinf} \|\Theta_{t}^{(\sig_1,\sig_2)}\|_{\infty\to\infty} = \max_a \sum_{b}\big|\Theta_{t,ab}^{(\sig_1,\sig_2)} \big|\le \max_a \sum_{b}\Theta_{t,ab}^{(+,-)}=\frac{1}{1-t}.\ee
	
	\item {\bf Polynomial and exponential decay}: For all $\sig_1,\sig_2\in \{+,-\}$, there exist constants $c_d,C_d>0$ (depending on $d$) such that the following bound holds: 
	\begin{equation}\label{prop:ThfadC}   
		\big|\Theta^{(\sig_1,\sig_2)}_{t}(0,a)\big|\le C_d B_{t, |a|}\cdot {e^{-c_d |a|/ {\ell}_t}},\quad \forall a\in \Zn .  
	\end{equation}
	Furthermore, when $\sig_1=\sig_2$, we have a much stronger exponential decay: there exist constants $c_\kappa,C_\kappa>0$ (depending on $d$ and $\kappa$) such that 
	\begin{equation}\label{prop:ThfadC_short}  \big|\Theta^{(\sig_1,\sig_2)}_{t}(0,a)\big|\le C_\kappa \left(1_{a=0} + \ilambda^2 e^{-c_{\kappa}|a|}\right),\quad \forall a\in \Zn.
	\end{equation}
	In the setting of the block Anderson model, these constants may also depend on $\ilambda^{-1}$ when $1\le \ilambda\le \fd^{-1}$. 
	
	\item {\bf First-order difference}: The following estimate holds for all $a, r\in \Zn$ satisfying $|r|\lesssim |a|$:  
	\begin{equation}\label{prop:BD1}     
		\left| \Theta^{(\sig_1,\sig_2)}_{t}(0, a+r)-\Theta^{(\sig_1,\sig_2)}_{t}(0, a)\right|\prec \frac{1}{\ilambda^2+|1-t|}\frac{|r|}{(|a|+1)^{d-1}}.
	\end{equation}

	\item {\bf Second-order difference}: The following estimate holds for all $a,r\in \Zn$ satisfying $|r|\lesssim |a|$: 
	\begin{equation}\label{prop:BD2}
		\left|\Theta^{(\sig_1,\sig_2)}_{t} (0,a+r) + \Theta^{(\sig_1,\sig_2)}_{t} (0,a-r)-  2\Theta^{(\sig_1,\sig_2)}_{t} (0,a) \right|
		\prec \frac{1}{\ilambda^2+|1-t|}\frac{|r|^2}{(|a|+1)^{d}} .
	\end{equation}
	
	\item {\bf Propagator without zero mode.} The following estimate holds for all $a \in \Zn$: 
	\be\label{prop:ThfadC0} 
	\big|\zTheta^{(\sig_1,\sig_2)}_{t} (0,a)\big|\prec \frac{1}{\ilambda^2+|1-t|}\frac{1}{(|a|+1)^{d-2}}. 
	\ee
\end{enumerate}

\end{lemma}

\begin{example} 
As shown in \cite{YY_25,RBSO1D}, the $2$-$\cK$ and $3$-$\cK$ loops are given by 
\begin{align}\label{Kn2sol}\cK^{(2)}_{t,\bsig,\ba}&= \sum_{b}\Theta_{t}^{(\sig_1,\sig_2)}(a_1,b)\cal M^{(2)}_{\bsig,(b,a_2)}=W^{-d} \left(\Theta_t^{(\sigma_1,\sigma_2)}  M^{(\sig_1,\sig_2)}\right)_{a_1a_2},\\
	\label{Kn3sol} \mathcal{K}^{(3)}_{t, \boldsymbol{\sigma}, \ba} &= \sum_{b_1,b_2,b_3}\Theta_{t}^{(\sig_1,\sig_2)}(a_1,b_1)\Theta_{t}^{(\sig_2,\sig_3)}(a_2,b_2)\Theta_{t}^{(\sig_3,\sig_1)}(a_3,b_3) \cal M^{(3)}_{\bsig,\mathbf b},
\end{align}
where $\bsig=(\sig_1,\ldots,\sig_\fn)$ and $\ba=(a_1,\ldots,a_\fn)$ for $\fn\in \{2,3\}$, and the $M$-loops are defined in \eqref{eq:KMloop}. 
\end{example}

\subsection{Proof of the main results}

The main results, \Cref{MR:locSC,MR:QuDiff}, for random band matrices follow directly from the following key lemmas on the $G$-loop estimates, while the proof of \Cref{MR:decol_BA} for the block Anderson model will be given separately in \Cref{sec:ext-to-BA}. Recall the notions of stochastic domination, $B_{t, K}$, and $\ell_t$ defined in \eqref{stoch_domination}, \eqref{eq_B_param}, and \eqref{eq:ellt}, respectively.

\begin{lemma}[$G$-loop estimates]\label{ML:GLoop}
In the setting of \Cref{MR:locSC}, fix any $z=\hat{E}+\ii \eta\in \mathbf D_{\kappa,\e}$ and consider the flow framework in \Cref{zztE}. For each fixed $n\in \N$, the following estimate holds uniformly in $t\in [0,t_0]$: 
\begin{align}\label{Eq:L-KGt}
	\max_{\boldsymbol{\sigma}, \mathbf{a}}\left|{\cal L}^{(n)}_{t, \boldsymbol{\sigma}, \mathbf{a}}-{\cal K}^{(n)}_{t, \boldsymbol{\sigma}, \mathbf{a}}\right|&\prec  (W^{-d}B_{t,0})^{ n}, \\
	\label{Eq:L-KGt2}
	\max_{\boldsymbol{\sigma}, \mathbf{a}}\left|{\cal L}^{(n)}_{t, \boldsymbol{\sigma}, \mathbf{a}} \right|&\prec (W^{-d}B_{t,0})^{ n-1}.
\end{align}
\end{lemma}

\begin{lemma}[$2$-loop estimates]\label{ML:GLoop_expec}
In the setting of \Cref{ML:GLoop}, for $\boldsymbol{\sigma}\in\{+,-\}^2$ and $ \ba=(a_1, a_2)\in (\Zn)^2$, the following pointwise estimate holds uniformly in $t\in [0,t_0]$ for any large constant $D>0$: 
\begin{equation}\label{Eq:Gdecay}
	\left| {\cal L}^{(2)}_{t, \boldsymbol{\sigma}, \mathbf{a}}-{\cal K}^{(2)}_{t, \boldsymbol{\sigma}, \mathbf{a}}\right|\prec 
	\p{W^{-d}B_{t, 0}}^{1/5}\cdot \p{W^{-d}B_{t,  |a_1-a_2| }}e^{-\left({| a_1-a_2| }/{\ell_t}\right)^{1/2}} + W^{-D}. 
\end{equation}
Moreover, the expectation of a $2$-$G$-loop satisfies the following $L^\infty$-estimate uniformly in $t\in [0,t_0]$: 
\begin{equation}\label{Eq:Gtlp_exp}
	\max_{\boldsymbol{\sigma}, \ba}\left|\mathbb E{\cal L}^{(2)}_{t, \boldsymbol{\sigma}, \ba}-{\cal K}^{(2)}_{t, \boldsymbol{\sigma}, \ba}\right|\prec \p{W^{-d}B_{t,0}}^2 \p{(\ilambda^2W^d)^{-1/5}+W^{-d}B_{t,0}}
	. 
\end{equation}
In the above estimates, the exponent $1/5$ can be replaced by any other positive constant less than 1/4.
\end{lemma}

\begin{lemma}[Local law for $G_t$]\label{ML:GtLocal}
In the setting of \Cref{ML:GLoop}, the following entrywise local law holds uniformly in $t\in [0,t_0]$: 
\begin{equation}\label{Gt_bound}
	\|(G_{t}-M)_{xy}\|_{\max}^2\prec W^{-d}B_{t,(|x-y|/W)} .
\end{equation} 
\end{lemma}

With these lemmas, we readily conclude the proofs of our main results, \Cref{MR:locSC,MR:QuDiff}.
\begin{proof}[\bf Proof of \Cref{MR:locSC,MR:QuDiff}]
For a fixed $z=\hat{E}+\ii \eta\in \mathbf D_{\kappa,\e}$, we can choose the flow framework as in \Cref{zztE}. 
Then, with \eqref{eq:zztE}, \eqref{eq:BtBt}, and \eqref{Kn2sol}, we observe that \Cref{ML:GtLocal}, \Cref{ML:GLoop} (with $\fn=1$), and \Cref{ML:GLoop_expec} respectively give the entrywise local law \eqref{G_bound}, the averaged local law \eqref{G_bound_ave}, and the quantum diffusion estimates in \Cref{MR:QuDiff} at each \emph{fixed $z$}. To extend these estimates uniformly to all $z$, we can use a standard $N^{-C}$-net and perturbation argument, whose details we omit. 
\end{proof}

Before concluding this section, we now outline the proofs of \Cref{ML:GLoop,ML:GLoop_expec,ML:GtLocal}. At $t=0$, we have $G_{0}(\sigma)=M(\sigma)$ for $\sig\in\{+,-\}$. Together with Definitions \ref{Def:G_loop} and \ref{Def_Ktza}, it implies that for any fixed $\fn\in \N$:   
$${\cL}^{(\fn)}_{0, \boldsymbol{\sigma},\ba}= {\cK}^{(\fn)}_{0, \boldsymbol{\sigma},\ba},\quad \forall \; \boldsymbol{\sigma}\in \{+,-\}^\fn,\ \ \ba\in (\Zn)^\fn. 
$$
Now, for $t>0$, we will establish the following theorem, which provides an induction result that extends the $G$-loop estimate progressively along the stochastic flow.
 
\begin{theorem}\label{lem:main_ind} 
In the setting of \Cref{MR:locSC}, fix any $z=\hat{E}+\ii \eta\in \mathbf D_{\kappa,\e}$ and consider the flow framework in \Cref{zztE}. 
Suppose the estimates \eqref{Eq:L-KGt}, \eqref{Eq:Gdecay}, \eqref{Eq:Gtlp_exp}, and \eqref{Gt_bound} hold at a fixed $s\in [0,t_0]$, that is: 
\begin{itemize}
	\item[(a)] {\bf $G$-loop estimate}: For each fixed $n\ge 1$, we have
	\be\label{Eq:L-KGt+IND}
	\max_{\boldsymbol{\sigma}, \ba}\left|{\cal L}^{(n)}_{s, \boldsymbol{\sigma}, \ba}-{\cal K}^{(n)}_{s, \boldsymbol{\sigma}, \ba}\right|\prec (W^{-d}B_{s,0})^{n}.
	\ee
	\item[(b)] {\bf 2-loop estimate}: 
	For $\boldsymbol{\sigma}\in\{+,-\}^2$, $ \ba=(a_1,a_2)\in (\Zn)^2,$ and any large constant $D>0$, we have 
	\be\label{Eq:Gdecay+IND}
	\left| {\cal L}^{(2)}_{s, \boldsymbol{\sigma}, \ba}-{\cal K}^{(2)}_{s, \boldsymbol{\sigma}, \ba}\right|\prec \p{W^{-d}B_{s, 0}}^{1/5}\cdot \p{W^{-d}B_{s, |a_1-a_2|}} e^{-\left({| a_1-a_2| }/{\ell_s}\right)^{1/2}} +W^{-D}.
	\ee
	Furthermore, if $1-s\ge \ilambda^2$ (where $\ell_s= 1$ by \eqref{eq:ellt}), we assume a stronger estimate:
	\be\label{Eq:Gdecay+IND_s<g}
	\left| {\cal L}^{(2)}_{s, \boldsymbol{\sigma}, \ba}-{\cal K}^{(2)}_{s, \boldsymbol{\sigma}, \ba}\right|\prec \p{W^{-d}B_{s, 0}}^2 e^{-| a_1-a_2|^{1/2}} +W^{-D}.
	\ee
	
	\item[(c)] {\bf Local law}: We have the (maximum) entrywise local law 
	\be \label{Gt_bound+IND}
	\|G_{s}-M\|_{\max} \prec (W^{-d}B_{s,0})^{1/2}.
	\ee
	
	\item[(d)] {\bf Expected $2$-loop estimate}: For all $\boldsymbol{\sigma}\in\{+,-\}^2$ and $ \ba=(a_1,a_2)\in (\Zn)^2,$ we have 
	\be \label{Eq:Gtlp_exp+IND}
	\max_{\boldsymbol{\sigma}, \ba}\left|\mathbb E{\cal L}^{(2)}_{s, \boldsymbol{\sigma}, \ba}-{\cal K}^{(2)}_{s, \boldsymbol{\sigma}, \ba}\right|\prec \p{W^{-d}B_{s,0}}^2 \p{(\ilambda^2W^d)^{-1/5}+W^{-d}B_{s,0}}. 
	\ee
\end{itemize} 
Then, there exists a constant $0<\fc_d \le 10^{-2}$ (depending on $d$, $\kappa$, $\e$, and $\fd$ in \eqref{eq:WO}) such that for any $s< t < 1$ satisfying 
\begin{equation}\label{con_st_ind}
	\p{W^{-d}B_{t,0}}^{\fc_d}\le \frac{1-t}{1-s} < 1, 
\end{equation}
the estimates \eqref{Eq:L-KGt}--\eqref{Gt_bound} hold.  
In the proof, we do not track the exact value of $\fc_d$ (although it can be done by keeping track of the constants carefully in our proof).
In addition, if $1-t\ge \ilambda^2$, we have
\be\label{Eq:Gdecay+s<g}
\left| {\cal L}^{(2)}_{t, \boldsymbol{\sigma}, \ba}-{\cal K}^{(2)}_{t, \boldsymbol{\sigma}, \ba}\right|\prec \p{W^{-d}B_{t, 0}}^2 e^{-| a_1-a_2|^{1/2}} +W^{-D},\quad \forall \boldsymbol{\sigma}\in\{+,-\}^2,\  \ba \in (\Zn)^2.
\ee
\end{theorem}

With \Cref{lem:main_ind} in hand, we can establish \Cref{ML:GLoop,ML:GLoop_expec,ML:GtLocal} via induction on $t$. 
\begin{proof}[\bf Proof of \Cref{ML:GLoop,ML:GLoop_expec,ML:GtLocal}]
We first perform induction from $t=0$ up to $t=1-\ilambda^2$ (when $\ilambda^2\le 1/2$) or $t=1/2$ (when $\ilambda^2>1/2$), establishing \eqref{Eq:L-KGt}--\eqref{Gt_bound} at $t_1:=(1-\ilambda^2)\vee (1/2)$. In the case $\ilambda^2\le 1/2$, we additionally maintain the stronger 2-loop estimate \eqref{Eq:Gdecay+s<g} throughout this step. Next, starting from $t_1$, we continue the induction using \Cref{lem:main_ind} up to $t=t_0$, thereby completing the proof of the estimates stated in \Cref{ML:GLoop,ML:GLoop_expec,ML:GtLocal}.    
\end{proof}

The proof of \Cref{lem:main_ind} is divided into six steps, which comprise the remainder of this paper. Throughout these steps, we assume the hypotheses of \Cref{lem:main_ind} hold. Moreover, each step builds upon the results obtained in the preceding ones.

\medskip 
\noindent
\textbf{Step 1} (A priori $G$-loop bound):  
The $n$-$G$-loops satisfy the a priori bound: 
\begin{equation}\label{lRB1}
{\cal L}^{(n)}_{u,\boldsymbol{\sigma}, \mathbf{a}}\prec \p{\frac{1-s}{1-u}}^{n-1}  
(W^{-d}B_{s,0})^{n-1},\quad  \forall  u\in [s, t].
\end{equation}
Furthermore, a weak local law holds in the sense 
\begin{equation}\label{Gtmwc}
\|G_u-M\|_{\max}\prec  (W^{-d}B_{u,0})^{1/4},\quad \forall  u\in [s, t] .
\end{equation} 

\medskip 
\noindent 
\textbf{Step 2}  (A priori $2$-$G$-loop decay): 
The following sharp local laws hold uniformly for all $u\in [s,t]$: 
\begin{align}\label{Gt_bound_flow}
	|(G_u-M)_{xy}|^2&\prec W^{-d}B_{u,(|x-y|/W)},\quad \forall x,y\in \ZL, \\
	\label{Gt_avgbound_flow} 
	\max_{a\in \Zn}\left|\tr\p{(G_u-M)E_a}\right| &\prec   W^{-d}B_{u,0}.
\end{align}
In particular, the entrywise local law  \eqref{Gt_bound} and the $1$-loop estimate in \eqref{Eq:L-KGt} hold at time $t$. 
In addition, there exists a constant $C_d>0$ (independent of $\fc_d$ in \eqref{con_st_ind}) such that, for any $\boldsymbol{\sigma}\in\{+,-\}^2$, $\ba=(a_1,a_2)\in(\Zn)^2$, and $u\in[s,t]$, the following estimate holds for arbitrarily large constant $D>0$: 
\begin{equation}\label{Eq:Gdecay_w}
	\left| {\cal L}^{(2)}_{u, \boldsymbol{\sigma}, \mathbf{a}}-{\cal K}^{(2)}_{u, \boldsymbol{\sigma}, \mathbf{a}}\right| \prec \p{\frac{1-s}{1-u}}^{C_d}\p{W^{-d}B_{u, 0}}^{1/5}\cdot \p{W^{-d}B_{u,  |a_1-a_2| }} e^{-\left({| a_1-a_2| }/{\ell_u}\right)^{1/2}} +W^{-D} . 
\end{equation}
Hence, at this step, the estimate \eqref{Eq:Gdecay+IND} deteriorates at time $t$ by a factor of $(|1-s|/|1-t|)^{C_d}$. 
We note that the exponent $1/5$---as an arbitrarily chosen positive constant less than $1/4$---in \eqref{Eq:Gdecay_w} is not optimal. Achieving the optimal decay rate for a 2-loop estimate would require this exponent to be 1, which is beyond the scope of the current paper (see also \Cref{rmk_bottleneck} below).

\medskip
\noindent 
\textbf{Step 3}  (Sharp $n$-loop bound): 
The following bound on $n$-$G$-loops holds for any fixed $n\in \N$: 
\begin{equation}\label{Eq:LGxb}
	\max_{\boldsymbol{\sigma}, \mathbf{a}}\left| {\cal L}^{(n)}_{u, \boldsymbol{\sigma}, \mathbf{a}} \right|
	\prec 
	(W^{-d}B_{u, 0})^{-n+1} ,\quad \forall u\in[s,t] .
\end{equation}
In particular, this gives the estimate \eqref{Eq:L-KGt2} at time $t$. 

\medskip
\noindent 
\textbf{Step 4} (Sharp  $(\cal L - \cal K)$-loop estimate): The following  estimate on $(\cL-\cK)$-loops holds for any fixed $n\in \N$: 
\begin{equation}\label{Eq:L-KGt-flow}
	\max_{\boldsymbol{\sigma}, \mathbf{a}}\left|{\cal L}^{(n)}_{u, \boldsymbol{\sigma}, \mathbf{a}}-{\cal K}^{(n)}_{u, \boldsymbol{\sigma}, \mathbf{a}}\right|\prec (W^{-d}B_{u, 0})^{-n},\quad \forall u\in[s,t] .
\end{equation}
In particular, this gives the estimate \eqref{Eq:L-KGt} at time $t$.

\medskip
\noindent 
\textbf{Step 5} (Pointwise estimate for $2$-loops): For any $\boldsymbol{\sigma}\in\{+,-\}^2$, $ \ba=(a_1,a_2)\in(\Zn)^2,$ and $u\in[s,t]$, the following estimate holds for any large constant $D>0$:
\begin{equation}\label{Eq:Gdecay_flow}
	\left| {\cal L}^{(2)}_{u, \boldsymbol{\sigma}, \ba}-{\cal K}^{(2)}_{u, \boldsymbol{\sigma}, \ba}\right| \prec  \p{W^{-d}B_{u, 0}}^{1/5}\cdot \p{W^{-d}B_{u,  |a_1-a_2| }} e^{-\left({| a_1-a_2| }/{\ell_u}\right)^{1/2}} +W^{-D} . 
\end{equation}
Hence, the pointwise decay estimate \eqref{Eq:Gdecay} holds at time $t$. Furthermore, if $1-t\ge \ilambda^2$, then we have 
\be\label{Eq:Gdecay+s<g_flow}
\left| {\cal L}^{(2)}_{u, \boldsymbol{\sigma}, \ba}-{\cal K}^{(2)}_{u, \boldsymbol{\sigma}, \ba}\right|\prec \p{W^{-d}B_{u, 0}}^2 e^{-| a_1-a_2|^{1/2}} +W^{-D},\quad \forall u\in[s,t].
\ee

\noindent 
\textbf{Step 6} (Expected 2-$G$-loop estimate): 
The following estimate holds: 
\begin{equation}\label{Eq:Gtlp_exp_flow}
	\max_{\boldsymbol{\sigma}, \ba}\left|\mathbb E{\cal L}^{(2)}_{u, \boldsymbol{\sigma}, \ba}-{\cal K}^{(2)}_{u, \boldsymbol{\sigma}, \ba}\right|\prec \p{W^{-d}B_{u,0}}^2 \p{(\ilambda^2W^d)^{-1/5}+W^{-d}B_{u,0}},\quad  
	\forall u \in [s, t] .
\end{equation}
Hence, the estimate \eqref{Eq:Gtlp_exp} holds at time $t$. 

\medskip

We remark that the estimates established in each step hold uniformly in $u\in[s,t]$ (recall \eqref{stoch_domination}) due to a standard $N^{-C}$-net and perturbation argument. For simplicity of presentation, we will not emphasize this uniformity at every step of the proof.

\section{Steps 1 and 2: A priori \texorpdfstring{$G$}{G}-loop estimates}\label{Sec:Steps12}

The remainder of the paper is devoted to proving \Cref{lem:main_ind}, according to the six steps outlined following its statement.
Most of the arguments extend directly to the block Anderson model, with technical differences arising in Step 1 and in the treatment of the light-weight term \smash{$\mathcal{E}^{\Gc,(2)}$} (defined in \eqref{def_EwtG}) when controlling the 2-$G$-loops---specifically, in the proofs of \Cref{lem:LWterm,lem: EWGn2_N} below. 
We will describe the necessary modifications to adapt the argument to the block Anderson model in \Cref{sec:ext-to-BA}.

\subsection{Proof of Step 1} 

Our proof depends crucially on the following lemma, which provides estimates on the resolvent entries via bounds on 2-$G$-loops. Specifically, \eqref{GiiGEX} and \eqref{GavLGEX} establish bounds on both the entrywise and averaged differences between $G$ and $M$ in the $\max$-norm sense, while \eqref{GijGEX} provides a finer estimate on the off-diagonal resolvent entries, which allows us to derive the decay of the resolvent entries from the decay of the 2-$G$ loops.

\begin{lemma}\label{lem_GbEXP}
	Consider the setting of the random band matrix model. For any $t\in [0,t_0]$, define the event 
	\begin{equation}\label{def_asGMc}
		\Omega(t, \e_0) := \big\{\|G_t - M\|_{\max} \leq W^{-\e_0} \big\}
	\end{equation}
	for a small constant $\e_0>0$. Then, the following \emph{entrywise local law} holds for any constants $\tau,D>0$: 
	\begin{equation}
		\P\Big(\mathbf 1(\Omega(t, \e_0))\cdot \|G_{t}-M\|_{\max}^2 \le W^\tau \max_{a,b\in \Zn} \cL^{(2)}_{t,(-,+),(a,b)} \Big)\ge 1-W^{-D} \, .\label{GiiGEX}
	\end{equation}
	Furthermore, we have the following pointwise decay estimate of $G_t$ for all $a,b\in \Zn$:
	\begin{align}\label{GijGEX}
		\mathbf 1(\Omega(t, \e_0))\cdot \max_{\substack{x\in [a], y \in [b], x\ne y}} |(G_t)_{xy}|^{2}  \prec & \sum_{\substack{|a'-a|\le 1,|b'-b|\le 1,\\ \bsig\in \{(+,-),(-,+)\}}}  \cL^{(2)}_{t,\bsig,(a',b')}+ W^{-d}\mathbf 1_{|a-b|\le 1}  \, .
	\end{align}
	Finally, suppose the following estimates hold for a deterministic control parameter $W^{-d/2}\le \Psi_t \le W^{-\e_0}$:\footnote{This parameter $\Psi_t$ should be distinguished from the matrix $\Psi$ appearing in the block Anderson Hamiltonian \eqref{eq:H_blocka}.}
	\be\label{initialGT2} 
	\|G_{t} - M \|_{\max}\prec W^{-\e_0} ,\quad \max_{a,b\in \Zn} \cL^{(2)}_{t,(-,+),(a,b)} \prec \Psi_t^2 \, . 
	\ee
	Then, we have the \emph{averaged local law}: 
	\begin{align}
		\max_{a} \left|\tr\p{\left(G_{t}-M\right)E_{a}} \right| &\prec \Psi_t^2  \, . \label{GavLGEX}
	\end{align}
\end{lemma}
\begin{proof}
	These estimates have been proven as Lemma 4.1 in \cite{YY_25} for 1D random band matrices. However, their proofs are dimension-independent and use standard arguments based on resolvent identities and large deviation estimates as in \cite{erdHos2012rigidity,erdos2013delocalization}.     
\end{proof}

Step 1 of the proof of \Cref{lem:main_ind} is similar to that in \cite[Section 5.1]{YY_25}, where the core is to establish the following continuity estimate for the $G$-loops. 

\begin{lemma}\label{lem_ConArg}
	Fix any \(\e \le s \leq t \leq 1\) for a constant \(\e > 0\). In the flow setting given by \Cref{def_flow} and \Cref{zztE}, assume that the following bound holds at time \(s\) for any fixed $n\in \N$: 
	\begin{equation}\label{55}
		\max_{\boldsymbol{\sigma}, {\ba}} \abs{{\cal L}^{(n)}_{s, \boldsymbol{\sigma}, {\ba}}} \prec(W^{-d}B_{s,0})^{n-1}.
	\end{equation}
	Then, on the event \(\Omega_t= \left\{\|G_{t}\|_{\max} \leq C_0\right\}\) for a constant $C_0>0$, the following estimate holds for any fixed \(n \in \N\) with $n\ge 2$:
	\begin{equation}\label{res_lo_bo_eta}
		{\bf 1}(\Omega_t) \cdot \max_{\boldsymbol{\sigma}, {\ba}} \big|{\cal L}^{(n)}_{t, \boldsymbol{\sigma}, {\ba}} \big|\prec \left(
		(W^{-d}B_{s,0})\cdot\frac{\eta_{s}}{\eta_{t}}\right)^{n-1}
		\le  \left(
		(W^{-d}B_{t,0})\cdot\frac{\eta_{s}}{\eta_{t}}\right)^{n-1}.
	\end{equation}
\end{lemma} 
\begin{proof}
	The proof of this lemma is exactly the same as that for Lemma 5.1 in \cite{YY_25}, except for some minor changes in notations. Hence, we omit the details.\end{proof}

With \Cref{lem_GbEXP,lem_ConArg}, Step 1 of the proof of \Cref{lem:main_ind} for the random band matrix model (i.e., the proof of \eqref{lRB1} and \eqref{Gtmwc}) is the same as that in \cite[Section 5.1]{YY_25}. Hence, we omit the details. 

\subsection{Dynamics of  \texorpdfstring{$(\cal L-\cal K)$-loops}{L-K}}\label{subsec:DyLK}

We begin by presenting some representations of the $G$-loop dynamics, formulated using the loop hierarchy \eqref{eq:mainStoflow}, Duhamel’s principle, and certain evolution kernels, which we introduce below. This dynamics has already been derived for random band matrices in \cite{YY_25} and for the block Anderson model in \cite{RBSO1D}, and will be the basis for Step 2 and subsequent steps for the proof of \Cref{lem:main_ind}. 
For any fixed $\fn\in \N$, combining the loop hierarchy \eqref{eq:mainStoflow} and the convolution tree equation \eqref{pro_dyncalK}, we obtain the following SDE for the $(\cL-\cK)$-loops as shown in equation (5.12) of \cite{YY_25}: 
\begin{align}\label{eq_L-Keee}
	\dd(\mathcal{L} - \mathcal{K})^{(\fn)}_{t, \boldsymbol{\sigma}, \ba} = &\left[\cK^{(2)} \sim (\mathcal{L} - \mathcal{K})\right]^{(\fn)}_{t, \boldsymbol{\sigma}, \ba} \, \dd t+\sum_{\lenk=3}^\fn \left[\cK^{(\lenk)}\sim (\mathcal{L} - \mathcal{K})\right]^{(\fn)}_{t, \boldsymbol{\sigma}, \ba}\, \dd t \nonumber\\
	&+ \mathcal{E}^{(\cL-\cK)\times (\cL-\cK),(\fn)}_{t, \boldsymbol{\sigma}, \ba}\dd t+\mathcal{E}^{\Gc,(\fn)}_{t, \boldsymbol{\sigma}, \ba}\dd t + \dd\mathcal{E}^{M,(\fn)}_{t, \boldsymbol{\sigma}, \ba} .
\end{align} 
Here, every \smash{$\cK^{(\lenk)}$}, $2\le \lenk\le n$, is regarded as a linear operator acting on the $(\cL-\cK)$-loops, defined as:
\begin{align}\label{DefKsimLK}
	&\left[\cK^{(\lenk)}\sim (\mathcal{L} - \mathcal{K})\right]_{t, \boldsymbol{\sigma}, \ba}^{(\fn)}:= W^d \sum_{1\le k < l \leq \fn : l-k=\lenk-1} \sum_{a,b} 
	(\mathcal{L} - \mathcal{K})^{(\fn-\lenk+2)}_{t, \cutL^{(a)}_{k, l}\left(\boldsymbol{\sigma},\, \ba\right)} 
	S^{\LK}_{ab}\mathcal{K}^{(\lenk)}_{t,\cutR^{(b)}_{k,l}\left(\boldsymbol{\sigma} ,\ba\right)} \nonumber\\
	&\quad +  W^d \sum_{1 \leq k < l \leq \fn:l-k=\fn-\lenk+1} \sum_{a,b} \mathcal{K}^{(\lenk)}_{t, \cutL^{(a)}_{k, \,l}\left(\boldsymbol{\sigma},\ba\right)} S^{\LK}_{ab}
	\left(\cL-\cK\right)^{(\fn-\lenk+2)}_{t,\cutR^{(b)}_{k, l}\left(\boldsymbol{\sigma},\ba\right)} ,
\end{align}
$\mathcal{E}^{\Gc}$ and $\dd\mathcal{E}^{M}$ are defined in \eqref{def_EwtG} and \eqref{def_Edif}, respectively, and the term $\mathcal{E}^{(\cL-\cK)\times (\cL-\cK)}$ is defined by 
\begin{equation}\label{def_ELKLK}
	\mathcal{E}^{(\cL-\cK)\times (\cL-\cK),(\fn)}_{t, \boldsymbol{\sigma}, \ba} :=
	W^d \sum_{1 \leq k < l \leq \fn} \sum_{a,b} 
	(\mathcal{L} - \mathcal{K})^{(\fn+k-l+1)}_{t, \cutL^{(a)}_{k, l}\left(\boldsymbol{\sigma},\, \ba\right)} 
	S^{\LK}_{ab} (\cL-\mathcal{K})^{(l-k+1)}_{t,\cutR^{(b)}_{k,l}\left(\boldsymbol{\sigma} ,\ba\right)}\,  .
\end{equation}
In \eqref{eq_L-Keee}, the superscript ``$(n)$" indicates the length of the $(\cL-\cK)$-loop on the LHS of the equation, and we will sometimes omit it from our notations when the value of $n$ is clear from the context.

We now rewrite \eqref{eq_L-Keee} into an integral equation using Duhamel’s principle. First, we introduce the evolution kernel associated with this integral equation.

\begin{definition}[Evolution kernel]\label{DefTHUST}
	For each $t\in [0,1)$, fixed $n\ge 2$ and $\bsig=(\sigma_1,\ldots,\sig_\fn)\in \{+,-\}^\fn$, we define the linear operator ${\varTheta}^{(n)}_{t, \boldsymbol{\sigma}}$ acting on $n$-dimensional tensors ${\cal A}: (\Zn)^{n}\to \mathbb C$ as follows (recall \eqref{eq:Msig}):
	\begin{align}\label{def:op_thn}
		\left({\varTheta}^{(\fn)}_{t, \boldsymbol{\sigma}} \circ \mathcal{A}\right)_{\ba} & = \sum_{i=1}^\fn \sum_{b_i\in\Zn} \left(\frac{M^{(\sig_i,\sig_{i+1})}S^{(\sB)}}{1 - t M^{(\sig_i,\sig_{i+1})}S^{(\sB)}}\right)_{a_ib_i}  \mathcal{A}_{\ba^{(i)}(b_i)},
	\end{align} 
	where $\ba=(a_1,\ldots, a_\fn)\in (\Zn)^\fn$, \(\ba^{(i)}(b_i):= (a_1, \ldots, a_{i-1}, b_i, a_{i+1}, \ldots, a_\fn),\) and we adopt the cyclic convention that $\sigma_{\fn+1}=\sig_1$.  
	The evolution kernel corresponding to ${\varTheta}^{(\fn)}_{t, \boldsymbol{\sigma}}$ is given by
	\begin{align}\label{def_Ustz}
		\left(\mathcal{U}_{s, t, \boldsymbol{\sigma}}^{(\fn)} \circ \mathcal{A}\right)_{\ba} = \sum_{\mathbf{b} = (b_1, \ldots, b_{\fn})} \prod_{i=1}^\fn \left(\frac{1 - s \cdot M^{(\sig_i,\sig_{i+1})} S^\LK }{1 - t \cdot M^{(\sig_i,\sig_{i+1})} S^\LK}\right)_{a_i b_i} \cdot \mathcal{A}_{\mathbf{b}} . 
	\end{align}
\end{definition}

By the definition of the 2-$\cK$-loop in \eqref{Kn2sol}, we observe that the first term on the RHS of \eqref{eq_L-Keee} can be rewritten as ${\varTheta}^{(n)}_{t, \boldsymbol{\sigma}} \circ (\mathcal{L} - \mathcal{K})^{(n)}_{t,\bsig}$. With this fact and using Duhamel's principle, we obtain the following lemma.

\begin{lemma}[Integrated loop hierarchy, Lemma 5.3 of \cite{YY_25}] \label{Sol_CalL}
	First, \eqref{eq_L-Keee} is equivalent to the integral equation: for any stopping time $\tau\ge s$ with respect to the matrix Brownian motion $\{H_t\}$,
	\begin{align}\label{int_K-L_ST}
		({\cal L}-{\cal K})^{(n)}_{\tau,\boldsymbol{\sigma},{\ba}}&=({\cal L}-{\cal K})^{(n)}_{s,\boldsymbol{\sigma},{\ba}}+\int_{s}^{\tau}\p{\varTheta^{(n)}_{u,\boldsymbol{\sigma}}\circ({\cal L}-{\cal K})^{(n)}_{u,\boldsymbol{\sigma}}}_{{\ba}}\dd u+\sum_{l_\mathcal{K}=3}^n \int_s^\tau \Big[\mathcal{K}^{(\lenk)} \sim (\mathcal{L} - \mathcal{K})\Big]^{(n)}_{u, \boldsymbol{\sigma}, {\ba}}\dd u \nonumber\\
		&+ \int_s^\tau \mathcal{E}^{(\mathcal{L} - \mathcal{K}) \times (\mathcal{L} - \mathcal{K}),(n)}_{u, \boldsymbol{\sigma}, {\ba}}\dd u + \int_{s}^{\tau}\mathcal{E}^{\Gc,(n)}_{u, \boldsymbol{\sigma}, {\ba}}\dd u + \int_s^\tau \dd \mathcal{E}^{M,(n)}_{u, \boldsymbol{\sigma}, {\ba}} .
	\end{align}
	Second, applying Duhamel's principle to \eqref{eq_L-Keee}, we obtain the following \emph{integrated loop hierarchy}:
		\begin{align}\label{int_K-LcalE}
			& (\mathcal{L} - \mathcal{K})^{(\fn)}_{t, \boldsymbol{\sigma}, \ba}  =
			\left(\mathcal{U}^{(\fn)}_{s, t, \boldsymbol{\sigma}} \circ (\mathcal{L} - \mathcal{K})^{(\fn)}_{s, \boldsymbol{\sigma}}\right)_{\ba} + \sum_{l_\mathcal{K} =3}^\fn \int_{s}^t \left(\mathcal{U}^{(\fn)}_{u, t, \boldsymbol{\sigma}} \circ \Big[\cK^{(\lenk)}\sim (\mathcal{L} - \mathcal{K})\Big]^{(\fn)}_{u, \boldsymbol{\sigma}}\right)_{\ba} \dd u \nonumber \\
			&+ \int_{s}^t \left(\mathcal{U}^{(\fn)}_{u, t, \boldsymbol{\sigma}} \circ \mathcal{E}^{(\cL-\cK)\times(\cL-\cK),(\fn)}_{u, \boldsymbol{\sigma}}\right)_{\ba} \dd u  + \int_{s}^t \left(\mathcal{U}^{(\fn)}_{u, t, \boldsymbol{\sigma}} \circ \mathcal E^{\Gc,(\fn)}_{u, \boldsymbol{\sigma}}\right)_{\ba} \dd u + \int_{s}^t \left(\mathcal{U}^{(\fn)}_{u, t, \boldsymbol{\sigma}} \circ \dd \mathcal E^{M,(\fn)}_{u, \boldsymbol{\sigma}}\right)_{\ba} .
		\end{align}
	\end{lemma}

	Before returning to the proof of \Cref{lem:main_ind}, we introduce the following notation that corresponds to the quadratic variation of the martingale term in \eqref{def_Edif}.

	\begin{definition}[Quadratic variation loop]\label{def:CALE} 
		For $t\in [0,1]$ and $\bsig=(\sigma_1,\ldots,\sig_\fn)\in \{+,-\}^\fn$, we introduce the $(2\fn)$-dimensional \emph{quadratic variation tensor} for any \(\ba=(a_1,\ldots, a_\fn) \) and \(\ba'=(a_1',\ldots, a_\fn')\):
		\be\label{defEOTE}
		\left( \cal E \otimes \cal E  \right)^{M,(\fn)}_{t, \boldsymbol{\sigma}, \ba, \ba'} := 
		\sum_{k=1}^\fn \left( \cal E \otimes \cal E \right)^{M,(n;k)}_{t,  \boldsymbol{\sigma}, \ba, \ba'},\ \ \text{where}\ \ \left( \cal E \otimes \cal E \right)^{M,(n;k)}_{t,  \boldsymbol{\sigma}, \ba, \ba'} :
		=   W^d \sum_{b,b'} S^{\LK}_{bb'}{\cal L}^{(2\fn+2)}_{t, (\boldsymbol{\sigma}\otimes\overline\bsig)^{(k) },(\ba\otimes \ba')^{(k)}(b,b')}.
		\ee
		Here, ${\cal L}^{(2\fn+2)}$ denotes a $(2\fn+2)$-loop obtained by cutting the $k$-th edge of ${\cal L}^{(\fn)}_{t,\boldsymbol{\sigma},\ba}$ and then gluing it (with indices $\ba$) with its conjugate loop (with indices $\ba'$) along the newly introduced vertices $b$ and $b'$. Formally:
		\begin{align*}
			& {\cal L}^{(2\fn+2)}_{t, (\boldsymbol{\sigma}\otimes\overline\bsig)^{(k) },(\ba\otimes \ba')^{(k )}(b,b')}:=  \tr \bigg\{ \prod_{i=k}^\fn \left(G_{t}(\sigma_i) E_{a_i}\right) \cdot \prod_{i=1}^{k-1} \left(G_{t}(\sigma_i) E_{a_i}\right) \cdot G_t(\sigma_k) E_{b} G_t(-\sigma_k) \\
			&\qquad \times \prod_{i={1}}^{k-1} \left(E_{a_{k-i}'} G_{t}(-\sigma_{k-i}) \right)\cdot \prod_{i=k}^{\fn} \left(E_{a_{\fn+k-i}'} G_{t}(-\sigma_{\fn+k-i}) \right) E_{b'}\bigg\} ,
		\end{align*}
		where the notations $(\ba\otimes \ba')^{(k )}$ and $(\boldsymbol{\sigma}\otimes\overline \bsig)^{(k) }$ (with $\overline\bsig$ denoting $(-\sig_1,\ldots,-\sig_\fn)$) represent respectively
		\begin{align}\label{def_diffakn_k}
			(\ba\otimes \ba')^{(k )}(b,b')&=( a_k,\ldots, a_n, a_1,\ldots , a_{k-1}, b, a'_{k-1},\ldots, a_1', a_n',\ldots, a'_{k},b'),\nonumber \\
			(\boldsymbol{\sigma}\otimes \overline\bsig)^{(k )}&=(  \sigma_k, \ldots, \sigma_\fn,   \sigma_1,\ldots ,\sigma_{k}, -\sigma_{k}, \ldots, -\sigma_1 ,   -\sigma_\fn, \ldots, -\sigma_{k }).
		\end{align}
		
	\end{definition}

	Under the above notations, applying the Burkholder-Davis-Gundy inequality, we obtain the following lemma that provides high-moment bounds on the martingale term. 
	
	\begin{lemma}[Lemma 5.5 of \cite{YY_25}]\label{lem:DIfREP}
		Let $\tau$ be a stopping time with respect to the matrix Brownian motion $\{H_t\}$. Then, for any fixed $p\in \N$, there exists a constant $C_{n,p}$ such that
		\begin{align}\label{aaswtghh}
			\mathbb{E} \left [    \int_{s}^\tau \dd  \mathcal{E}^{M,(n)}_{u, \boldsymbol{\sigma}, {{\ba}}}  \right ]^{2p} 
			&\le C_{n,p}  
			\mathbb{E} \left(\int_{s}^\tau 
			\left( 
			\left( \mathcal{E} \otimes  \mathcal{E} \right)
			^{M,(n)}_{u, \boldsymbol{\sigma}, \ba,\ba}  
			\right)\dd u 
			\right)^{p},\\
			\label{alu9_STime}
			\mathbb{E} \left [    \int_{s}^t\left(\mathcal{U}^{(\fn)}_{u, t, \boldsymbol{\sigma}} \circ \dd\cal E^{M,(\fn)}_{u, \boldsymbol{\sigma}}\right)_{\ba} \right ]^{2p} 
			&\le C_{n,p}
			\mathbb{E} \left(\int_{s}^t 
			\left(\left(
			\mathcal{U}^{(\fn)}_{u, t,  \boldsymbol{\sigma}}
			\otimes 
			\mathcal{U}^{(\fn)}_{u, t,  \overline{\boldsymbol{\sigma}}}
			\right) \;\circ  \;
			\left( \cal E \otimes \cal E\right)^{M,(\fn)}
			_{u, \boldsymbol{\sigma}  }
			\right)_{\ba, \ba}\dd u 
			\right)^{p}.
		\end{align} 
		Here, $\mathcal{U}^{(\fn)}_{u, t,  \boldsymbol{\sigma}} \otimes    \mathcal{U}^{(\fn)}_{u, t,  \overline{\boldsymbol{\sigma}}}$ denotes the tensor product of the evolution kernel defined in \eqref{def_Ustz}: 
		$$
		\left[\left(
		\mathcal{U}^{(\fn)}_{u, t,  \boldsymbol{\sigma}}
		\otimes 
		\mathcal{U}^{(\fn)}_{u, t,  \overline{\boldsymbol{\sigma}}}
		\right)\circ \cal A\right]_{\ba,\ba'}=
		\sum_{\bfb,\bfb'\in(\Zn)^{\fn}}  \prod_{i=1}^\fn \left(\frac{1 - u\cdot M^{(\sig_i,\sig_{i+1})} S^\LK}{1 - t\cdot  M^{(\sig_i,\sig_{i+1})} S^\LK}\right)_{a_i b_i} \prod_{i=1}^\fn \left(\frac{1 - u\cdot M^{(-\sig_i,-\sig_{i+1})} S^\LK }{1 - t\cdot M^{(-\sig_i,-\sig_{i+1})} S^\LK }\right)_{a'_i b'_i} \mathcal{A}_{\bfb,\bfb'}
		$$
		for any $(2\fn)$-dimensional tensor ${\cal A}: (\Zn)^{2\fn}\to \mathbb C$ and $\bfb=(b_1,\ldots, b_\fn)$, $\bfb'=(b_1',\ldots, b_\fn')$.
		
	\end{lemma}

	\subsection{Proof of Step 2}\label{subsec:step2_pf}
	
	In this subsection, we focus on the proof of \eqref{Eq:Gdecay_w}. The local laws \eqref{Gt_bound_flow} and \eqref{Gt_avgbound_flow} then follow directly from \eqref{Eq:Gdecay_w} together with Lemma \ref{lem_GbEXP}. 
	Because of the hierarchical structure of \eqref{eq:mainStoflow}, establishing the pointwise stability estimate \eqref{Eq:Gdecay_w} would, in principle, require \emph{pointwise control} of all higher-order $G$-loops with length $n \ge 3$. Truncating this “pointwise” loop hierarchy (even if possible) would be substantially more involved than handling the “maximum” loop hierarchy used in our current approach.
	Fortunately, by combining two key technical ingredients—the diagrammatic techniques developed for the light-weight estimate and the $\contract$ inequality (see \Cref{ygdhmsgq0}) used in the martingale estimate—we can overcome this difficulty and effectively truncate the pointwise loop hierarchy at order $n=2$. This is one of the pivotal components of our proof: without it, both Step 2 and the $L^\infty$-stability of the tree approximation of the loop hierarchy, established in Steps 3 and 4 below, would break down. 
	
	We begin by defining the following tail functions to quantify the pointwise decay.

	\begin{definition}[Tail functions]\label{def: TTfunc}
		For any $t\in [0,1)$, we define a tail function $\cal T_{t}:[0,\infty)\to [0,\infty)$ as
		\begin{align}\label{defTUL}
			{\cal T}_{t}(r):=  B_{t,r}\cdot e^{- \left( 
				r/\ell_t \right) ^{1/2} },\quad \forall r\ge 0,
		\end{align}
		which corresponds to the $\Theta$-propagator bound in \eqref{prop:ThfadC}. 
		Given any $0\le \ell \le L$ and large constant $D>0$, we will also use a truncated tail function \smash{$\widetilde{{\cal T}}$} defined as: 
		\begin{align}\label{defWTTlD}
			\wT^{\ell}_{t,D}(r):= \max\left( {\cal T}_{t}(r\wedge \ell), \;  W^{-D}\right).
		\end{align}
	\end{definition}
	
	Note that ${\cal T}_{t}$ is a decreasing function in $r\ge 0$. Moreover, for $0\le r\le L$, when $1-t\ge \ilambda^2/L^{2}$, the term $(L^{d}|1-t|)^{-1}$ is always dominated by the term $(\ilambda^2+|1-t|)^{-1}/(r+1)^{d-2}$ in $B_{t,r}$, while when $1-t\le \ilambda^2/L^{2}$, $\ell_t=L$ and the exponential factor \smash{$\exp (- (r/\ell_t ) ^{1/2})$} in \eqref{defTUL} is of constant order.  
	The tail function $\cT_t$ satisfies the following elementary estimate, whose proof is provided in \Cref{sec:pfpropT}. 
	
	\begin{lemma}[Property of $\cal T$]\label{lem:propT}
		There exists a constant $C_d>0$ (depending only on $d$) such that the following bound holds 
		for any $0\le u\le t<1$ satisfying (i) $1-u\ge 1- t\ge \ilambda^2/L^{2}$, or (ii) $1-t\le 1-u \le \ilambda^2/L^{2}$:
		\begin{align}\label{TTT2}
			\sum_{c\in \Zn}{\cal T}_{u}(|a-c|)\cdot {\cal T}_{t}(| c -b |) \le \frac{C_d}{1-u}\cdot {\cal T}_{t}(|a-b|) ,\quad \forall a,b\in \Zn.
		\end{align}
	\end{lemma}

	Our core strategy for the proof of step 2 involves iterating the following self-improving estimates: for any large constant $D>0$, we have that  
	\begin{align}\label{awi2iks}
		\sup_{u\in [s,t]}\max_{\ba=(a,b)} \frac{\big|{\cal L}^{(2)}_{u,(+,-),\ba}\big|}{W^{-d} \widetilde{{\cal T}}^{\ell}_{u,D}\p{ |a-b|}}\prec 1  &\implies \sup_{u\in [s,t]}\max_{\ba=(a,b)}\max_{\bsig} \frac{\big|{(\cal L-\cal K)}^{(2)}_{u,\bsig,\ba}\big|}{W^{-d} \widetilde{{\cal T}}^{\ell} _{u,D}\p{|a-b|}}\prec \p{\frac{1-s}{1-t}}^{C_d}(W^{-d}B_{t,0})^{1/5} \nonumber
		\\
		&\implies \sup_{u\in [s,t]}\max_{\ba=(a,b)} \frac{\big|{\cal L}^{(2)}_{u,(+,-),\ba}\big|}{W^{-d} \widetilde{{\cal T}}^{\ell'}_{u,D}\p{ |a-b|}}\prec 1 ,
	\end{align}
	where the two scales $\ell'>\ell$ are roughly connected through the relation \({\cal T}_{t} \p{\ell'}\asymp  (W^{-d}B_{t,0})^{1/6}\cdot {\cal T}_{t}\p{\ell}.\)
	Here, the decrease from $1/5$ to $1/6$ accounts for the prefactor $(|1-s|/|1-t|)^{C_d}$. In other words, \eqref{awi2iks} shows that once we have obtained a sharp 2-$G$-loop bound up to the scale $\ell$, then after one iteration, we can push this bound to a slightly larger scale $\ell'$. After $\OO(1)$ iterations, we can improve the 2-$G$-loop bound up to a scale \smash{$\wh\ell$} with \smash{$\cal T_t(\wh\ell)=\OO(W^{-D})$}. At this scale, we have 
	\[{(\cal L-\cal K)}^{(2)}_{u,\bsig,\ba}\prec W^{-d}\widetilde{{\cal T}}^{\wh\ell}_{u,D}\p{|a-b|}\asymp W^{-d}\widetilde{{\cal T}}^{L}_{u,D}\p{|a-b|}.\]
	Together with the middle step of \eqref{awi2iks}, it implies the desired estimate \eqref{Eq:Gdecay_w}. 
	
	To establish \eqref{awi2iks}, we use equation \eqref{int_K-L_ST} with $n=2$:
	\begin{align}\label{LK_simple}
		({\cal L}-{\cal K})^{(2)}_{\tau,\boldsymbol{\sigma},{\ba}}&=({\cal L}-{\cal K})^{(2)}_{s,\boldsymbol{\sigma},{\ba}}+\int_{s}^{\tau}\p{\varTheta^{(2)}_{u,\boldsymbol{\sigma}}\circ({\cal L}-{\cal K})^{(2)}_{u,\boldsymbol{\sigma}} + \mathcal{E}^{(\mathcal{L} - \mathcal{K}) \times (\mathcal{L} - \mathcal{K})}_{u, \boldsymbol{\sigma}} + \mathcal{E}^{\Gc}_{u, \boldsymbol{\sigma}} }_{{\ba}}\dd u + \int_s^\tau \dd \mathcal{E}^{M}_{u, \boldsymbol{\sigma}, {\ba}} ,
	\end{align} 
	where we omit ``$(2)$" from some superscripts. 
	Define the following (random) control parameter:
	\begin{align}\label{defCALJ}
		\wh{\cal J}^{\ell}_{u,D}:=\max_{\ba=(a,b)} \max_{\bsig}{\big|\left({\cal L-\cal K}\right)^{(2)}_{u,\bsig, \ba}\big|}\big/\big[ W^{-d} \widetilde{{\cal T}}^{\ell}_{u,D}\p{|a-b|}\big].
	\end{align}
	The first two terms on the RHS of \eqref{LK_simple} can be bounded as in the following lemma, whose proof is postponed to  
	\Cref{subsec:pf_lem:newKLK}.
	
	\begin{lemma}\label{lem:newKLK} 
		In the setting of \Cref{lem:main_ind}, suppose the weak local law \eqref{Gtmwc} holds. Then, there exists a constant $C>0$ (depending only on $c_d$ and $C_d$ in \eqref{prop:ThfadC}) such that the following estimates hold with high probability for all $0\le \ell \le L$:
		\begin{align}\label{juwo2=klk}      
			\max_{\ba=(a,b)} \max_{\bsig}\abs{\p{\varTheta^{(2)}_{u,\boldsymbol{\sigma}}\circ({\cal L}-{\cal K})^{(2)}_{u,\boldsymbol{\sigma}}}_{{\ba}}}\Big/\br{W^{-d}\widetilde{{\cal T}}^{\ell}_{u,D}\p{|a-b|}} & \le  \frac{C}{1-u}\cdot  \wh{\cal J}^{\ell}_{u, D }  ,\\
			\label{juwo=Lklk}      
			\max_{\ba=(a,b)}\max_{\bsig} \left|\mathcal{E}^{(\mathcal{L} - \mathcal{K}) \times (\mathcal{L} - \mathcal{K})}_{u, \boldsymbol{\sigma}, {\ba}} \right|\Big/\br{W^{-d}\widetilde{{\cal T}}^{\ell}_{u,D}\p{ |a-b|}} &\le  \frac{C}{1-u} \cdot \left( \wh{\cal J}^{\ell}_{u,D} + \left(\wh{\cal J}^{\ell}_{u,D} \right)^2\cdot {\bf 1}_{\ell\ge 1}\right)  .
		\end{align}  
	\end{lemma}

	Next, we bound the light-weight term $\cal E^{\Gc}$ using the following two lemmas. \Cref{lem:LWterm} is applied to recover the polynomial decay $B_{t,\cdot}$ in \eqref{eq_B_param}, while \Cref{lem: EWGn2_N} is used to establish the tail behavior described by the tail function in \eqref{defWTTlD}.

	\begin{lemma}[Light-weight estimate: $B$-bound]\label{lem:LWterm}
		In the setting of \Cref{lem:main_ind}, assume that the bounds in \eqref{initialGT2} hold for some constant $\e_0>0$ and deterministic control parameter $ W^{-d/2}\le \Psi_t \le W^{-\e_0}$. Suppose we have the estimate 
		\be\label{eq:LW_assm}
		{\cal L}^{(2)}_{t,\bsig,(a,b)}\prec \Psi_t^2(|a-b|),\quad \forall \ \bsig\in\{(+,-),(-,+)\}, \ a,b\in \Zn, \ee
		for a class of deterministic parameters $0<\Psi_t(|a-b|)\le W^{-\e_0}$. Without loss of generality, assume that $\Psi_t(\ell)$ is monotonically decreasing for $\ell\ge 0$ (otherwise, one may replace it by the function $\sup_{\ell'\ge\ell}\Psi_t(\ell')$). 
		Suppose there exist constants $C_1,C_2>1$ such that the following relations hold for any constant $C>1$: 
		\be\label{eq:Psi}
		W^{-d/2}\lesssim \Psi_t(0)\asymp\Psi_t(\ell) \ \  \forall \ 0\le \ell\le C,\quad \text{and}\quad \Psi_t(\ell_1)/\Psi_t(\ell_2)\le C_1(\ell_2/\ell_1)^{C_2} \ \ \forall \ell_2\ge \ell_1\ge 1.
		\ee
		Then, the following estimate holds: 
		\be\label{eq:LW_conclusion}
		\mathcal{E}^{\Gc,(2)}_{t, \boldsymbol{\sigma}, (a,b)}\prec \frac{1}{\eta_t}\Psi_t(0) \cdot \Psi^2_t\p{|a-b|}, 
		\quad \forall \ \bsig\in \{+,-\}^2, \ a,b\in \Zn.
		\ee
		In particular, if we take \(\Psi_t(|a-b|)=  (W^{-c_0}B_{t, |a-b|\wedge K})^{1/2}\) for a constant $c_0>0$ and some $0\le K\le L$, then 
		\be\label{eq:LW_conclusion2}
		\mathcal{E}^{\Gc,(2)}_{t, \boldsymbol{\sigma}, (a,b)}\prec \frac{1}{\eta_t}\p{W^{-c_0}B_{t,0}}^{1/2}\cdot W^{-c_0} B_{t,|a-b|\wedge K}  .
		\ee
	\end{lemma}
	
	\begin{lemma}[Light-weight estimate: $\cal T$-bound]\label{lem: EWGn2_N} 
		Given any $t \in [0,1)$, assume that the bounds in \eqref{initialGT2} hold for some constant $\e_0>0$ and deterministic control parameter $W^{-d/2}\le \Psi_t\le W^{-\e_0}$. Suppose for some $0\le \ell \le (\log W)^{10}\ell_t$, the following estimate holds for any large constant $D>0$:
		\be\label{eq:LW_assm_exp}
		{\cal L}^{(2)}_{t,\bsig,(a,b)}\prec W^{-d}\wT_{t,D}^{\ell}(|a-b|),\quad \forall \ \bsig\in\{(+,-),(-,+)\}, \ a,b\in \Zn . \ee
		Then, the following estimate holds for any large constant $D>0$: 
		\be\label{eq:LW_conclusion_exp}
		\mathcal{E}^{\Gc,(2)}_{t, \boldsymbol{\sigma}, (a,b)} \prec \frac{1}{\eta_t}(W^{-d}B_{t,0})^{1/2} \cdot W^{-d} \wT_{t,D}^{\ell}(|a-b|)  , 
		\quad \forall \ \bsig\in \{+,-\}^2, \ a,b\in \Zn.
		\ee
	\end{lemma}
	\begin{remark}\label{rmk:poly_exp}
		The second condition in \eqref{eq:Psi} implies that the parameter $\Psi_t(r)$ decays polynomially in $r$. In particular, the parameter \smash{$W^{-d}\wT_{t,D}^{\ell}(r)$} in \eqref{eq:LW_assm_exp} does not satisfy this condition because of the exponential factor \smash{$\exp\p{-(r/\ell_t)^{1/2}}$} in \eqref{defTUL}. Consequently, \Cref{lem: EWGn2_N} cannot be deduced directly from \Cref{lem:LWterm}. The corresponding proof is technically more delicate, since factors of the form $\Psi_t(c|a-b|)$, for some constant $c\in(0,1)$, can no longer be replaced by $\Psi_t(|a-b|)$ as in the polynomial-decay setting.
		In \Cref{lem: EWGn2_N}, the assumption $\ell \le (\log W)^{10}\ell_t$ is made without loss of generality, since $\cal T_{t}(\ell) \le W^{-D}$ whenever $\ell > (\log W)^{10}\ell_t$. 
	\end{remark}
	
	The proofs of \Cref{lem:LWterm,lem: EWGn2_N} are based on bounding the high moments using Gaussian integration by parts and some diagrammatic tools developed in \cite{yang2021delocalization,yang2021random}. We will present the details in \Cref{Sec:graph}.  
	Finally, the martingale term is bounded as in the following lemma, whose proof will be postponed to \Cref{subsec:pf_lem: EMn2_N}. 
	
	\begin{lemma}[Martingale estimate]\label{lem: EMn2_N} 
		Given any $t\in[0,1)$, suppose the setting of \Cref{lem:LWterm} holds. Consider the term \smash{$\left( \cal E\otimes  \cal E \right)^{M}_{t, \boldsymbol{\sigma}, \ba,\ba}\equiv \left( \cal E\otimes  \cal E \right)^{M,(2;k)}_{t, \boldsymbol{\sigma}, \ba,\ba}$} defined in \eqref{defEOTE} for any \smash{$\ba=(a,b)\in (\Zn)^2$}, $\bsig=(\sig_1,\sig_2)\in \{+,-\}^2$, and $k\in\{1,2\}$.
		First, the following estimate holds:
		\be\label{eq:MG_conclusion}
		\left( {\cal E}\otimes  {\cal E} \right)^{M}_{t, \boldsymbol{\sigma}, \ba, \ba} \prec \frac{1}{\eta_t}\Psi_t(0) \cdot  \Psi_t^4\p{ |a-b|} .
		\ee
		In particular, if we take \(\Psi_t(|a-b|)= (W^{-c_0}B_{t, |a-b|\wedge K})^{1/2}\) for a constant $c_0>0$ and some $0\le K\le L$, then 
		\be\label{eq:MG_conclusion2}
		\left( \cal E\otimes  \cal E \right)^{M}_{t, \boldsymbol{\sigma}, \ba, \ba} \prec \frac{1}{\eta_t}\p{W^{-c_0}B_{t,0}}^{1/2}\cdot  \p{W^{-c_0} B_{t,|a-b|\wedge K}}^2.
		\ee
		Second, suppose for some $0\le \ell\le (\log W)^{10}\ell_t$, the estimate \eqref{eq:LW_assm_exp} holds for any large constant $D>0$. Then, the following estimate holds for any large constant $D>0$:
		\be\label{eq:MG_conclusion3}
		(\mathcal{E} \otimes \mathcal{E})^{M}_{t,\bsig, \ba,\ba} \prec \frac{1}{\eta_t} \left[\left(W^{-d}B_{t,0}\right)^{1/2} +  \left(\wh{\mathcal{J}}_{t,D}^{\ell}\right)^3\right]\cdot \left(W^{-d} \wT^{\ell}_{t,D}(|a-b|)\right)^2.
		\ee
	\end{lemma}
	\begin{remark}\label{rmk_bottleneck}
		We remark that the martingale estimate is the main bottleneck of our proof, preventing us from improving the exponent $1/5$ in \eqref{Eq:Gdecay_flow} to the optimal value $1$. This might be overcome by incorporating certain (not necessarily sharp) spatial decay properties of higher-order $G$-loops into the loop hierarchy analysis. However, such a refinement is far from a straightforward extension of the current argument. Even if achievable, it would entail substantial additional technical complexity while yielding only a ``marginal" improvement in the quantum diffusion estimates. Since this refinement does not affect the main results---delocalization (\Cref{MR:decol}), local laws (\Cref{MR:locSC}), QUE (\Cref{MR:QUE}), and bulk universality (\Cref{Thm: B_Univ})---we leave this direction for future investigation.
	\end{remark}
	
	With the above Lemmas \ref{lem:newKLK}--\ref{lem: EMn2_N} as inputs, we are ready to control the terms in equation \eqref{LK_simple} and complete Step 2 for the proof of \Cref{lem:main_ind}. 
	
	\begin{proof}[\bf Step 2: Proof of \eqref{Gt_bound_flow}--\eqref{Eq:Gdecay_w}] 
		Suppose the estimate \eqref{Eq:Gdecay_w} has been established. Then, using \eqref{Kn2sol} and \eqref{prop:ThfadC}, we obtain 
		\be\label{eq:kn2sol_decay}
		{\cal K}^{(2)}_{u, (-,+), (a_1,a_2)}\prec W^{-d}B_{u,|a_1-a_2|}, \quad \forall a_1,a_2\in \Zn.
		\ee
		Together with \eqref{Eq:Gdecay_w}, this implies---provided $\fc_d$ in \eqref{con_st_ind} is chosen sufficiently small---the 2-$G$-loop bound
		\be\label{eq:L2_decay}
		{\cal L}^{(2)}_{u, (-,+), (a_1,a_2)}\prec W^{-d}B_{u,|a_1-a_2|}, \quad \forall a_1,a_2\in \Zn.
		\ee
		Applying the weak local law \eqref{Gtmwc} from Step 1 to verify the first condition in \eqref{initialGT2}, and then invoking \Cref{lem_GbEXP}, we obtain the entrywise local law \eqref{Gt_bound_flow} and the averaged local law \eqref{Gt_avgbound_flow}.
		
		It remains to establish the estimate \eqref{Eq:Gdecay_w}. From the bound \eqref{lRB1} proved in Step 1, we know 
		\begin{equation} \label{lokis2}
			\max_{\boldsymbol{\sigma},{\ba}}\abs{{\cal L}^{(2)}_{u,\bsig,{\ba}}}\prec \frac{1-s}{1-t}(W^{-d}B_{t,0}) =: W^{-c_0},\quad \forall u\in [s,t],
		\end{equation}
		where $c_0\equiv c_0(W)\gtrsim 1$ under the condition \eqref{con_st_ind}. We first prove the following maximum bound:
		\be\label{eq:opt_L2}
		\max_{\boldsymbol{\sigma},{\ba}}\abs{\left({\cal L}-{\cal K}\right)^{(2)}_{u,\bsig,{\ba}}}\prec  W^{-d}B_{u,0},\quad \forall u\in [s,t].
		\ee
		Using \Cref{lem:newKLK} (for the case $\ell=0$), \Cref{lem:LWterm}, and the inductive hypothesis \eqref{Eq:L-KGt+IND} with $n=2$, we obtain from the flow \eqref{LK_simple} (with $\tau=t$) that 
		\begin{align}\label{l>0EQ}
			\max_{\bsig,\ba}\left|\left({\cal L}-{\cal K}\right)^{(2)}_{t,\bsig, {\ba}}\right|
			\le \int_{s}^t \frac{C_0}{1-u}\max_{\bsig,\ba}\left|\left({\cal L}-{\cal K}\right)^{(2)}_{u,\bsig,\ba}\right| \dd u  +\OO_\prec \p{(W^{-d}B_{s,0})^2+W^{-\frac32c_0}} + \abs{\int_s^t \dd \mathcal{E}^{M}_{u, \boldsymbol{\sigma}, {\ba}}}   
		\end{align}
		for a constant $C_0>0$ depending only on $c_d$ and $C_d$ in \eqref{prop:ThfadC}. Combining the estimate \eqref{eq:MG_conclusion} in \Cref{lem: EMn2_N} with \Cref{lem:DIfREP}, and applying Markov's inequality, we obtain that 
		\[\abs{\int_s^t \dd \mathcal{E}^{M}_{u, \boldsymbol{\sigma}, {\ba}}} \prec \p{\int_{s}^{t}\frac{W^{-5c_0/2}}{\eta_u}\dd u}^{1/2}\prec W^{-\frac54c_0}.\]
		Together with \eqref{l>0EQ}, it gives
		\be\label{eq:Gronwall_2L_max}\max_{\bsig,{\ba}}\left|\left({\cal L}-{\cal K}\right)^{(2)}_{t, \bsig, {\ba} }\right|\le \int_{s}^t \frac{C_0}{1-u}\max_{\bsig,\ba}\left|\left({\cal L}-{\cal K}\right)^{(2)}_{u,\bsig,\ba}\right| \dd u  +\OO_\prec \p{(W^{-d}B_{s,0})^2+W^{-\frac54c_0 }} .
		\ee
		Recall the following classical Gr{\"o}nwall's inequality, where $\beta$ is non-negative and $\al$ is non-decreasing:
		\be\label{Gronwall_inequality}
		\begin{aligned}
			&f(t) \leq \alpha(t)+\int_s^t \beta(u) f(u) \mathrm{d} u \implies f(t) \leq \alpha(t) \exp \left(\int_s^t \beta(u) \mathrm{d} u\right).
		\end{aligned}
		\ee
		Applying it to \eqref{eq:Gronwall_2L_max} yields that  
		\be\label{eq:L-K2max} \max_{\bsig,\ba}\left|\left({\cal L}-{\cal K}\right)^{(2)}_{t,\bsig,\ba}\right|\prec 
		\p{\frac{1-s}{1-t}}^{C_0}\left((W^{-d}B_{t,0})^2 + W^{-\frac54c_0}\right)\le W^{-d}B_{t,0},
		\ee
		where, in the last step, we use that 
		\[ \p{\frac{1-s}{1-t}}^{C_0}W^{-\frac54c_0} \le W^{-d}B_{t,0} \cdot \p{\frac{1-s}{1-t}}^{C_0+5/4}\p{W^{-d}B_{t,0}}^{1/4} \ll W^{-d}B_{t,0} \]
		as long as we choose $\fc_d$ in \eqref{con_st_ind} sufficiently small depending on $C_0$. Combining \eqref{eq:L-K2max} with \eqref{eq:bcal_k} (for the $n=2$ case), we conclude \eqref{eq:opt_L2} at $u=t$. Obviously, the same result applies to each $u\in [s,t]$, and a standard $N^{-C}$-net and perturbation argument extends it uniformly to all $u\in[s,t]$, which concludes \eqref{eq:opt_L2}.

		We can obtain from \eqref{eq:opt_L2} that for $\ell^{(0)}=0$ and any large constant $D>0$,
		\begin{align}\label{ksjjuw}
			\max_{\ba=(a,b)}\max_{\bsig}\big|\cal L^{(2)}_{u, \bsig, {\ba} }\big|\Big/\p{W^{-d} \widetilde{\cal {T}}^{\ell^{(0)}}_{u,D}\p{|a-b|}} \prec 1, \quad \forall u\in[s,t],
		\end{align}
		where we also use that 
		\be\label{eq:simpleboundK}\cK^{(2)}_{u,\bsig,\ba}\prec W^{-d}\wT^{L}_{u,D}\p{|a-b|}\ee 
		by \eqref{prop:ThfadC} and the definition of 2-$\cK$-loop in \eqref{Kn2sol}. 
		Now, fix a large constant $D>0$. Suppose we have the following bound for a collection of length scales $\{K_u\ge 0:u\in [s,t]\}$:
		\begin{align}\label{kwr3juw}
			\max_{\ba=(a,b)}\max_{\bsig}\big|{\cal L}^{(2)}_{u,\bsig,{\ba} }\big|\Big/\p{W^{-d}\widetilde{\cal {T}}^{K_u}_{u,D}\p{|a-b|}}\prec 1,\quad \forall u\in[s,t]. 
		\end{align}
		Moreover, assume that $\cal T_u(K_u)$ is non-decreasing in $u$, i.e., 
		\be\label{eq:monotone_Ku}
		\cal T_{u}(K_u) \le \cal T_v(K_v), \quad  \wT_{u,D}^{K_u}(r) \le \wT_{v,D}^{K_v}(r),\quad \forall s\le u\le v \le t, \ r\ge 0.
		\ee
		Our goal is to establish the following estimate for a constant $C_0>0$ depending only on $c_d$ and $C_d$ in \eqref{prop:ThfadC}:
		\be\label{eq:Gronwall_dervJuD}
		\wh{\mathcal{J}}_{u,D}^{K_u} \prec\left(|1-s|/|1-u|\right)^{C_0}(W^{-d} {B} _{t,0})^{1/5}  ,\quad \forall u\in [s,\tau],  
		\ee
		where $\tau$ is a stopping time defined as
		\be\label{eq:def2_stopping} 
		\tau :=t\wedge T,\quad \text{with}\quad T:=\inf\left\{u\ge s: \wh{\cal J}_{u,D}^{K_u} \ge (W^{-d}B_{u,0})^{1/6}\right\}.\ee

		For the proof of \eqref{eq:Gronwall_dervJuD}, using \Cref{lem: EWGn2_N}, we obtain that 
		\be\label{eq:boundGcterm}
		\int_{s}^{\tau}\mathcal{E}^{\Gc}_{u, \boldsymbol{\sigma}, {\ba}}\dd u \prec \int_{s}^\tau \frac{1}{\eta_u} (W^{-d}B_{u,0})^{1/2} \cdot \p{W^{-d} \wT_{u,D}^{K_u}(|a-b|)}\dd u  \prec (W^{-d}B_{\tau,0})^{1/2} \cdot \p{W^{-d} \wT_{\tau,D}^{K_\tau}(|a-b|)},
		\ee
		where, in the second step, we also use the fact that $B_{u,0}$ and $\wT_{u,D}^{K_u}(r)$ are monotonically increasing in $u$. Similarly, using \Cref{lem: EMn2_N} along with \Cref{lem:DIfREP} and the definition of the stopping time $\tau$, we obtain that 
		\be\label{eq:boundmgterm}
		\begin{split}
			\int_s^\tau \dd \mathcal{E}^{M}_{u, \boldsymbol{\sigma}, {\ba}} &\prec  \Big[\left(W^{-d}B_{\tau,0}\right)^{1/4}  + \sup_{u\in[s,\tau]}\p{\wh{\cal J}_{u,D}^{K_u}}^{3/2}\Big]\cdot \left(W^{-d} \wT^{K_\tau}_{\tau,D}(|a-b|)\right) \\
			&\lesssim \left(W^{-d}B_{\tau,0}\right)^{1/4} \cdot \left(W^{-d} \wT^{K_\tau}_{\tau,D}(|a-b|)\right).\end{split}
		\ee
		Applying \Cref{lem:newKLK}, \eqref{eq:boundGcterm}, and \eqref{eq:boundmgterm} to equation \eqref{LK_simple} yields that 
		\begin{align}\label{LKu2p2mmvj}
			\left|\left({\cal L}-{\cal K}\right)^{(2)}_{\tau,\bsig,\ba}\right| \le &~
			\left|\left({\cal L}-{\cal K}\right)^{(2)}_{s, \bsig,\ba}\right|+ \int_s^\tau  \frac{C_0}{1-u} \wh{\mathcal{J}}^{K_u}_{u,D} \cdot W^{-d} \widetilde{{\cal T}}^{K_u} _{u,D}\p{|a-b|} \dd u  
			\\\nonumber
			&~+\OO_\prec\p{\p{W^{-d} {B}_{\tau,0}}^{1/4}\cdot W^{-d} \widetilde{{\cal T}}^{K_\tau}_{\tau,D}\p{|a-b|}}.
		\end{align}
		From this equation, using the induction hypothesis \eqref{Eq:Gdecay+IND} at time $s$ and the definition \eqref{defCALJ}, we obtain that 
		$$
		\wh{\mathcal{J}}^{K_\tau}_{\tau,D}\le  \int_s^\tau \frac{C_0}{1-u} \wh{\mathcal{J}}^{K_u}_{u,D}  \dd u + \OO_\prec \p{(W^{-d} {B} _{\tau,0})^{1/5}}.
		$$
		Applying Gr{\"o}nwall's inequality \eqref{Gronwall_inequality} again to this equation, we derive \eqref{eq:Gronwall_dervJuD}.

		By the induction hypothesis \eqref{Eq:Gdecay+IND} at time $s$, the stopping time $T$ defined in \eqref{eq:def2_stopping} satisfies $T>s$ with high probability. In \eqref{eq:Gronwall_dervJuD}, using again \eqref{con_st_ind} with $\fc_d$ chosen sufficiently small depending on $C_0$, we can ensure that with high probability, \smash{\(\wh{\mathcal{J}}_{u,D}^{K_u} \ll (W^{-d} {B} _{u,0})^{1/6}\)} for all \(u\in [s,\tau]\). This implies that $T\ge t$ with high probability, so that \eqref{eq:Gronwall_dervJuD} holds for all $u\in[s,t]$. 
		Next, we define a new collection of parameters $\{K_u'\ge 0:u\in [s,t]\}$ as follows: for each $u\in[s,t]$, let $K_u'$ be the unique positive solution to  
		\be\label{eq:def_ell1} 
		\cal T_u(K_u')=\cal T_u(K_u)\cdot (W^{-d} {B}_{u,0})^{1/6} . \ee
		Under this definition, the monotonicity relation \eqref{eq:monotone_Ku} continues to hold for the family $\{K_u'\}_{u\in[s,t]}$. Furthermore, from \eqref{eq:simpleboundK} and \eqref{eq:Gronwall_dervJuD}, we obtain
		\[ {\cal L}^{(2)}_{u,\bsig, \ba}\prec W^{-d} \p{\widetilde{{\cal T}}^{K_u}_{u,D}(|a-b|)\cdot (W^{-d} {B} _{u,0})^{1/6} + \wT_{u,D}^{L}(|a-b|)} \lesssim W^{-d}\widetilde{{\cal T}}^{K_u'}_{u,D}(|a-b|),\]
		where, in the second step, we use \eqref{eq:def_ell1} along with the relation $\widetilde{{\cal T}}^{K_u}_{u,D}(|a-b|)\asymp \widetilde{{\cal T}}^{L}_{u,D}(|a-b|)$ for $K_u\ge L$. We now use this as the input (with $K_u$ replaced by $K_u'$) in \eqref{kwr3juw} for the next iteration of the above argument, and show that \eqref{kwr3juw} continues to hold at an even larger scale $K_u''$. From \eqref{eq:def_ell1}, it follows that, for any fixed $D>0$, after performing at most $\OO(1)$ iterations (where the number of iterations depends on $D$), we reach scales $\{K_u:u\in[s,t]\}$ such that $\cal T_u(K_u)\le W^{-D}$ for all $u\in[s,t]$. Consequently, we have 
		\[\wT_{u,D}^{K_u}(|a-b|)\asymp \wT_{u,D}^{L}(|a-b|), \quad \forall u\in [s,t], \ a,b\in \Zn.\]
		Finally, applying \eqref{eq:Gronwall_dervJuD} once more yields \eqref{Eq:Gdecay_w}.
	\end{proof}
	
	\begin{remark}
		We briefly explain why the stability estimate \eqref{Eq:Gdecay_w} \emph{cannot} be established continuously along the flow using Duhamel's formula, as was done in the $d \in \{1,2\}$ case \cite{YY_25,DYYY25,erdHos2025zigzag}.
		For a large constant $D>0$, these works consider the random control parameter \smash{$J_{u,D}\equiv \wh{\cal J}^{L}_{u,D}$} (defined in \eqref{defCALJ}) at the full-system scale $\ell=L$.   
		Roughly speaking, the previous approach begins with an initial high-probability bound $J_{s,D}\le W^{-c}$ (for some constant $c>0$ at $u=s$) and then propagates this bound along a $W^{-C}$-net of $[s,t]$. The essential idea is that if $J_{u,D}\le W^{-c}$ holds, a perturbative argument yields a slightly weaker bound $J_{u',D}\le W^{-c+\e}$ at $u'=u+W^{-C}$. This can then be bootstrapped back to $J_{u',D}\le W^{-c}$ using the loop hierarchy \eqref{int_K-LcalE}, at the cost of a probability loss $W^{-D}$, which remains acceptable since the total accumulation over the net is only $W^{-D+C}$.
		However, this mechanism breaks down once we attempt to incorporate the light-weight estimates from \Cref{lem:LWterm,lem: EWGn2_N}. They rely on a moment method that degrades the probability bound much more severely---from $W^{-D}$ in the original assumptions \eqref{eq:LW_assm} (resp.~\eqref{eq:LW_assm_exp}) to $W^{-\delta D}$ in \eqref{eq:LW_conclusion} (resp.~\eqref{eq:LW_conclusion_exp}) for some constant $\delta <1$. As a result, \Cref{lem:LWterm,lem: EWGn2_N} can each be applied $\OO(1)$ times only, and thus the flow-based propagation strategy is no longer applicable. 
		Instead, we must prove the light-weight estimate simultaneously for all $u\in[s,t]$, starting from the weak local law \eqref{Gtmwc} established in Step 1. This necessitates the self-improving argument developed in \eqref{awi2iks}.

		Moreover, since our initial control of $J_{u,D}$ is too weak (it can be as large as $W^D$), the quadratic term \smash{$\cal E^{(\cL-\cK)\times (\cL-\cK)}$} cannot be effectively bounded by $J_{u,D}^2$, unlike in \cite{YY_25,DYYY25,erdHos2025zigzag}. We must therefore rely on the bound \eqref{juwo=Lklk}. However, this estimate invalidates the Duhamel-based argument in dimensions $d\in\{1,2\}$: when inserted into Duhamel’s formula, it contributes a term of size \smash{$(|1-s|/|1-t|)\cdot W^{-d}\widetilde{{\cal T}}^{\ell}_{u,D}\p{|a-b|}$}, which already breaks the first step of \eqref{awi2iks}. As a result, we must treat \smash{$\cal E^{(\cL-\cK)\times (\cL-\cK)}$} as a main term rather than an error term, and reorganize the proof around a new Grönwall-type argument.
		
	\end{remark}

	\subsection{Proof of Lemma \ref{lem:newKLK}}\label{subsec:pf_lem:newKLK}
	
	Without loss of generality, it suffices to consider the case where 
	\begin{align}\label{uwu3po2[]}
		{\cal T}_{u}\p{\ell}\ge  W^{-D}. 
	\end{align} 
	Otherwise, we can find $K\le \ell$ such that ${\cal T}_{u}\p{K} = W^{-D}$, where there is always \(\widetilde{{\cal T}}^{\ell}_{u,D}\p{|a-b|}=\widetilde{{\cal T}}^{K}_{u,D}\p{|a-b|}\) by definition \eqref{defWTTlD}. 
	For the proof of \eqref{juwo2=klk}, by the definitions \eqref{def:op_thn}, \eqref{defWTTlD}, and \eqref{defCALJ}, we have that 
	\begin{align} \label{uu2j2ois} W^d\abs{\p{\varTheta^{(2)}_{u,\boldsymbol{\sigma}}\circ({\cal L}-{\cal K})^{(2)}_{u,\boldsymbol{\sigma}}}_{{\ba}}}&\; \le \;C \wh{\cal J}^{\ell}_{u,D}\sum_c \big|\big(\Theta_{t}^{(\sig_1,\sig_2)}M^{(\sig_1,\sig_2)}\big)_{ac}\big|\cdot \widetilde{{\cal T}}^{\ell}_{u,D}\p{|c-b|} \nonumber \\
		&\; \le \;  C \wh{\cal J}^{\ell}_{u,D}\sum_{c:|c-b|\ge \ell} \big|\Theta_{t,ac}^{(\sig_1,\sig_2)}\big|\cdot {\cal T}_{u}\p{\ell} + C \wh{\cal J}^{\ell}_{u,D}\sum_{c:|c-b|< \ell} \cal T_{u}(|a-c|){\cal T}_{u}\p{|c-b|} \nonumber\\
		&\; \le \;  \frac{C}{1-u} \wh{\cal J}^{\ell}_{u,D}\br{{\cal T}_{u}\p{\ell} +\cal T_{u}(|a-b|)} \le \frac{C}{1-u} \wh{\cal J}^{\ell}_{u,D}\wT_{u,D}^{\ell}(|a-b|),
	\end{align} 
	where, in the second step, we use the estimate \eqref{prop:ThfadC}, and the facts that under the condition \eqref{uwu3po2[]}, \be\label{eq;facts}\widetilde{{\cal T}}^{\ell}_{u,D}\p{|c-b|}\le {{\cal T}} _{u}\p{\ell}\ \ \text{for}\ \ |c-b|\ge \ell,\quad \widetilde{{\cal T}}^{\ell}_{u,D}\p{|c-b|}\le {{\cal T}} _{u}\p{|c-b|}\ \ \text{for}\ \ |c-b|<\ell,
	\ee
	and in the third step, we use \eqref{eq:THETAinftinf} and \eqref{TTT2}. This concludes \eqref{juwo2=klk}.

	The proof of \eqref{juwo=Lklk} follows a similar argument. We have that for any $\bsig_1,\bsig_2\in\{+,-\}^2$,  
	\begin{align}\label{i2kk2zgg}
		&~ W^d \sum_{c_1,c_2} \left|  \p{\cal L-\cal K}_{u,\bsig_1,(a,c_1)}^{(2)} S^{(\sB)}_{c_1c_2} \p{\cal L-\cal K}^{(2)}_{u,\bsig_2,(c_2,b)}\right|\\
		\;\le \;&~  \sum_{c } \left(|{\cal L}^{(2)}_{u,\bsig_1,(a,c)}|+|{\cal K}^{(2)}_{u,\bsig_1,(a,c)}|  \right)\cdot 
		\max_{|c-b|\ge (\ell-1)_+}
		\wh{\cal J}^{\ell}_{u,D} \widetilde{\mathcal{T}}_{u,D}^{\ell}\p{|c-b|}\nonumber \\ \nonumber
		\;+ \;&~ \sum_{c } \left(|{\cal L}^{(2)}_{u,\bsig_2,(c,b)}|+|{\cal K}^{(2)}_{u,\bsig_2,(c,b)}|  \right)\cdot 
		\max_{|c-a|\ge  (\ell - 1)_+}
		\wh{\cal J}^{\ell}_{u,D} \widetilde{\mathcal{T}}_{u,D}^{\ell}\p{|a-c|} 
		\\\nonumber
		\;+ \;&~ \mathbf 1_{\ell\ge 1}\cdot W^{-d}\left(\wh{\cal J}^{\ell}_{u,D}\right)^2
		\sum_{c:|c-a|\vee |c-b| < \ell} \widetilde{\mathcal{T}}_{u,D}^{\ell}\p{|a-c|} \widetilde{\mathcal{T}}_{u,D}^{\ell}\p{|c-b|} .\nonumber
	\end{align}
	For the first term on the RHS of \eqref{i2kk2zgg}, applying the Cauchy–Schwarz inequality and Ward's identities \eqref{WI_calL} and \eqref{WI_calK}, we get that with high probability,
	\begin{align*}
		\sum_{c } \left(|{\cal L}^{(2)}_{u,\bsig_1,(a,c)}|+|{\cal K}^{(2)}_{u,\bsig_1,(a,c)}|  \right) &\le 
		\sum_{c } \p{ {\cal L}^{(2)}_{u,(-,+),(a,c)}+{\cal K}^{(2)}_{u,(-,+),(a,c)}}\\
		&=\frac{\im m+\max_a \im \tr\p{G_uE_a}}{W^{d} \eta_u} =\frac{1+\oo(1)}{W^{d} (1-u)}, 
	\end{align*}
	where we use the weak local law \eqref{Gtmwc} and the relation \eqref{eta} in the last step. With this estimate, we can bound the first term on the RHS of \eqref{i2kk2zgg} as follows with high probability:
	$$
	\frac{1+\oo(1)}{W^d(1-u)} \max_{|c-b|\ge (\ell-1)_+}
	\wh{\cal J}^{\ell}_{u,D} \widetilde{\mathcal{T}}_{u,D}^{\ell}\p{|c-b|} \le \frac{C}{W^d(1-u)} \wh{\cal J}^{\ell}_{u,D} {\mathcal{T}}_{u}\p{\ell} ,
	$$
	where we have used \eqref{eq;facts} and $\widetilde{\mathcal{T}}_{u}\p{K+1}\asymp \widetilde{\mathcal{T}}_{u}\p{K}$ for any $K\ge 0$. The second term on the RHS of \eqref{i2kk2zgg} can be bounded in the same way. For the last term on the RHS of \eqref{i2kk2zgg}, using \eqref{eq;facts}, we bound it by 
	\be\label{eq:i2kk2zgglast}\mathbf 1_{\ell\ge 1} \cdot CW^{-d}\left(\wh{\cal J}^{\ell}_{u,D}\right)^2
	\sum_{c}{\mathcal{T}}_{u}\p{|a-c|}  {\mathcal{T}}_{u}\p{|c-b|} \le \mathbf 1_{\ell\ge 1} \cdot \frac{C}{W^d(1-u)}\left(\wh{\cal J}^{\ell}_{u,D}\right)^2 \cal T_u(|a-b|),\ee
	where we use \eqref{TTT2} in the second step. Combining the above two bounds, we conclude \eqref{juwo=Lklk}.

	\subsection{Proof of Lemma \ref{lem: EMn2_N}}\label{subsec:pf_lem: EMn2_N}
	
	We only control the term \smash{\((\mathcal{E} \otimes \mathcal{E})_{t, \boldsymbol{\sigma}, \boldsymbol{a}, \boldsymbol{a}}^M\equiv (\mathcal{E} \otimes \mathcal{E})_{t, \boldsymbol{\sigma}, \boldsymbol{a}, \boldsymbol{a}}^{M,(2;1)}\)} without loss of generality. By definition \eqref{defEOTE}, we can write it as follows with $\ba(c,c')=(a,b,c',b,a,c)$ and $(\boldsymbol{\sigma}\otimes \bsig)^{(1)} = (\sigma_1, \sigma_2, \sigma_1, -{\sigma}_1, -{\sigma}_2, -{\sigma}_1)$:
	\begin{equation}\label{eq_mart-sum}
		(\mathcal{E} \otimes \mathcal{E})_{t, \boldsymbol{\sigma}, \boldsymbol{a}, \boldsymbol{a}}^M = W^d \sum_{c, c'}S^{(\sB)}_{cc'} \mathcal{L}^{(6)}_{t, (\boldsymbol{\sigma}\otimes \bsig)^{(1)}, \ba(c,c')}.
	\end{equation}
	We can represent the RHS using the following graph: 
	\begin{equation}\label{eq:2martingale_quad}
		\parbox[c]{0.3\linewidth}{
			\scalebox{0.95}{
				\begin{tikzpicture}[scale=1.2, every node/.style={font=\small}]
					
					\coordinate (a1) at (0, 1.2);
					\fill (a1) circle (2pt) node[above=2pt]{$a$};
					
					\coordinate (a2) at (2, 1.2);
					\fill (a2) circle (2pt) node[above=2pt]{$b$};
					
					\coordinate (b) at (0.75, 0.6);
					\fill (b) circle (2pt) node[left=2pt]{$c$};
					
					\coordinate (b') at (1.25, 0.6);
					\fill (b') circle (2pt) node[right=2pt]{$c'$};
					
					\coordinate (a1') at (0, 0);
					\fill (a1') circle (2pt) node[below=2pt]{$a$};
					
					\coordinate (a2') at (2, 0);
					\fill (a2') circle (2pt) node[below=2pt]{$b$};
					
					\node at (-1,0.6) {$W^d\sum_{c\sim c'}S_{cc'}^{\LK}$};
					
					\draw[thick] (a1) -- (b) node[midway, left] {$\sig_1$};
					\draw[thick] (b') -- (a2) node[midway, right] {$\sig_1$};
					\draw[thick] (b') -- (a2') node[midway, right] {$-\sig_1$};
					\draw[thick] (b) -- (a1') node[midway, left] {$-\sig_1$};
					\draw[thick] (a1) -- (a2) node[midway, above] {$\sig_2$};
					\draw[thick] (a1') -- (a2') node[midway, below] {$-\sig_2$};

			\end{tikzpicture}}
		}
	\end{equation}
	Each solid edge between two block vertices (e.g., $a$ and $b$) with charge $\sig\in\{+,-\}$ denotes a resolvent entry $(G_t(\sig))_{xy}$ or $(G_t(\sig))_{yx}$ with $x\in [a]$ and $y\in [b]$. The orientation of the graph is assumed to be clockwise.
	To illustrate the idea for the proof of \eqref{eq:MG_conclusion}, assume without loss of generality that the vertex $c$ is closer to $b$ than to $a$. Then, the graph in \eqref{eq:2martingale_quad} already contains four ``long legs": in addition to the two edges connecting $a$ and $b$, the two edges connecting $a$ and $c$ contribute a factor $\Psi_t^2(|a-c|) \ge \Psi_t^2(|a-b|/2) \gtrsim \Psi_t^2(|a-b|)$. Summing over the two remaining edges connected to $c'$ and applying Cauchy–Schwarz together with Ward’s identity yields an additional factor of $\eta_t^{-1}$.
	However, compared with the target estimate \eqref{eq:MG_conclusion}, we still miss the factor $\Psi_t(0)$. To recover it, we require a sharper treatment of the 6-$G$-loop that avoids applying Cauchy–Schwarz directly to the two edges connected to $c'$. Indeed, if we can directly perform the global summation over these two edges without taking absolute values, Ward’s identity \eqref{eq_Ward0} produces a solid edge, which contributes the desired small factor $\Psi_t(0)$.
	This is precisely achieved via the following \emph{$\contract$ inequality} for the 6-$G$-loop in \eqref{eq:2martingale_quad}, which can be regarded as a pointwise version of the more general $\contract$ inequality \eqref{u2jzooi-2} below.
	For simplicity, we abbreviate $G\equiv G_t$ in the following proof.

	\begin{lemma} [$\Contract$ inequality]\label{ygdhmsgq0}
		For any subset $\cal A\subset \Zn$, we have that 
		\begin{equation}\label{eq_sym_loop_bound}
			\sum_{c'\in \cal A}\sum_{c\sim c'} \mathcal{L}^{(6)}_{t, (\boldsymbol{\sigma}\otimes \bsig)^{(1)}, \ba(c,c')} \le \frac{1}{W^d\eta_t}\max_{c'\in \cal A}\left(\mathcal{L}^{(4)}_{t,\boldsymbol{\sigma}^{(\alt)}, (c',b,c',b)}\right)^{1/2}\cdot \max_{\sig\in\{+,-\}} \big|\mathcal{L}^{(3)}_{t, (\sigma,\sig_2,-\sig_2), (a,b,a)} \big| ,
		\end{equation}
		where $\boldsymbol{\sigma}^{(\alt)}=(\sig_1,-\sig_1,\sig_1,-\sig_1)$ is an alternating loop. By symmetry, a similar inequality holds: 
		\begin{equation}\label{eq_sym_loop_bound2}
			\sum_{c\in \cal A}\sum_{c'\sim c}\mathcal{L}^{(6)}_{t, (\boldsymbol{\sigma}\otimes \bsig)^{(1)}, \ba(c,c')} \le \frac{1}{W^d\eta_t}\max_{c\in \cal A} \left(\mathcal{L}^{(4)}_{t,\bsig^{(\alt)}, (c,a,c,a)}\right)^{1/2} \cdot \max_{\sig\in \{+,-\}}\big|\mathcal{L}^{(4)}_{t, (\sig, -{\sigma}_2, \sigma_2), (b,a,b)}\big|.
		\end{equation}
	\end{lemma}
	\begin{proof}
		
		We only prove the inequality \eqref{eq_sym_loop_bound}, while the proof of \eqref{eq_sym_loop_bound2} is similar by switching the roles of $(a,c)$ and $(b,c')$. We view the symmetric $6$-loop as a quadratic form. Fix any $c,c'\in \Zn$ with $|c-c'|\le 1$ and $z\in [c]$, we define the column vector \smash{$\psi^{z}\in \C^{W^d}$} and the $W^d\times W^d$ matrix \smash{$A^{(c')}$} as
		\begin{align*}
			\psi^{z}_{y} &=  (G(\sig_1)E_aG(\sig_2))_{zy}, \ \ \forall y\in [b], \quad \text{and}\quad 
			A^{(c')}_{y_1  y_2} = (G(\sig_1)E_{c'} G(-\sig_1))_{y_1 y_2}, \ \ \forall y_1, y_2 \in [b].
		\end{align*}
		Then, the left-hand side (LHS) of \eqref{eq_sym_loop_bound} can be written as
		\begin{equation}\label{eq_6loop_quadform}
			\mathcal{L}^{(6)}_{t, (\boldsymbol{\sigma}\otimes \bsig)^{(1)}, \ba(c,c')} = W^{-3d} \sum_{z\in[c]} \left(\psi^{z}\right)^\ast A^{(c')} \psi^{z} \le W^{-3d} \|A^{(c')}\|  \sum_{z\in[c]} \|\psi^{z}\|_2^2.
		\end{equation}
		Since the operator norm of $A$ is bounded above by the Hilbert-Schmidt norm, we have
		\begin{equation}\label{eq:L4loop}
			\|A^{(c')}\|^2 \le \sum_{y_1,y_2\in [b]} (G(\sig_1)E_{c'} G(-\sig_1))_{y_1y_2} (G (\sig_1)E_{c'}G(-\sig_1))_{y_2 y_1} =: W^{2d} \mathcal{L}^{(4)}_{t,\bsig^{(\alt)}, (c',b,c',b)}.
		\end{equation}
		On the other hand, the sum of the squared $L^2$-norms of $\psi^z$ is equal to
		\begin{equation*}
			\sum_{z\in[c]} \|\psi^{z}\|_2^2 =  \sum_{z\in\cal [c],y\in [b]} (G(\sig_1)E_a G(\sig_2))_{z y} (G(-\sig_2)E_a G(-\sig_1))_{y z} = W^{2d} \mathcal{L}^{(4)}_{t, \boldsymbol{\sigma}', (a,b,a,c)},
		\end{equation*}
		where $\boldsymbol{\sigma}' = (\sigma_1, \sigma_2, -{\sigma}_2, -{\sigma}_1)$. Plugging the previous two displays into \eqref{eq_6loop_quadform} yields 
		\begin{align}\label{eq_sym_loop_bound_old}
			\mathcal{L}^{(6)}_{t, (\boldsymbol{\sigma}\otimes \bsig)^{(1)}, \ba(c,c')} \le \left(\mathcal{L}^{(4)}_{t,\boldsymbol{\sigma}^{(\alt)}, (c',b,c',b)}\right)^{1/2} \cdot \mathcal{L}^{(4)}_{t, \boldsymbol{\sigma}', (a,b,a,c)} .
		\end{align}
		Summing this inequality over $c'\in \cal A$ and $c\sim c'$, we obtain 
		\begin{align*}
			\sum_{c'\in \cal A}\sum_{c\sim c'} \mathcal{L}^{(6)}_{t, (\boldsymbol{\sigma}\otimes \bsig)^{(1)}, \ba(c,c')} &\le \max_{c'\in \cal A}\left(\mathcal{L}^{(4)}_{t,\boldsymbol{\sigma}^{(\alt)}, (c',b,c',b)}\right)^{1/2}\cdot \sum_{c}\mathcal{L}^{(4)}_{t, \boldsymbol{\sigma}', (a,b,a,c)} ,
		\end{align*}
		where the summation over $c$ has been extended to all of $\Zn$. Finally, applying Ward’s identity \eqref{WI_calL} to the sum over $c$ yields the desired bound \eqref{eq_sym_loop_bound}. 
	\end{proof}

	Now, we are ready to complete the proof of \Cref{lem: EMn2_N}. We first prove the bound \eqref{eq:MG_conclusion}. We split the summation over $c$ into two regions according to the distances from $c$ to the fixed indices $a$ and $b$: 
	\begin{equation}\label{eq:S1+S2}
		(\mathcal{E} \otimes \mathcal{E})_{t, \boldsymbol{\sigma}, \boldsymbol{a}, \boldsymbol{a}}^M \lesssim  W^d \sum_{\substack{c'\sim c \\ |c-b|> |c-a|}}  \mathcal{L}^{(6)}_{t, (\boldsymbol{\sigma}\otimes \bsig)^{(1)}, \ba(c,c')} + W^d \sum_{\substack{c'\sim c\\ |c-b|\le |c-a|}} \mathcal{L}^{(6)}_{t, (\boldsymbol{\sigma}\otimes \bsig)^{(1)}, \ba(c,c')} =: \mathcal{S}_1 + \mathcal{S}_2.
	\end{equation}
	By symmetry, it suffices to prove the bound \eqref{eq:MG_conclusion} for the sum $\mathcal{S}_1$ using \eqref{eq_sym_loop_bound}. (To control the sum $\cal S_2$, we only need to utilize the inequality \eqref{eq_sym_loop_bound2} in place of \eqref{eq_sym_loop_bound} in the following proof.) 
	Under the conditions \(|c-c'|\le 1\) and \(|c-b| \ge |a-b|/2\), we bound each off-diagonal $G$-entry of the $4$-loop \smash{$\mathcal{L}^{(4)}_{u,\bsig^{(\alt)}, (c',b,c',b)}$} using \eqref{GijGEX}, and each diagonal $G$-entry using \eqref{GiiGEX}. This gives
	\begin{align}\label{eq:pointwise_loop}
		\p{ \mathcal{L}^{(4)}_{u,\bsig^{(\alt)}, (c',b,c',b)}}^{1/2} &\prec  \max_{\substack{|b'-b|\le 1,|c''-c'|\le 1,\\ \bsig\in \{(+,-),(-,+)\}}}\mathcal{L}^{(2)}_{u,\bsig, (b',c'')}+W^{-d}\mathbf 1_{|b-c'|\le 1}\prec \Psi_t^2(|b-c'|) \lesssim \Psi_t^2(|a-b|).
	\end{align}
	Here, the second step uses the assumption \eqref{eq:LW_assm} on the 2-loops together with condition \eqref{eq:Psi}, while the third step follows from $|c-c'|\le 1$, $|c-b|\ge |a-b|/2$, and another application of \eqref{eq:Psi}. Plugging this bound into the $\contract$ inequality \eqref{eq_sym_loop_bound}, we get that
	\begin{equation}\label{eq:reduce4to3}
		\mathcal{S}_1 \prec \frac{1}{\eta_t} \Psi_t^2\left(|a-b|\right) \max_{\sig\in\{+,-\}}\left|\mathcal{L}^{(3)}_{t, (\sig,\sig_2,-\sig_2), (a,b,a)}\right|.
	\end{equation}
	To bound the 3-$G$-loop on the RHS, we write
	\[\mathcal{L}^{(3)}_{t, (\sig,\sig_2,-\sig_2), (a,b,a)}=W^{-3d}\sum_{x,x'\in [a]}\sum_{y\in[b]}G_{xx'}(\sig)G_{x'y}(\sig_2)G_{yx}(-\sig_2), \]
	and apply an argument analogous to that in \eqref{eq:pointwise_loop}, obtaining 
	\begin{align}
		\big| \mathcal{L}^{(3)}_{t, (\sig,\sig_2,-\sig_2), (a,b,a)}\big|& \prec \Psi_t(0) \cdot \Psi_t^2(|a-b|).\label{eq:reduce4_bdd3}
	\end{align} 
	Roughly speaking, the factor $ \Psi_t^2(|a-b|)$ arises from the two long legs---namely, the entries $G_{x'y}$ and $G_{yx}$ between $[a]$ and $[b]$. The short leg $G_{xx'}$, with both indices in $[a]$, contributes a factor of $\Psi_t(0)+W^{-d}$ by \eqref{GiiGEX}. (More precisely, each off-diagonal entry of $G$ contributes a factor of $\Psi_t(0)$, while diagonal entries contribute $W^{-d}$.) Substituting \eqref{eq:reduce4_bdd3} into \eqref{eq:reduce4to3} gives the desired bound \eqref{eq:MG_conclusion}.

	In order to show the bound \eqref{eq:MG_conclusion3}, we define two cutoff scales 
	\be\label{eq:cutoff_scales}
	\ell_t^\ast = (\log W)^{3/2}\ell_t,\quad \ell_t^{\dag} = (\log W)^{7/4}\ell_t.\ee
	Then, we divide the proof into two cases. 
	First, suppose $|a-b| \le \ell_t^{\dag}$. In this case, the function $\widetilde{\mathcal{T}}^{\ell}_{t,D}\p{|a-b|}$ exhibits no exponential decay. Hence, it suffices to apply the polynomial decay bound \eqref{eq:MG_conclusion} with the profile function \smash{\(\Psi_t(r) = (W^{-d} B_{t, r})^{1/2}.\)}  
	The conclusion \eqref{eq:MG_conclusion3} then follows directly from \eqref{eq:MG_conclusion} and the facts 
	\[
	\br{\Psi_t(r)}^2 \prec  W^{-d}\widetilde{\mathcal{T}}^{\ell}_{t,D}\p{r} , \quad W^{-d}\widetilde{\mathcal{T}}^{\ell}_{t,D}\p{r} \prec \br{\Psi_t(r)}^2, \quad \forall \ 0\le r\le \ell_t^\dag.
	\]
	It remains to deal with the case 
	\be\label{eq:largea-b}|a-b| > \ell_t^{\dag} = (\log W)^{7/4}\ell_t.\ee 
	In this case, we split the summation \eqref{eq_mart-sum} into three parts as follows: 
	\begin{align}
		(\mathcal{E} \otimes \mathcal{E})_{t, \boldsymbol{\sigma}, \boldsymbol{a}, \boldsymbol{a}}^M &\lesssim W^d \bigg(\sum^\star_{\substack{|c-a| \le \ell_t^\ast \text{ or } |c'-b| > \ell}} + \sum^\star_{\substack{|c'-b| \le \ell_t^\ast \text{ or } |c-a| > \ell}} + \sum^\star_{\substack{\ell_t^\ast < |c'-b|,|c-a| \le \ell}}\bigg) \mathcal{L}^{(6)}_{t, (\boldsymbol{\sigma}\otimes \bsig)^{(1)}, \ba(c,c')}\nonumber\\
		&=: \widetilde{\mathcal{S}}_1 + \widetilde{\mathcal{S}}_2 + \widetilde{\mathcal{S}}_3.\label{eq;S123}
	\end{align}
	where $\sum^\star$ refers to the summation subject to the constraint $|c-c'|\le 1$. 
	To estimate $\widetilde{\mathcal{S}}_1$, we combine \eqref{GijGEX} with \eqref{eq:LW_assm_exp} and proceed analogously to \eqref{eq:pointwise_loop}. This yields  
	\begin{align}\label{eq:pointwise_loop2}
		\p{ \mathcal{L}^{(4)}_{u,\bsig^{(\alt)}, (c',b,c',b)}}^{1/2} \prec W^{-d}\wT_{t,D}^{\ell}(|b-c'|) \prec W^{-d}\wT_{t,D}^{\ell}(|a-b|),
	\end{align}
	where the second inequality is justified as follows: (1) if $|c'-b|>\ell$, then \smash{\(\wT_{t,D}^{\ell}(|b-c'|) = \wT_{t,D}^{\ell}(\ell)\le \wT_{t,D}^{\ell}(|a-b|);\)} (2) if $|c-a|\le \ell_t^*$, then by \eqref{eq:largea-b}, we have $|b-c'|\ge |a-b|-(\ell_t^*+1) =(1+\oo(1))|a-b|$, which implies 
	\[ \wT_{t,D}^{\ell}(|b-c'|) \le \wT_{t,D}^{\ell}\p{|a-b|-(\ell_t^*+1)} \prec \wT_{t,D}^{\ell}(|a-b|). \]
	Substituting \eqref{eq:pointwise_loop2} into the $\contract$ inequality \eqref{eq_sym_loop_bound}, we obtain  
	\begin{align}
		\widetilde{\mathcal{S}}_1 &\prec  \frac{1}{\eta_t} \left[W^{-d} \widetilde{\mathcal{T}}^{\ell}_{t,D}\p{|a-b|}\right] \max_{\sig\in\{+,-\}}\left|\mathcal{L}^{(3)}_{t, (\sig,\sig_2,-\sig_2), (a,b,a)}\right|.\label{eq_S1tilde}
	\end{align}
	Next, using \eqref{GijGEX} and \eqref{eq:LW_assm_exp}, we obtain a similar bound on the 3-$G$-loop as in \eqref{eq:reduce4_bdd3}:
	\[\left|\mathcal{L}^{(3)}_{t, (\sig,\sig_2,-\sig_2), (a,b,a)}\right|\prec (W^{-d}B_{t,0})^{1/2}\cdot W^{-d} \widetilde{\mathcal{T}}^{\ell}_{t,D}\p{|a-b|}.\]
	Plugging this into \eqref{eq_S1tilde} gives
	\begin{equation}\label{eq:boundwtS_1}
		\widetilde{\mathcal{S}}_1 \prec \frac{1}{\eta_t} \left(W^{-d}B_{t,0}\right)^{1/2}\left[W^{-d} \widetilde{\mathcal{T}}^{\ell}_{t,D}\p{|a-b|}\right]^2 , 
	\end{equation}
	which is controlled by the RHS of \eqref{eq:MG_conclusion3}. By symmetry, an identical bound holds for $\wt {\cal S}_2$.

	It remains to control the sum $\widetilde{\mathcal{S}}_3$. In this case, we bound all legs of the original $6$-loop directly using \eqref{GijGEX}. In this case, all legs of this loop have lengths (i.e., $|a-b|$, $|a-c|$, and $|c'-b|$) at least $\ell_t^\ast$. Then, instead of using the assumption \eqref{eq:LW_assm_exp}, we will bound the $2$-loops using the parameter defined in \eqref{defCALJ} as follows: for any $a,b\in \Zn$ satisfying $|a-b|\gtrsim \ell_t^\ast$,  
	\begin{equation}\label{eq_L2-J}
		\mathcal{L}^{(2)}_{t,(-,+),(a,b)} \prec \wh{\mathcal{J}}^{\ell}_{t,D} \cdot W^{-d} \widetilde{\mathcal{T}}^{\ell}_{t,D}\p{|a-b|} + \mathcal{K}^{(2)}_{t,(-,+),(a,b)} \lesssim \big(\wh{\mathcal{J}}^{\ell}_{t,D} +W^{-D}\big)\cdot W^{-d} \widetilde{\mathcal{T}}^{\ell}_{t,D}\p{|a-b|} ,
	\end{equation}
	where the second step is due to the exponential decay of the 2-$\mathcal{K}$-loop given by \eqref{prop:ThfadC}. 
	Then, using \eqref{GijGEX} to bound the 6 legs of the 6-loop, we obtain that 
	\begin{align}
		\widetilde{\mathcal{S}}_3 \prec &~W^{-2d}\p{\wh{\mathcal{J}}^{\ell}_{t,D}+W^{-D}}^3\sum_{\substack{\ell_t^\ast < |c'-b|,|c-a| \le \ell}}^\star \widetilde{\mathcal{T}}^{\ell}_{t,D}\p{|a-b|} \widetilde{\mathcal{T}}^{\ell}_{t,D}\p{|a-c|}  \widetilde{\mathcal{T}}^{\ell}_{t,D}\p{|c-b|} ,\label{eq:boundwtS31}
	\end{align}
	where we also use that 
	$\widetilde{\mathcal{T}}_{u,D}^{\ell}\p{|c'-b|}\asymp \widetilde{\mathcal{T}}_{u,D}^{\ell}\p{|c-b|}$ for $|c-c'|\le 1$. We can bound the RHS of \eqref{eq:boundwtS31} by
	\begin{align*}
		&~\left(\wh{\mathcal{J}}^{\ell}_{t,D}+W^{-D}\right)^3\cdot W^{-2d} \widetilde{\mathcal{T}}^{\ell}_{t,D}\p{|a-b|}\cdot \sum_c \br{\mathcal{T}_{t}\p{|a-c|}+W^{-D}}\br{\mathcal{T}_{t}\p{|c-b|}+W^{-D}} \\
		\lesssim &~\frac{(\wh{\mathcal{J}}^{\ell}_{t,D}+W^{-D})^3}{1-t} \cdot W^{-2d} \widetilde{\mathcal{T}}^{\ell}_{t,D}\p{|a-b|}\cdot \br{\mathcal{T}_{t}\p{|a-b|} +W^{-D}}\lesssim \frac{(\wh{\mathcal{J}}^{\ell}_{t,D}+W^{-D})^3}{\eta_t} \cdot \left(W^{-d} \widetilde{\mathcal{T}}^{\ell}_{t,D}\p{|a-b|}\right)^2,
	\end{align*}
	where we use \eqref{TTT2} and the fact $\sum_{a}\cal T_t\p{a}\lesssim (1-t)^{-1}$. This finishes the estimation of \smash{$\wt{\cal S}_3$}, and hence completes the proof of the bound \eqref{eq:MG_conclusion3}.

	\section{Steps 3 and 4: Sharp maximum estimates for \texorpdfstring{$G$}{G}-loops}\label{Sec:Steps34}

	In this section, we apply the integrated loop hierarchy \eqref{int_K-LcalE} to complete the proofs of Steps 3 and 4. The key task is to show that the second through fifth terms in \eqref{int_K-LcalE} are errors in the max-norm sense.
	
	The integrands in the second to fourth terms on the RHS of \eqref{int_K-LcalE} each involve a global sum over an index of an $\cal L$- or $(\cL-\cK)$-loop (see \eqref{DefKsimLK}, \eqref{def_ELKLK}, and \eqref{def_EwtG}). In low dimensions $d\in \{1,2\}$, these terms can be controlled directly using their max-norms, together with a factor $\ell_t^d$ that captures the range of exponential decay.
	For $d\in\{1,2\}$, the factor $\ell_t^d$ is bounded by $(1-t)^{-1}$, which incurs only a harmless logarithmic contribution upon integration in time. However, in higher dimensions $d\ge 3$, an additional factor $\ell_t^{d-2}$ arises on top of the $(1-t)^{-1}$, creating a substantial term that cannot be canceled. 
	To obtain sufficiently sharp bounds without relying on the precise pointwise decay of higher-order $G$-loops, we develop a new \emph{$\contract$ inequality} \eqref{yi2oslxj2}, derived from the Cauchy–Schwarz inequality and the Ward's identities \eqref{WI_calL} and \eqref{WI_calK}. Although technically simple, this inequality provides a powerful tool: it allows us to control the absolute sum of a higher-order $G$-loop in terms of lower-order ones in a sharp max-norm sense, and crucially produces the desired $(1-t)^{-1}$ factor without introducing any excess powers of $\ell_t$.
	Among all these terms, one remains that cannot be effectively controlled using the $\contract$ inequalities—namely, the term of the form $(\cal L-\cal K)^{(2)}\diamond(\cal L-\cal K)^{(n)}$ (recall the notation in \eqref{eq:diamond}) in \eqref{def_ELKLK}. This term is instead shown to be an error by applying the pointwise $(\cal L-\cal K)^{(2)}$-estimate \eqref{Eq:Gdecay_w}.

	Bounding the martingale term requires controlling a $(2n+2)$-loop (recall \Cref{def:CALE}). The corresponding loop structure is symmetric and involves two summation indices along the loop (see e.g., \eqref{eq:2martingale_quad}). To handle it, we derive another $\contract$ inequality \eqref{u2jzooi-2}, obtained by extending the argument in the proof of \Cref{ygdhmsgq0}. This inequality yields a non-sharp but sufficient bound for the martingale term, introducing an additional power of $W$, i.e., the factor $(W^{-d}B_{t,0})^{-1/(2p)}$ in \eqref{eq:MG_nloop}. However, this factor remains small enough for our proof for large enough $p$.

	Our overall strategy for Steps 3 and 4 can now be summarized as follows. Using the $\contract$ inequalities together with a good stability estimate \eqref{Eq:Gdecay_w} for $({\cal L}-{\cal K})^{(2)}$-loops, we can bound each term in \eqref{int_K-LcalE} either by lower-order $(\cL-\cK)$-loops or by higher-order $(\cL-\cK)$-loops with an extra small factor $W^{-\varepsilon}$. Then, with the integrated hierarchy \eqref{int_K-LcalE}, we can bootstrap from a sequence of weaker bounds to tighter ones, ultimately yielding improved control on the entire hierarchy of $(\cL -\cK)$-loops in each induction. As in \cite{YY_25, DYYY25, erdHos2025zigzag}, after $\OO(1)$ iterations, this process yields (almost) sharp max-norm bounds for all $(\cL-\cK)$-loops.

	\begin{lemma} [$\Contract$ inequality]\label{ygdhmsgq}
		For each $1\le k \le n-1$, we have the bound
		\begin{align}\label{yi2oslxj2}
			\max_{\boldsymbol{\sigma}}\sum_{a_n} |{\cal L}^{(n)}_{t, \bsig, \ba}|\le \frac 1 {W^{d}\eta_t} \left(\max_{\boldsymbol{\sigma},\mathbf{b}}  \left|{\cal L}^{(2k-1)}_{t, \bsig, \mathbf{b}}\right|\cdot \max_{\boldsymbol{\sigma},\mathbf{b}}  \left|{\cal L}^{(2n-2k-1)}_{t, \bsig, \mathbf{b} }\right|\right)^{\frac 1 2}   .
		\end{align}
		Furthermore, for any $n\ge 4$, $1\le k<j<l\le n-1$, and subsets $\cal A(a_n)\subset \Zn$ (which may depend on $a_n$) of cardinality $|\cal A(a_n)|\le C$ for a constant $C>0$, the following bound holds for any fixed $p\ge 1$: 
		\begin{align}\label{u2jzooi-2}
			\max_{\boldsymbol{\sigma}} \sum_{a_n}\sum_{a_j\in \cal A(a_n)} |{\cal L}^{(n)}_{t,\bsig, {\ba}}|\le  &~\frac{C}{W^d\eta_t} \left(\max_{\boldsymbol{\sigma},\mathbf{b}}  \left|{\cal L}^{(2k-1)}_{t, \bsig, \mathbf{b}}\right|\cdot \max_{\boldsymbol{\sigma},\mathbf{b}}  \left|{\cal L}^{(2n-2l-1)}_{t,\bsig,\mathbf{b}}\right|\right)^{\frac 1 2}   \left(\max_{\boldsymbol{\sigma},\bfb}   \left|{\cal L}^{(2(l-k)p)}_{t,\bsig, \mathbf b}\right|\right)^{\frac1{2p}} .
		\end{align}
	\end{lemma}
	
	\begin{proof}
		We first prove \eqref{yi2oslxj2}.  Applying the Cauchy-Schwarz inequality with respect to the averages over $[a_k]$ and $[a_n]$ in the loop \smash{${\cal L}^{(n)}_{t,\bsig, \ba}$}, we obtain that for $\bsig=(\sig_1,\ldots, \sig_n)$ and $\ba=(a_1,\ldots, a_n)$, 
		\begin{align}\label{eq:CS1}
			|{\cal L}^{(n)}_{t, \bsig,\ba}| 
			\; \le \;&
			\left({\cal L}^{(2k)}_{t,\boldsymbol{\sigma}_1,\ba_1 }\cdot  {\cal L}^{(2n-2k)}_{t,\bsig_2,\ba_2}\right)^{1/2},
		\end{align}
		where $\ba_1$, $\ba_2$, $\bsig_1$, and $\bsig_2$ are defined as     
		\begin{align}
			{\ba}_1= (a_1,\cdots, a_{k-1}, a_k, a_{k-1},\cdots ,a_1,a_n), &\quad 
			\boldsymbol{\sigma}_1 = (\sigma_1, \cdots, \sigma_{k}, -{\sigma_{k}},\cdots, -\sigma_1) ,\nonumber
			\\\nonumber
			{\ba}_2= (a_{n-1},\cdots, a_{k+1}, a_k, a_{k+1}, \cdots ,a_{n-1},a_n), &\quad 
			\boldsymbol{\sigma}_2 = (-{\sigma_n}, \cdots, -\sig_{k+1}, \sigma_{k+1}, \cdots ,\sigma_n).
		\end{align}  
		Note that the loops ${\cal L}^{(2k)}_{t,\boldsymbol{\sigma}_1,{\ba}_1}$ and ${\cal L}^{(2n-2k)}_{t, \boldsymbol{\sigma}_2,{\ba}_2}$ are both non-negative since $\bsig_1$ and $\bsig_2$ are symmetric.  Then, applying the Cauchy-Schwarz inequality to \eqref{eq:CS1} again, we obtain that 
		\be\label{eq;genWard0}
		\sum_{a_n} |{\cal L}^{(n)}_{t, \boldsymbol{\sigma},\ba}| \le  
		\Big(\sum_{a_n}{\cal L}^{(2k)}_{t, \boldsymbol{\sigma}_1,\ba_1}\Big)^{1/2} \Big(\sum_{a_n}{\cal L}^{(2n-2k)}_{t, \boldsymbol{\sigma}_2,\ba_2}\Big)^{1/2}.
		\ee
		Using Ward's identity \eqref{WI_calL}, we can express the $G$-loops on the RHS as  
		\be\label{eq;genWard1}
		\sum_{a_n}{\cal L}^{(2k)}_{t, \boldsymbol{\sigma}_1,\ba_1} =\frac{{\cal L}^{(2k-1)}_{t, \boldsymbol{\sigma}^+_1,\ba'_1}
			-{\cal L}^{(2k-1)}_{t, \boldsymbol{\sigma}^-_1,\ba'_1}}{2\ii W^d\eta_t} ,\quad
		\sum_{a_n}{\cal L}^{(2n-2k)}_{t, \boldsymbol{\sigma}_2,\ba_2} =\frac{{\cal L}^{(2n-2k-1)}_{t, \boldsymbol{\sigma}^+_2,\ba'_2}
			-{\cal L}^{(2n-2k-1)}_{t, \boldsymbol{\sigma}^-_2,\ba'_2}}{2\ii W^d\eta_t},
		\ee
		where $\ba'_1$, $\ba'_2$, $\boldsymbol{\sigma}^{\pm}_1$, and $\boldsymbol{\sigma}^{\pm}_2$ are defined as 
		\begin{align*}
			\ba'_1=(a_1,\cdots, a_{k-1}, a_k, a_{k-1}\cdots, a_1),\quad &\boldsymbol{\sigma}^{\pm}_1 = ( \pm,\sigma_2, \cdots, \sigma_{k}, -{\sigma_{k}},\cdots,-{\sigma_2}), \\
			\ba'_2=( a_{n-1}\cdots,   a_{k+1}, a_{k}, a_{k+1},\cdots, a_{n-1}), \quad 
			&\boldsymbol{\sigma}^{\pm}_2 = (\pm, -{\sigma_{n-1}}, \cdots, - \sigma_{k+1}, \sigma_{k+1},\cdots  \sigma_{n-1}).
		\end{align*}
		Plugging \eqref{eq;genWard1} into \eqref{eq;genWard0}, we  conclude \eqref{yi2oslxj2}. 
		
		For $\bsig=(\sigma_1,  \ldots, \sigma_n)\in \{+,-\}^n$ and \smash{$\ba = (a_1, \ldots, a_{n-1})\in (\Zn)^{n-1}$}, denote a $G$-chain of length $n$ by 
		\begin{equation}\label{defC=GEG}
			{\cal C}^{(n)}_{t, \boldsymbol{\sigma}, \ba} = \prod_{i=1}^{n-1}\left(G_t(\sigma_i) E_{a_i}\right)\cdot  G_t(\sigma_n) .
		\end{equation}
		To show \eqref{u2jzooi-2}, we split the loop ${\cal L}^{(n)}_{t, \boldsymbol{\sigma}, \ba}$ at $a_k$, $a_l$, and $a_n$, and write it as 
		$${\cal L}^{(n)}_{t, \boldsymbol{\sigma}, \ba} = W^{-3d}\sum_{x_k\in\br{a_k} }\sum_{x_l\in\br{a_l}} \sum_{x_n\in \br{a_n}} \p{{\cal C}^{(k)}_{t, \boldsymbol{\sigma}_1, \ba_1}}_{x_n x_k}\p{{\cal C}^{(l-k)}_{t, \boldsymbol{\sigma}_2, \ba_2}}_{x_kx_l}\p{{\cal C}^{(n-l)}_{t, \boldsymbol{\sigma}_3, \ba_3}}_{x_lx_n} ,$$
		where $\ba_1=(a_1,\ldots,a_{k-1})$, $\bsig_1=(\sig_1,\ldots,\sig_k)$, $\ba_2=(a_{k+1},\ldots,a_{l-1})$, $\bsig_2=(\sig_{k+1},\ldots,\sig_l)$, $\ba_3=(a_{l+1},\ldots,a_{n-1})$, and $\bsig_3=(\sig_{l+1},\ldots,\sig_n)$. We can bound the RHS by using the operator norm of the $W^d\times W^d$ matrix $A=(A_{xy}:x\in  \br{a_k}, y\in \br{a_l})$ with \smash{$A_{xy}= ({\cal C}^{(l-k)}_{t, \boldsymbol{\sigma}_2, \ba_2})_{xy}$}:
		\begin{align}
			\left|{\cal L}^{(n)}_{t, \ba, \boldsymbol{\sigma}}\right|\le 
			\frac{1}{W^{2d}}\sum_{x_n\in \br{a_n}}
			\bigg(\sum_{x_k\in \br{a_k}}\Big|\p{{\cal C}^{(k)}_{t, \boldsymbol{\sigma}_1, \ba_1}}_{x_n x_k}\Big|^2\bigg)^{1/2}
			\bigg(\sum_{x_l\in \br{a_l}}\Big|\p{{\cal C}^{(n-l)}_{t, \boldsymbol{\sigma}_3, \ba_3}}_{x_l x_n}\Big|^2\bigg)^{1/2} \cdot \frac{1}{W^d}\|A\| .\label{eq:Apsipsi}
		\end{align}
		To control the operator norm $\|A\|$, we use the simple linear algebra fact 
		\(\|A\|\le \left\{\tr[(AA^*)^p]\right\}^{\frac{1}{2p}} \)
		for any $p\in \N$, which gives that  
		\begin{align}\label{eq:Apsipsi1}
			\frac{1}{W^d}\|A\| \le  \left\{\frac{1}{W^{2pd}}\tr \br{\left(AA^*\right)^{p}}\right\}^{\frac1{2p}}
			\le \left(\max_{\boldsymbol{\sigma},\bfb}   \left|{\cal L}^{(2(l-k)p)}_{t, \boldsymbol{\sigma} , \bfb}\right|\right)^{\frac1{2p}}
		\end{align}
		Plugging \eqref{eq:Apsipsi1} into \eqref{eq:Apsipsi} and applying the Cauchy–Schwarz inequality, we obtain that 
		\begin{align*}
			\sum_{a_n}\sum_{a_j\in \cal A(a_n)} \left|{\cal L}^{(n)}_{t, \boldsymbol{\sigma}, \ba}\right|&\le C  \left(\max_{\boldsymbol{\sigma},\bfb}  \left| {\cal L}^{(2(l-k)p)}_{t, \boldsymbol{\sigma}, \bfb}\right|\right)^{\frac1{2p}} \bigg(\frac{1}{W^{2d}}\sum_{a_n}\sum_{x_n\in\br{a_n}}\sum_{x_k\in\br{a_k}}\Big|\p{{\cal C}^{(k)}_{t, \boldsymbol{\sigma}_1, \ba_1}}_{x_n x_k}\Big|^2\bigg)^{1/2} \\
			&\qquad \qquad \times \bigg(\frac{1}{W^{2d}}\sum_{a_n}\sum_{x_n\in\br{a_n}}\sum_{x_l\in\br{a_l}}\Big|\p{{\cal C}^{(n-l)}_{t, \boldsymbol{\sigma}_3, \ba_3}}_{x_l x_n}\Big|^2\bigg)^{1/2}   \\
			&\le \frac{C}{W^d\eta_t}\left(\max_{\boldsymbol{\sigma},\bfb}   \left|{\cal L}^{(2(l-k)p)}_{t, \bfb, \boldsymbol{\sigma}}\right|\right)^{\frac1{2p}} \bigg(\max_{\boldsymbol{\sigma},\bfb}  \left|{\cal L}^{(2k-1)}_{t, \boldsymbol{\sigma}, \bfb}\right|\cdot \max_{\boldsymbol{\sigma},\bfb}  \left|{\cal L}^{(2n-2l-1)}_{t, \boldsymbol{\sigma}, \bfb}\right| \bigg)^{1/2},
		\end{align*}
		where we applied Ward's identity \eqref{WI_calL} in the second step. This concludes \eqref{u2jzooi-2}.
	\end{proof}
	
	In the course of proving \Cref{ML:Kbound}, we will also establish the following bound, whose proof is deferred to \Cref{Sec:CalK}.
	
	\begin{lemma}\label{lem_wardineq_K}
		For any $n\ge 2$ and $t\in[0,1)$, we have that
		\begin{align}\label{wardineq_K}
			\max_{\boldsymbol{\sigma}\in\{+,-\}^n}\sum_{a_n} |{\cal K}^{(n)}_{t, \bsig, \ba}|\prec \frac 1 {W^{d}\eta_t} (W^{-d}B_{t,0})^{n-2}.  
		\end{align} 
	\end{lemma}

	For any $\fn\ge 1$, let $\Xi^{({\cal L})}_{t, \fn}\ge 1$ and $\Xi^{({\cal L}-{\cal K})}_{t, \fn}\ge 1$ be \emph{deterministic} control parameters for $\cL$-loops and $(\cL-\cK)$-loops of length $\fn$ such that the following bounds hold: 
	\begin{align}\label{def:XiL}
		\wh\Xi^{({\cal L})}_{t, \fn} &:= 1+ \max_{\boldsymbol{\sigma}\in \{+, -\}^\fn}\max_{\ba\in (\Zn)^\fn} \left|{\cal L}^{(\fn)}_{t, \boldsymbol{\sigma}, \ba} \right|\Big/ \left(W^{-d}B_{t,0}\right)^{\fn-1}\prec \Xi^{({\cal L})}_{t, \fn},\\
		\wh\Xi^{({\cal L}-{\cal K})}_{t, \fn} &:=1+\max_{\boldsymbol{\sigma}\in \{+, -\}^\fn}\max_{\ba\in (\Zn)^\fn} \left|({\cal L}-{\cal K})^{(\fn)}_{t, \boldsymbol{\sigma}, \ba} \right|\Big/\left(W^{-d}B_{t,0}\right)^{\fn} \prec \Xi^{({\cal L}-{\cal K})}_{t, \fn}.\label{def:XIL-K}
	\end{align}
	Using these control parameters and the $\contract$ inequalities in \Cref{ygdhmsgq,lem_wardineq_K}, we can control the terms on the RHS of equation \eqref{int_K-LcalE} as follows.

	\begin{lemma}[Estimates of $\cal E$ terms]\label{lem:SEforLn}
		In the setting of \Cref{lem:main_ind}, suppose the local laws \eqref{Gt_bound_flow} and \eqref{Gt_avgbound_flow} and the $2$-$G$-loop estimate \eqref{Eq:Gdecay_w} hold. Then, the following estimates hold for any fixed $n\ge 2$: 
		\begin{enumerate}
			
			\item The light-weight term defined in \eqref{def_EwtG} satisfies that  
			\begin{align}\label{bEwGn}
				\max_{\bsig,\ba}\abs{{\cal E}^{\Gc, (n)}_{t,\bsig,\ba}}\prec \frac{\p{W^{-d}B_{t,0}}^n}{\eta_t} \cdot\p{\wh\Xi^{({\cal L})}_{t, \fn_1}\cdot \wh\Xi^{({\cal L})}_{t, \fn_2}}^{\frac12}, 
			\end{align} 
			where $n_1=n-1$ and $n_2=n+1$ if $n$ is even, and $n_1=n_2=n$ if $n$ is odd.
			
			\item For $3\le \lenk \le n$, the $[\cK^{(\lenk)}\sim (\cL-\cK)]^{(n)}$ term defined in \eqref{DefKsimLK} satisfies that
			\begin{align}\label{eq:KsimL-K}
				\max_{\bsig,\ba} \abs{[\cK^{(\lenk)}\sim (\cL-\cK)]^{(n)}_{t,\bsig,\ba}}\prec \frac{\left(W^{-d}B_{t,0}\right)^{n}}{\eta_t}\cdot \wh\Xi^{({\cal L}-{\cal K})}_{t, n-\lenk+2}. 
			\end{align}

			\item The ${\cal E}^{(\cL-\cK)\times (\cL-\cK),(n)}$ term defined in \eqref{def_ELKLK} satisfies that 
			\begin{align}\label{eq:L-KsimL-K}
				\max_{\bsig,\ba} \abs{ {\cal E}^{(\cL-\cK)\times (\cL-\cK),(n)}_{t,\bsig,\ba}}\prec &~ 
				\frac{\left(W^{-d}B_{t,0}\right)^{n}}{\eta_t}  \sum_{n'= \lceil n/2\rceil +1}^{n-1} \wh\Xi^{({\cal L}-{\cal K})}_{t,n+2-n'} \cdot \left(\wh\Xi^{({\cal L})}_{t,n_1'} \cdot  \wh\Xi^{({\cal L})}_{t,n_2'}\right)^{1/2} \nonumber\\
				&~ +\frac{\left(W^{-d}B_{t,0}\right)^{n}}{\eta_t}  \cdot (W^{-d} {B}_{t,0})^{\frac 1 6}\wh\Xi^{({\cal L}-{\cal K})}_{t,n},
			\end{align}
			where $n'_1=n_2'=n'-1$ if $n'$ is even, and $n'_1=n'-2$ and $n_2'=n'$ if $n'$ is odd. Note that the first term on the RHS is zero when $n\in\{2,3\}$. 
			
			\item The term $\left( \cal E\otimes  \cal E \right)^{M}_{t, \boldsymbol{\sigma}, \ba,\ba'}\equiv \left( \cal E\otimes  \cal E \right)^{M,(n;k)}_{t, \boldsymbol{\sigma}, \ba,\ba'}$ defined in \eqref{defEOTE} satisfies the following estimate for each $k \in \qqq{n}$ and any fixed $p\in \N$:
			\begin{align}\label{eq:MG_nloop}
				\max_{\bsig}\max_{\ba,\ba'}\left( \cal E\otimes  \cal E \right)^{M}_{t, \boldsymbol{\sigma}, \ba,\ba'}\prec \frac{(W^{-d} {B}_{t,0})^{2n-\frac{1}{2p}}}{\eta_t}  \cdot \wh\Xi^{({\cal L})}_{t,2n-1}  \left(\wh\Xi^{({\cal L})}_{t,4p}\right)^{\frac{1}{2p}}. 
			\end{align}
		\end{enumerate}
	\end{lemma}
	\begin{proof}
		Using the inequality \eqref{yi2oslxj2} (with $k=\lceil n/2\rceil$) and the averaged local law \eqref{Gt_avgbound_flow} established in Step 2, we can bound that 
		\begin{align*}
			\abs{\mathcal{E}^{\Gc,(\fn)}_{t, \boldsymbol{\sigma}, \ba}} &\le   {W}^d \sum_{k=1}^\fn \sum_{a, b\in \Zn} \; 
			\abs{\tr\p{ \Gc_t(\sigma_k) E_{a} }}
			S^{\LK}_{ab} 
			\abs{\left( {\cut}^{(b)}_{k} \circ {\cal L}^{(\fn)}_{t, \boldsymbol{\sigma}, \ba} \right)} \\
			& \prec W^d\cdot (W^{-d}B_{t,0}) \cdot \frac{1}{W^d\eta_t} \left(\max_{\boldsymbol{\sigma},\bfb}  \left|{\cal L}^{(n_1)}_{t, \boldsymbol{\sigma}, \bfb}\right|\cdot \max_{\boldsymbol{\sigma},\bfb}  \left|{\cal L}^{(n_2)}_{t, \boldsymbol{\sigma}, \bfb}\right|\right)^{\frac 1 2},
		\end{align*}
		which concludes \eqref{bEwGn} together with the definition \eqref{def:XiL}. For \eqref{eq:KsimL-K}, with the $\cK$-loop bounds \eqref{eq:bcal_k} and \eqref{wardineq_K}, we obtain that 
		\begin{align*}
			\max_{\bsig,\ba}\abs{[\cK^{(\lenk)}\sim (\cL-\cK)]^{(n)}_{t,\bsig,\ba}} &\lesssim W^d \p{\max_{\bsig,\ba}\abs{(\mathcal{L} - \mathcal{K})^{(\fn-\lenk+2)}_{t, \bsig, \ba}}} \cdot  \max_{\bsig,\ba} \sum_{a_{\lenk}}\abs{\cK^{(\lenk)}_{t,\bsig,\ba}}    \\
			&\lesssim \frac{1}{\eta_t}\p{\max_{\bsig,\ba}\abs{(\mathcal{L} - \mathcal{K})^{(\fn-\lenk+2)}_{t, \bsig, \ba}}} \cdot (W^{-d}B_{t,0})^{\lenk-2} ,
		\end{align*} 
		which concludes \eqref{eq:KsimL-K} together with the definition \eqref{def:XIL-K}. 
		
		For the estimate \eqref{eq:L-KsimL-K}, we first consider the case $n\in\{2,3\}$. In this case, we have that  
		\begin{align*}
			\max_{\bsig,\ba} \abs{\mathcal{E}^{(\cL-\cK)\times (\cL-\cK),(n)}_{t, \boldsymbol{\sigma}, \ba}} &\lesssim  \p{\max_{\bsig,\ba}\abs{(\mathcal{L} - \mathcal{K})^{(n)}_{t, \bsig, \ba}}}\cdot \max_{\bsig,\ba}W^{d}\sum_{a_2}\abs{(\mathcal{L} - \mathcal{K})^{(2)}_{t, \bsig,\ba}} \prec \frac{(W^{-d} {B}_{t,0})^{\frac16}}{\eta_t}{\max_{\bsig,\ba}\abs{(\mathcal{L} - \mathcal{K})^{(n)}_{t, \bsig, \ba}}},
		\end{align*}
		which concludes \eqref{eq:L-KsimL-K} for $n\in\{2,3\}$ by using the definition \eqref{def:XIL-K}. Above, in the second step, we use the $2$-$G$-loop estimate \eqref{Eq:Gdecay_w} to get that 
		\be\label{eq:sumtwoloop}W^{d}\sum_{a_2}\abs{(\mathcal{L} - \mathcal{K})^{(2)}_{t, \bsig,\ba}} \prec \p{\frac{1-s}{1-u}}^{C_d}(W^{-d} {B}_{t,0})^{1/5}\cdot \sum_{a_2} {\cal T}_{t}(|a_1-a_2|)+W^{-D} \prec \frac{(W^{-d} {B}_{t,0})^{1/6}}{\eta_t} \ee
		under the condition \eqref{con_st_ind} as long as $\fc_d$ is chosen sufficiently small depending on $C_d$. 
		For general $n\ge 4$, we need to bound that 
		\begin{align}\nonumber
			\max_{\bsig,\ba}\abs{\mathcal{E}^{(\cL-\cK)\times (\cL-\cK),(\fn)}_{t, \boldsymbol{\sigma}, \ba}} &\lesssim \sum_{n'= \lceil n/2\rceil +1}^{n-1} \p{\max_{\bsig,\ba}\abs{({\cal L}-{\cal K})^{(n+2-n')}_{t,\bsig,\ba}}} \cdot \max_{\bsig,\ba}  
			W^d \sum_{a_{n'}}\abs{({\cal L}-{\cal K})^{(n')}_{t,\bsig,\ba}} \\
			&+ \p{\max_{\bsig,\ba}\abs{({\cal L}-{\cal K})^{(n)}_{t,\bsig,\ba}}} \cdot \max_{\bsig,\ba}  
			W^d \sum_{a_{2}}\abs{({\cal L}-{\cal K})^{(2)}_{t,\bsig,\ba}}.\label{eq:sumtwoloop2}
		\end{align}
		Note that the second term on the RHS can be controlled using \eqref{eq:sumtwoloop} again, while the first term on the RHS can be handled with the $\contract$ inequality \eqref{yi2oslxj2} for $\cL$-loops and the bound \eqref{wardineq_K} for $\cK$-loops:
		\begin{align*}
			W^d \sum_{a_{n'}}\abs{({\cal L}-{\cal K})^{(n')}_{t,\bsig,\ba}} &\le W^d \sum_{a_{n'}}\p{\abs{{\cal L}^{(n')}_{t,\ba,\bsig}}+\abs{{\cal K}^{(n')}_{t,\ba,\bsig}}}\\
			&\le \frac{1}{\eta_t}\min_{k=1}^{n'-1}\left(\max_{\boldsymbol{\sigma},\bfb}  \left|{\cal L}^{(2k-1)}_{t, \boldsymbol{\sigma}, \bfb}\right|\cdot \max_{\boldsymbol{\sigma},\bfb}  \left|{\cal L}^{(2n'-2k-1)}_{t, \boldsymbol{\sigma},\bfb}\right|\right)^{1/2}+ \frac{1}{\eta_t}(W^{-d}B_{t,0})^{n'-2}.
		\end{align*}
		Then, using \eqref{def:XiL} and \eqref{def:XIL-K}, and setting $k=\lceil n'/2\rceil$, we can bound the RHS of \eqref{eq:sumtwoloop2} by that of \eqref{eq:L-KsimL-K}.
		
		Finally, to establish the estimate \eqref{eq:MG_nloop}, we apply the inequality \eqref{u2jzooi-2}---with $n$ replaced by $2n+2$, and choosing $k=n$, $j=n+1$, and $l=n+2$). This yields that for any $p\ge 1$,
		\begin{align*}
			\max_{\bsig}\max_{\ba,\ba'}\abs{\left( \cal E\times  \cal E \right)^{M}_{t, \boldsymbol{\sigma}, \ba,\ba'}}&\lesssim \max_{\boldsymbol{\sigma},\ba} W^d \sum_{a_{n+1}\sim a_{2n+2}}  |{\cal L}^{(2n+2)}_{t, \boldsymbol{\sigma}, \ba}|  \lesssim \frac{1}{\eta_t} \p{\max_{\boldsymbol{\sigma},\bfb}  \left|{\cal L}^{(2n-1)}_{t, \boldsymbol{\sigma}, \bfb }\right|} \cdot  \left(\max_{\boldsymbol{\sigma},\bfb} \left|  {\cal L}^{(4p)}_{t, \boldsymbol{\sigma}, \bfb}\right|\right)^{\frac{1}{2p}}. 
		\end{align*}
		Combined with the definition \eqref{def:XiL}, this completes the proof of \eqref{eq:MG_nloop}.
	\end{proof}
	
	With the estimates from \Cref{lem:SEforLn}, we now proceed to analyze the integrated loop hierarchy in \eqref{int_K-LcalE}. 
	Without loss of generality, the proof can be divided into two cases, depending on whether (i) $1-t \ge \ilambda^2/L^{2}$, or (ii) $1-s \le \ilambda^2/L^{2}$.\footnote{If $1-t < \ilambda^2/L^{2} < 1-s$, then we can add a middle time $u=1-\ilambda^2/L^{2}$ and perform the proofs for case (i) from $s$ to $u$, and for case (ii) from $u$ to $t$.} 
	In case (i), we employ the sum-zero operator introduced in \cite{YY_25}. In contrast, case (ii) requires a new approach based on the removal of zero modes from the $\cL$-loops.
	
	\subsection{Proof of Step 3: The case \texorpdfstring{$1-t \ge \ilambda^2/L^{2}$}{1-t >= L-2}}
	
	Throughout this subsection, we always assume that $1-t \ge \ilambda^2/L^{2}$, in which case we have 
	\be\label{eq:Bu0asymp}
	B_{u,0}\asymp (\ilambda^2+|1-u|)^{-1},\quad  \ell_u\asymp \ilambda\eta_u^{-1/2}+1, \quad \forall u\in [s,t].
	\ee 
	In Step 2 of the proof of \Cref{lem:main_ind}, we have established an exponential decay of the 2-$G$-loops beyond the scale $\ell_u$ as shown in \eqref{Eq:Gdecay_w}. With \eqref{GijGEX}, we can easily extend this decay to general $G$-loops.
	
	\begin{definition}[Fast decay property]\label{Def_decay}
		Let ${\cal A}: (\Zn)^{\fn}\to \mathbb C$ be an $\fn$-dimensional tensor for a fixed $\fn\ge 2$. Given $u\in[s,t]$ and constants $\e,D>0$, we say $\cal A$ satisfies the $(u, \e, D)$-decay property if 
		\begin{equation}\label{deccA}
			\max_{i,j\in \qqq{\fn}}|a_i-a_j|\ge W^{\e}\ell_u \ \implies \ {\cal A}_{\ba}=\OO(W^{-D})\quad \text{for}\quad \ba=(a_1,a_2,\ldots, a_\fn).
		\end{equation}
	\end{definition}

	It is easy to see that the $G$-loops satisfy the $(u, \e, D)$-decay property for any constants $\e,D>0$ under the estimate \eqref{Eq:Gdecay_w} by using \Cref{lem_GbEXP}. Moreover, in \Cref{Sec:CalK}, we will present a tree representation formula for the $\cK$-loops, which is formed with the $\Theta$-propagators. Thus, the $\cK$-loops also satisfy the $(u, \e, D)$-decay property for any constants $\e,D>0$ by using \eqref{prop:ThfadC}.
	
	\begin{claim}\label{lem_decayLoop} 
		Suppose the estimates \eqref{Gt_bound_flow} and \eqref{Eq:Gdecay_w} hold. For any $\fn\ge 2$, $\bsig\in \{+,-\}^\fn$, $u\in[s,t]$, and constants $\e,D>0$, the loops \smash{${\cal L}^{(\fn)}_{u,\bsig,\ba}$} and \smash{${\cal K}^{(\fn)}_{u,\bsig,\ba}$} satisfy the $(u, \e, D)$-decay property with probability $1-\OO(W^{-D'})$ for any large constant $D'>0$. In other words, we have that
		\begin{align}\label{res_decayLK}
			\mathbb P\left( \max_{\boldsymbol{\sigma}}  \Big(\left|{\cal L}^{(\fn)}_{u,\boldsymbol{\sigma},\ba}\right|+\left|{\cal K}^{(\fn)}_{u,\boldsymbol{\sigma},\ba}\right|\Big)\cdot{\bf 1}\left(\max_{  i, j\in \qqq{\fn} } |a_i-a_j|\ge W^{\e}\ell_u \right) \ge W^{-D}\right)\le W^{-D'}. 
		\end{align}
	\end{claim}
	
	Due to the fast decay property of the $G$-loops and $\cK$-loops, when the evolution kernels act on them, we can apply the evolution kernel estimates in \Cref{lem:sum_decay}. In particular, for non-alternating loops $\bsig$, using the estimates in \Cref{lem:SEforLn} and the estimate \eqref{sum_res_2_NAL} in \Cref{lem:sum_decay}, we readily establish the following lemma.

	\begin{lemma}[Non-alternating loops]\label{lem:STOeq_NQ}  
		Under the assumptions of \Cref{lem:main_ind}, suppose $1-t\ge \ilambda^2/L^{2}$ and the estimates \eqref{Gt_bound_flow}--\eqref{Eq:Gdecay_w} hold uniformly in $u\in[s,t]$. Fix any $\fn\ge 2$ and $\bsig\in \{+,-\}^\fn$ satisfying 
		\begin{equation}\label{NALsigm}
			\sigma_k=\sigma_{k+1} \quad \text{for some}\quad  k\in\qqq{\fn}.
		\end{equation} 
		(Recall that $\sig_{\fn+1}=\sig_1$ as a convention.) 
		Then, we have the following estimate for any fixed $p\ge 1$: 
		\begin{align}\label{am;asoiuw}
			\max_{\ba\in (\Zn)^\fn} \big|({\cal L}-{\cal K})^{(\fn)}_{t, \boldsymbol{\sigma}, \ba} \big|  \big/ (W^{-d}B_{t,0})^n\prec \sup_{u\in [s,t]}
			\left( (W^{-d}B_{t,0})^{\frac 1 6}\wh\Xi^{(\cal L-\cK)}_{u, n}+ \max_{n'=2}^{\fn-1} \Xi^{({\cal L-\cal K})}_{u,n'}+ \max_{n'=\fn-1}^{\fn+1} \Xi^{({\cal L})}_{u,n'} \right)\\
			+\sup_{u\in [s,t]}
			\left(\max_{n'=\lceil n/2\rceil + 1}^{n-1} \Xi^{({\cal L-\cal K})}_{u,n+2-n'} \p{\Xi^{({\cal L})}_{u,n_1'} \Xi^{({\cal L})}_{u,n_2'}}^{\frac 12 } +(W^{-d}B_{t,0})^{-\frac{1}{4p}}(\Xi^{(\cal L)}_{u, {2\fn-1}})^{\frac 1 2}(\Xi^{(\cal L)}_{u, 4p})^{\frac{1}{4p}}\right),\nonumber
		\end{align} 
		where, as defined below \eqref{eq:L-KsimL-K}, $n'_1=n_2'=n'-1$ if $n'$ is even, and $n'_1=n'-2$ and $n_2'=n'$ if $n'$ is odd. 
	\end{lemma}
	
	\begin{proof}
		By \Cref{lem_decayLoop}, the light-weight term $\cal E_{t,\bsig,\ba}^{\Gc,(n)}$ satisfies condition \eqref{deccA0} below for any constants $\e,D>0$. Then, under the assumption \eqref{NALsigm}, using the bound \eqref{bEwGn} together with the evolution kernel estimate \eqref{sum_res_2_NAL} below, we can control the fourth term on the RHS of \eqref{int_K-LcalE} as  
		\begin{align*}
			\int_{s}^t \left(\mathcal{U}^{(\fn)}_{u, t, \boldsymbol{\sigma}} \circ \mathcal E^{\Gc,(\fn)}_{u, \boldsymbol{\sigma}}\right)_{\ba} \dd u &\prec \int_{s}^t \p{\frac{\ilambda^2+|1-u|}{\ilambda^2+|1-t|}}^{\fn-1} \frac{\p{W^{-d}B_{u,0}}^n}{\eta_u} \p{\wh\Xi^{({\cal L})}_{u, \fn_1}\cdot \wh\Xi^{({\cal L})}_{u, \fn_2}}^{\frac12} \dd u \\
			&\prec \p{W^{-d}B_{t,0}}^n\int_{s}^t\frac1{\eta_u} \max_{n'=n-1}^{n+1} \Xi^{(\cal L)}_{u,n'}\dd u ,
		\end{align*}
		where in the second step, we use $\frac{\ilambda^2+|1-u|}{\ilambda^2+|1-t|} B_{u,0} \le B_{t,0}$ and the control parameter in \eqref{def:XiL}.
		Using the bounds \eqref{eq:KsimL-K} and \eqref{eq:L-KsimL-K}, together with the assumption \eqref{Eq:L-KGt+IND} on \smash{$(\mathcal{L} - \mathcal{K})^{(\fn)}_{s, \boldsymbol{\sigma}, \ba}$} and the evolution kernel estimate  \eqref{sum_res_2_NAL}, the first three terms on the RHS of \eqref{int_K-LcalE} can be bounded in the same manner. This yields that 
		\begin{align}\label{sahwNQ}
			&(\mathcal{L} - \mathcal{K})^{(n)}_{t, \boldsymbol{\sigma}, \ba}\big/ (W^{-d}B_{t,0})^n \prec 
			1 + (W^{-d}B_{t,0})^{-n} \int_{s}^t \left(\mathcal{U}^{(\fn)}_{u, t, \boldsymbol{\sigma}} \circ  \dd  \mathcal{E}^{M,(\fn)}_{u, \boldsymbol{\sigma}}\right)_{\ba} +   \int_{s}^t \frac1{\eta_u} \max_{n'=2}^{\fn-1} \Xi^{({\cal L-\cal K})}_{u,n'}\dd u  
			\\
			&+ \int_{s}^t \frac1{\eta_u} \max_{n'=n-1}^{n+1} \Xi^{(\cal L)}_{u,n'}\dd u +  \int_{s}^t \frac1{\eta_u} \max_{n'=\lceil n/2\rceil + 1}^{n-1} \Xi^{({\cal L-\cal K})}_{u,n+2-n'} \p{\Xi^{({\cal L})}_{u,n_1'} \Xi^{({\cal L})}_{u,n_2'}}^{\frac 12 } \dd u + (W^{-d}B_{t,0})^{\frac16} \int_{s}^t \frac{1}{\eta_u}  \wh\Xi^{({\cal L-\cal K})}_{u,n} \dd u  .\nonumber
		\end{align}
		For the martingale term, by Lemma \ref{lem:DIfREP}, and using \eqref{eq:MG_nloop} together with \eqref{sum_res_2_NAL}, we obtain that
		\begin{align*} 
			(W^{-d}B_{t,0})^{-n}\int_{s}^t \left(\mathcal{U}^{(\fn)}_{u, t, \boldsymbol{\sigma}} \circ  \dd  \mathcal{E}^{M,(\fn)}_{u, \boldsymbol{\sigma}}\right)_{\ba} & \prec (W^{-d}B_{t,0})^{-n} \left\{\int_{s}^t 
			\left(\left(
			\mathcal{U}^{(\fn)}_{u,t,  \boldsymbol{\sigma}}
			\otimes 
			\mathcal{U}^{(\fn)}_{u,t,  \overline{\boldsymbol{\sigma}}}
			\right) \;\circ  \;
			\left( \cE \otimes  \cE \right)^{M,(\fn)}
			_{u, \boldsymbol{\sigma}  }
			\right)_{\ba, \ba}\dd u\right\}^{1/2} \nonumber\\
			&\prec (W^{-d}B_{t,0})^{-\frac{1}{4p}}\left\{\int_{s}^t 
			\frac{1}{\eta_u}\Xi^{(\cal L)}_{u, {2\fn-1}}\p{\Xi^{(\cal L)}_{u, 4p}}^{\frac{1}{2p}}\dd u\right\}^{1/2} 
		\end{align*} 
		for any fixed $p\ge 1$. Plugging it into \eqref{sahwNQ} and performing the integral over $u$, we conclude \eqref{am;asoiuw}.\end{proof}

	\noindent{\textbf{Alternating cases}}: It remains to deal with the case with alternating signs, where \eqref{NALsigm} does not occur:
	\begin{equation}\label{NALsig_diff}
		\sigma_k=- \sigma_{k+1}, \quad \forall \ k\in\qqq{\fn}\, .
	\end{equation}
	For this purpose, we introduce another key tool---a sum-zero operator ${\cal Q}_t$.
	
	\begin{definition}[Partial sum and sum-zero operator]\label{Def:QtPt} 
		Let ${\cal A}: (\Zn)^{\fn}\to \mathbb C$ be an $\fn$-dimensional tensor for a fixed $\fn\in \N$ with $\fn\ge 2$. Define the partial sum operator ${\cal P}$ as 
		$$
		\left( {\cal P} \circ {\cal A}\right)_{a_1}:= \sum_{a_i: i\in\qqq{2,\fn}}  {\cal A}_{\ba},\quad \ba=(a_1,\ldots, a_n). 
		$$
		We say a tensor $\cal A$ satisfies the \emph{sum-zero property} if $ {\cal P} \circ {\cal A}\equiv 0$. For $t\in[0,1)$, we define the operator $\cal Q_t$ as
		\be\label{eq:sumzero_op}
		\left({\cal Q}_{t}\circ {\cal A}\right)_{\ba}:={\cal A}_{\ba}-\left({\cal P} \circ {\cal A}\right)_{a_1}\dthn^{(\fn)}_{t, \ba},\ee
		where the tensor $\dthn^{(\fn)}_{t, \ba}$ is a mollifier satisfying \be\label{eq:suma1chi}\sum_{a_2,\ldots,a_n}\dthn^{(\fn)}_{t, \ba}\equiv 1,\quad \forall a_1\in \Zn,
		\ee
		along with the following estimates for a constant $c>0$: 
		\be\label{eq:derv_Theta} \dthn_{t,\ba}^{(\fn)} \prec (\ell_t^d)^{-(\fn-1)}\exp\p{-c\sum_{i=2}^n|a_i-a_1|/\ell_t} ,\quad 
		\max_{\ba}\|\partial_t\dthn_{t,\ba}^{(\fn)}\|_\infty \prec |1-t|^{-1}(\ell_t^d)^{-(\fn-1)}. \ee  
		The detailed form of $\dthn^{(\fn)}_{t, \ba}$ is not important to us, and we will give an example in \Cref{rmk:choosechi} below. 
		With equation \eqref{eq:suma1chi}, we see that \smash{$ {\cal P} \circ  \dthn^{(\fn)}_{t, \ba}\equiv 1$}, ${\cal P} \circ {\cal Q}_t =0$, and that for any tensor $\cal A$,
		\be\label{eq:sum0PA}
		{\cal P} \circ {\cal A} \equiv 0\ \ \implies\ \ {\cal P} \circ \left(\varTheta^{(\fn)}_{t, \boldsymbol{\sigma}}  \circ {\cal A}\right) \equiv 0,
		\ee
		where we recall that \(\varTheta^{(\fn)}_{t, \boldsymbol{\sigma}} \) is the operator defined in \Cref{DefTHUST}. In other words, if \({\cal A}\) satisfies the sum-zero property, then so does \smash{$\varTheta^{(\fn)}_{t, \boldsymbol{\sigma}}  \circ {\cal A}$}. 
	\end{definition}

	\begin{example}\label{rmk:choosechi}
		To construct the mollifier tensor \(\dthn_{t,\ba}^{(\fn)}\), we first choose a compactly supported, smooth, non-negative function \smash{$f\in C_c^{\infty}(\R^d)$} that is not identically equal to zero, and rescale it as \smash{\(f_t(a):=\ell_t^{-d}f(|a|/\ell_t)\)} for $a\in \Zn$. We then define, with an appropriate normalization constant $C_{f,t}$,  
		\begin{align}
			\dthn^{(\fn)}_{t, \ba}:=C_{f,t}\cdot \prod_{i=2}^n   f_t(a_i-a_1) ,\quad \forall \ba\in (\Zn)^n .
		\end{align}
		It is easy to see that this function satisfies the required properties in \eqref{eq:derv_Theta}. 
	\end{example}

	We will use the sum-zero operator to get improved estimates on the terms on the RHS of \eqref{int_K-LcalE} when \eqref{NALsig_diff} holds. 
	Roughly speaking, we will decompose a tensor $\cal A$ as \smash{${\cal A}_{\ba}= ({\cal Q}_{t}\circ {\cal A})_{\ba} + \left({\cal P} \circ {\cal A}\right)_{a_1}\dthn^{(\fn)}_{t, \ba}$} using \eqref{eq:sumzero_op}. For the first part, we can get an improvement by using the evolution kernel estimate \eqref{sum_res_2}, while for the second part, we can apply Ward's identity to ${\cal P} \circ {\cal A}$. 
	
	We first claim that when $\bsig$ satisfies \eqref{NALsig_diff}, the following estimate holds uniformly in $u\in[s,t]$:
	\be\label{eq:Ward_typeP}
	\left[\cal P\circ (\mathcal{L} - \mathcal{K})^{(\fn)}_{u, \boldsymbol{\sigma}}\right]_{a_1} \dthn^{(\fn)}_{u, \ba} \prec   (W^{-d}B_{u,0})^{\fn}  \Xi^{(\mathcal{L}-\mathcal{K})}_{u, \fn-1}.
	\ee
	To see this, we apply Ward's identities in \Cref{lem_WI_K} at the vertex $a_\fn$ and get  
	$$\left[\cal P\circ (\mathcal{L} - \mathcal{K})^{(\fn)}_{u, \boldsymbol{\sigma}}\right]_{a_1}= \frac{1}{2\ii W^d \eta_u}\left[\cal P\circ (\mathcal{L} - \mathcal{K})^{(\fn-1)}_{u, \wh\bsig^{(+,\fn)}}-\cal P\circ (\mathcal{L} - \mathcal{K})^{(\fn-1)}_{u, \wh\bsig^{(-,\fn)}}\right]_{a_1}.$$
	By \eqref{def:XIL-K}, the two $(\cL-\cK)^{(\fn-1)}$-loops on the RHS are controlled by 
	$$(\mathcal{L} - \mathcal{K})^{(\fn-1)}_{u, \wh\bsig^{(\pm,\fn)},\wh \ba^{(n)}} \prec (W^{-d}B_{u,0})^{\fn-1}\Xi^{(\mathcal{L}-\mathcal{K})}_{u, \fn-1}.$$ Moreover, due to the fast decay property of the $(\cal L-\cal K)$-loops, the partial sums over the remaining $(\fn-2)$ vertices lead to an additional \smash{$\ell_u^{d(\fn-2)}$} factor up to a negligible error $\OO(W^{-D})$. This leads to
	\begin{align}\label{jywiiwsoks}
		\left[\cal P\circ (\mathcal{L} - \mathcal{K})^{(\fn)}_{u, \boldsymbol{\sigma}}\right]_{a_1}  \prec (W^d\eta_u)^{-1}(\ell_u^{d})^{n-2}\cdot (W^{-d}B_{u,0})^{n-1}   \Xi^{(\mathcal{L}-\mathcal{K})}_{u, \fn-1}.
	\end{align}
	Together with \eqref{eq:derv_Theta}, it implies that 
	\begin{align*}
		\left[\cal P\circ (\mathcal{L} - \mathcal{K})^{(\fn)}_{u, \boldsymbol{\sigma}}\right]_{a_1} \dthn^{(\fn)}_{u, \ba} \prec (W^d\ell_u^d\eta_u)^{-1}\cdot (W^{-d}B_{u,0})^{n-1}\Xi^{(\mathcal{L}-\mathcal{K})}_{u, \fn-1} \lesssim (W^{-d}B_{u,0})^{n}\Xi^{(\mathcal{L}-\mathcal{K})}_{u, \fn-1} 
	\end{align*}
	uniformly in $u\in [s,t]$, where, in the second step, we use $(\ell_u^d\eta_u)^{-1}\lesssim B_{u,0}$ by \eqref{eq:Bu0asymp}.

	It remains to control ${\cal Q}_t \circ (\mathcal{L} - \mathcal{K})^{(\fn)}_{t, \boldsymbol{\sigma}, \ba}$. For this purpose, we need the following claim on the $(\infty\to \infty)$-norm of the sum-zero operator, which follows easily from \Cref{Def:QtPt} and the estimate \eqref{eq:derv_Theta}.  
	
	\begin{claim}\label{lem_+Q} 
		Let ${\cal A}: (\Zn)^{\fn}\to \mathbb C$ be an $\fn$-dimensional tensor for a fixed $\fn\ge 2$. If $\cal A$ satisfies the $(t, \e, D)$-decay property, then we have that 
		\begin{align}\label{normQA}
			\left\|{\cal Q}_t\circ \cal A\right\|_{\infty} \le W^{C_\fn \e} \|\cal A\|_\infty +W^{-D+C_\fn}
		\end{align}
		for a constant $C_\fn$ that does not depend on $\e$ or $D$. Furthermore, if $\|\cal A\|_\infty\le W^C$ for a constant $C>0$, then $\cal A_\fa-({\cal Q}_t\circ \cal A)_{\ba}=\left({\cal P} \circ {\cal A}\right)_{a_1}\dthn^{(\fn)}_{t, \ba}$ satisfies the $(t, \e', D')$-decay property for any constants $\e',D'>0$. 
	\end{claim}

	We derive from equation \eqref{eq_L-Keee} that 
	\begin{align} 
		\dd{\cal Q}_t \circ (\mathcal{L} - \mathcal{K})^{(\fn)}_{t, \boldsymbol{\sigma}, \ba} 
		&= {\cal Q}_t \circ \left[{\cK}^{(2)} \sim (\mathcal{L} - \mathcal{K})\right]^{(\fn)}_{t, \boldsymbol{\sigma}, \ba} 
		+ \sum_{l_\mathcal{K}=3}^\fn {\cal Q}_t \circ \Big[{\cK}^{(\lenk)}\sim (\mathcal{L} - \mathcal{K})\Big]^{(\fn)}_{t, \boldsymbol{\sigma}, \ba} + {\cal Q}_t \circ \mathcal{E}^{(\cL-\cK)\times(\cL-\cK),(\fn)}_{t, \boldsymbol{\sigma}, \ba}\dd t \nonumber
		\\
		& + {\cal Q}_t \circ \mathcal{E}^{\Gc,(\fn)}_{t, \boldsymbol{\sigma}, \ba} \dd t
		+ {\cal Q}_t \circ \dd\mathcal{E}^{M,(\fn)}_{t, \boldsymbol{\sigma}, \ba} 
		- \left[{\cal P} \circ \left(\mathcal{L} - \mathcal{K}\right)^{(\fn)}_{t,\bsig}\right]_{a_1} \big(\partial_t \dthn_{t,\ba}^{(\fn)}\big) \dd t.\label{zjuii1}
	\end{align} 
	Recalling \Cref{DefTHUST}, we can rewrite the first term on the RHS as 
	\begin{align}\label{zjuii2}
		{\cal Q}_t \circ \Big[{\cK}^{(2)}\sim (\mathcal{L} - \mathcal{K})\Big]_{t, \boldsymbol{\sigma},\ba}^{(\fn)}
		& = \varTheta^{(\fn)}_{t, \boldsymbol{\sigma}} \circ \left[{\cal Q}_t \circ (\mathcal{L} - \mathcal{K})_{t, \boldsymbol{\sigma}}^{(\fn)}\right]_{\ba} 
		+ \left[{\cal Q}_t , \varTheta^{(\fn)}_{t, \boldsymbol{\sigma}} \right] \circ (\mathcal{L} - \mathcal{K})_{t, \boldsymbol{\sigma}, \ba}^{(\fn)},
	\end{align}
	where  $[{\cal Q}_t , \varTheta^{(\fn)}_{t, \boldsymbol{\sigma}} ] = {\cal Q}_t \circ \varTheta^{(\fn)}_{t, \boldsymbol{\sigma}} - \varTheta^{(\fn)}_{t, \boldsymbol{\sigma}}\circ \cal Q_t$ denotes the commutator
	between ${\cal Q}_t $ and $\varTheta^{(\fn)}_{t, \boldsymbol{\sigma}}$. Since \({\cal P} \circ {\cal Q}_t = 0\), we notice that the first 5 terms on the RHS of \eqref{zjuii1} satisfy the sum zero property. 
	Since \smash{${\cal P}  \circ \dthn^{(\fn)}_{t,  \ba}\equiv 1$}, we have \smash{${\cal P} \circ \big(\partial_t \dthn_{t,\ba}^{(\fn)}\big)=0$}, so the last term on the RHS of \eqref{zjuii1} also satisfies the sum zero property.
	Next, due to \eqref{eq:sum0PA}, the first term on the RHS of \eqref{zjuii2} also satisfies the sum-zero property. 
	Finally, since the LHS of \eqref{zjuii2} has the sum-zero property, the second term on the RHS of \eqref{zjuii2} also satisfies the sum-zero property.

	With Duhamel's principle, we can derive from \eqref{zjuii1} and \eqref{zjuii2} the following counterpart of \eqref{int_K-LcalE}:
	\begin{align}\label{int_K-L+Q}
		{\cal Q}_t\circ   (\mathcal{L} - \mathcal{K})^{(n)}_{t, \boldsymbol{\sigma}, \ba} &= 
		\left(\mathcal{U}^{(n)}_{s, t, \boldsymbol{\sigma}} \circ  {\cal Q}_s\circ  {\cal B}_0(s)\right)_{\ba} +   \int_{s}^t \left(\mathcal{U}^{(n)}_{u, t, \boldsymbol{\sigma}} \circ  {\cal Q}_u\circ \sum_{k=1}^5{\cal B}_k(u) \right)_{\ba} \dd u \nonumber\\
		&+  \int_{s}^t \left(\mathcal{U}^{(n)}_{u, t, \boldsymbol{\sigma}} \circ  {\cal Q}_u\circ \dd\mathcal{E}^{M,(n)}_{u, \boldsymbol{\sigma}}\right)_{\ba} ,
	\end{align}
	where the tensors $\cal B_i$ for $i\in \qqq{0,5}$ are defined as follows:
	\begin{align*}
		&{\cal B}_0(s):=(\mathcal{L} - \mathcal{K})^{(n)}_{s, \boldsymbol{\sigma}},\quad {\cal B}_1(u):=\sum_{\lenk=3}^n \Big[\mathcal{K}^{(\lenk)} \sim (\mathcal{L} - \mathcal{K})\Big]^{(n)}_{u, \boldsymbol{\sigma}},\quad {\cal B}_2(u):=\mathcal{E}^{(\mathcal{L} - \mathcal{K}) \times (\mathcal{L} - \mathcal{K}),(n)}_{u, \boldsymbol{\sigma}},\\
		&{\cal B}_3(u):=\mathcal{E}^{\Gc,(n)}_{u, \boldsymbol{\sigma}},\quad {\cal B}_4(u):= \big[{\cal Q}_u , \varTheta^{(n)}_{u,\boldsymbol{\sigma}} \big]\circ (\mathcal{L} - \mathcal{K})^{(n)}_{u, \boldsymbol{\sigma} },\quad {\cal B}_5(u):= -\br{{\cal P} \circ\left(\mathcal{L} - \mathcal{K}\right)^{(n)}_{u, \boldsymbol{\sigma}}}\cdot \partial_t  \dthn^{(n)}_{u } . 
	\end{align*}
	We can control the terms on the RHS of \eqref{int_K-L+Q} by applying the improved evolution kernel estimate \eqref{sum_res_2}, which exploits both the sum-zero and fast-decay properties of the terms on the RHS of \eqref{zjuii1}. Combining this with \eqref{eq:Ward_typeP} and \Cref{lem:STOeq_NQ} for the non-alternating case, we obtain the following inductive bootstrap bounds for the $\Xi$-parameters. The proof proceeds by estimating the RHS of \eqref{int_K-L+Q} 
	in a manner analogous to the proof of \Cref{lem:STOeq_NQ}, using the bounds established in \Cref{lem:SEforLn} together with straightforward controls of the $\mathcal{B}_4$ and $\mathcal{B}_5$ terms, based on the properties in \eqref{eq:derv_Theta}. Hence, we defer the detailed proof to \Cref{subsec:support}.

	\begin{lemma}[Inductive bootstrap bound for $\Xi$-parameters]
		\label{lem:STOeq_Qt} 
		Under the assumptions of \Cref{lem:main_ind}, suppose $1-t\ge \ilambda^2/L^{2}$ and the estimates \eqref{Gt_bound_flow}--\eqref{Eq:Gdecay_w} hold uniformly in $u\in[s,t]$. Then, for any fixed $\fn\ge 2$ and $p\ge 1$, the following bound holds uniformly in $u\in[s,t]$: 
		\begin{align}\label{am;asoi222}
			\sup_{v\in[s,u]} \wh\Xi^{({\cal L-\cal K})}_{v,\fn} &\prec \sup_{v\in[s,u]} \left(\p{W^{-d}B_{u,0}}^{-\frac{1}{4p}}(\Xi^{(\cal L)}_{v, {2\fn-1}})^{\frac 1 2}(\Xi^{(\cal L)}_{v, 4p})^{\frac{1}{4p}}\right)\\
			&+ \sup_{v\in[s,u]} \left(\max_{n'=1}^{\fn-1} \Xi^{({\cal L-\cal K})}_{v,n'}+ \max_{n'=\fn-1}^{\fn+1} \Xi^{({\cal L})}_{v,n'} + \max_{n'=\lceil n/2\rceil + 1}^{n-1} \Xi^{({\cal L-\cal K})}_{v,n+2-n'} \p{\Xi^{({\cal L})}_{v,n_1'} \Xi^{({\cal L})}_{v,n_2'}}^{\frac 12 } \right).\nonumber
		\end{align}
	\end{lemma}

	Now, similar to the argument in Section 5.6 of \cite{YY_25}, we will iterate the bootstrap bound \eqref{am;asoi222} to obtain the sharp $L^\infty$-bound \eqref{Eq:LGxb} on the $G$-loops in the regime $1-t\ge \ilambda^2/L^{2}$. That is, for any fixed $\fn\in \N$, 
	\be\label{eq:Xiiter}
	\sup_{u\in [s,t]} \wh\Xi^{(\cL)}_{u,\fn} \prec 1.
	\ee
	Observe that when $1-s>\ilambda^2$, we have \(B_{u,0}\asymp |1-u|^{-1}\) for all $u\in [s,1-\ilambda^2]$. Consequently, the $G$-loop bound \eqref{Eq:LGxb} follows directly from \eqref{lRB1}. 
	Hence, in the remainder of the proof, it suffices to consider the case $1-t\le 1-s \le \ilambda^2$, where \(W^{-d}B_{u,0}\asymp (\ilambda^2W^d)^{-1}\) for all \(u\in [s,t].\) 
	First, the averaged local law \eqref{Gt_avgbound_flow} gives that \smash{$\wh\Xi_{u,1}^{(\cL-\cK)}\prec 1=:\Xi_{u,1}^{(\cL-\cK)}$} uniformly for $u\in[s,t]$. Second, by the $\cK$-loop bound \eqref{eq:bcal_k}, we have 
	\begin{align}\label{rela_XILXILK}
		\wh\Xi^{(\cal L)}_{u,\fn}\prec 1+(\ilambda^2W^{d})^{-1} \wh\Xi^{({\cal L-\cal K})}_{u, \fn},\quad \quad \wh\Xi^{({\cal L-\cal K})}_{u, \fn}\prec  (\ilambda^2W^{d})\left(\wh\Xi^{(\cal L)}_{u,\fn}+1\right),\quad \text{for}\quad n\ge 2 .
	\end{align}
	Moreover, the a priori $G$-loop bound \eqref{lRB1} provides the following initial estimates, uniform in $u\in[s,t]$:
	\begin{align}\label{sef8w483r324}
		\wh\Xi^{(\cal L)}_{u,\fn}\prec (\eta_s/\eta_u)^{\fn-1}, \quad 
		\wh\Xi^{({\cal L-\cK})}_{u, \fn}\prec (\eta_s/\eta_u)^{\fn-1}\cdot (\ilambda^2W^{d}).
	\end{align}
	We then introduce the control parameter
	\begin{equation}\label{adsyzz0s8d6}
		\Psi(\fn,k;s,t):= (\ilambda^2W^{d})^{3/4} +\p{{\eta_s}/{\eta_t}}^{\fn-1}(\ilambda^2W^{d})^{1-k/8}.
	\end{equation} 
	The iteration will proceed simultaneously in the indices $n$ and $k$. The outcome of each step is summarized in the following lemma, whose proof—being analogous to that of (5.109) in \cite{YY_25}—is deferred to \Cref{subsec:support}.

	\begin{lemma}\label{lem:iterations}
		In the setting of \Cref{lem:main_ind}, suppose \eqref{lRB1}--\eqref{Eq:Gdecay_w} and \eqref{am;asoi222} hold uniformly for all $u\in[s,t]$ with $ 1-s \le \ilambda^2$. Fix any $(\fn,k)\in \N^2$ with $\fn\ge 2$ and $k\ge 1$. Assume that, uniformly for $u\in[s,t]$, 
		\be\label{eq:iteration_induc}
		\sup_{v\in[s,u]}\wh \Xi^{({\cal L-\cK})}_{v,r}\prec   \Psi(r,l;s,u)\ee
		holds for all index pairs $(r,l)\in \{(r, k): 2\le r\le n-1\}\cup \{(r, k-1): 2\le r\le n+2\}.$ 
		Then, \eqref{eq:iteration_induc} also holds for $(r,l)=(n, k)$. 
	\end{lemma}
	
	Roughly speaking, this lemma states if we have already established a ``good" bound for all shorter $G$-loops of length $r\le \fn-1$, along with a ``weaker" bound for $G$-loops of length $r\le \fn+2$, then we can derive the ``good" bound for all $G$-loops of length $\fn$. 
	Using \Cref{lem:iterations}, we apply a simple iterative argument to complete the proof of \eqref{Eq:LGxb} in Step 3 for the case $1-t\ge \ilambda^2/L^{2}$. 
	
	\begin{proof}[\bf Proof of \eqref{Eq:LGxb} when $1-t\ge \ilambda^2/L^{2}$]
		As discussed above, it remains to deal with the case $1-s \le \ilambda^2$. By \eqref{sef8w483r324}, we initially have a weak bound for $G$-loops of arbitrarily large lengths, meaning that \eqref{eq:iteration_induc} holds with $l=0$ for every fixed $r\in \N$. Applying \Cref{lem:iterations} once, we obtain a slightly improved bound \eqref{eq:iteration_induc} for $r=1$ and $l=1$. Then, continuing the iteration in $r$ while keeping $l=1$ fixed, we establish the bound \eqref{eq:iteration_induc} for each fixed $r\in \N$ with $l=1$. 
		Next, applying the iteration in \Cref{lem:iterations} again yields an even stronger bound \eqref{eq:iteration_induc} with $r=2$ and $l=2$. Repeating the iteration in $r$ with $l=2$ fixed, we further establish the bound \eqref{eq:iteration_induc} for every fixed $r\in \N$ with $l=2$. This process continues, progressively improving the bound to \eqref{eq:iteration_induc} for each fixed $r\in \N$ with $l=3$, and so forth.
		
		For any given $(\fn,k)\in \N^2$, by repeating the above procedure for $\OO(1)$ times, we conclude that the estimate \eqref{eq:iteration_induc} holds for $(r,l)=(n,k)$.  
		In particular, if we choose $\fc_d$ in condition \eqref{con_st_ind} sufficiently small and take $k$ large enough so that $(\eta_s/\eta_t)^{\fn-1}\cdot  (\ilambda^2W^{d})^{1-k/8} \ll (\ilambda^2W^{d})^{3/4}$, then we get from \eqref{eq:iteration_induc} that 
		\begin{equation}\label{eq:iterative_pf_largeeta}\sup_{u\in[s,t]}\wh\Xi^{({\cal L-\cK})}_{u,\fn}\prec \Psi(\fn,k;s,t) \lesssim (\ilambda^2W^{d})^{3/4}.\end{equation}
		Together with \eqref{rela_XILXILK}, this implies the estimate \eqref{eq:Xiiter}, 
		which completes the proof of the bound \eqref{Eq:LGxb}.
	\end{proof}

	\subsection{Proof of Step 3: The case \texorpdfstring{$1-s \le \ilambda^2/L^{2}$}{1-s <= L-2}}\label{sec:1-s<L-2}

	In this setting, we have $\ell_u\equiv L$ for all $u\in[s,t]$. Here, a key difference from the $d = 2$ case in \cite{DYYY25} arises in the intermediate regime $\ilambda^2/L^{d} \leq 1 - t \leq \ilambda^2/L^{2}$, where the polynomial decay mode $\ilambda^{-2} / (|a - b| + 1)^{d - 2}$ and the zero mode $(L^d |1 - t|)^{-1}$ mix in the quantity $B_{t, |a - b|}$ (recall \eqref{eq_B_param}).
	In this regime, we lose the fast decay property for $G$-loops that is required by \Cref{lem:sum_decay} below, and instead the weaker evolution kernel estimate \eqref{sum_res_Ndecay} becomes relevant. This estimate introduces a factor of $(1 - s)/(1 - t)$, which transforms the zero mode $(L^d |1 - s|)^{-1}$ at time $s$ into the zero mode $(L^d |1 - t|)^{-1}$ at time $t$. However, this transformation breaks the polynomial decay mode, potentially destabilizing the tree approximation at time $t$.
	To address this difficulty, we observe that the two modes actually propagate independently along the flow. Motivated by this, we introduce a zero-mode-removing operator, which decomposes the $(\cL - \cK)$-loops into two parts: (1) the zero-mode components, which can be controlled using Ward’s identities, and (2) the zero-mode-free components, whose evolution kernel satisfies sharper estimates thanks to the bound \eqref{prop:ThfadC0} (see \eqref{sum_res_Ndecay_nonzero} below).

	\begin{definition}[Zero-mode-removing operators]\label{def;zero_mode_remove}
		Let ${\cal A}: (\Zn)^{\fn}\to \mathbb C$ be an arbitrary $\fn$-dimensional tensor for a fixed $\fn\ge 1$. Define $P^{(i)}$ as the \emph{partial averaging operator} with respect to the $i$-th index of the loop: 
		$$
		(P^{(i)}\circ {\cal A})_{(a_1,\ldots, a_{i-1},b,a_{i+1},\ldots,a_n)}=L^{-d} \sum_{a_i}{\cal A}_{(a_1,\ldots, a_{i-1},a_i,a_{i+1},\ldots,a_n)},\quad \forall i\in \qqq{n}, \ b\in \Zn.$$
		Correspondingly, we define the \emph{zero-mode-removing} operator as $Q^{(i)}:= I-P^{(i)}.$ Furthermore, for any subset $A\subset \qqq{n}$, we denote 
		$$P^{(A)}:=\prod_{i\in A} P^{(i)},\quad \text{and}\quad  
		Q^{(A)}:=\prod_{i\in A} Q^{(i)}.
		$$
		As a convention, when $A=\emptyset$, we define $P^{(\emptyset)}$ and $Q^{(\emptyset)}$ as the identity operator. 
		Note that for any subsets $A$ and $A'$, the operators $P^{(A)}$, $P^{(A')}$, $Q^{(A)}$, and $Q^{(A')}$ all commute with each other. 
	\end{definition}
	
	By definition, the $(\infty\to \infty)$-operator norm of $Q^{(i)}$ is trivially bounded:
	\begin{align}\label{normQA2}
		\big\|Q^{(i)}\circ \cal A\big\|_{\infty} \le 2\|\cal A\|_\infty .
	\end{align}
	Given any $\bsig\in \{+,-\}^n$, we denote
	\( I_{\mathrm{diff}}(\bsig):=\{i\in\qqq{n}:\sig_i\ne\sig_{i+1}\},\)
	where we again adopt the cyclic convention $\sig_{n+1}=\sig_1$. Given any $\cL$- or $\cK$-loop, we can express it as a linear combination of $(Q^{(A)}\circ \cL)_{t,\bsig,\ba}$ or $(Q^{(A)}\circ \cK)_{t,\bsig,\ba}$ loops with $A\supset I_{\mathrm{diff}}(\bsig)$, by repeatedly applying \eqref{WI_calL} and \eqref{WI_calK}. 
	This is summarized in \Cref{lem: newPQ}, whose proof is deferred to \Cref{subsec:support} below.

	\begin{lemma}\label{lem: newPQ}
		For any fixed $n\in \N$, $\bsig\in \{+,-\}^n$, $\ba\in(\Zn)^n$, and subset $A\subset\qqq{n}$, we have the expansion: 
		\begin{align}\label{yurenAL}
			\left(Q^{(A)}\circ {\cal L}^{(n)}\right)_{t,\boldsymbol{\sigma},\ba}=\left(Q^{(A_n)}\circ {\cal L}^{(n)}\right)_{t,\boldsymbol{\sigma},\ba}+\sum_\al \frac{\xi_{\al}}{(2\ii N\eta_t)^{n-k_\al}} \left( Q^{(A_\al)}\circ {\cal L}^{(k_\al)}\right)_{t,\boldsymbol{\sigma}_{\al},\ba_{\al}},
		\end{align}
		where $A_n:=A\cup I_{\mathrm{diff}}(\bsig)$, $\{\al\}$ denotes a collection of $\OO(1)$ many labels, $1\le k_\al \le n-1$, $\ba_{\al}\in (\Zn)^{k_\al}$ consists of a subset of indices in $\ba$, $\bsig_\al\in \{+,-\}^{k_\al}$, $\qqq{k_\al}\supset A_\al \supset I_{\mathrm{diff}}(\bsig_\al)$, and $\xi_{\al}$ denote some deterministic and integer-valued coefficients of order $\OO(1)$. The same expansion also holds for the $\cK$-loop:
		\begin{align}\label{yurenAK}
			\left(Q^{(A)}\circ {\cal K}^{(n)}\right)_{t,\boldsymbol{\sigma},\ba}=\left(Q^{(A_n)}\circ {\cal K}^{(n)}\right)_{t,\boldsymbol{\sigma},\ba}+ \sum_\al \frac{\xi_{\al}}{(2\ii N\eta_t)^{n-k_\al}} \left( Q^{(A_\al )}\circ {\cal K}^{(k_\al)}\right)_{t,\boldsymbol{\sigma}_\al,\ba_\al }.
		\end{align}
		As a special case that is of particular importance to us, the above expansions hold for $A=\emptyset$. 
	\end{lemma}

	Given any $n\ge 2$ and $A\subset\qqq{n}$, we derive from equation \eqref{eq_L-Keee} that 
	\begin{align} 
		\dd Q^{(A)} \circ (\mathcal{L} - \mathcal{K})^{(\fn)}_{t, \boldsymbol{\sigma}, \ba} 
		&= Q^{(A)} \circ \left[\varTheta_{t,\bsig}^{(n)} \circ (\mathcal{L} - \mathcal{K})^{(n)}_{t, \boldsymbol{\sigma}}\right]_{\ba} \dd t
		+ \sum_{l_\mathcal{K}=3}^\fn Q^{(A)} \circ \Big[{\cK}^{(\lenk)}\sim (\mathcal{L} - \mathcal{K})\Big]^{(\fn)}_{t, \boldsymbol{\sigma}, \ba}\dd t \nonumber
		\\
		&+ Q^{(A)} \circ \mathcal{E}^{(\cL-\cK)\times(\cL-\cK),(\fn)}_{t, \boldsymbol{\sigma}, \ba}\dd t  + Q^{(A)} \circ \mathcal{E}^{\Gc,(\fn)}_{t, \boldsymbol{\sigma}, \ba} \dd t
		+ Q^{(A)} \circ \dd\mathcal{E}^{M,(\fn)}_{t, \boldsymbol{\sigma}, \ba} .\label{eq_L-Keee_nonzeromode}
	\end{align} 
	Due to the translation invariance of $S^{(\sB)}$ and $M^{(\sig_i,\sig_{i+1})}$, they both have an eigenvector $\mathbf e$ with $\mathbf e(x)\equiv L^{-d/2}$ for all $x\in \Zn$. As a consequence, the operators $\varTheta^{(n)}$ and $\mathcal{U}^{(\fn)}$ defined in \eqref{def:op_thn} and \eqref{def_Ustz} commute with the operators $P_i$, and hence also commute with the operators $Q^{(A)}$ for all $A\subset \qqq{n}$. Hence, the first term on the RHS of \eqref{eq_L-Keee_nonzeromode} can be written as 
	\[\left[ \varTheta_{t,\bsig}^{(n)} \circ Q^{(A)} \circ (\mathcal{L} - \mathcal{K})^{(n)}_{t, \boldsymbol{\sigma}}\right]_{\ba}.\]
	Then, applying Duhamel's principle to the equation \eqref{eq_L-Keee_nonzeromode}, we derive the following integral equation: 
	\begin{align}\label{iisuwjyys}
		Q^{(A )}\circ (\mathcal{L} - \mathcal{K})^{(n)}_{t, \boldsymbol{\sigma}, \ba} 
		&=    \left(Q^{(A)}\circ \mathcal{U}^{(n)}_{s, t, \boldsymbol{\sigma}} \circ  {\cal B}_0(s)\right)_{\ba} +   \int_{s}^t \left( Q^{(A)}\circ\mathcal{U}^{(n)}_{u, t, \boldsymbol{\sigma}} \circ \sum_{k=1}^3{\cal B}_k(u) \right)_{\ba} \dd u \nonumber\\
		& +  \int_{s}^t \left(Q^{(A)}\circ\mathcal{U}^{(n)}_{u, t, \boldsymbol{\sigma}} \circ   \dd\mathcal{E}^{M,(n)}_{u, \boldsymbol{\sigma}}\right)_{\ba} , 
	\end{align}
	where we use the notations in \eqref{int_K-L+Q}. 
	If $A\supset I_{\mathrm{diff}}(\bsig)$, then the new kernel \smash{$Q^{(A)}\circ \mathcal{U}^{(n)}_{s, t, \boldsymbol{\sigma}}$} has a better $(\infty\to \infty)$-norm estimate (as stated in \Cref{lem:sum_decay_nonzero} below) than the original kernel \smash{$\mathcal{U}^{(n)}_{s, t, \boldsymbol{\sigma}}$} when $1-s\le \ilambda^2/L^{2}$. In fact, corresponding to each $i\in I_{\mathrm{diff}}(\bsig)$, the $(\infty\to \infty)$-norm estimate of \smash{$\mathcal{U}^{(n)}_{s, t, \boldsymbol{\sigma}}$} will be weaker by an $\eta_s/\eta_t$ factor.  
	
	Combining the expansions \eqref{yurenAL} and \eqref{yurenAK} (for the case $A=\emptyset$), the equation \eqref{iisuwjyys}, the evolution kernel estimate in \Cref{lem:sum_decay_nonzero}, and the estimates in \Cref{lem:SEforLn}, we can establish a similar result as in \Cref{lem:STOeq_Qt}. We postpone the proof of \Cref{lem:STOeq_Qt_nonzero} to \Cref{subsec:support} below.
	
	\begin{lemma}\label{lem:STOeq_Qt_nonzero} 
		Under the assumptions of \Cref{lem:main_ind}, suppose $1-s\le \ilambda^2/L^{2}$ and the estimates \eqref{Gt_bound_flow}, \eqref{Gt_avgbound_flow}, and \eqref{Eq:Gdecay_w} hold uniformly in $u\in[s,t]$. 
		Then, for any fixed $\fn\ge 2$ and $p\ge 1$, the bound \eqref{am;asoi222} holds uniformly in $u\in[s,t]$. 
	\end{lemma}
	
	Now, we are ready to complete the proof of Step 3 for the case $1-s\le \ilambda^2/L^{2}$.

	\begin{proof}[\bf Proof of \eqref{Eq:LGxb} when $1-s\le \ilambda^2/L^{2}$]
		Given \Cref{lem:STOeq_Qt_nonzero}, we can prove a similar iterative result to \Cref{lem:iterations}, but with a different control parameter defined as follows: 
		\begin{equation}\label{eq:psipara_smalletacase}
			\Psi_{u}(\fn,k;s,t):=  (W^{-d}B_{s,0})^{-3/4}+(\eta_s/\eta_t)^{\fn-1}(W^{-d}B_{s,0})^{k/8-1}.
		\end{equation} 
		More precisely, suppose the estimate \eqref{eq:iteration_induc} holds uniformly in $u\in[s,t]$ for all index pairs $(r,l)\in \{(r, k): 2\le r\le n-1\}\cup \{(r, k-1): 2\le r\le n+2\}$. 
		Then, we have the following estimate uniformly in $u\in[s,t]$:
		\be\label{eq:iteration_improve_smalleta}
		\sup_{v\in[s,u]}\wh \Xi^{({\cal L-\cK})}_{v,\fn}\prec   \Psi(\fn,k;s,u).\ee
		Since the proof of \eqref {eq:iteration_improve_smalleta} is very similar to that for \Cref{lem:iterations} by using \Cref{lem:STOeq_Qt_nonzero}, we omit the details. 
		With this result, performing exactly the same iterative argument as in the case $1-t\ge \ilambda^2/L^{2}$ (i.e., the argument around \eqref{eq:iterative_pf_largeeta}), we can conclude \eqref{Eq:LGxb} for the case $1-s\le \ilambda^2/L^{2}$.
	\end{proof}

	\subsection{Proof of Step 4}
	
	We again divide the proof of \eqref{Eq:L-KGt-flow} into two cases according to whether $1-t \ge \ilambda^2/L^{2}$ or $1-s \le \ilambda^2/L^{2}$. In the former case, we apply \Cref{lem:STOeq_Qt} established in Step 3, while in the latter we apply \Cref{lem:STOeq_Qt_nonzero}. 
	At this step, using the sharp $G$-loop bound \eqref{Eq:LGxb} together with the averaged local law \eqref{Gt_avgbound_flow}, we may choose the optimal parameters \smash{$\Xi^{(\cal L)}_{v,n'}=1$} for all $G$-loops and \smash{$\Xi^{(\cal L-\cK)}_{v,1}=1$}. Then, from \eqref{am;asoi222}, we get the following bound for any fixed $n\ge 2$ and $p\ge 1$: 
	\begin{align} \label{saww02}
		\sup_{v\in[s,u]}\wh\Xi_{v,\fn}^{(\mathcal{L}-\mathcal{K})}  \prec (W^{-d}B_{s,0})^{-\frac{1}{4p}}+ \sup_{v\in [s,u]}
		\left(\max_{n'=2}^{\fn-1}\Xi^{({\cal L-\cal K})}_{v,n'} \right).
	\end{align}
	First, taking $\fn=2$ in \eqref{saww02}, the second term on the RHS vanishes. Since $p$ can be chosen arbitrarily large, we conclude that \smash{\(\sup_{v\in[s,u]}\wh\Xi_{v,2}^{(\mathcal{L}-\mathcal{K})}\prec 1.\)} 
	Starting from this base case, we can derive that \smash{\(\sup_{v\in[s,u]}\wh\Xi_{v,\fn}^{(\mathcal{L}-\mathcal{K})}\prec 1\)} 
	for any fixed $\fn\in\N$ by applying \eqref{saww02} inductively in $\fn$. This completes the proof of \eqref{Eq:L-KGt-flow} in Step 4.

	\subsection{Evolution kernel estimates}\label{ks-statements}
	
	In the above proof, we have used the following estimates on the evolution kernel defined in \Cref{DefTHUST}. Their proofs are postponed to \Cref{ks}. We first have an easy bound on the $({\infty\to \infty})$-norm of \smash{$\mathcal U^{(n)}_{t,\bsig,\ba}$}. 
	
	\begin{lemma}\label{lem:sum_Ndecay}
		Let ${\cal A}: (\Zn)^{\fn}\to \mathbb C$ be an $\fn$-dimensional tensor for a fixed $\fn\ge 2$. Then, for any $0\le s \le t < 1$, we have that 
		\begin{align}\label{sum_res_Ndecay}
			\| {\cal U}^{(n)}_{s,t,\boldsymbol{\sigma}}\;\circ {\cal A}\|_{\infty} \le \left(  \frac{1-s}{1-t}\right)^{\fn }\cdot \|{\cal A}\|_{\infty} , 
		\end{align}
		where the $L^\infty$-norm of ${\cal A}$ is defined as $\|{\cal A}\|_{\infty}=\max_{\ba\in (\Zn)^\fn}|\cal A_{\ba}|$.
	\end{lemma}

	If the tensor $\cal A$ exhibits faster-than-polynomial decay on scales larger than $\ell_s$, we show that the $({\infty \to \infty})$-norm of the evolution kernel \smash{${\cal U}^{(\fn)}_{s,t,\boldsymbol{\sigma}}$} satisfies a better bound. This bound can be further improved when $\boldsymbol{\sigma}$ is non-alternating or when $\cal A$ satisfies a sum-zero property.

	\begin{lemma}\label{lem:sum_decay}
		Let ${\cal A}: (\Zn)^{\fn}\to \mathbb C$ be an $\fn$-dimensional tensor for a fixed $\fn\ge 2$. Suppose it satisfies the following fast-decay property for some small constant $\e\in(0,1)$ and large constant $D>1$:  
		\begin{equation}\label{deccA0}
			|\cal A_{\ba}|\le W^{-D}, \quad \forall\ \ba=(a_1,\ldots, a_\fn)\in (\Zn)^{\fn} \ \ \text{with}\ \	\max_{i,j\in \Zn}|a_i-a_j|\ge W^{\e}\ell_s  \, .
		\end{equation}
		Fix any $0\le s \le t \le 1 -\ilambda^2/L^{2}$ such that $(1-t)/(1-s)\ge W^{-1}$. 
		Then, there exists a constant $C_\fn>0$ that does not depend on $\e$ or $D$ such that the following bound holds:
		\begin{align}\label{sum_res_1}
			\left\|{\cal U}^{(\fn)}_{s,t,\boldsymbol{\sigma}} \circ {\cal A}\right\|_\infty \le W^{C_\fn\e}\frac{\ell_t^2 }{\ell_s^2 }\p{\frac{\ilambda^2+|1-s|}{\ilambda^2+|1-t|}}^{\fn}\|{\cal A}\|_\infty
			+W^{-D+C_\fn} . 
		\end{align}
		This estimate can be further improved in the following cases:
		
		\begin{itemize}
			\item[(I)] 
			If $\bsig$ is non-alternating, i.e.,  $\sigma_k=\sigma_{k+1}$ for some $k\in \qqq{n}$, then we have
			\begin{align}\label{sum_res_2_NAL}
				\left\|{\cal U}^{(\fn)}_{s,t,\boldsymbol{\sigma}} \circ {\cal A}\right\|_\infty \le W^{C_\fn\e} \p{\frac{\ilambda^2+|1-s|}{\ilambda^2+|1-t|}}^{\fn-1}  \|{\cal A}\|_{\infty}  + W^{-D+C_\fn} 
			\end{align} 
			for a constant $C_\fn>0$ that does not depend on $\e$ or $D$.
			
			\item[(II)] If ${\cal A}$ satisfies the following sum-zero property: 
			\begin{align}\label{sumAzero}
				\sum_{a_2,\ldots,a_\fn\in \Zn}{\cal A}_{\ba}=0, \quad \forall a_1\in \Zn \, ,
			\end{align}
			then there exists a constant $C_\fn>0$ that does not depend on $\e$ or $D$ such that
			\begin{align}\label{sum_res_2}
				\left\|{\cal U}^{(\fn)}_{s,t,\boldsymbol{\sigma}} \circ {\cal A}\right\|_\infty \le W^{C_\fn\e} \p{\frac{\ilambda^2+|1-s|}{\ilambda^2+|1-t|}}^{\fn} \|{\cal A}\|_{\infty} 
				+W^{-D+C_\fn}. 
			\end{align} 
		\end{itemize}
	\end{lemma} 
	If $1-s \le \ilambda^2/L^{2}$, we obtain the following estimates for the evolution kernels with zero modes removed.
	\begin{lemma}\label{lem:sum_decay_nonzero}
		Fix any $1-\ilambda^2/L^{2}\le s \le t<1$ and $\bsig\in \{+,-\}^n$. Let ${\cal A}: (\Zn)^{\fn}\to \mathbb C$ be an arbitrary $\fn$-dimensional tensor for a fixed $\fn\ge 2$. Then, for any subset $A\subset \qqq{n}$ satisfying $A\supset I_{\mathrm{diff}}(\bsig)$, we have  
		\begin{align}\label{sum_res_Ndecay_nonzero}
			\| Q^{(A)}\circ {\cal U}^{(n)}_{s,t,\boldsymbol{\sigma}}\;\circ {\cal A}\|_{\infty} \prec \|{\cal A}\|_{\infty} .
		\end{align}
	\end{lemma}

	\subsection{Proofs of supporting lemmas}\label{subsec:support}
	
	This subsection presents the proofs of several supporting lemmas.
	
	\begin{proof}[\bf Proof of Lemma \ref{lem:STOeq_Qt}] 
		\Cref{lem:STOeq_NQ} already gives a good enough bound for non-alternating $(\cL-\cK)$-loops. It remains to control \smash{\(\max_{\ba} |({\cal L}-{\cal K})^{(\fn)}_{u, \boldsymbol{\sigma}, \ba} |\big/ (W^{-d}B_{u,0})^{n}\)} for alternating $\bsig$ using \eqref{eq:Ward_typeP} and equation \eqref{int_K-L+Q}. 
		
		First, the partial sum term \([\cal P\circ (\mathcal{L} - \mathcal{K})^{(\fn)}_{u, \boldsymbol{\sigma}}]_{a_1} \dthn^{(\fn)}_{u, \ba}\) has been bounded in \eqref{eq:Ward_typeP}. Second, using  \Cref{lem_+Q}, the induction hypothesis \eqref{Eq:L-KGt+IND} at time $s$, the estimates established in \eqref{bEwGn}--\eqref{eq:MG_nloop}, and the evolution kernel estimate \eqref{sum_res_2}, and adopting a similar argument as in the proof of \Cref{lem:STOeq_NQ}, we can control the terms involving $\cal B_{i}$ for $i\in \{0,1,2,3\}$ and the martingale term on the RHS of \eqref{int_K-L+Q} as follows:
		\begin{align}
			&  [\cal P\circ (\mathcal{L} - \mathcal{K})^{(\fn)}_{u, \boldsymbol{\sigma}}]_{a_1} \dthn^{(\fn)}_{u, \ba}+ \left(\mathcal{U}^{(n)}_{s, u, \boldsymbol{\sigma}} \circ  {\cal Q}_s\circ  {\cal B}_0(s)\right)_{\ba} +  \int_{s}^u \left(\mathcal{U}^{(n)}_{v, u, \boldsymbol{\sigma}} \circ  {\cal Q}_v\circ \sum_{k=1}^3{\cal B}_k(v)  \right)_{\ba} \dd v + \int_{s}^u \left(\mathcal{U}^{(\fn)}_{v, u, \boldsymbol{\sigma}} \circ  {\cal Q}_v\circ \dd\mathcal{E}^{M,(\fn)}_{v, \boldsymbol{\sigma}}\right)_{\ba} \nonumber\\
			\prec&~  (W^{-d}B_{u,0})^n \sup_{v\in [s,u]}
			\left((W^{-d}B_{u,0})^{\frac{1}{6}}\wh\Xi^{(\cal L-\cK)}_{v, n}+ \max_{n'=2}^{\fn-1} \Xi^{({\cal L-\cal K})}_{v,n'}+ \max_{n'=\fn-1}^{\fn+1} \Xi^{({\cal L})}_{v,n'} +\max_{n'=\lceil n/2\rceil + 1}^{n-1} \Xi^{({\cal L-\cal K})}_{v,n+2-n'} \p{\Xi^{({\cal L})}_{v,n_1'} \Xi^{({\cal L})}_{v,n_2'}}^{\frac 12 } \right) \nonumber\\
			&~+ (W^{-d}B_{u,0})^n \sup_{v\in[s,u]} \left((W^{-d}B_{u,0})^{-\frac{1}{4p}}(\Xi^{(\cal L)}_{v, {2\fn-1}})^{\frac 1 2}(\Xi^{(\cal L)}_{v, 4p})^{\frac{1}{4p}}\right) .\label{eq:alternatecase1}
		\end{align}
		It remains to address the terms involving $\cal B_{4}$ and $\cal B_{5}$ on the RHS of \eqref{int_K-L+Q}. We claim the following bounds: 
			\begin{align}\label{y27kasdfg}
				\| {\cal B}_4(u)\|_{\infty}+\|{\cal B}_5(u)\|_{\infty}\prec \eta_u^{-1} (W^{-d}B_{u,0})^n \Xi^{(\mathcal{L}-\mathcal{K})}_{u, n-1}.
			\end{align}
			First, combining \eqref{jywiiwsoks} with \eqref{eq:Bu0asymp} and the second estimate in \eqref{eq:derv_Theta}, we obtain that  
			\be\label{A5}  
			\left\|\left[ {\cal P} \circ \left(\mathcal{L} - \mathcal{K}\right)^{(\fn)}_{u, \boldsymbol{\sigma}}\right] \partial_u\dthn_{u}^{(\fn)}\right\|_\infty \prec \frac{1}{\eta_u}\frac{\ilambda^2+|1-u|}{|1-u|\ell_u^{d}}\cdot (W^{-d}B_{u,0})^n  \Xi^{(\mathcal{L}-\mathcal{K})}_{u, \fn-1}\lesssim \eta_u^{-1} (W^{-d}B_{u,0})^n \Xi^{(\mathcal{L}-\mathcal{K})}_{u, \fn-1}.
			\ee
			Second, using the definition of $\cal Q_t$ in \eqref{eq:sumzero_op} and the definition of $\varTheta^{(\fn)}_{u,\bsig}$ in \eqref{def:op_thn}, we can bound that 
			\begin{align}\label{A4}  
				\br{{\cal Q}_u , \varTheta^{(\fn)}_{u,\boldsymbol{\sigma}} }\circ (\mathcal{L} - \mathcal{K})^{(\fn)}_{u, \boldsymbol{\sigma}, \ba}
				&= \varTheta^{(\fn)}_{u, \boldsymbol{\sigma}} \circ \left[ \left({\cal P} \circ \left( \mathcal{L} - \mathcal{K} \right)^{(\fn)}_{u,\bsig}\right) \cdot \dthn^{(\fn)}_u \right]_{\ba} - \left[\left( {\cal P} \circ \varTheta^{(\fn)}_{u, \boldsymbol{\sigma}} \circ (\mathcal{L} - \mathcal{K})^{(\fn)}_{u,\bsig}\right)  \cdot \dthn^{(\fn)}_u \right]_{\ba} \nonumber
				\\ 
				& \prec \eta_u^{-1}\big\|\dthn^{(\fn)}_{u}\big\|_\infty \cdot   \left\| {\cal P} \circ \left(  \mathcal{L} - \mathcal{K} \right)^{(\fn)}_{u,\boldsymbol{\sigma}} \right\|_{\infty} \prec \eta_u^{-1} (W^{-d}B_{u,0})^n \Xi^{(\mathcal{L}-\mathcal{K})}_{u, \fn-1},
			\end{align}
			where in the second step, we use the simple fact that $ \|\varTheta_{u,\bsig}^{(\fn)}\|_{\infty\to\infty} \lesssim (1-u)^{-1}$ for any $\bsig\in \{+,-\}^\fn$ due to \eqref{eq:THETAinftinf}, and in the third step, we use \eqref{jywiiwsoks} and the first bound in \eqref{eq:derv_Theta}. Combining \eqref{A5} and \eqref{A4} yields \eqref{y27kasdfg}. 
			Now, applying the evolution kernel estimate \eqref{sum_res_2} and the estimates in \eqref{y27kasdfg}, and performing the integral over $v$, we can obtain that 
			\begin{align}
				&  \int_{s}^u \left(\mathcal{U}^{(n)}_{v, u, \boldsymbol{\sigma}} \circ  {\cal Q}_v\circ \sum_{k=4}^5{\cal B}_k(v)  \right)_{\ba} \dd v  \prec  (W^{-d}B_{u,0})^n \sup_{v\in[s,u]}\left(\Xi^{({\cal L-\cal K})}_{v,n-1} \right). \label{eq:alternatecase2}
			\end{align}

			Finally, combining \Cref{lem:STOeq_NQ} with \eqref{eq:alternatecase1} and \eqref{eq:alternatecase2}, we obtain that 
			\begin{align*}
				\sup_{v\in[s,u]} \wh\Xi^{({\cal L-\cal K})}_{v,\fn} &\prec (W^{-d}B_{u,0})^{\frac{1}{6}}\sup_{v\in [s,u]}
				\wh\Xi^{(\cal L-\cK)}_{v, n}  + \sup_{v\in[s,u]} \left((W^{-d}B_{u,0})^{-\frac{1}{4p}}(\Xi^{(\cal L)}_{v, {2\fn-1}})^{\frac 1 2}(\Xi^{(\cal L)}_{v, 4p})^{\frac{1}{4p}}\right)\\
				&+\sup_{v\in [s,u]}
				\left(\max_{n'=1}^{\fn-1} \Xi^{({\cal L-\cal K})}_{v,n'}+ \max_{n'=\fn-1}^{\fn+1} \Xi^{({\cal L})}_{v,n'} + \max_{n'=\lceil n/2\rceil + 1}^{n-1} \Xi^{({\cal L-\cal K})}_{v,n+2-n'} \p{\Xi^{({\cal L})}_{v,n_1'} \Xi^{({\cal L})}_{v,n_2'}}^{\frac 12 } \right),\nonumber
			\end{align*}
			which yields \eqref{am;asoi222} at each fixed $u\in[s,t]$ upon solving for \smash{$\sup_{v\in[s,u]} \wh\Xi^{({\cal L-\cal K})}_{v,\fn}$}. Then, applying a standard $N^{-C}$-net argument extends it uniformly to all $u\in [s,t]$. 
		\end{proof}

		\begin{proof}[\bf Proof of Lemma \ref{lem:iterations}]
			By the averaged local law \eqref{Gt_avgbound_flow}, we have \smash{$\wh\Xi_{u,1}^{(\cL-\cK)}\prec 1=:\Xi_{u,1}^{(\cL-\cK)}$} uniformly for all $u\in[s,t]$. Substituting this into \eqref{am;asoi222}, we obtain the following bound, also uniformly in $u\in [s,t]$ under the assumption $1-s\le \ilambda^2$: 
			\begin{align}\label{am;asoi333}
				\sup_{v\in[s,u]} \wh \Xi^{({\cal L-\cal K})}_{v,\fn}&\prec \sup_{v\in[s,u]} \left(\p{W^{-d}B_{u,0}}^{-\frac{1}{4p}}(\Xi^{(\cal L)}_{v, {2\fn-1}})^{\frac 1 2}(\Xi^{(\cal L)}_{v, 4p})^{\frac{1}{4p}}\right) \\
				&+\sup_{v\in [s,u]}
				\left( \max_{n'=2}^{\fn-1}\Xi^{({\cal L-\cal K})}_{v,n'}+\max_{n'=n-1}^{n+1}\Xi^{({\cal L})}_{v,\fn'}+\max_{n'=3}^{n-1} \Xi^{({\cal L-\cal K})}_{v,n+2-n'} \p{\Xi^{({\cal L})}_{v,n_1'} \Xi^{({\cal L})}_{v,n_2'}}^{\frac 12 } \right),\nonumber
			\end{align}
			where we also use that $\lceil n/2\rceil + 1\le n-1$ only when $n\ge 4$, in which case we have $\lceil n/2\rceil + 1\ge 3$. 
			
			To show \Cref{lem:iterations} using the bootstrap bound \eqref{am;asoi333}, we first claim that under the assumption \eqref{eq:iteration_induc}, the following bound holds uniformly in $u\in[s,t]$ 
			as long as we choose $p\ge 4$: 
			\begin{align} \label{xiu2n+2psi}
				\sup_{v\in[s,u]}\left((\ilambda^2W^d)^{\frac{1}{4p}}(\wh\Xi^{(\cal L)}_{v, {2\fn-1}})^{\frac 1 2}(\wh\Xi^{(\cal L)}_{v, 4p})^{\frac{1}{4p}}\right)  \prec \Psi(\fn,k;s,u) \ .
			\end{align} 
			Given the bound \eqref{xiu2n+2psi}, we obtain from \eqref{am;asoi333} that
			\begin{equation} 
				\sup_{v\in[s,u]}\wh\Xi_{v,\fn }^{(\mathcal{L}-\mathcal{K})} \prec \Psi(\fn,k;s,u) + 
				\sup_{v\in [s,u]}
				\left( \max_{n'=2}^{\fn-1}\Xi^{({\cal L-\cal K})}_{v,n'}+\max_{n'=n-1}^{n+1}\Xi^{({\cal L})}_{v,\fn'}+\max_{n'=3}^{n-1} \Xi^{({\cal L-\cal K})}_{v,n+2-n'} \p{\Xi^{({\cal L})}_{v,n_1'} \Xi^{({\cal L})}_{v,n_2'}}^{\frac 12 } \right)\label{sadui_w0}
			\end{equation}
			uniformly in $u\in[s,t]$. By the induction hypothesis \eqref{eq:iteration_induc}, we may choose the $\Xi^{({\cal L-\cal K})}$-parameters as 
			\be \nonumber
			\Xi^{({\cal L-\cal K})}_{v,n'}= \Psi(n',k;s,u)\le \Psi(n,k;s,u) \ \  \text{for}\ \ 2\le n' \le \fn-1, 
			\ee
			and, by \eqref{rela_XILXILK} and \eqref{eq:iteration_induc}, we can choose the $\Xi^{({\cal L})}$-parameters as
			\begin{align*}
				\Xi^{({\cal L})}_{v,n'}=\begin{cases}
					1+(\ilambda^2W^{d})^{-1} \Psi(n',k;s,u), \ & \text{for} \ \  n'\le n-1,\\  
					1+(\ilambda^2W^{d})^{-1} \Psi(n',k-1;s,u)\ & \text{for} \ \ n\le n'\le n+2. 
				\end{cases} 
			\end{align*}
			With these choices, we readily verify that  
			\[\max_{n'=n-1}^{n+1}\Xi^{({\cal L})}_{v,\fn'} \le  \Psi(\fn,k;s,u),\quad \text{and}\quad  \p{\Xi^{({\cal L})}_{v,n_1'} \Xi^{({\cal L})}_{v,n_2'}}^{\frac 12 } \lesssim 1+\p{{\eta_s}/{\eta_u}}^{n'}(\ilambda^2W^d)^{-\frac{k}{8}} \ \ \text{for} \ \ 3\le n'\le {n-1}.\]
			Substituting these bounds into \eqref{sadui_w0}, we obtain  
			\begin{align*}
				\sup_{v\in[s,u]}\wh\Xi_{v,\fn }^{(\mathcal{L}-\mathcal{K})}\prec \Psi(\fn,k;s,u) +  \max_{n'=3}^{\fn-1} \Psi(n+2-n',k;s,u)\bigg[1+\p{\frac{\eta_s}{\eta_u}}^{n'}(\ilambda^2W^d)^{-\frac{k}{8}}\bigg]\lesssim \Psi(\fn,k;s,u), 
			\end{align*}
		which concludes the proof of \Cref{lem:iterations}.

			It remains to prove \eqref{xiu2n+2psi}. For $n\in\{2,3\}$, using the a priori bound \eqref{sef8w483r324} for \smash{$\wh\Xi^{(\cal L)}_{v, {2\fn-1}}$ and $\wh\Xi^{(\cal L)}_{v, 4p}$}, along with the relation \eqref{rela_XILXILK}, we obtain for $p\ge 4$ that,
			\begin{align*} 
				\sup_{v\in[s,u]}\left((\ilambda^2W^d)^{\frac{1}{4p}}(\wh\Xi^{(\cal L)}_{v, {2\fn-1}})^{\frac 1 2}(\wh\Xi^{(\cal L)}_{v, 4p})^{\frac{1}{4p}}\right)  \prec (\ilambda^2W^d)^{\frac{1}{16}}\left( {\eta_s}/{\eta_u}\right)^{n}  \le \Psi(\fn,k;s,u).
			\end{align*} 
			We now consider the $n\ge 4$ case. Recall the $G$-chain defined in \eqref{defC=GEG}. Given an arbitrary $(2\fn-1)$-$G$-chain \smash{$\cL^{(2\fn-1)}_{v,\bsig,\ba}$} with 
			$$\bsig=(\sig,\sig_1,\ldots, \sig_{\fn-1}, \sig',\sig_1',\ldots, \sig_{\fn-2}'),\quad \ba=(b,a_1,\ldots,a_{\fn-1},b',a_1',\ldots,a_{\fn-2}'),$$ 
			we can write it is 
			\begin{align*}
				\cL^{(2\fn-1)}_{v,\bsig,\ba}=W^{-2d}\sum_{x\in[b],x'\in [b']} \left(\cal C^{(\fn)}_{v,\bsig_1,\ba_1}\right)_{xx'} \left(\cal C^{(\fn-1)}_{v,\bsig_2,\ba_2}\right)_{x'x} ,  
			\end{align*}
			where $\bsig_1=(\sig_1,\ldots, \sig_{\fn-1}, \sig')$, $\bsig_2=(\sig_1',\ldots, \sig_{\fn-2}',\sig)$, $\ba_1=([a_1],\ldots,[a_{\fn-1}])$, and $\ba_2=([a_1'],\ldots,[a_{\fn-2}'])$. Applying the Cauchy-Schwarz inequality gives
			\begin{equation}\label{suauwiioo1}
				\left|\cL^{(2\fn-1)}_{v,\bsig,\ba}\right|
				\le 
				\bigg( W^{-2d}\sum_{x\in[b],x'\in[b']}
				\left|\left({\cal C}^{(\fn)} _{v,\boldsymbol{\sigma}_1,\ba_1}\right)_{xx'} 
				\right|^2 \bigg)^{1/2}\bigg( W^{-2d}\sum_{x\in[b],x'\in[b']}
				\left|\left({\cal C}^{(\fn-1)} _{v,\boldsymbol{\sigma}_2,\ba_2}\right)_{x'x} 
				\right|^2 \bigg)^{1/2}.
			\end{equation}
			Following the argument for equation (5.118) in \cite{YY_25}, we obtain that 
			\begin{equation}\label{eq:boundtwochains}
				\begin{split}
					& W^{-2d}\sum_{x\in[b],x'\in[b']}
					\left|\left({\cal C}^{(\fn)} _{v,\boldsymbol{\sigma}_1,\ba_1}\right)_{xx'} 
					\right|^2 \le (\ilambda^2W^d) \Xi^{(\cL)}_{v,2l_1}\Xi^{(\cL)}_{v,2l_2}, \\  &W^{-2d}\sum_{x\in[b],x'\in[b']}
					\left|\left({\cal C}^{(\fn-1)} _{v,\boldsymbol{\sigma}_2,\ba_2}\right)_{x'x} 
					\right|^2 \le (\ilambda^2W^d) \Xi^{(\cL)}_{v,2l_3}\Xi^{(\cL)}_{v,2l_4},
				\end{split}
			\end{equation}
			where $l_1=\lceil n/2\rceil$, $l_2=\lfloor n/2\rfloor$, $l_3=\lceil (n-1)/2\rceil$, and $l_4=\lfloor (n-1)/2\rfloor$. 
			By the induction hypothesis \eqref{eq:iteration_induc} (noting $\max_i\p{2l_i}\le \fn+1$) and the relation \eqref{rela_XILXILK}, we have
			\begin{align}\label{auskoppw2.00}
				\wh\Xi_{v,2l_i}^{(\cal L)} \prec 1+ (\ilambda^2 W^{d})^{-1} \Psi(k-1,2l_i;s,u) \prec 1+ \p{{\eta_s}/{\eta_u}}^{2l_i-1} (\ilambda^2W^d)^{-(k-1)/8}.
			\end{align}
			Plugging \eqref{auskoppw2.00} into \eqref{eq:boundtwochains} and then applying \eqref{suauwiioo1}, we obtain for $k\ge 1$ that, 
			\begin{align}\label{auskoppw2}
				\sup_{v\in[s,u]}\big(\wh\Xi^{({\cal L})}_{v, 2\fn-1}\big)^{\frac{1}{2}}&\prec  \sup_{v\in[s,u]} (\ilambda^2W^d)^{1/2}\left[\prod_{i=1}^4\bigg(1+ \p{{\eta_s}/{\eta_u}}^{2l_i-1} (\ilambda^2W^d)^{-(k-1)/8}\bigg)\right]^{1/4} \nonumber\\
				&\lesssim (\ilambda^2W^d)^{1/2}\bigg(1+ \p{{\eta_s}/{\eta_u}}^{n} (\ilambda^2W^d)^{-(k-1)/8}\bigg).
			\end{align}
			On the other hand, by \eqref{sef8w483r324}, we have 
			\(\sup_{v\in[s,u]} (\wh\Xi^{(\cal L)}_{v, 4p})^{\frac{1}{4p}} \prec {\eta_s}/{\eta_u} .\)
			Combining it with \eqref{auskoppw2}, and using the condition \eqref{con_st_ind}, we conclude \eqref{xiu2n+2psi} when $p\ge 4$.  
		\end{proof}

		\begin{proof}[\bf Proof of Lemma \ref{lem: newPQ}] 
			We only prove the expansion \eqref{yurenAL} using Ward's identity \eqref{WI_calL}, while \eqref{yurenAK} follows from the same argument by using the identity \eqref{WI_calK} instead. 
			We first record the following consequence of \eqref{WI_calL}: for any $\bsig\in \{+,-\}^n$ with $\sig_i\ne \sig_{i+1}$ and any subset $A\subset \qqq{n}\setminus\{i\}$, 
			\begin{align}\label{y2ussz}
				\left(Q^{(A)} \circ P^{(i)}\circ {\cal L}^{(n)}\right)_{t,\boldsymbol{\sigma},\ba}
				= \frac{1}{ 2\ii N\eta_t}
				\left(\left(Q^{(A_{(i)})} \circ {\cal L}^{(n-1)}\right)_{t,\boldsymbol{\sigma}_+,\ba_+}- \left(Q^{(A_{(i)})}\circ {\cal L}^{(n-1)}\right)_{t,\boldsymbol{\sigma}_-,\ba_-}\right), 
			\end{align}
			where $A_{(i)}:=\{1\le j<i: j\in A\} \cup \{i\le j\le n-1: j+1\in A\}$, $\ba_\pm = (a_1,\ldots, a_{i-1}, a_{i+1},\ldots, a_n)$, and $\boldsymbol{\sigma}_\pm = (\sigma_1\cdots \sigma_{i-1}, \pm, \sigma_{i+2},\cdots \sigma_n)$. 
			We now prove \eqref{yurenAL} by induction in $n$ and using \eqref{y2ussz}. 
			
			First, it is easy to see that \eqref{yurenAL} trivially holds for $n=1$ and $|A|\in\{0,1\}$. 
			Next, suppose we have shown that there exists a decomposition \eqref{yurenAL} for any $n\le l-1$. We prove \eqref{yurenAL} for $n=l$. Given any $A\subset\qqq{n}$, if $A\supset I_{\mathrm{diff}}(\bsig)$, then \eqref{yurenAL} holds trivially with only one term on the RHS (i.e., the term $Q^{(A)}\circ {\cal L}^{(n)}$ itself). Otherwise, suppose there exists $i\in I_{\mathrm{diff}}(\bsig)\setminus A$, then we write
			\be\label{eq:y2ussz}
			\left(Q^{(A)}\circ {\cal L}^{(n)}\right)_{t,\boldsymbol{\sigma},\ba}= \left(Q^{(A\cup \{i\} )}\circ {\cal L}^{(n)}\right)_{t,\ba,\boldsymbol{\sigma}}+\left(Q^{(A )}\circ P^{(i)}\circ{\cal L}^{(n)}\right)_{t,\boldsymbol{\sigma},\ba}. 
			\ee
			Applying \eqref{y2ussz} to the second term on the RHS and using the induction hypothesis for $Q^{(A_{(i)})} \circ {\cal L}^{(n-1)}$, we can expand \smash{$(Q^{(A )}\circ P^{(i)}\circ{\cal L}^{(n)})_{t,\boldsymbol{\sigma},\ba}$} into an expression that matches the form of a summand on the RHS of \eqref{yurenAL}. 
			On the other hand, in the first term of \eqref{eq:y2ussz}, the number of elements in $I_{\mathrm{diff}}(\bsig)\setminus (A\cup\{i\})$ is reduced by 1 relative to $I_{\mathrm{diff}}(\bsig)\setminus A$. If $I_{\mathrm{diff}}(\bsig)\setminus (A\cup\{i\})=\emptyset$, then the induction is completed. Otherwise, we choose the next element $j\in I_{\mathrm{diff}}(\bsig)\setminus (A\cup\{i\})$ and repeat the same argument with $A$ and $i$ replaced by $A\cup\{i\}$ and $j$, respectively. Iterating this procedure at most $|I_{\mathrm{diff}}(\bsig)\setminus A|$ times, we finally express $Q^{(A)}\circ {\cal L}^{(n)}$ in the form \eqref{yurenAL}. This completes the induction, thereby concluding the proof of \Cref{lem: newPQ}. 
		\end{proof}
		
		\begin{proof}[\bf Proof of Lemma \ref{lem:STOeq_Qt_nonzero}]
			Given any $\bsig\in\{+,-\}^n$ and $\ba\in (\Zn)^n$, we expand $\cL^{(n)}_{t,\bsig,\ba}$ and $\cK^{(n)}_{t,\bsig,\ba}$ as in \eqref{yurenAL} and \eqref{yurenAK} with $A_n=I_{\mathrm{diff}}(\bsig)$. This yields 
			\begin{align}
				\p{\cal L-\cal K}^{(n)}_{t,\boldsymbol{\sigma},\ba}=Q^{(A_n)}\circ \p{\cal L-\cal K}^{(n)}_{t,\boldsymbol{\sigma},\ba}+\sum_\al \frac{\xi_{\al}}{(2\ii N\eta_t)^{n-k_\al}}   Q^{(A_\al)}\circ \p{\cal L-\cK}^{(k_\al)} _{t,\boldsymbol{\sigma}_{\al},\ba_{\al}}.\label{eq:expandQAempty}
			\end{align}
			For the $k_\al=1$ terms, using \eqref{normQA2} and the averaged local law \eqref{Gt_avgbound_flow}, we bound 
			\begin{align}\label{eq:kalpha1}
				\frac{\xi_{\al}}{(2\ii N\eta_t)^{n-k_\al}}   Q^{(A_\al)}\circ (\cal L-\cK)^{(k_\al)} _{t,\boldsymbol{\sigma}_{\al},\ba_{\al}} \prec (N\eta_t)^{-(n-1)}\cdot W^{-d}B_{t,0} \lesssim (W^{-d}B_{t,0})^n.
			\end{align}
			
			All other terms on the RHS of \eqref{eq:expandQAempty} involve loops of length $\ge 2$, and therefore satisfy equation \eqref{iisuwjyys}.
			Take the first term as an example. Using the evolution kernel estimate in \Cref{lem:sum_decay_nonzero}, the bounds \eqref{bEwGn}--\eqref{eq:MG_nloop}, and the assumption \eqref{Eq:L-KGt+IND} on \smash{$(\mathcal{L} - \mathcal{K})^{(\fn)}_{s, \boldsymbol{\sigma}, \ba}$}, we obtain that
			\begin{align}\label{sahwNQ_smalleta}
				&\frac{ Q^{(A_n)}\circ (\mathcal{L} - \mathcal{K})^{(n)}_{t, \boldsymbol{\sigma}, \ba}}{(W^{-d}B_{t,0})^{n}} \prec 
				1 + (W^{-d}B_{t,0})^{\frac{1}{6}} \int_{s}^t   \wh\Xi^{({\cal L-\cal K})}_{u,n} \frac{\dd u}{\eta_u} +   \int_{s}^t  \max_{n'=2}^{\fn-1} \Xi^{({\cal L-\cal K})}_{u,n'}\frac{\dd u}{\eta_u}  + \int_{s}^t \max_{n'=n-1}^{n+1} \Xi^{(\cal L)}_{u,n'}\frac{\dd u}{\eta_u}    \nonumber \\
				&+  \int_{s}^t \max_{n'=\lceil n/2\rceil + 1}^{n-1} \Xi^{({\cal L-\cal K})}_{u,n+2-n'} \p{\Xi^{({\cal L})}_{u,n_1'} \Xi^{(\cal L)}_{u,n_2'}}^{\frac 12 } \frac{\dd u}{\eta_u}   + \frac{1}{(W^{-d}B_{t,0})^{n}} \int_{s}^t \left(Q^{(A_n)}\circ \mathcal{U}^{(\fn)}_{u, t, \boldsymbol{\sigma}} \circ  \dd  \mathcal{E}^{M,(\fn)}_{u, \boldsymbol{\sigma}}\right)_{\ba}.
			\end{align}
			By Lemma \ref{lem:DIfREP}, together with \eqref{eq:MG_nloop} and \Cref{lem:sum_decay_nonzero}, we bound the martingale term as  
			\begin{align} \label{sahwNQ2}
				(W^{-d}B_{t,0})^{-n} \int_{s}^t \left(Q^{(A_n)}\circ \mathcal{U}^{(\fn)}_{u, t, \boldsymbol{\sigma}} \circ  \dd  \mathcal{E}^{M,(\fn)}_{u, \boldsymbol{\sigma}}\right)_{\ba}&\prec \p{W^{-d}B_{t,0}}^{-\frac{1}{4p}}\left\{\int_{s}^t 
				\eta_u^{-1}\Xi^{(\cal L)}_{u, {2\fn-1}}\p{\Xi^{(\cal L)}_{u, 4p}}^{\frac{1}{2p}}\dd u\right\}^{1/2} 
			\end{align} 
			for any fixed $p\ge 1$. Plugging it into \eqref{sahwNQ_smalleta} and performing the integral over $u$, we obtain 
			\begin{align}\label{am;asoiuw_smalleta}
				\frac{ Q^{(A_n)}\circ (\mathcal{L} - \mathcal{K})^{(n)}_{t, \boldsymbol{\sigma}, \ba}}{(W^{-d}B_{t,0})^{n}} &\prec \sup_{u\in [s,t]}
				\left((W^{-d}B_{t,0})^{\frac{1}{6}}\wh\Xi^{(\cal L-\cK)}_{u, n}+ \max_{n'=2}^{\fn-1} \Xi^{({\cal L-\cal K})}_{u,n'}+ \max_{n'=\fn-1}^{\fn+1} \Xi^{({\cal L})}_{u,n'} \right)\\
				&+\sup_{u\in [s,t]}
				\left(\max_{n'=\lceil n/2\rceil + 1}^{n-1} \Xi^{({\cal L-\cal K})}_{u,n+2-n'} \p{\Xi^{({\cal L})}_{u,n_1'} \Xi^{({\cal L})}_{u,n_2'}}^{\frac 12 } +\p{W^{-d}B_{t,0}}^{-\frac{1}{4p}}(\Xi^{(\cal L)}_{u, {2\fn-1}})^{\frac 1 2}(\Xi^{(\cal L)}_{u, 4p})^{\frac{1}{4p}}\right).\nonumber
			\end{align} 
			This bound also applies to each term \smash{$Q^{(A_\al)}\circ \p{\cal L-\cK}^{(k_\al)} _{t,\boldsymbol{\sigma}_{\al},\ba_{\al}}$} in \eqref{eq:expandQAempty}, with $n$ replaced by $k_\al$. Using the fact \smash{$(N\eta_t)^{-1}\le W^{-d}B_{t,0}$}, every summand of \eqref{eq:expandQAempty}, after rescaled by $(W^{-d}B_{t,0})^{-n}$, is controlled by the RHS of \eqref{am;asoiuw_smalleta}.
			Now, taking the maximum of both sides of \eqref{eq:expandQAempty} over $(\bsig,\ba)$ gives
			\begin{align*}
				\sup_{u\in[s,t]} \wh\Xi^{({\cal L-\cal K})}_{u,\fn} &\prec (W^{-d}B_{t,0})^{\frac{1}{6}}\sup_{u\in [s,t]}
				\wh\Xi^{(\cal L-\cK)}_{u, n}  + \sup_{u\in[s,t]} \left(\p{W^{-d}B_{t,0}}^{-\frac{1}{4p}}(\Xi^{(\cal L)}_{u, {2\fn-1}})^{\frac 1 2}(\Xi^{(\cal L)}_{u, 4p})^{\frac{1}{4p}}\right)\\
				&+\sup_{u\in [s,t]}
				\left(\max_{n'=2}^{\fn-1} \Xi^{({\cal L-\cal K})}_{u,n'}+ \max_{n'=\fn-1}^{\fn+1} \Xi^{({\cal L})}_{u,n'} + \max_{n'=\lceil n/2\rceil + 1}^{n-1} \Xi^{({\cal L-\cal K})}_{u,n+2-n'} \p{\Xi^{({\cal L})}_{u,n_1'} \Xi^{({\cal L})}_{u,n_2'}}^{\frac 12 } \right),\nonumber
			\end{align*}
			solving which yields \eqref{am;asoi222} at $u=t$. Obviously, the same argument applies uniformly to all $u\in [s,t]$. Finally, a standard $N^{-C}$-net argument extends \eqref{am;asoi222} uniformly to the entire interval $u\in [s,t]$. 
		\end{proof}

		\section{Step 5: Pointwise estimate for \texorpdfstring{$(\cL-\cK)$-loops}{L-Kloop}}\label{Sec:Step5}

		As in the proof for Step 3, by adding intermediate times if necessary, we can divide the proof into the following four cases: (i) $\ilambda^2/L^{2}\le 1- t\le 1-s \le \ilambda^2$, (ii) $\ilambda^2/L^{d}\le 1-t\le 1-s\le \ilambda^2/L^{2}$, (iii) $1- s\ge 1-t \ge \ilambda^2$, and (iv) $1-t\le 1-s\le \ilambda^2/L^{d}$. 
		In case (iv), $\ell_t=L$ and $B_{t,0}$ is dominated by the $(L^d|1-t|)^{-1}$ term, so the desired bound already follows from the conclusion \eqref{Eq:L-KGt-flow} in Step 4. Hence, we only need to focus on the first three cases in the following proof. Our goal is to remove the $(|1-s|/|1-u|)^{C_d}$ factor in \eqref{Eq:Gdecay_w}. 
		Our proofs below are all based on the integrated loop hierarchy in equation \eqref{int_K-LcalE} with $n=2$:
		\begin{align} \label{int_K-LcalE_n=2}
			(\mathcal{L} - \mathcal{K})^{(2)}_{t, \boldsymbol{\sigma}, \ba}= &
			\left(\mathcal{U}^{(2)}_{s, t, \boldsymbol{\sigma}} \circ (\mathcal{L} - \mathcal{K})^{(2)}_{s, \boldsymbol{\sigma}}\right)_{\ba}+ \int_{s}^t \left(\mathcal{U}^{(2)}_{u, t, \boldsymbol{\sigma}} \circ \mathcal{E}^{(\mathcal{L} - \mathcal{K}) \times (\mathcal{L} - \mathcal{K}),(2)}_{u, \boldsymbol{\sigma}}\right)_{\ba} \dd u  \nonumber \\
			&+ \int_{s}^t \left(\mathcal{U}^{(2)}_{u, t, \boldsymbol{\sigma}} \circ \mathcal{E}^{\Gc,(2)}_{u, \boldsymbol{\sigma}}\right)_{\ba} \dd u + \int_{s}^t \left(\mathcal{U}^{(2)}_{u, t, \boldsymbol{\sigma}} \circ \dd\mathcal{E}^{M,(2)}_{u, \boldsymbol{\sigma}}\right)_{\ba} . 
		\end{align}
		With the stability estimate \eqref{Eq:Gdecay_w} in hand, we can now show that the quadratic term $\mathcal{E}^{(\mathcal{L} - \mathcal{K}) \times (\mathcal{L} - \mathcal{K}),(2)}$ yields a negligible contribution. Using the light-weight estimate \eqref{eq:LW_conclusion_exp} and the martingale estimate \eqref{eq:MG_conclusion3} established in Step 2, we can also obtain sufficient control of the third and fourth terms on the RHS of \eqref{int_K-LcalE_n=2}.
		The remaining technical challenge arises from the contribution of the initial condition, i.e., the first term on the RHS of \eqref{int_K-LcalE_n=2}.
		In case (i), bounding this term requires leveraging a CLT-type cancellation mechanism, similar to that developed in \cite{DYYY25}.\footnote{Interestingly, in two dimensions \cite{DYYY25}, the CLT cancellation was used in Step~3 but not in Step~5, whereas in our case the situation is reversed. In lower dimensions, Step~5 becomes almost trivial once Steps~2 and~4 are established; however, in dimensions $d\ge 3$, the additional complexity in Step 5 arises once again due to the polynomial decay of the 2-$G$-loops.}
		In case (ii), we apply the zero-mode-removing operator introduced in \Cref{def;zero_mode_remove}. Finally, case (iii) can be regarded as a special instance of the lower-dimensional proofs for $d\in \{1,2\}$, as explained in \Cref{sec:Step5_larget}. 

		If $|1-s|/|1-t|\le (\log W)^{10}$, then the factor $(|1-s|/|1-u|)^{C_d}\le (\log W)^{10C_d}$ can be absorbed into the stochastic domination notation ``$\prec$". Hence, throughout the following proof, we always assume that 
		\be\label{eq:assmtlarge} |1-s|/|1-t| > (\log W)^{10} \implies t\ge 1-(\log W)^{-10}.\ee

		\subsection{The case \texorpdfstring{$\ilambda^2/L^{2}\le 1- t\le 1-s \le \ilambda^2$}{1-t>L-2}}
		In this case, we have \(B_{u,0}\asymp \ilambda^{-2}\) and $\ell_u\asymp \ilambda|1-u|^{-1/2}$ for all \(u\in[s,t].\)  
		In Step 2, we have established the following estimates 
		for any large constant $D>0$:  
		\be\label{eq:Step2_inputs}
		{\cal L}^{(2)}_{u,\bsig,(a_1,a_2)}\prec   W^{-d}\widetilde{\mathcal{T}}_{u,D}^{L}\p{|a_1-a_2|},\ \ ({\cal L}-{\cal K})^{(2)}_{u,\bsig,(a_1,a_2)}\prec  (\ilambda^2W^{d})^{-\frac16}\cdot   W^{-d}\widetilde{\mathcal{T}}_{u,D}^{L}(|a_1-a_2|),\ \ \forall u\in[s,t].
		\ee
		We now apply \eqref{eq:Step2_inputs} to \eqref{i2kk2zgg} with $\ell=L$. In this case, only the third term on the RHS remains nonzero, and it can be bounded as in \eqref{eq:i2kk2zgglast}. Consequently, we obtain that
		\be \label{S5WG+M000}
		\mathcal{E}_{u, \bsig, \ba}^{(\mathcal{L} - \mathcal{K}) \times (\mathcal{L} - \mathcal{K}),(2)} 
		\prec   \frac{(\ilambda^2W^{d})^{-\frac13}}{\eta_u}  \br{W^{-d}\widetilde{{\cal T}}^{L} _{u,D}(|a_1-a_2|) }.
		\ee
		Moreover, using \Cref{lem: EWGn2_N,lem: EMn2_N}, we get that for any $\bsig\in \{+,-\}^2$,
		\be \begin{aligned}\label{S5WG+M}
		{\cal E}^{\Gc,(2)}_{u,\bsig, \ba}&\prec \frac{ (\ilambda^2W^{d})^{-\frac12}}{\eta_u}   \br{W^{-d} \widetilde{  \cal T} ^{L}_{u, D}(|a_1-a_2|)},\\  
		(\mathcal{E} \otimes \mathcal{E})^{M,(2)}_{u,\bsig, \ba,\ba}&\prec \frac{ (\ilambda^2W^{d})^{-\frac12} }{\eta_u} \br{W^{-d}\widetilde{\cal T}^{L}_{u,D}\p{|a_1-a_2|}}^2.
		\end{aligned}\ee
		We recall the definition of $\cal U_{s,t,\bsig}^{(2)}$ in \eqref{def_Ustz} and notice the following identity (recall that $t\gtrsim1$ under \eqref{eq:assmtlarge}): 
		\be\label{eq:decompU}\frac{1 - s \cdot M^{(\sig_1,\sig_2)}S^{(\sB)}}{1 - t \cdot M^{(\sig_1,\sig_2)}S^{(\sB)}} = \frac{s}{t} + \frac{t-s}{t}\Theta^{(\sig_1,\sig_2)}_t.\ee
		If $\sig_1=\sig_2$, then $\Theta_{t}^{(\sig_1,\sig_2)}$ decay exponentially by \eqref{prop:ThfadC_short}. Using the estimates \eqref{S5WG+M000} and \eqref{S5WG+M}, the assumption \eqref{Eq:Gdecay+IND}, the decomposition \eqref{eq:decompU}, and the estimate \eqref{prop:ThfadC_short}, we obtain from \eqref{int_K-LcalE_n=2} that 
		\[ (\mathcal{L} - \mathcal{K})^{(2)}_{t, \boldsymbol{\sigma}, \ba} \prec (\ilambda^2W^{d})^{-\frac15}\cdot \br{W^{-d}\wT^{L}_{t,D}(|a_1-a_2|)}.\]
		Since the proof is considerably simpler than in the case $\sig_1\ne \sig_2$  discussed below, we omit the details.

		We now focus on the key challenge with $\sig_1\ne \sig_2$.
		When $\sig_1\ne \sig_2$, with the estimates in \eqref{S5WG+M} and using the decomposition \eqref{eq:decompU}, we can bound the second and third terms on the RHS of \eqref{int_K-LcalE_n=2} as
		\begin{align}
			\int_s^t\left(\mathcal{U}^{(2)}_{u, t, \boldsymbol{\sigma}} \circ \mathcal{E}_{u, \boldsymbol{\sigma}}^{(\mathcal{L}-\mathcal{K}) \times(\mathcal{L}-\mathcal{K}),(2)}\right)_{\mathbf{a}} \dd u&\prec (\ilambda^2W^{d})^{-\frac13}\int_s^t \frac{{\cal A}_{u,t, \ba} }{\eta_u} \dd u \prec (\ilambda^2W^{d})^{-\frac13}\sup_{u\in[s,t]}{\cal A}_{u,t, \ba},\label{jymwons-L-K}
			\\ 
			\int_s^t\left(\mathcal{U}^{(2)}_{u, t, \boldsymbol{\sigma}} \circ \mathcal{E}_{u, \boldsymbol{\sigma}}^{\Gc,(2)}\right)_{\mathbf{a}} \dd u
			&\prec (\ilambda^2W^{d})^{-\frac12}\int_s^t \frac{{\cal A}_{u,t,\ba}}{\eta_u} \dd u\prec (\ilambda^2W^{d})^{-\frac12}\sup_{u\in[s,t]}{\cal A}_{u,t, \ba},\label{jymwons-Gc}
		\end{align}
		where the two-dimensional tensor $\cal A$ is defined as 
		\begin{align}
			{\cal A}_{u,t, \ba}:=&~W^{-d}\widetilde{\cal T}^{L}_{u,D} \p{|a_1-a_2|}+(t-u)\sum_b \br{\Theta_{t,a_1b}^{(+,-)}\cdot  W^{-d}\widetilde{  \cal T}^{L}_{u,D}\p{|b-a_2|}
				+ W^{-d}\widetilde{  \cal T}^{L}_{u,D}\p{|a_1-b|}\cdot \Theta_{t, b a_2}^{(+,-)}} \nonumber\\  
			&~+(t-u)^2\sum_{b_1,b_2}\Theta_{t,a_1 b_1}^{(+,-)} \cdot  W^{-d}\widetilde{ \cal T}^{L}_{u,D}\p{|b_1-b_2|} \cdot \Theta_{t,b_2a_2}^{(+,-)},\quad \forall \ba=(a_1,a_2)\in (\Zn)^2. \label{eq:def_calA5}
		\end{align}
		Similarly, we write the martingale term as
		\begin{align}
			\int_{s}^t \left(\mathcal{U}^{(2)}_{u, t, \boldsymbol{\sigma}} \circ \dd \mathcal{E}^{M,(2)}_{u, \boldsymbol{\sigma}}\right)_{\ba} 
			&=\int_{s}^t \frac {u^2}{t^2}\dd \mathcal{E}^{M,(2)}
			_{u, \boldsymbol{\sigma}, {{\ba}}} 
			+ \left( \Theta_{t}^{(+,-)}\cdot \p{\int_{s}^t\frac{t-u}{t}
				\dd \mathcal{E}^{M,(2)}
				_{u, \boldsymbol{\sigma}}}\right)_{{\ba}}
			+ \left(\left(\int_{s}^t
			\frac{t-u}{t}\dd\mathcal{E}^{M,(2)}
			_{u, \boldsymbol{\sigma}}\right)\cdot\Theta_{t}^{(+,-)}  \right)_{{\ba}}\nonumber
			\\
			&\quad +
			\left(\Theta_{t}^{(+,-)}\cdot\left(\int_{s}^t
			\frac{(t-u)^2}{t^2}
			\dd\mathcal{E}^{M,(2)}
			_{u, \boldsymbol{\sigma}}\right)\cdot\Theta_{t}^{(+,-)}  \right)_{{\ba}}.\label{jymwons}
		\end{align}
		Combining the second bound in \eqref{S5WG+M} with the Burkholder-Davis-Gundy inequality (recall \eqref{aaswtghh}), we get   
		\begin{align*}
			\int_{s}^t (t-u)^r\cdot 
			\dd \mathcal{E}^{M,(2)}_{u, \boldsymbol{\sigma}, {{\ba}}} \prec (1-s)^r (\ilambda^2W^{d})^{-\frac14} \cdot W^{-d}    \widetilde{{\cal T}}^{L} _{t,D}\p{|a_1-a_2|},\quad \forall r\in\{0,1,2\}.  
		\end{align*}
		Plugging it into \eqref{jymwons} and using the $(\infty\to \infty)$-norm of $\Theta_t^{(+,-)}$ given by \eqref{eq:THETAinftinf}, we obtain 
		\begin{align}\label{uuwmskiow}
			\int_{s}^t \left(\mathcal{U}^{(2)}_{u, t, \boldsymbol{\sigma}} \circ \dd\mathcal{E}^{M,(2)}_{u, \boldsymbol{\sigma}}\right)_{{\ba}}
			\prec \frac{(1-s)^2}{(1-t)^2} \cdot (\ilambda^2W^{d})^{-\frac14} \max_{u\in[s,t]}{\cal A}_{u,t, \ba}.
		\end{align}
		Now, inserting the bounds \eqref{jymwons-L-K}, \eqref{jymwons-Gc}, and \eqref{uuwmskiow} into \eqref{int_K-LcalE_n=2} gives that 
		\begin{align}\label{iois,mtx}
			(\mathcal{L} - \mathcal{K})^{(2)}_{t, \boldsymbol{\sigma}, {\ba}} \prec &
			\left(\mathcal{U}^{(2)}_{s, t, \boldsymbol{\sigma}} \circ (\mathcal{L} - \mathcal{K})^{(2)}_{s, \boldsymbol{\sigma}}\right)_{{\ba}}+ \frac{(1-s)^2}{(1-t)^2}\cdot (\ilambda^2W^{d})^{-\frac14} \sup_{u\in [s,t]} {\cal A}_{u,t,{\ba}}  .
		\end{align}
		Applying the estimate \eqref{prop:ThfadC} to $\Theta_t^{(+,-)}$, along with the bound \eqref{TTT2}, we obtain that
		\begin{align}\label{uwp2,92kj}
			(1-u) \sum_b \Theta_{t,a_1b}^{(+,-)} \cdot \widetilde{  \cal T}^{L}_{u,D}\p{|b-a_2|} &\prec (1-u)\sum_b \cal T_{t}\p{|a_1-b|} \br{\cal T_{u}\p{|b-a_2|} +W^{-D}} \nonumber\\
			&\prec \cal T_{t}\p{|a_1-a_2|} + \frac{1-u}{1-t}W^{-D} \lesssim \widetilde{  \cal T}^{L}_{t,D-1}\p{|a_1-a_2|}.
		\end{align}
		Inserting this estimate into the definition of $\cal A$ in \eqref{eq:def_calA5}, we can bound ${\cal A}_{u,t,{\ba}}$ as   
		\begin{align}\label{uwftgwesj}
			\max_{u\in [s,t]} {\cal A}_{u,t,{\ba}}\prec \frac{1-s}{1-t}\cdot W^{-d}\widetilde{  \cal T}^{L}_{t,D-2}\p{|a_1-a_2|}.
		\end{align}
		Combining this with \eqref{iois,mtx}, we obtain, for any large constant $D>0$,
		\begin{align}\label{iois,mtx2}
			(\mathcal{L} - \mathcal{K})^{(2)}_{t, \boldsymbol{\sigma}, {\ba}} \prec &
			\left(\mathcal{U}^{(2)}_{s, t, \boldsymbol{\sigma}} \circ (\mathcal{L} - \mathcal{K})^{(2)}_{s, \boldsymbol{\sigma}}\right)_{{\ba}}+ (\ilambda^2W^{d})^{-\frac15} \cdot W^{-d}\widetilde{  \cal T}^{L}_{t,D}\p{|a_1-a_2|}.
		\end{align}
		For the first term on the RHS, we claim that
		\begin{align}\label{iksjuwjx0}
			\left(\mathcal{U}^{(2)}_{s, t, \boldsymbol{\sigma}} \circ (\mathcal{L} - \mathcal{K})^{(2)}_{s, \boldsymbol{\sigma}}\right)_{{\ba}} \prec (\ilambda^2W^{d})^{-\frac15}  \cdot W^{-d} {\cal T}_{t}\p{|a_1-a_2|}+W^{-D}.
		\end{align}
		Combining \eqref{iois,mtx2} and \eqref{iksjuwjx0}, we conclude \eqref{Eq:Gdecay_flow} for the case $\ilambda^2/L^{2}\le 1- t\le 1-s \le \ilambda^2$. 
		
		\begin{proof}[\bf Proof of \eqref{iksjuwjx0}]   
		Using the decomposition \eqref{eq:decompU}, we can expand the LHS of \eqref{iksjuwjx0} as 
			\begin{align}
				\left(\mathcal{U}^{(2)}_{s, t, \boldsymbol{\sigma}} \circ (\mathcal{L} - \mathcal{K})^{(2)}_{s, \boldsymbol{\sigma}}\right)_{{\ba}}=~&\frac{s^2}{t^2}\left(\mathcal{L} - \mathcal{K}\right)^{(2)}_{s, \boldsymbol{\sigma},{{\ba}}}
				+\frac{t-s}{t} \left(\Theta_{t}^{(+,-)} \cdot \left (\mathcal{L} - \mathcal{K}\right)^{(2)}_{s, \boldsymbol{\sigma}}\right)_{{\ba}}
				+\frac{t-s}{t} \left(\left (\mathcal{L} - \mathcal{K}\right)^{(2)}_{s, \boldsymbol{\sigma}}\cdot\Theta_{t}^{(+,-)}  \right)_{{\ba}} \nonumber\\
				&~+\frac{(t-s)^2}{t^2}  \left(\Theta_{t}^{(+,-)} \cdot\left (\mathcal{L} - \mathcal{K}\right)^{(2)}_{s, \boldsymbol{\sigma}}\cdot\Theta_{t}^{(+,-)}  \right)_{{\ba}} .\label{uwkxkisjwj0}
			\end{align}
		With the inductive hypothesis \eqref{Eq:Gdecay+IND} on $\left(\mathcal{L} - \mathcal{K}\right)_{s,\bsig,\ba}^{(2)}$ and using \eqref{uwp2,92kj} again, we get that 
			\be\label{uwkxkisjwj}
			(t-s) \left|\left(\Theta_{t}^{(+,-)} \cdot \left (\mathcal{L} - \mathcal{K}\right)^{(2)}_{s, \boldsymbol{\sigma}}\right)_{{\ba}}\right|+(t-s) \left| \left(\left (\mathcal{L} - \mathcal{K}\right)^{(2)}_{s, \boldsymbol{\sigma}}\cdot\Theta_{t}^{(+,-)}  \right)_{{\ba}}\right|\prec  (\ilambda^2W^{d})^{-\frac15} \cdot W^{-d}\widetilde{  \cal T}^{L}_{t,D}\p{|a_1-a_2 |}.
			\ee
		From \eqref{Eq:Gdecay+IND}, \eqref{uwkxkisjwj0}, and \eqref{uwkxkisjwj}, we obtain that  
			\begin{align}\label{ilowkidjsw}
				\left(\mathcal{U}^{(2)}_{s, t, \boldsymbol{\sigma}} \circ (\mathcal{L} - \mathcal{K})^{(2)}_{s, \boldsymbol{\sigma}}\right)_{{\ba}}&\prec
				(\ilambda^2W^{d})^{-\frac15} \cdot W^{-d}\widetilde{  \cal T}^{L}_{t,D}\p{|a_1-a_2|}\nonumber\\
				&+(1-s)^2 \left| \left(\Theta_{t}^{(+,-)} \cdot\left (\mathcal{L} - \mathcal{K}\right)^{(2)}_{s, \boldsymbol{\sigma}}\cdot\Theta_{t}^{(+,-)}  \right)_{{\ba}}\right| .
			\end{align}
			To show \eqref{iksjuwjx0}, it remains to prove the following estimate for $\ba=(a_1,a_2)$ and $\bsig\in \{(+,-),(-,+)\}$:   
			\begin{align}\label{iksjuwjx}
				(1-s)^2 \left(\Theta_t^{(+,-)} \cdot(\mathcal{L}-\mathcal{K})^{(2)}_{s, \bsig} \cdot \Theta_t^{(+,-)}\right)_{\ba}\prec (\ilambda^2W^{d})^{-\frac15}  \cdot W^{-d} {\cal T}_{t}\p{|a_1-a_2|}+W^{-D}.
			\end{align}

			Recall that we assume \eqref{eq:assmtlarge}, under which we have  
			\be\label{eq:ells_to_ellt}
			\ell_t\ge \ell_s\cdot (\log W)^{5}.
			\ee
			Due to the exponential decay of $\Theta_t^{(+,-)}$ and $({\cal L-\cal K})^{(2)}_{s,\bsig}$ given by \eqref{prop:ThfadC} and \eqref{Eq:Gdecay+IND}, respectively, we only need to focus on the following regime: 
			\be\label{eq:ells_to_ellt2}
			|a_1-a_2|\le \frac{1}{2}(\log W)^{3/2}\cdot \ell_t+
			(\log W)^{5/2}\cdot \ell_s \le (\log W)^{3/2}\cdot \ell_t;
			\ee
			otherwise the LHS of \eqref{iksjuwjx} is an error of order $W^{-D}$ for any large constant $D>0$. Under this condition, we have \smash{$\exp\big(-\p{|a_1-a_2|/\ell_t}^{1/2}\big)\prec 1$}, so proving \eqref{iksjuwjx} amounts to establishing the following bound:
			\be\label{iksjuwjx2} (1-s)^2 \left(\Theta_t^{(+,-)} \cdot(\mathcal{L}-\mathcal{K})^{(2)}_{s, \bsig} \cdot \Theta_t^{(+,-)}\right)_{\ba}\prec \frac{(\ilambda^2W^{d})^{-\frac65}}{|a_1-a_2|^{d-2}+1}. \ee
			To show this estimate, we write the LHS of \eqref{iksjuwjx2} as 
			\begin{align*}
				&\ \quad (1-s)^2 \sum_{b_1,b_2}\Theta_{t,a_1b_1}^{(+,-)}(\mathcal{L}-\mathcal{K})^{(2)}_{s, \bsig,(b_1,b_2)}  \p{\Theta_{t,b_2a_2}^{(+,-)}-\Theta_{t,b_1a_2}^{(+,-)}}+(1-s)^2 \sum_{b_1,b_2}\p{\Theta_{t,a_1b_1}^{(+,-)}\Theta_{t,a_2b_1}^{(+,-)}} (\mathcal{L}-\mathcal{K})^{(2)}_{s, \bsig,(b_1,b_2)} \\
				&=(1-s)^2 \sum_{b_1,b_2}\Theta_{t,a_1b_1}^{(+,-)}(\mathcal{L}-\mathcal{K})^{(2)}_{s, \bsig,(b_1,b_2)}  \p{\Theta_{t,b_2a_2}^{(+,-)}-\Theta_{t,b_1a_2}^{(+,-)}}+(1-s)^2 \sum_{b_1} \Theta_{t,a_1b_1}^{(+,-)} \Theta_{t,a_2b_1}^{(+,-)}  \frac{\im \tr\p{(G_s-m)E_{b_1}}}{W^d\eta_s} \\
				&=: f_{\ba} + g_{\ba}.
			\end{align*}
			Above, we have applied Ward's identities \eqref{WI_calL} and \eqref{WI_calK} to express $\sum_{b_2} (\mathcal{L}-\mathcal{K})^{(2)}_{s, \bsig,(b_1,b_2)}$ as 
			\be\label{eq:PcalBterm}
			\sum_{b_2} (\mathcal{L}-\mathcal{K})^{(2)}_{s, \bsig,(b_1,b_2)}=\frac{\im {({\cal L-\cal K})^{(1)}_{s,+, b_1}}}{W^d\eta_s}=\frac{\im \tr\p{(G_s-M)E_{b_1}}}{W^d\eta_s}\prec \frac{(\ilambda^2W^d)^{-1}}{W^{d}\eta_s}, \ee
			where we use the averaged local law \eqref{Gt_avgbound_flow} for $\tr\p{(G_s-M)E_{b_1}}$ in the last step. With \eqref{eq:PcalBterm}, we can bound $g_{\ba}$ as follows for any large constant $D>0$:
			\begin{align}
				g_{\ba} &\prec \frac{1-s}{\ilambda^2W^{2d}} \sum_{b_1} \Theta_{t,a_1b_1}^{(+,-)} \Theta_{t,a_2b_1}^{(+,-)}  \lesssim \frac{1-s}{\ilambda^2W^{2d}} \sum_{b_1} \cT_{t}(|a_1-b_1|)\cT_{t}(|a_2-b_1|) \nonumber\\
				&\lesssim \frac{1-s}{1-t} \p{\ilambda^2W^{d}}^{-1}\cdot W^{-d} \cT_{t}(|a_1-a_2|) \le \frac{(\ilambda^2W^{d})^{-\frac65} }{|a_1-a_2|^{d-2}+1},\label{eq:boundga}
			\end{align}
			where, in the second step, we apply the estimate \eqref{prop:ThfadC}, and in the third step, we use the bound \eqref{TTT2}.

			To conclude \eqref{iksjuwjx2}, it remains to show that 
			\be\label{iksjuwjx3} 
			f_{\ba}\prec \frac{(\ilambda^{2}W^d)^{-\frac65} }{|a_1-a_2|^{d-2}+1}. \ee
			To simplify notation, we abbreviate ${\cal B}:=(\mathcal{L} - \mathcal{K})^{(2)}_{s, \boldsymbol{\sigma}}$. By the assumption \eqref{Eq:Gdecay+IND}, we have
			\be\label{eq:propcalB} 
			{\cal B}_{b_1b_2}\prec \frac{(\ilambda^{2}W^d)^{-\frac65} }{|b_1-b_2|^{d-2}+1},\quad \text{and} \quad  {\cal B}_{b_1b_2}\prec W^{-D}\quad \text{whenever}\ \  |b_1-b_2|\ge (\log W)^3\ell_s ,
			\ee
			for any large constant $D>0$. 
			To show \eqref{iksjuwjx3}, we decompose $f_{\ba}$ into the following two parts:
			\begin{align*}
				f_{\ba}^{\mathrm{near}}&:=(1-s)^2\sum_{b_1: |b_1-a_1|\wedge |b_1-a_2| \le (\log W)^4\ell_s} \sum_{b_2} \Theta_{t,a_1b_1}^{(+,-)}\cal B_{b_1b_2}\p{\Theta_{t,b_2a_2}^{(+,-)}-\Theta_{t,b_1a_2}^{(+,-)}},\\  f_{\ba}^{\mathrm{far}}&:=(1-s)^2\sum_{b_1: |b_1-a_1|\wedge |b_1-a_2| > (\log W)^4\ell_s} \sum_{b_2} \Theta_{t,a_1b_1}^{(+,-)}\cal B_{b_1b_2}\p{\Theta_{t,b_2a_2}^{(+,-)}-\Theta_{t,b_1a_2}^{(+,-)}}.
			\end{align*}
			For the term $f_{\ba}^{\mathrm{near}}$, using \eqref{eq:propcalB} and \eqref{prop:ThfadC}, we get that for any constant $D>0$, 
			\begin{align*}
				f_{\ba}^{\mathrm{near}} &\prec  (1-s)^2\sum_{b_1}^* \sum_{b_2:|b_2-b_1|\le (\log W)^3\ell_s}\frac{\ilambda^{-2}}{|a_1-b_1|^{d-2}+1}\frac{(\ilambda^{2}W^d)^{-\frac65}}{|b_1-b_2|^{d-2}+1} \p{\frac{\ilambda^{-2}}{|b_2-a_2|^{d-2}+1}+\frac{\ilambda^{-2}}{|b_1-a_2|^{d-2}+1}} + W^{-D}
				\\
				&\prec  (1-s)^2\sum_{b_1}^* \frac{(\ilambda^{2}W^d)^{-\frac65}}{|a_1-b_1|^{d-2}+1}\cdot {\frac{\ilambda^{-4}\ell_s^2}{|b_1-a_2|^{d-2}+1}}
				\prec (\ilambda^{2}W^d)^{-\frac65} \cdot \frac{(1-s)^2 \ilambda^{-4}\ell_s^4}{|a_1-a_2|^{d-2}+1}  \lesssim \frac{(\ilambda^{2}W^d)^{-\frac65}}{|a_1-a_2|^{d-2}+1},
			\end{align*}
			where $\sum^*_{b_1}$ refers to the summation over $\{b_1:|b_1-a_1|\wedge |b_1-a_2| \le (\log W)^4\ell_s\}$. 
			To control the term \smash{$f_{\ba}^{\mathrm{far}}$}, we employ a CLT-type cancellation mechanism, extending the approach developed for 2D random band matrices in \cite{DYYY25}.

			\begin{lemma}\label{lem;CLT}
				In the setting of \Cref{lem:main_ind}, fix any $\ilambda^2/L^{2}\le 1-t \le 1-s\le \ilambda^2$ so that \eqref{con_st_ind} holds. Then, $f_{\ba}^{\mathrm{far}}$ satisfies the following estimate under the conditions \eqref{eq:ells_to_ellt} and \eqref{eq:ells_to_ellt2}:
				\be\label{iksjuwjx3_far} f_{\ba}^{\mathrm{far}}\prec \frac{(\ilambda^{2}W^d)^{-\frac65} }{|a_1-a_2|^{d-2}+1}. \ee
			\end{lemma}
			With this lemma, we conclude the estimate \eqref{iksjuwjx3}, which implies \eqref{iksjuwjx}, and hence completes the proof of \eqref{iksjuwjx0} for the case $\ilambda^2/L^{2}\le 1-t \le 1-s\le \ilambda^2$ together with \eqref{ilowkidjsw}. 
		\end{proof}

		\begin{proof}[\bf Proof of \Cref{lem;CLT}]
			We decompose $f_{\ba}^{\mathrm{far}}$ as its expectation $\E f_{\ba}^{\mathrm{far}}$ plus the fluctuation part $\IE f_{\ba}^{\mathrm{far}}$, where $\IE$ denotes $\mathbb {IE}:=1-\E$, and $\E f_{\ba}^{\mathrm{far}}$ and $\IE f_{\ba}^{\mathrm{far}}$ represent the expressions obtained by replacing $\cal B$ in $f_{\ba}^{\mathrm{far}}$ with $\E\cal B$ and $\IE \cal B$, respectively. 
			Due to the translation invariance and symmetry of $\E\cal B$, the expectation part $\E f_{\ba}^{\mathrm{far}}$ involves a second-order difference of the $\Theta^{(+,-)}$-propagator, which satisfies the improved bound \eqref{prop:BD2} that is summable in $a$. 
			For the fluctuation part $\IE f_{\ba}^{\mathrm{far}}$, the $\Theta^{(+,-)}$-propagator effectively transfers the estimate of $\IE \cal B$ from shorter scales of order $\ell_s$ to the larger scale $\ell_t$. Moreover, intuitively, the random variables \smash{$\IE \cal B_{b_1b_2}$ and $\IE \cal B_{b_1'b_2'}$} become asymptotically independent when $|b_1-b_1'|$ exceeds the typical decay scale $\ell_s$ of $\cal B$. 
			Thus, the term $\IE f_{\ba}^{\mathrm{far}}$ can be viewed as a superposition of roughly $\ell_t^d/\ell_s^d$ asymptotically independent, centered random variables, yielding an additional improvement of order through a CLT-type argument.

			We now proceed to the formal proof. 
			By the translation invariance and symmetry of our model on the block level, $\E \cal B$ is also translationally invariant and symmetric in the following sense:
			\[ \E \cal B(a+c,b+c)=\E\cal B(a,b),\quad \E \cal B(a,b)=\E\cal B(b,a),\quad \forall a,b,c\in \Zn.\]
			These imply that \( \E	\cal B(b_1,b_2)=\E \cal B(b_2,b_1)=\E\cal B(b_1,2b_1-b_2)\), with which we get 
			\begin{align*}
				\sum_{b_2} \E\cal B_{b_1b_2}\p{\Theta_{t,b_2a_2}^{(+,-)}-\Theta_{t,b_1a_2}^{(+,-)}}&=\sum_{b_2} \cal \E \cal B_{b_1,(2b_1-b_2)}\p{\Theta_{t,b_2a_2}^{(+,-)}-\Theta_{t,b_1a_2}^{(+,-)}} =\sum_{b_2} \E \cal B_{b_1b_2}\p{\Theta_{t,(2b_1-b_2)a_2}^{(+,-)}-\Theta_{t,b_1a_2}^{(+,-)}}.
			\end{align*}
			Thus, we can rewrite $\E f_{\ba}^{\mathrm{far}}$ as 
			\begin{align*}
				\E	f_{\ba}^{\mathrm{far}}&=(1-s)^2\sum_{b_1}^\star \sum_{b_2} \Theta_{t,a_1b_1}^{(+,-)}\E\cal B_{b_1b_2}\cdot \frac 12 \p{\Theta_{t,b_2a_2}^{(+,-)}+\Theta_{t,(2b_1-b_2)a_2}^{(+,-)} -2\Theta_{t,b_1a_2}^{(+,-)}},
			\end{align*}
			where $\sum^\star_{b_1}$ refers to the summation over $\{b_1:|b_1-a_1|\wedge |b_1-a_2| > (\log W)^4\ell_s\}$. Using \eqref{prop:ThfadC} and \eqref{prop:BD2}, along with the properties of $\cal B$ in \eqref{eq:propcalB}, we can bound the above expression by 
			\begin{align}
				\E	f_{\ba}^{\mathrm{far}} &\prec (1-s)^2\sum_{b_1}^\star \frac{\ilambda^{-2}}{|a_1-b_1|^{d-2}+1}\frac{\ilambda^{-2}\ell_s^2}{|a_2-b_1|^{d}+1}  \sum_{b_2:|b_2-b_1|\le (\log W)^3\ell_s} \frac{(\ilambda^{2}W^d)^{-\frac65}}{|b_1-b_2|^{d-2}+1} + W^{-D} \nonumber\\
				&\prec \frac{(\ilambda^{2}W^d)^{-\frac65}\cdot (1-s)^2\ilambda^{-4}\ell_s^4}{|a_1-a_2|^{d-2}+1} \lesssim \frac{(\ilambda^{2}W^d)^{-\frac65}}{|a_1-a_2|^{d-2}+1} \label{eq:boundEfar}
			\end{align}
			for arbitrarily large constant $D>0$.
			
			To deal with the term $\IE f_{\ba}^{\mathrm{far}}$, we need to explore the CLT cancellation within it as in \cite[Section 7]{DYYY25} and \cite[Appendix A.10]{RBSO1D}. By Markov's inequality and using $\ilambda^{4}(1-s)^{-2}\asymp\ell_s^4$, it suffices to prove the following $(2p)$-th moment bound for any fixed $p\in \N$:
			\be\label{eq:main_challenge3}
			\E\abs{\ilambda^4(1-s)^{-2}(\ilambda^{2}W^d)^{\frac65}\cdot  \IE f_{\ba}^{\mathrm{far}}}^{2p} \prec \left(\frac{\ell_s^4}{|a_1-a_2|^{d-2}+1}\right)^{2p} . 
			\ee
			Abbreviating \smash{$\Xi_{a_2}\p{b_1,b_2}:=\ilambda^2(\Theta_{t,b_2a_2}^{(+,-)}-\Theta_{t,b_1a_2}^{(+,-)})$} and $\mathsf B:=(\ilambda^{2}W^d)^{\frac65}\cal B$, we can write the LHS of \eqref{eq:main_challenge3} as 
			\be\label{eq:2p_product}\sum_{\bfb^{(1)},\ldots, \bfb^{(2p)}}^{\star} \prod_{k=1}^{2p}\p{\ilambda^2\Theta_{t}^{(+,-)} (a_1,b_1^{(k)})\cdot \Xi_{a_2}(b_1^{(k)},b_2^{(k)})} \cdot \E\bigg\{\prod_{k=1}^{p}   \IE \mathsf B_{b_1^{(k)}b_2^{(k)}} \cdot \prod_{k=p+1}^{2p} \IE \overline{\mathsf B}_{b_1^{(k)}b_2^{(k)}}\bigg\}+\OO(W^{-D})\ee
			for any large constant $D>0,$ where $\sum^\star$ refers to the summation of \smash{$\bfb^{(k)}=(b_1^{(k)},b_2^{(k)})$} over the regions 
			\be\label{eq:sumregionsforb}\left\{(b_1^{(k)},b_2^{(k)}):|b_1^{(k)}-a_1|\wedge |b_1^{(k)}-a_2| > (\log W)^4\ell_s,\ |b_1^{(k)}-b_2^{(k)}| \le (\log W)^3\ell_s\right\}.\ee 
			
			First, suppose the following pairing condition holds: 
			\be\label{eq:pairingcond}
			\text{for any }\ k\in \qqq{2p},\ \text{ there exists } \ l\ne k \ \text{ such that } \ |b^{(k)}_1-b^{(l)}_1|\le 10(\log W)^3\ell_s.\ee  
			We can divide $\qqq{2p}$ into a disjoint union of $r$ subsets $\qqq{2p}=\sqcup_{i=1}^r A_i,$ constructed such that for every $k\in A_i$, any $l\in \qqq{2p}$ satisfying \smash{$|b^{(k)}_1-b^{(l)}_1|\le 10(\log W)^3\ell_s$} also belongs to $A_i$. According to the pairing condition, we have $r\le p$. Without loss of generality, suppose every subset $A_i$ contains an index $k_i$. Subject to the condition \eqref{eq:pairingcond} and the above decomposition $\qqq{2p}=\sqcup_{i=1}^r A_i$ (the corresponding summation being denoted by \smash{$\sum^{\star_1}$}), we can bound \eqref{eq:2p_product} as follows:
			\begin{align}
				\label{eq:2p_product_pair}
				&~\sum^{\star_1}_{\bfb^{(1)},\ldots, \bfb^{(2p)}} \prod_{i=1}^{r}\left(\prod_{k\in A_i}\p{\ilambda^2\Theta_{t}^{(+,-)} (a_1,b_1^{(k)})\cdot \Xi_{a_1}(b_1^{(k)},b_2^{(k)})}\cdot \big|\IE \mathsf B_{b_1^{(k)}b_2^{(k)}}\big|\right) \\
				\prec &~\sum^{\star_1}_{\bfb^{(1)},\ldots, \bfb^{(2p)}} \prod_{i=1}^{r}\left( \frac{1}{\big|a_1-b_1^{(k_i)}\big|^{d-2}+1}\right)^{|A_i|}\p{\frac{\ell_s}{\big|a_2-b_1^{(k_i)}\big|^{d-1}+1}}^{|A_i|}\prod_{k\in A_i}\p{\frac{1}{\big|b_1^{(k)}-b_2^{(k)}\big|^{d-2}+1}} \nonumber\\
				\prec &~\sum^{\star_2}_{b_1^{(k_1)},\ldots, b_1^{(k_r)}} \prod_{i=1}^{r}\left( \frac{1}{\big|a_1-b_1^{(k_i)}\big|^{d-2}+1}\right)^{|A_i|}\p{\frac{\ell_s}{\big|a_2-b_1^{(k_i)}\big|^{d-1}+1}}^{|A_i|} (\ell_s^{d+2})^{|A_i|-1}\ell_s^2 \nonumber \\
				\lesssim &~\prod_{i=1}^{r}\left( \frac{1}{|a_1-a_2|^{d-2}+1}\right)^{|A_i|}\frac{\ell_s^d}{(\ell_s^{d-2})^{|A_i|}} \cdot (\ell_s^{d+2})^{|A_i|-1} \ell_s^2=\left(\frac{\ell_s^4}{|a_1-a_2|^{d-2}+1}\right)^{2p}.\nonumber 
			\end{align}
			Here, in the first step, we use the estimates \eqref{prop:ThfadC}, \eqref{prop:BD1}, and \eqref{eq:propcalB}, along with the conditions in \eqref{eq:sumregionsforb} and \eqref{eq:pairingcond}. In the second step, we take the summation over $\bfb^{(k)}$ for $k\in A_i\setminus\{k_i\}$ and over \smash{$b_2^{(k_i)}$} under the restrictions given by \eqref{eq:sumregionsforb} and \eqref{eq:pairingcond}, yielding a factor $(\ell_s^{d+2})^{|A_i|-1}\ell_s^2$ for each $i\in\qqq{r}$, and $\sum^{\star_2}$ refers to the summation under the constraints \smash{$|b_1^{(k_i)}-a_1|\wedge |b_1^{(k_i)}-a_2| > (\log W)^4\ell_s$} for $i\in\qqq{r}$. In the third step, we have used the following bound for any $k\ge 2$:
			\begin{align}\label{eq:simplecalculus}
				\sum_{b:|b-a_1|\wedge |b-a_2| > (\log W)^4\ell_s} \left( \frac{1}{|a_1-b|^{d-2}+1}\right)^{k}\p{\frac{\ell_s}{|a_2-b|^{d-1}+1}}^{k}\lesssim \p{\frac{1}{|a_1-a_2|^{d-2}+1}}^k \frac{\ell_s^d}{(\ell_s^{d-2})^k} . 
			\end{align}
			To see why this holds, consider the case $|b-a_1|\ge |b-a_2|$---the case $|b-a_1|\le |b-a_2|$ can be handled similarly. We can bound the corresponding summation by 
			\begin{align*}
				\left( \frac{1}{|a_1-a_2|^{d-2}+1}\right)^{k} \sum_{b:|b-a_1|\wedge |b-a_2| > (\log W)^4\ell_s} \p{\frac{\ell_s}{|a_2-b|^{d-1}+1}}^{k}\prec \left( \frac{1}{|a_1-a_2|^{d-2}+1}\right)^{k} \frac{\ell_s^d}{(\ell_s^{d-2})^k} , 
			\end{align*}
			where we use that $(|a_2-b|^{d-1}+1)^{-k}$ is summable when $d\ge 3$ and $k\ge 2$. This gives \eqref{eq:simplecalculus}.

			Now, with \eqref{eq:2p_product_pair}, we obtain \eqref{eq:main_challenge3} under the pairing condition \eqref{eq:pairingcond}. It remains to control \eqref{eq:2p_product} when \eqref{eq:pairingcond} does not hold. In fact, we can prove the following result: if there exists an isolated vertex $\bfb^{(i)}$ such that \smash{\(\min_{j:j\ne i}|b_1^{(i)}-b_1^{(j)}|\ge 10(\log W)^3 \ell_s\)}, then for any large constant $D>0$,
			\be\label{eq:bound_isolated}
			\E\bigg\{\prod_{k=1}^{p}   \IE \mathsf B_{b_1^{(k)}b_2^{(k)}} \cdot \prod_{k=p+1}^{2p} \IE \overline{\mathsf B}_{b_1^{(k)}b_2^{(k)}}\bigg\} \le W^{-D}. 
			\ee
			Since the proof of this bound is exactly the same as that for \cite[equation (7.39)]{DYYY25} and \cite[equation (A.112)]{RBSO1D} (which does not depend on the dimension $d$), we omit the details here. This completes the proof of \eqref{eq:main_challenge3}, and hence concludes \Cref{lem;CLT} together with \eqref{eq:boundEfar}. 
		\end{proof}
		
		\subsection{The case \texorpdfstring{$\ilambda^2/L^{d} \leq 1-t \leq 1-s \leq \ilambda^2/L^{2}$}{1-t>L-d}}
		
		The $\sig_1=\sig_2$ case is straightforward by analyzing equation \eqref{int_K-LcalE_n=2} with the estimates \eqref{Eq:Gdecay_w}, \eqref{eq:LW_conclusion_exp}, \eqref{eq:MG_conclusion3}, the decomposition \eqref{eq:decompU}, and the exponential bound \eqref{prop:ThfadC_short}. We therefore omit the details and focus on the case $\sig_1\ne \sig_2$ in the following.
		In this case, we first remove the zero mode from the $(\cL-\cK)^{(2)}$-loop using the zero-mode-removing operator introduced in \Cref{def;zero_mode_remove}. Once the zero mode is removed, the rescaled propagator \smash{$(1-s)\zTheta_t^{(+,-)}$} becomes summable by \eqref{prop:ThfadC0}.

		Applying Ward's identities \eqref{WI_calL} and \eqref{WI_calK}, we can express \smash{$(\mathcal{L} - \mathcal{K})_{t, \boldsymbol{\sigma}}^{(2)}$} as: 
		\begin{align}\label{zYU1}
			(\mathcal{L} - \mathcal{K})^{(2)}_{t, \boldsymbol{\sigma}, {\ba}} - \left[Q^{(1)}\circ(\mathcal{L} - \mathcal{K})^{(2)}_{t, \boldsymbol{\sigma}}\right]_{{\ba}}&=
			\frac{\im {({\cal L-\cal K})^{(1)} _{t,+, a_2}}}{N\eta_t}  \prec {\frac{(\ilambda^{2}W^{d})^{-1}}{N\eta_t}},
		\end{align}
		where, in the second step, we apply the averaged local law \eqref{Gt_avgbound_flow}. 
		Then, we analyze the term \smash{$Q^{(1)}\circ(\mathcal{L} - \mathcal{K})^{(2)}_{t, \boldsymbol{\sigma}}$} using equation \eqref{iisuwjyys} with $n=2$. By definition, we have  \smash{$Q^{(1)}\circ\Theta_t^{(+,-)}=\zTheta_t^{(+,-)}$}. Using \eqref{prop:ThfadC0}, we obtain the following counterparts to the equations \eqref{eq:THETAinftinf} and \eqref{uwp2,92kj}: 
		\begin{align}\label{ThetaBcirc_infint} 
			\|\zTheta_t^{(\sB)}\|_{\infty\to \infty} = \max_a \sum_{b}\big|\zTheta^{(+,-)}_{t,ab}\big| & \prec \ilambda^{-2}L^2, \\ \label{uwp2,92kj00}
			(1-u) \sum_b \zTheta_{t,a_1b}^{(\sB)} \cdot \widetilde{  \cal T}^{L}_{u,D}\p{|b-a_2|} &\prec \widetilde{  \cal T}^{L}_{t,D}\p{|a_1-a_2|}.
		\end{align}
		With these estimates in hand, and employing a similar argument to the one above (i.e., the one leading to \eqref{iois,mtx2} and \eqref{ilowkidjsw}), we can derive the following estimate for  \smash{$Q^{(1)}\circ(\mathcal{L} - \mathcal{K})^{(2)}_{t, \boldsymbol{\sigma}}$}: 
		\begin{align}\label{zYU2}
			\left[Q^{(1)}\circ(\mathcal{L} - \mathcal{K})^{(2)}_{t, \boldsymbol{\sigma}}\right]_{{\ba}}
			& \prec (\ilambda^{2}W^d)^{-\frac15}\cdot W^{-d}\widetilde{  \cal T}^{L }_{t,D}\p{|a_1-a_2|}+ (1-s)^2 \left(\zTheta_t^{(+,-)} \cdot(\mathcal{L}-\mathcal{K})^{(2)}_{s,  \boldsymbol{\sigma}} \cdot \Theta_t^{(+,-)}\right)_{\ba} .
		\end{align}
		To control the second term on the RHS, we apply the bounds \eqref{prop:ThfadC}, \eqref{prop:ThfadC0}, and \eqref{Eq:Gdecay+IND} to get that 
		\begin{align*}
			&~(1-s)^2 \left(\zTheta_t^{(+,-)} (\mathcal{L}-\mathcal{K})^{(2)}_{s,  \boldsymbol{\sigma}}  \Theta_t^{(+,-)}\right)_{\ba}\\
			\prec &~  (1-s)^2 \frac{(\ilambda^2 W^{d})^{-\frac15}}{W^{d}} \sum_{b_1,b_2}\frac{\ilambda^{-2}}{|a_1-b_1|^{d-2}+1}\p{\frac{\ilambda^{-2}}{|b_1-b_2|^{d-2}+1}+\frac{1}{L^d(1-s)}} \p{ \frac{\ilambda^{-2}}{|b_2-a_2|^{d-2}+1} + \frac{1}{L^d(1-t)}}  \\
			\prec &~ (1-s) \frac{(\ilambda^2 W^{d})^{-\frac15}}{W^{d}}  \sum_{b_1}\frac{\ilambda^{-2}}{|a_1-b_1|^{d-2}+1} \left( \frac{\ilambda^{-2}}{|b_1-a_2|^{d-2}+1}+\frac{1}{L^d (1-t)}\right) \\
			\prec &~ \frac{(\ilambda^2 W^{d})^{-\frac15}}{W^{d}} \left(\frac{(1-s)\ilambda^{-4}L^2}{|a_1-a_2|^{d-2}+1}+\frac{(1-s)\ilambda^{-2}L^2}{L^{d}(1-t)}\right) \lesssim (\ilambda^2 W^{d})^{-\frac15} \cdot W^{-d}\widetilde{\mathcal{T}}_{t,D}^{L}\p{ |a_1-a_2|},
		\end{align*}
		where we use $B_{s,0}\asymp \ilambda^{-2}$ in the first step, and $1-s\le \ilambda^2/L^{2}$ in the second and fourth steps. Plugging it into \eqref{zYU2} and using \eqref{zYU1}, we complete the proof of \eqref{Eq:Gdecay_flow} for the case $\ilambda^2/L^{d} \leq 1-t \leq 1-s \leq \ilambda^2/L^{2}$. 
		
		\subsection{The case \texorpdfstring{$1-t \ge \ilambda^2$}{1-t>lambda2}}\label{sec:Step5_larget}

		In this case, the length scale $\ell_u$ remains identically equal to 1 throughout the evolution $u\in[s,t]$ (recall \eqref{eq:ellt}). With the estimate \eqref{Eq:Gdecay_w} established in Step 2, the proof of \eqref{Eq:Gdecay+s<g_flow} is relatively straightforward and may be viewed as a special case of the argument in \cite[Section 5.3]{YY_25}, where decay estimates for 2-$G$-loops were derived in the context of 1D random band matrices. 
		The key reason that the argument of \cite{YY_25} applies in the regime $1-t \ge \ilambda^2$ is that the prefactor $(W^{-d}B_{u,0})^{2}$ in the target estimate \eqref{Eq:Gdecay+s<g_flow} is of order $(W^d|1-u|)^{-2}$. Unlike the regime $1-u\ll \ilambda^2$, this quantity has no polynomial dependence on $|a_1-a_2|$, precisely matching the setting in dimension 1. 
		In fact, our proof here is much simpler, since we have already established exponential decay in \eqref{Eq:Gdecay_w} and obtained the optimal (maximum) $(\cal L - \cal K)$-loop estimate \eqref{Eq:L-KGt-flow}. 
		
		Given $u\in[s,t]$ and a sufficiently large constant $D>0$, we introduce the following tail function to control the tail behavior of the $(\cL-\cK)$-loops: 
		\begin{align}   
			{T} _{u, D} (r) :=  (W^d|1-u|)^{-2}
			\exp(- r^{1/2})+W^{-D},\quad \forall r\ge 0 .\label{def_WTuD}
		\end{align} 
		Let ${\cal J}^* _{u,D}\ge 1$ be a \emph{deterministic control parameter} such that 
		\be\label{eq:def_new_J*}
		\max_{\ba=(a,b)} \max_{\bsig}{\big|\left({\cal L-\cal K}\right)^{(2)}_{u,\bsig, \ba}\big|}\big/{ T}_{u,D}\p{|a-b|} \prec {\cal J}^* _{u,D}.
		\ee
		With the tail function \eqref{def_WTuD} and the control parameter ${\cal J}^* _{u,D}$, we can obtain the following estimates on the second to fourth terms on the RHS of \eqref{int_K-LcalE_n=2}.

		\begin{lemma}\label{lem_dec_calE} In the setting of  \Cref{lem:main_ind}, suppose $1-t\ge \ilambda^2$ and the estimates \eqref{Gt_bound_flow}--\eqref{Eq:L-KGt-flow} hold. Then, for any $u\in [s,t]$, $\bsig\in\{+,-\}^2$, $\ba=(a_1,a_2)\in(\Zn)^2$, and large enough $D$ (such that $W^D\ge N$), the following estimates hold:  
			\begin{align}\label{res_deccalE_lk}
				& {\cal E}^{(\cL-\cK)\times(\cL-\cK),(2)}_{u, \boldsymbol{\sigma},\ba} \big/ 
				{T}_{u,D}(|a_1-a_2|) \prec  (1-u)^{-1}\cdot   \left(W^d|1-u|\right)^{-1}
				\left({\cal J}^{*}_{u,D} \right)^{2},\\
				&\cal E^{\Gc,(2)}_{u, \boldsymbol{\sigma},\ba}  \big/ 
				{T}_{u,D}(|a_1-a_2|)
				\prec (1-u)^{-1}\br{{\bf 1}\left(|a_1-a_2|\le (\log W)^{\frac32}\right)  +  \left(W^d|1-u|\right)^{-\frac12}  \p{{\cal J}^{*}_{u,D}}^{\frac32}  }.
				\label{res_deccalE_wG}
			\end{align} 
			Furthermore, suppose $\ba'=(a_1',a_2')\in (\Zn)^2$ satisfies that \(\max_{i=1}^2 |a_i-a_i'|\le (\log W)^{3/2}.\) Then, we have that 
			\begin{align}
				\label{res_deccalE_dif}
				\left(
				\mathcal{E}\otimes  \mathcal{E} \right)^{M,(2)}
				_{u, \bsig, \ba ,\ba'} \big/  
				{T}_{u,D}^2(|a_1-a_2|)
				& \prec (1-u)^{-1}\br{{\bf 1}\left(|a_1-a_2|\le 4(\log W)^{\frac32}\right)  +  \left(W^d|1-u|\right)^{-\frac12} \left({\cal J}^{*}_{u,D} \right)^{3}}.   
			\end{align} 
		\end{lemma}
		\begin{proof}
			The proof of this lemma is a special case of the argument for \cite[Lemma 5.7]{YY_25}, as explained above. In fact, many of the arguments can be simplified using the estimates \eqref{Gt_bound_flow}--\eqref{Eq:L-KGt-flow} established in the previous steps. We omit the details.
		\end{proof}
		
		Another key ingredient in the proof is the following evolution kernel estimate for ${\cal U}^{(2)}_{s,t,\boldsymbol{\sigma}}$.  
		
		\begin{lemma}	\label{TailtoTail}
			Suppose $1-s\ge 1-t\ge \ilambda^2$ and ${\cal A}$ satisfies that 
			$$
			|{\cal A}_{\ba}|\le T_{s,D}(|a_1-a_2|),\quad \forall \ba=(a_1,a_2)\in(\Zn)^2,
			$$
			for some constant $D>0$. Then, we have  
			\begin{align}
				\label{neiwuj} 
				\left({\cal U}^{(2)}_{s,t,\boldsymbol{\sigma}} \circ 
				{\cal A}\right)_{\ba} & \lesssim
				{\cal T}_{t,D}( |a_1-a_2|)+(|1-s|/|1-t|)^2\cdot W^{-D} . 
			\end{align}
		\end{lemma}
		\begin{proof}
			The estimate \eqref{neiwuj} can be proved easily with \eqref{sum_res_Ndecay}, the estimate \eqref{prop:ThfadC} on $\Theta$-propagators, and the following basic calculus fact:
			$\max_{a\in \mathbb R^d} \int_{x\in \mathbb R^d} \exp\big(-\sqrt{|a-x|}-\sqrt{|x|}+ \sqrt{|a|} \big)\dd x\lesssim 1.$  
			We omit the details. 
		\end{proof}
		
		Now, for an arbitrarily small constant $\e>0$, we define the stopping time
		\begin{align}\label{eq:def_TTT}
			T:=\inf \Big\{u\ge s:  \max_{\ba=(a,b)} \max_{\bsig}{\big|\left({\cal L-\cal K}\right)^{(2)}_{u,\bsig, \ba}\big|}\big/{ T}_{u,D}\p{|a-b|} \ge W^\e \Big\} .
		\end{align}
		Using \Cref{lem_dec_calE,TailtoTail}, we can establish the following lemma by analyzing the equation \eqref{int_K-LcalE_n=2}. 
		
		\begin{lemma}\label{lem:pf_step5} 
			In the setting of \Cref{lem:main_ind}, suppose $1-t\ge \ilambda^2$ and the estimates \eqref{Gt_bound_flow}--\eqref{Eq:L-KGt-flow} hold. Then, for sufficiently small constant $\e>0$, we have that $T\ge t$ with high probability. 
		\end{lemma}  
		\begin{proof}
			The proof of this lemma is analogous to, and in fact much simpler than, the proof of equation (2.76) in \cite[Section 5.3]{YY_25}, by using the estimates \eqref{res_deccalE_lk}--\eqref{res_deccalE_dif} and the evolution kernel estimate \eqref{neiwuj}. Hence, we omit the details of the proof. 
		\end{proof}
		
		By \Cref{lem:pf_step5}, we immediately obtain \eqref{Eq:Gdecay+s<g_flow}, since $\e$ is arbitrary. Moreover, this estimate is stronger than \eqref{Eq:Gdecay_flow}. This completes Step 5 in the case $1-s\ge 1-t\ge \ilambda^2$.

\section{Step 6: Expected 2-loop estimates}\label{subsection:step6}

Our proof of Step 6 relies on analyzing the expectation of the loop hierarchy in \eqref{int_K-LcalE_n=2}: 
\begin{align} \label{Eexpint_K-L}
	\E(\mathcal{L} - \mathcal{K})^{(2)}_{t, \boldsymbol{\sigma}, \ba}=
	\left(\mathcal{U}^{(2)}_{s, t, \boldsymbol{\sigma}} \circ \E(\mathcal{L} - \mathcal{K})^{(2)}_{s, \boldsymbol{\sigma}}\right)_{\ba}+ \int_{s}^t \left(\mathcal{U}^{(2)}_{u, t, \boldsymbol{\sigma}} \circ \E\mathcal{E}^{(\mathcal{L} - \mathcal{K}) \times (\mathcal{L} - \mathcal{K})}_{u, \boldsymbol{\sigma}} + \mathcal{U}^{(2)}_{u, t, \boldsymbol{\sigma}} \circ \E\mathcal{E}^{\Gc}_{u, \boldsymbol{\sigma}}\right)_{\ba} \dd u ,  
\end{align}
where we omit the superscript ``$(2)$" from some of the notation for brevity. 
In addition, we use the following expected averaged local law, which improves the (sharp) averaged local law \eqref{Gt_avgbound_flow} from Step 2 by an extra factor of $W^{-d}B_{u,0}$ arising from the expectation.

\begin{lemma}\label{lem:improve_exp_aver}
	In the setting of \Cref{lem:main_ind}, suppose the estimates \eqref{Gt_avgbound_flow} and \eqref{Eq:L-KGt-flow} hold uniformly in $u\in[s,t]$. Then, we have 
	\begin{align}\label{res_ELK_n=1} 
		\max_{a}  \left|\mathbb E\tr\p{ (G_u-M) E_{a}}\right|
		\prec (W^{-d}B_{u,0})^{2},\quad \forall  s\le u \le t .
	\end{align} 
\end{lemma}
\begin{proof}
	The proof of this lemma is the same as that for \cite[Lemma 5.15]{YY_25} by using \eqref{Gt_avgbound_flow} and \eqref{Eq:L-KGt-flow}. 
\end{proof}

For the first term on the RHS of \eqref{Eexpint_K-L}, $\E(\mathcal{L} - \mathcal{K})^{(2)}_{s, \boldsymbol{\sigma},\ba}$ satisfies \eqref{Eq:Gtlp_exp+IND}. For the $\E\mathcal{E}^{(\mathcal{L} - \mathcal{K}) \times (\mathcal{L} - \mathcal{K})}$ term, combining the pointwise bound \eqref{Eq:Gdecay_flow} from Step 5 with the maximum bound \eqref{Eq:L-KGt-flow} from Step 4 yields, for $1-u \ge \ilambda^2/L^{d}$ and any $\ba=(a_1,a_2)\in (\Zn)^2$,
\begin{align}
	\E{\cal E}^{(\cL-\cK)\times(\cL-\cK)}_{u,\boldsymbol{\sigma},\ba}&\prec 
	(W^{-d}B_{u,0})^{\frac{11}{5}} \sum_{b} B_{u,|a_1-b|}e^{-(|a_1-b|/\ell_u)^{1/2}}\lesssim (1-u)^{-1}(W^{-d}B_{u,0})^{\frac{11}{5}},\label{eq:Exp(L-K)1}
\end{align}
while for $1-u\le \ilambda^2/L^{d}$, the maximum bound \eqref{Eq:L-KGt-flow} gives 
\begin{align}
	\E{\cal E}^{(\cL-\cK)\times(\cL-\cK)}_{u,\boldsymbol{\sigma},\ba}&\prec W^d\cdot L^d(N|1-u|)^{-4} =(1-u)^{-1} (N|1-u|)^{-3}.\label{eq:Exp(L-K)2}
\end{align}
Finally, consider the light-weight term in \eqref{Eexpint_K-L}. By the definition \eqref{def_EwtG}, we may write 
\be\label{eq:LW-n=2}
\mathcal{E}^{\Gc}_{u, \boldsymbol{\sigma}, \ba}= {W}^d \sum_{k=1}^2 \sum_{\al, \beta\in \Zn}  
\tr\big ( \Gc_u(\sigma_k)  E_{\al} \big)
S^{\LK}_{\al\beta} 
\left( {\cut}^{(\beta)}_{k} \circ {\cal L}^{(2)}_{u, \boldsymbol{\sigma}, \ba} \right) . \ee
When $1-u\le \ilambda^2/L^{d}$, its expectation can be bounded as
\begin{align}
	\E\mathcal{E}^{\Gc}_{u, \boldsymbol{\sigma}, \ba}&= {W}^d \sum_{k=1}^2 \sum_{\al, \beta}  
	\br{\E \tr\big( \Gc_u(\sigma_k) E_{\al} \big) S^{\LK}_{\al\beta} \left( {\cut}^{(\beta)}_{k} \circ {\cal K}^{(2)}_{u, \boldsymbol{\sigma}, \ba} \right) + \E \tr \big ( \Gc_u(\sigma_k)  E_{\al} \big) S^{\LK}_{\al\beta} \left( {\cut}^{(\beta)}_{k} \circ \p{\cal L-\cK}^{(2)}_{u, \boldsymbol{\sigma}, \ba} \right)} \nonumber\\
	&\prec W^d\cdot L^d(N|1-u|)^{-4} =(1-u)^{-1} (N|1-u|)^{-3}, \label{eq:ExpLWn=2_smalleta}
\end{align}
where the second step uses \eqref{res_ELK_n=1} and \eqref{eq:bcal_k} (with $n=3$) to bound the first term, and \eqref{Gt_avgbound_flow} together with \eqref{Eq:L-KGt-flow} (with $n=3$) to bound the second. On the other hand, when $1-u\ge \ilambda^2/L^{d}$, we control \smash{$\E\mathcal{E}^{\Gc}_{u, \boldsymbol{\sigma}, \ba}$} as in the following lemma, which may be viewed as an analogue of \Cref{lem:LWterm} in the sense of expectation. Its proof uses the same diagrammatic techniques as in the proof of \Cref{lem:LWterm}, albeit in a considerably simpler form, and is therefore deferred to \Cref{subsec:pf-LWterm_EXP}.

\begin{lemma}[Expected light-weight estimate]\label{lem:LWterm_EXP}
	In the setting of \Cref{lem:main_ind}, suppose $1-t\ge \ilambda^2/L^{d}$ and the estimates \eqref{Gt_bound_flow}--\eqref{Eq:Gdecay_flow} hold. Then, for any $\bsig\in\{+,-\}^2$ and $\ba\in(\Zn)^2$, we have  
	\begin{align}\label{eq:ExpLWn=2}
		\E\mathcal{E}^{\Gc,(2)}_{t, \boldsymbol{\sigma}, \ba}\prec (1-t)^{-1} (W^{-d}B_{t,0})^{\frac{5}{2}}. 
	\end{align}
\end{lemma}
With the above estimates in place, we can now complete the proof of \eqref{Eq:Gtlp_exp_flow} using the tools developed in Step 3 (cf.~\Cref{Sec:Steps34}).
Here, we can obtain an additional small factor $(\ilambda^2W^d)^{-1/5}+W^{-d}B_{u,0}$ compared to \eqref{Eq:L-KGt-flow}, because the leading error in the integrated loop hierarchy \eqref{int_K-LcalE_n=2} arises from the martingale term, which vanishes after taking the expectation.

\begin{proof}[\bf Step 6: Proof of \eqref{Eq:Gtlp_exp_flow}]
	In the case $1-s\ge 1-t\ge \ilambda^2$, we have $B_{u,0}\asymp |1-u|^{-1}$, while in the case $1-t\le 1-s\le \ilambda^2/L^d$, we have $B_{u,0}\asymp (L^d|1-u|)^{-1}$. In these regimes, substituting \eqref{Eq:Gtlp_exp+IND} and \eqref{eq:Exp(L-K)1}--\eqref{eq:ExpLWn=2} into \eqref{Eexpint_K-L}, applying the evolution kernel estimate \eqref{sum_res_Ndecay} with $n=2$, and integrating over $u$, we obtain  
	\be\label{eq:ExpL-Keasy} \E(\mathcal{L} - \mathcal{K})^{(2)}_{t, \boldsymbol{\sigma}, \ba} \prec (W^{-d}B_{t,0})^{2}\p{(\ilambda^2W^d)^{-1/5}+W^{-d}B_{t,0}}
	.\ee
	It remains to deal with the regimes (i) $\ilambda^2/L^{2}\le 1- t\le 1-s \le \ilambda^2$, and (ii) $\ilambda^2/L^{d}\le 1-t\le 1-s\le \ilambda^2/L^{2}$. In these regimes, when $\sig_1=\sig_2$, substituting \eqref{Eq:Gtlp_exp+IND}, \eqref{eq:Exp(L-K)1}, and \eqref{eq:ExpLWn=2} into \eqref{Eexpint_K-L}, applying the evolution kernel estimate \eqref{sum_res_2_NAL}, and integrating over $u$, we obtain \eqref{eq:ExpL-Keasy}. 
	The remaining and more delicate case is when $\sig_1\ne \sig_2$. To handle this, we again employ the sum-zero operator from \Cref{Def:QtPt} in regime (i), and the zero-mode-removing operator from \Cref{def;zero_mode_remove} in regime (ii).
	
	\medskip
	
	\noindent
	\textbf{Regime (i)}: Using Ward's identities \eqref{WI_calL} and \eqref{WI_calK}, along with \eqref{res_ELK_n=1} and \eqref{eq:derv_Theta}, we obtain 
	\be\label{eq:EPL-K}
	\br{\cal P\circ \mathbb E (\mathcal{L} - \mathcal{K})^{(2)}_{t, \boldsymbol{\sigma}}}_{a_1}\dthn_{t,\ba}^{(2)}= \frac{\im \E \tr\p{(G_t -M)E_{a_1}}}{W^d\eta_t}\dthn_{t,\ba}^{(2)} \prec (W^{-d}B_{t,0})^{-3}.
	\ee
	To estimate $\cal Q_t\circ \mathbb E (\mathcal{L} - \mathcal{K})^{(2)}_{t, \boldsymbol{\sigma}, \ba}$, we take the expectation of equation \eqref{int_K-L+Q} with $\fn=2$, yielding
	\begin{align}\label{int_K-L+QE}
		& {\cal Q}_t\circ  \E(\mathcal{L} - \mathcal{K})^{(2)}_{t, \boldsymbol{\sigma}, \ba}   = 
		\left(\mathcal{U}^{(2)}_{s, t, \boldsymbol{\sigma}} \circ  {\cal Q}_s\circ \E(\mathcal{L} - \mathcal{K})^{(2)}_{s, \boldsymbol{\sigma}}\right)_{\ba} + \int_{s}^t \left(\mathcal{U}^{(2)}_{u, t, \boldsymbol{\sigma}} \circ  {\cal Q}_u\circ \E \mathcal{E}^{(\cL-\cK)\times(\cL-\cK),(2)}_{u, \boldsymbol{\sigma}}\right)_{\ba} \dd u \nonumber\\
		& + \int_{s}^t \left(\mathcal{U}^{(2)}_{u, t, \boldsymbol{\sigma}} \circ  {\cal Q}_u\circ \E\mathcal{E}^{\Gc,(2)}_{u, \boldsymbol{\sigma}}\right)_{\ba} \dd u  + \int_{s}^t 
		\left(\mathcal{U}^{(2)}_{u, t, \boldsymbol{\sigma}} 
		\circ \left( \left[{\cal Q}_u , \varTheta^{(2)}_{u,\boldsymbol{\sigma}} \right]\circ \E(\mathcal{L} - \mathcal{K})^{(2)}_{u, \boldsymbol{\sigma} }\right) \right)_{\ba} \dd u \nonumber\\
		& - \int_{s}^t \left(\mathcal{U}^{(2)}_{u, t, \boldsymbol{\sigma}} \circ \left\{\left[ {\cal P} \circ \E\left(\mathcal{L} - \mathcal{K}\right)^{(2)}_{u, \boldsymbol{\sigma}}\right] \partial_u\dthn_{u}^{(2)} \right\}\right)_{\ba} \dd u.
	\end{align}
	We now estimate the RHS of \eqref{int_K-L+QE}. For the first three terms, we apply \Cref{lem_+Q}, together with the estimates \eqref{Eq:Gtlp_exp+IND}, \eqref{eq:Exp(L-K)1}, and \eqref{eq:ExpLWn=2}, and use the improved evolution kernel estimate \eqref{sum_res_2}, which exploits the sum-zero and fast-decay properties. Integrating over $u$, we can bound them by $(W^{-d}B_{t,0})^{-2}[(\ilambda^2W^d)^{-1/5}+W^{-d}B_{t,0}]$.

	It remains to control the last two terms in \eqref{int_K-L+QE}. Arguing as in \eqref{eq:EPL-K}, and again using Ward's identity, \eqref{res_ELK_n=1}, and \eqref{eq:derv_Theta}, we find
	\be\label{eq:boundELKQ1}
	\left\|\left[ {\cal P} \circ \E\left(\mathcal{L} - \mathcal{K}\right)^{(2)}_{u, \boldsymbol{\sigma}}\right] \partial_u\dthn_{u}^{(2)}\right\|_\infty \prec (1-u)^{-1} (W^{-d}B_{u,0})^{-3}. 
	\ee 
	For the commutator term, using the same argument as in \eqref{A4} (with $n=2$), we obtain
	\begin{align}
		\big[{\cal Q}_u , \varTheta^{(2)}_{u,\boldsymbol{\sigma}} \big]\circ \E(\mathcal{L} - \mathcal{K})^{(2)}_{u, \boldsymbol{\sigma}, \ba} 
		&\prec (1-u)^{-1}\left\|\dthn^{(2)}_{u}\right\|_\infty \cdot   \left\| {\cal P} \circ \E\left(  \mathcal{L} - \mathcal{K} \right)^{(2)}_{u,\boldsymbol{\sigma}} \right\|_{\infty}\prec (1-u)^{-1} (W^{-d}B_{u,0})^{-3},\label{eq:boundcommutator}
	\end{align} 
	where the second step again follows from Ward's identity and the estimates \eqref{res_ELK_n=1} and \eqref{eq:derv_Theta}. 
	Applying \Cref{lem_+Q}, along with the bounds \eqref{eq:boundELKQ1} and \eqref{eq:boundcommutator}, and using the evolution kernel estimate \eqref{sum_res_2}, we conclude that the contribution of the last two terms in \eqref{int_K-L+QE} is bounded by $(W^{-d}B_{t,0})^3$. This completes the proof of \eqref{Eq:Gtlp_exp_flow} in case (i).
	
	\medskip
	
	\noindent
	\textbf{Regime (ii)}: We use the operator $Q^{(\{1,2\})}=Q^{(1)}\circ Q^{(2)}$ in \Cref{def;zero_mode_remove} and Ward's identity to express 
	\begin{align}
		(\cal L-\cK)_{t,\bsig,\ba}^{(2)}=  Q^{(\{1,2\})} \circ (\cal L-\cK)_{t,\bsig,\ba}^{(2)}  + \frac{\im \br{\tr\p{(G_t-M)(E_{a_1}+E_{a_2})} -N^{-1}\tr\p{G-M}}}{N\eta_t}
	\end{align}
	for $\sig_1\ne \sig_2$. By \eqref{res_ELK_n=1}, the second term is bounded by $(W^{-d}B_{t,0})^2\cdot (N\eta_t)^{-1}\le (W^{-d}B_{t,0})^3$. To estimate the first term, we take the expectation of equation \eqref{iisuwjyys} with $\fn=2$ and $A=\{1,2\}$, yielding
	\begin{align}\label{iisuwjyys_exp}
		Q^{(A )}\circ \E(\mathcal{L} - \mathcal{K})^{(n)}_{t, \boldsymbol{\sigma}, \ba} 
		&=\left(Q^{(\{1,2\})}\circ\mathcal{U}^{(2)}_{s, t, \boldsymbol{\sigma}}  \circ \E(\mathcal{L} - \mathcal{K})^{(2)}_{s, \boldsymbol{\sigma}}\right)_{\ba} + \int_{s}^t \left(Q^{(\{1,2\})}\circ \mathcal{U}^{(2)}_{u, t, \boldsymbol{\sigma}}  \circ \E \mathcal{E}^{(\cL-\cK)\times(\cL-\cK),(2)}_{u, \boldsymbol{\sigma}}\right)_{\ba} \dd u \nonumber\\
		& + \int_{s}^t \left(Q^{(\{1,2\})}\circ \mathcal{U}^{(2)}_{u, t, \boldsymbol{\sigma}} \circ \E\mathcal{E}^{\Gc,(2)}_{u, \boldsymbol{\sigma}}\right)_{\ba} \dd u .
	\end{align}
	Applying \eqref{normQA2}, together with the estimates \eqref{Eq:Gtlp_exp+IND}, \eqref{eq:Exp(L-K)1}, and \eqref{eq:ExpLWn=2}, and using the evolution kernel estimate \eqref{sum_res_Ndecay_nonzero}, we integrate over $u$ to obtain the desired bound: $(W^{-d}B_{t,0})^{-2}[(\ilambda^2W^d)^{-1/5}+W^{-d}B_{t,0}]$. This completes the proof of \eqref{Eq:Gtlp_exp_flow} in case (ii). 
\end{proof}

\section{Estimation of the light-weight term}\label{Sec:graph}

This section is devoted to proving the key estimates for the light-weight term $\cal E^{\Gc,(2)}$---namely, \Cref{lem:LWterm,lem: EWGn2_N} in Step 2. 
The proofs build on several crucial ideas and refined graphical tools developed in earlier works \cite{yang2021delocalization,yang2022delocalization,yang2021random}. 
We present the full details in the setting of the random band matrix model, while the modifications required for the block Anderson model are described separately in \Cref{subsec:LWchange-to-BA} below. 
Recall the light-weight term takes the form \eqref{eq:LW-n=2}. For $\bsig=(\sig',\sig)$ and $\ba=(a,b)$, we have 
\begin{align}
	\mathcal{E}^{\Gc}_{t, \boldsymbol{\sigma}, \ba}&= {W}^d \sum_{a_1, a_2}   S^{\LK}_{a_1a_2} 
	\tr \big(\Gc_t(\sig) E_{a_1}\big)\tr\big( G_t(\sig) E_{a_2} G_t(\sig) E_{b} G_t(\sig') E_{a}\big) +\p{(\sig,a)\leftrightarrow (\sig',b)} \label{eq:EGC}\\
	&= W^{-2d} \sum_{x\in [a],y\in [b]} (G_t(\sig'))_{yx} \sum_{a_1, a_2}S^{\LK}_{a_1a_2}\sum_{\al \in [a_2]}     
	\tr\big( \Gc_t(\sig) E_{a_1} \big) (G_t(\sig))_{x\al }(G_t(\sig))_{\al y} +\p{(\sig,a)\leftrightarrow (\sig',b)},\nonumber
\end{align} 
where $\p{(\sig,a)\leftrightarrow (\sig',b)}$ denotes the term obtained by exchanging $(\sig,a)$ and $(\sig',b)$ in the preceding term. 
In the setting of \Cref{lem:LWterm}, using \eqref{GiiGEX} and \eqref{GijGEX}, we can bound $(G_t(\sig'))_{yx}$ as 
\be\label{eq:directG1}
|(G_t(\sig'))_{yx}|\prec  \delta_{xy}+\Psi_t(|a-b|).
\ee 
Note that when $1-t\le \ilambda^2/L^{2}$, we have $\ell_t=L$, so the exponential factor in the definition of $\cal T_t$ in \eqref{defTUL} is always of order 1 for $r\lesssim L$. In this case, \Cref{lem: EWGn2_N} is an immediate consequence of \Cref{lem:LWterm}. Hence, for the proof of \Cref{lem: EWGn2_N}, we only need to focus on the $1-t> \ilambda^2/L^{2}$ case, where  
\[ W^{-d}\cal T_{t,D}^{\ell}(|a-b|)\asymp \br{\sT_{t}(|a-b|\wedge \ell)}^2 + W^{-D}.\]
Here, for simplicity of notation, we introduce the function $\sT_{t}:[0,\infty)\to [0,\infty)$ defined as: 
\[\sT_{t}(r):=\frac{(\ilambda^2+|1-t|)^{-1/2}}{W^{ d /2}(r+1)^{(d-2)/2}}\exp\p{-\frac{1}{2}\sqrt{{r}/{\ell_t}}},\quad \forall r\ge 0. \]
Then, in the setting of \Cref{lem: EWGn2_N} and for $1-t> \ilambda^2/L^{2}$, we have the following estimate by \eqref{GiiGEX} and \eqref{GijGEX}:
\be\label{eq:directG2}
|(G_t(\sig'))_{yx}|\prec \delta_{xy}+\sT_{t}(|a-b|\wedge \ell) + W^{-D} .
\ee 
With \eqref{eq:directG1} and \eqref{eq:directG2} in place, the proofs of \Cref{lem:LWterm,lem: EWGn2_N} reduce to estimating
\begin{align}\label{fxyG_sum}
	f_{xy}(G)&=W^{-d}\sum_{a_1, a_2}S^{\LK}_{a_1a_2}\sum_{{\al\in [a_2],\beta \in [a_1]}}\mathbf 1_{\al\notin\{x,y\}}
	(\Gc_t)_{\beta\beta} (G_t)_{x\al }(G_t)_{\al y} \equiv \sum_{\al\notin\{x,y\}}\sum_{\beta}S_{\al\beta}
	\Gc_{\beta\beta} G_{x\al}G_{\al y},
\end{align} 
where, without loss of generality, we take $\sig = +$ and abbreviate $G \equiv G_t$. 
We denote by \smash{$\wt f_{xy}(G)$} the corresponding contribution with $\al\in\{x,y\}$. 
Using \eqref{eq:directG1}, \eqref{eq:directG2}, and the averaged local law \eqref{GavLGEX}, we get 
\be\label{eq:recolterm} 
\wt f_{xy}(G) \prec \Psi^2_t(0)\cdot \p{\delta_{xy}+\Psi_t(|a-b|)}
\ee
in the setting of \Cref{lem:LWterm}, and 
\be\label{eq:recolterm2} 
\wt f_{xy}(G) \prec (W^{-d}B_{t,0}) \cdot \p{\delta_{xy} + \sT_{t}(|a-b|\wedge \ell)+ W^{-D}}
\ee
in the setting of \Cref{lem: EWGn2_N}. 
Combining \eqref{eq:recolterm} with \eqref{eq:directG1}, and averaging over $x\in[a]$ and $y\in [b]$, yields  
\begin{align}\label{eq:recoltermwt}
	W^{-2d} \sum_{x\in [a],y\in [b]} (G_t(\sig'))_{yx} \wt f_{xy}(G) \prec \Psi^2_t(0)\left[W^{-d}\delta_{ab}+\Psi_t^2(|a-b|)\right] \lesssim \Psi^2_t(0)\cdot \Psi_t^2(|a-b|), 
\end{align} 
which is bounded by the RHS of \eqref{eq:LW_conclusion}. Similarly, combining \eqref{eq:recolterm2} with \eqref{eq:directG2} gives
\begin{align}\label{eq:recoltermwt2}
	W^{-2d} \sum_{x\in [a],y\in [b]} (G_t(\sig'))_{yx} \wt f_{xy}(G) \prec (W^{-d}B_{t,0}) \cdot W^{-d} \wT_{t,D}^{\ell}(|a-b|), 
\end{align}
which is bounded by the RHS of \eqref{eq:LW_conclusion_exp}. Therefore, it remains to estimate the main term $f_{xy}(G)$.

We first claim the following $\max$-bound on $f_{xy}(G)$:
\begin{align}\label{eq:boundfxyGinf}
	|f_{xy}(G)|\prec \Psi_t^2 \sum_{\al}|G_{x\al}||G_{\al y}| \le \Psi_t^2 \p{{\im G_{xx}}/{\eta_t}}^{1/2}\p{{\im G_{yy}}/{\eta_t}}^{1/2}\prec \eta_t^{-1}\Psi_t^2,
\end{align}
where the first step uses the averaged local law \eqref{GavLGEX}, and the second applies Cauchy–Schwarz together with Ward’s identity. In the setting of \Cref{lem:LWterm}, we take $\Psi_t = \Psi_t(0)$, whereas in the setting of \Cref{lem: EWGn2_N}, we take $\Psi_t = (W^{-d}B_{t,0})^{1/2}$. 
Taking \Cref{lem:LWterm} as an example, our goal is to establish the sharper bound $|f_{xy}(G)|\prec \eta_t^{-1}\Psi_t(0)\Psi_t(|a-b|)$.
Thus, the estimate \eqref{eq:boundfxyGinf} lacks the long-edge factor $\Psi_t(|a-b|)$, which captures the spatial decay in $|a-b|$. This missing factor is a key technical obstacle—previously noted around \eqref{eq:introGc} in the introduction—because all tools developed within the stochastic flow and loop hierarchy frameworks fail to extract it. 
We have to exploit a \emph{global fluctuation averaging} mechanism in the sum over $\al$ in \eqref{fxyG_sum}, which can yield additional powers of $L^{-1}$ or \smash{$\ell_t^{-1}$} to compensate for the missing long edge.
To realize this mechanism, we estimate the high moments of $f_{xy}(G)$, adapting ideas from \cite{yang2021delocalization,yang2022delocalization,yang2021random}. 
By Markov's inequality, \Cref{lem:LWterm,lem: EWGn2_N} follow from the next estimates on the high moments of $f_{xy}(G)$.
\begin{lemma}\label{lem:LW_moment}
	In the setting of \Cref{lem:LWterm}, for any fixed $p\in 2\N$, there exists a constant $c>0$ depending on $p$ such that the following estimate holds:
	\be\label{eq:LW_moment}
	\E\big|f_{xy}(G)\big|^{p}\prec  \eta_t^{-p}[\Psi_t(0)]^p \cdot \br{\Psi_t\p{c|a-b|}}^p . 
	\ee
\end{lemma}
\begin{lemma}\label{lem:LW_moment_exp}
	In the setting of \Cref{lem: EWGn2_N}, assume that $1-t>\ilambda^2/L^{2}$. Then, for each fixed $p\in 2\N$, the following estimate holds for any large constant $D>0$:
	\be\label{eq:LW_moment_exp}
	\E\big|f_{xy}(G)\big|^{p}\prec \eta_t^{-p} (W^{-d}B_{t,0})^{p/2} \cdot \br{\sT_{t}(|a-b|\wedge \ell) }^{p}+W^{-D}. 
	\ee
\end{lemma}
\begin{proof}[\bf Proof of \Cref{lem:LWterm,lem: EWGn2_N}]
	Applying Markov’s inequality to \eqref{eq:LW_moment} (with arbitrarily large $p$) yields 
	\[ f_{xy}(G)\prec \eta_t^{-1} \Psi_t(0) \cdot \Psi_t\p{ |a-b|/\log W} \prec \eta_t^{-1} \Psi_t(0) \cdot \Psi_t\p{ |a-b|},\]
	where the second step uses the condition \eqref{eq:Psi}. Combining this bound with \eqref{eq:directG1} and \eqref{eq:recoltermwt}, we obtain \eqref{eq:LW_conclusion}.  
	For the proof of \Cref{lem: EWGn2_N}, as discussed below \eqref{eq:directG1}, it suffices to assume that $1-t>\ilambda^2/L^{2}$. In this case, applying Markov’s inequality to \eqref{eq:LW_moment_exp} yields, for any large constant $D>0$,
	\[ f_{xy}(G)\prec \eta_t^{-1} (W^{-d}B_{t,0})^{1/2} \cdot \br{\sT_{t}(|a-b|\wedge \ell) } + W^{-D}.\]
	Together with \eqref{eq:directG2} and \eqref{eq:recoltermwt2}, it gives \eqref{eq:LW_conclusion_exp}.
\end{proof}

The remainder of this section is devoted to proving \Cref{lem:LW_moment,lem:LW_moment_exp}. We begin with the proof of \Cref{lem:LW_moment}; the proof of \Cref{lem:LW_moment_exp} will follow the same general strategy, with several additional technical ingredients. Note that when $|a-b| \le (\log W)^{3/2}$, both \eqref{eq:LW_moment} and \eqref{eq:LW_moment_exp} follow directly from \eqref{eq:boundfxyGinf}. Thus, in the following proof, we restrict attention to the case
\be\label{eq:far_ab}
|a-b| > (\log W)^{3/2}.
\ee
Following the approach of \cite{yang2021delocalization,yang2022delocalization,yang2021random}, we represent resolvent expressions (such as $\big|f_{xy}(G)\big|^{p}$) as graphs, expand them using Gaussian integration by parts, and then bound the resulting diagrams according to their graphical structures.
In the next subsection, we introduce the basic graphical notations, including the concepts of vertices, edges, molecules, graph values, and scaling sizes, and define the basic graph expansions that will be used in the proof.

\subsection{Graphical tools and local expansions}\label{sec:graphs}

Our graphs are composed of matrix indices as vertices, and various types of edges representing different matrix entries. Specifically, the graphs will encode entries from the following matrices: $I, \ S, \ G, \ \Gc=G-M$, and $S^{\pm}$, where $S^{\pm}$ are defined by 
\be\label{eq:def-Spm}
S^+_{xy}(z):=W^{-d}\big(S^{\LK}\Theta^{(+,+)}(z)\big)_{ab} \quad \text{for}\quad x\in[a], y\in [b],\quad \text{and}\quad S^-(z) :=(S^+(z))^*, 
\ee
with the notations $\Theta^{(+,+)}(z)$ and $S^{\LK}$ specified earlier in \eqref{def:Theta} and \eqref{eq:variancematrix}. 
The graphical notation is defined in a general setting for an arbitrary random band matrix $H$ and its associated variance matrix $S$. As a special case, these definitions apply to our matrix flow $H_t$, where the variance matrix is given by $S_t=tS$, and the associated resolvent is denoted by $G_t\equiv (H_t-z_t)^{-1}$.  

\begin{definition}[Graphs] \label{def_graph1} 
	Given a graph with vertices and edges, we assign the following structures and call it a vertex-level graph.  
	\begin{itemize}
		\item {\bf Vertices:} Vertices represent matrix indices in our expressions. Every graph has some external or internal vertices: external vertices represent indices whose values are fixed, while internal vertices represent summation indices that will be summed over. 
		
		\item{\bf  Solid edges and weights:} We will use $(x,y)$ to denote a solid edge from vertex $x$ to vertex $y$. Every solid edge represents a resolvent entry. More precisely: 
		\begin{itemize}
			\item A blue (resp.~red) oriented solid edge from $x$ to $y$ represents a $G_{xy}$ (resp.~$\bar G_{xy}$) factor. 
			
			\item A blue (resp.~red) oriented solid edge with a circle ($\circ$) from $x$ to $y$ represents a $(G-M)_{xy}$ (resp.~$\overline {(G-M)}_{xy} $) factor.
			
			\item A $G_{xx}$ (resp.~$\bar G_{xx}$) factor will be represented by a blue (resp.~red) self-loop on the vertex $x$, while a $(G-M)_{xx}$ (resp.~$\overline {(G-M)}_{xx}$) factor will be represented by a blue (resp.~red) self-loop on the vertex $x$ with a circle ($\circ$). Following the convention in \cite{yang2021delocalization}, we will also call $G_{xx}$ and $\bar G_{xx}$ as blue and red {\bf weights}, and call $(G-M)_{xx}$ and $\overline {(G-M)}_{xx}$ as blue and red {\bf light-weights}.
			
		\end{itemize}
		We assign a $+$ charge to each blue solid edge and a $-$ charge to each red solid edge.		
		\item {\bf Waved edges:}
		\begin{itemize}
			\item A black waved edge between $x$ and $y$ represents an $S_{xy}$ factor. 
			
			\item A blue (resp.~red) waved edge between $x$ and $y$ represents an $S^+_{xy}$ (resp.~$S^-_{xy}$) factor.
		\end{itemize}
		
		\item {\bf Dotted edges:} A black dotted edge between $x$ and $y$ represents a factor $\mathbf 1_{x=y}$ and a $\times$-dotted edge represents a factor $\mathbf 1_{x\ne y}$. There is at most one dotted or $\times$-dotted edge between every pair of vertices.

		\item{\bf Coefficient:} There is a coefficient associated with every graph. Every coefficient is of order $\OO(1)$ and is a polynomial of $m$, $\overline m$, $m^{-1}$, $\overline m^{-1}$, $(1-m^2)^{-1}$, and $(1-\overline m^2)^{-1}$. 
		
	\end{itemize}	
\end{definition}

Edges between internal vertices are called \emph{internal edges}, while edges with at least one end at an external vertex are called \emph{external edges}. The orientations of non-solid edges do not matter because they all represent entries of symmetric matrices. 
The dotted edges are introduced solely for organizing the proof in certain steps; otherwise, we will almost always identify vertices connected by dotted edges. 
To each graph, we assign a \emph{value} as follows. For simplicity, throughout this paper, we will always abuse the notation by identifying a graph (a geometric object) with its value (an analytic expression). 

\begin{definition}[Values of graphs]\label{ValG} 
	Given a graph $\mathcal G$, we define its value as follows. We first take the product of all edge factors together with the coefficient associated with the graph $\cal G$. We then sum over all internal indices corresponding to internal vertices, while keeping the external indices fixed at their prescribed values.
	For a linear combination of graphs $\sum_i c_i \cal G_i$, we define its value in the natural way, as the linear combination of the values of the individual graphs $\cal G_i$.
\end{definition}

As noted in the previous works \cite{yang2021delocalization,yang2022delocalization,Xu:2024aa,yang2021random}, our graphs have a two-level structure: some \emph{local structures} varying on scales of order $W$, called {\bf molecules}, which are the equivalence classes of vertices connected through dotted or waved edges, along with a \emph{global structure of blocks} varying on scales up to $L$. 
Given a graph defined in \Cref{def_graph1}, if we ignore the inner structure within each molecule, we will get a {\bf molecular graph} with vertices being molecules. The molecular graph will be very useful in organizing the graphs and understanding their global structures. We now give its formal definition.

\begin{definition}[Molecules and molecular graphs]\label{def_poly}
	We partition the set of all vertices into a union of disjoint subsets called molecules. Two vertices belong to the same molecule if and only if they are connected by a path of dotted and waved edges. Every molecule containing at least one external vertex is called an external molecule; otherwise, it is an internal molecule. An edge is said to be inside a molecule if both of its ending vertices belong to this molecule. 
	Given a graph $\cal G$, we define its molecular graph, denoted by $ \cal G_{\cal M}$, as follows:
	\begin{itemize}
		\item merge all vertices in the same molecule and represent them by a vertex;
		
		\item keep all solid edges between molecules;
		
		\item discard all the other components in $\cal G$ (i.e., $\times$-dotted edges, edges inside molecules, and coefficients).
	\end{itemize}
\end{definition}

By the exponential decay of the waved edges---derived from the definitions of $S$ and the estimate \eqref{prop:ThfadC_short}---we see that up to an error of order $\exp(-c(\log W)^{3/2})$ for a constant $c>0$:  
\be\label{scalemole}
\hbox{$x$, $y$ are in the same molecule} \implies  |x-y|\le  W(\log W)^{3/2}.
\ee
We remark that molecular graphs are used solely to help with the analysis of graph structures, while all graph expansions will be applied exclusively to vertex-level graphs. 

\begin{definition}[Normal graphs]  \label{defnlvl0}  
	We say a graph $\cal G$ is \emph{normal} if it satisfies the following properties:
	\begin{itemize}
		\item[(i)] It contains at most $\OO(1)$ many vertices and edges.  
		
		\item[(ii)] There are no dotted edges between vertices.
		
		\item[(iii)] Every pair of vertices in the graph are connected by a $\times$-dotted edge \emph{if and only if} they are connected by a solid edge.
		
		\item[(vi)] Every weight is a light-weight. 
	\end{itemize}
\end{definition}

In a normal graph, every $G$ edge is off-diagonal, while all the diagonal $G$ factors are represented by weights. Given any graph with $\OO(1)$ vertices and edges, we can rewrite it as a linear combination of normal graphs via the following \emph{dotted edge partition} operation. 

\begin{definition}[Dotted edge partition] \label{dot-def}
	Given a graph $\cal G$, if there is at least one $G$-edge but no $\times$-dotted edge between a pair of vertices $\al$ and $\beta$, we write
	$1=\mathbf 1_{\al=\beta} + \mathbf 1_{\al \ne \beta} .$
	If there is a $\times$-dotted edge $\mathbf{1}_{\al \ne \beta}$ but no $G$-edge between $\al$ and $\beta$, we write 
	$\mathbf 1_{\al\ne \beta} =1 - \mathbf 1_{\al = \beta}.$
	Expanding the product of all such identities, we can express $\cal G$ as
	\be\label{odot}\cal G := \sum {\Dot} \cdot \cal G,\ee
	where each ${\Dot}$ is a product of dotted and $\times$-dotted edges together with a sign $\pm$. If $\Dot$ is ``inconsistent"---that is, if two vertices are connected both by a $\times$-dotted edge and by a path of dotted edges—then $\Dot \cdot \cal G = 0$.  
	For each consistent ${\Dot}$, we merge vertices connected by dotted edges in ${\Dot} \cdot \cal G$. In particular, if there is a $\times$-dotted edge between $\al$ and $\beta$, then all $G$ edges between them are off-diagonal; otherwise, the $G$-edges between them become weights once $\al$ and $\beta$ are merged. Finally, in each resulting graph, every diagonal factor $G_{xx}$ (resp.~$\overline{G}_{xx}$) is decomposed into a light-weight $(G_{xx}-m)$ plus a coefficient $m$ (resp. $(\overline{G}_{xx}-\overline m)$ plus $\overline m$). Taking the product over all such decompositions yields a linear combination of normal graphs.
	
\end{definition}

Given a normal graph, we can define its scaling size as in \Cref{def scaling}.

\begin{definition}[Scaling size]\label{def scaling} 
	We define the scaling size of a normal graph $\Gamma$ as follows:
	\begin{align}
		\size(\Gamma): =&~ (L^d)^{n_M(\Gamma)}\cdot (\Psi_t)^{n_S(\Gamma)} \cdot W^{-d\left(n_W(\Gamma)  - n_V(\Gamma)\right)}.
		\label{eq_defsize}
	\end{align}
	where $n_S(\Gamma)$, $n_W(\Gamma)$, $n_V(\Gamma)$, and $n_M(\Gamma)$ denote the numbers of solid edges (including light-weights), waved edges, internal vertices, and internal molecules, respectively.
	If a graph $\Gamma$ can be written as a sum of $\OO(1)$ normal graphs $\Gamma_k$, i.e., $\Gamma=\sum_k \Gamma_k$, then we define 
	\be\label{eq_defsizemax}\size(\Gamma):=\max_k \size(\Gamma_k).\ee
	As a convention, for any (possibly non-normal) graph $\Gamma$ with $\OO(1)$ vertices and edges, we define $\size(\Gamma)$ by first expanding $\Gamma$ into a sum of normal graphs via the dotted edge partition, and then applying \eqref{eq_defsizemax}.
\end{definition}

First, by \eqref{GiiGEX}, each off-diagonal solid edge or light-weight contributes a factor of $\OO_\prec(\Psi_t)$ under the assumption \eqref{initialGT2}, which accounts for the \smash{$(\Psi_t)^{n_S(\Gamma)}$} factor in \eqref{eq_defsize}. 
Next, every molecule contains a “free vertex” that can range over the entire lattice \smash{$\ZL$}, giving rise to an \smash{$N^{n_M(\Gamma)}$} factor. The remaining vertices in the molecule are confined (up to a sufficiently small error, by \eqref{scalemole}) to a \smash{$\OO(W(\log W)^{3/2})$}-neighborhood of the free vertex, yielding the factor \smash{$W^{d(n_V(\Gamma)-n_M(\Gamma))}$}.
Finally, each waved edge contributes a $W^{-d}$ factor, by the definition of $S$ and the bound for $S^\pm$: there exist constants $c,C>0$ such that 
\be\label{eq:estSpm-W}S^\pm_{xy}\le CW^{-d}e^{-c|x-y|/W},\ee
which follows from the estimate \eqref{prop:ThfadC_short}. 
With these considerations, we obtain the following claim.

\begin{claim}
	We have the bound \(\Gamma\prec \size(\Gamma)\) for any normal graph $\Gamma$ in the setting of \Cref{lem:LWterm} (with $\Psi_t=\Psi_t(0)$) or \Cref{lem: EWGn2_N} (with $\Psi_t=(W^{-d}B_{t,0})^{1/2}$).  
\end{claim}

In the proofs, we will only need to keep track of the number of factors of $\Psi_t$ and $W^{-d/2}$ appearing in the scaling size. To this end, we introduce a more convenient notion of \emph{scaling order}. 

\begin{definition}[Scaling order]\label{def scaling order} 
	We define the \emph{scaling order} of a normal graph $\Gamma$ as
	\be\label{eq:ordG} \ord(\Gamma):=n_S(\Gamma)+ 2 \left(n_W(\Gamma)  -  n_V(\Gamma)\right) . \ee
Since $\Psi_t\ge W^{-d/2}$, we have the trivial identity \(\size(\Gamma)\le (L^d)^{n_M(\Gamma)}(\Psi_t)^{\ord(\Gamma)}.\)
	For a general graph $\Gamma$ that can be written as a sum of $\OO(1)$ many normal graphs $\Gamma_k$, we define its scaling order as 
	\(\ord(\Gamma)=\min_k \ord(\Gamma_k).\) 
\end{definition}


We next state the graph expansion rules given in \cite{yang2021delocalization}. These expansions are stated for $G(z)=(H-z)^{-1}$, but they also apply to $G_t(z)$ with $G$ and $S$ replaced by $G_t$ and $S_t=tS$, respectively.

\begin{lemma}[Weight expansion, Lemma 3.5 of \cite{yang2021delocalization}] \label{ssl} 
	Suppose $f$ is a differentiable function of $G$. Then,
	\begin{align} 
		\Gc_{xx}  f(G)=_{\E}  m \sum_{  \al} S_{x\al}  \Gc_{xx}\Gc_{\al\al} f (G)  &+m^3 \sum_{  \al,\beta}S^{+}_{x\al} S_{\al \beta}  \Gc_{\al\al}\Gc_{\beta\beta} f (G) \nonumber\\
		-  m  \sum_{ \al} S_{x \al} G_{\al x}\partial_{ h_{ \al x}} f (G) &-  m^3 \sum_{ \al,\beta} S^{+}_{x\al}S_{\al \beta} G_{\beta \al}\partial_{ h_{ \beta\al}} f(G)  \, ,\label{Owx}\end{align}
	where ``$=_{\E}$" means ``equal in expectation". 
\end{lemma}

\begin{lemma}[Edge expansion, Lemma 3.10 of \cite{yang2021delocalization}] \label{Oe14}
	Suppose $f$ is a differentiable function of $G$. Consider 
	\be\label{multi setting}
	\cal G := \prod_{i=1}^{k_1}G_{x y_i}  \cdot  \prod_{i=1}^{k_2}\overline G_{x y'_i} \cdot \prod_{i=1}^{k_3} G_{ w_i x} \cdot \prod_{i=1}^{k_4}\overline G_{ w'_i x} \cdot f(G).
	\ee
	If $k_1\ge 1$, then we have that 
	\begin{align} 
		\cal G & =_{\E} m\delta_{xy_1} {\cal G}/{G_{xy_1}} + m   \sum_\al S_{x\al }\Gc_{\al \al} \cal G \label{Oe1x}\\
		&+\sum_{i=1}^{k_2} |m|^2  \left( \sum_\al S_{x\al }G_{\al y_1} \overline G_{\al y'_i}\right)\frac{\cal G}{G_{x y_1} \overline G_{xy_i'}}+ \sum_{i=1}^{k_3} m^2 \left(\sum_\al S_{x\al }G_{\al y_1} G_{w_i \al} \right)\frac{\cal G}{G_{xy_1}G_{w_i x}}      \nonumber \\
		& + \sum_{i=1}^{k_2}m  \Gc^-_{xx} \left( \sum_\al S_{x\al }G_{\al y_1} \overline G_{\al y'_i}\right)\frac{\cal G}{G_{x y_1} \overline G_{xy_i'}}  + \sum_{i=1}^{k_3} m  \Gc_{xx}  \left(\sum_\al S_{x\al }G_{\al y_1} G_{w_i \al} \right)\frac{\cal G}{G_{xy_1}G_{w_i x}}   \nonumber\\
		&+(k_1-1) m  \sum_\al S_{x\al } G_{x \al} G_{\al y_1}\frac{ \cal G}{G_{xy_1}}   + k_4 m   \sum_\al S_{x\al }\overline G_{\al x} G_{\al y_1}  \frac{\cal G}{G_{xy_1}} - m    \sum_\al S_{x\al } \frac{\cal G}{G_{x y_1}f(G)}G_{\al y_1}\partial_{ h_{\al x}}f (G)   \, .\nonumber
	\end{align}
	Here, the fractions are used to simplify the expression. For example, the fraction ${\cal G}/({G_{x y_1} \overline G_{xy_i'}})$ is the  graph obtained by removing the factor \smash{${G_{x y_1} \overline G_{xy_i'}}$} from the product in \eqref{multi setting}.
	We refer to the above expansion as the \emph{edge expansion with respect to $G_{xy_1}$}. The edge expansion with respect to other $G_{xy_i}$ can be defined in the same way. The edge expansions with respect to $\overline G_{x y'_i}$, $ G_{ w_i x}$, and $ \overline G_{ w'_i x}$ can be defined similarly by taking complex conjugates or matrix transpositions of \eqref{Oe1x}.   
\end{lemma}

\begin{lemma} [$GG$ expansion, Lemma 3.14 of \cite{yang2021delocalization}]\label{T eq0}
	Consider a graph $\cal G= G_{xy}   G_{y' x }  f (G)$, where $f$ is a differentiable function of $G$. We have that 
	\begin{align}
		\cal G& =_{\E} m\delta_{xy}G_{y' x}f(G)+  m^3  S^{+}_{xy} G_{y' y} f(G)  + m \sum_\al  S_{x\al} \Gc_{\al \al} \cal G + m^3 \sum_{\al,\beta}  S^{+}_{x\al}  S_{\al\beta} \Gc_{\beta \beta} G_{\al y}   G_{y'\al} f(G) \nonumber \\
		&  + m\Gc_{xx }    \sum_\al S_{x\al}G_{\al y}   G_{y'\al} f(G) + m^3 \sum_{\al,\beta}  S^+_{x\al} S_{\al\beta}  \Gc_{\al\al }  G_{\beta y}   G_{y'\beta} f(G)  \nonumber\\
		& - m \sum_{ \al}  S_{x\al}G_{\al y} G_{y' x} \partial_{ h_{\al x}}f(G) - m^3 \sum_{\al,\beta} S^{+}_{x\al} S_{\al\beta}G_{\beta y} G_{y' \al} \partial_{ h_{\beta \al}}f(G)  . \label{Oe2x}
	\end{align}
	The $\overline G \overline G$ expansion with respect to $\overline G_{xy}  \overline G_{y' x }$ can be derived by taking complex conjugate.
\end{lemma} 

Corresponding to the three lemmas above, we can define graph operations that represent the weight, edge, and $GG$ expansions. Following \cite{yang2021delocalization}, we refer to all these operations as \emph{local expansions at a vertex $x$}, meaning that they do not create new molecules: every new vertex introduced by such an expansion is connected to $x$ by a path of dotted or waved edges. 
As discussed in \cite[Section 3]{yang2021delocalization}, these expansions decrease the scaling size (equivalently, increase the scaling order) of a graph. Moreover, relative to the original graph $\cal G$ prior to expansion, each new graph produced by an expansion either becomes ``smaller" by a factor of $\Psi_t$ in scaling size, or moves ``closer" to being locally standard (as defined in \Cref{deflvl1}), in one of the following ways:
\begin{itemize}
	
	\item It contains one fewer weight than $\cal G$.
	
	\item The solid-edge degree of a vertex is reduced by 2, a new vertex of degree 2 is introduced, and the degrees of all other vertices remain unchanged. Here, the degree refers only to the number of solid edges incident to the vertex.
	
	\item A pair of edges of the form $G_{xy} G_{y x }$ (resp.~$\overline G_{xy} \overline G_{yx}$) is replaced by $m^4 S^{+}_{xy}$ (resp.~$\overline m^4 S^{-}_{xy}$).
	
\end{itemize}
Consequently, by repeatedly applying the local expansions from Lemmas \ref{ssl}--\ref{T eq0}, any normal graph can be expanded into a sum of \emph{locally standard graphs}, defined as follows.

\begin{definition} [Locally standard graphs] \label{deflvl1}
	A graph is \emph{locally standard} if: 
	\begin{itemize}
		\item[(i)] It is a normal graph. 
		
		\item[(ii)] It has no self-loops (i.e., weights or light-weights) on vertices.
		
		\item[(iii)] Each internal vertex is either standard neutral, or it is not incident to any solid ($G$ or $\Gc$) edge.
	\end{itemize}
	Here, a vertex is called \emph{standard neutral} if it satisfies the following two properties: 
	\begin{itemize}
		\item It is connected to exactly two solid edges carrying opposite charges. 
		
		\item It has a neutral charge, where the \emph{charge} of a vertex is defined by counting the incoming and outgoing blue solid edges (with $+$ charge) and red solid edges (with $-$ charge): 
		\be\label{eq:neutralcharge}
		\qquad \#\{\text{incoming $+$ or outgoing $-$ solid edges}\}- \#\{\text{outgoing $+$ or incoming $-$ solid edges}\} .\ee
	\end{itemize}
	By definition, the two solid edges attached to a standard neutral vertex $x$ take the form $G_{xy}\overline G_{xy'}$ or $G_{yx}\overline G_{y'x}$.
\end{definition}

Simply speaking, \emph{locally standard graphs} are those in which each $G$ edge is paired with a unique $\overline G$ edge at internal vertices. As discussed in \cite[Section 3.4]{yang2021delocalization}, by repeatedly applying local expansions, any normal graph can be expanded into a sum of locally standard graphs, up to a sufficiently small error, as summarized in the following lemma. For the reader’s convenience (and for later use in the proof of \Cref{lem:localregular}), we will recall the local expansion strategy from \cite[Section 3]{yang2021delocalization} in \Cref{sec:pflocalregular}.

\begin{lemma}[Lemma 3.22 of \cite{yang2021delocalization}] \label{lvl1 lemma}
	Let $\Gamma$ be an arbitrary graph with $\OO(1)$ vertices and edges. For any large constant $D>0$, we can expand $\Gamma$ into a sum of $\OO(1)$ many locally standard graphs $\Gamma_\mu$:
	\begin{align}\label{expand lvl1}
		\Gamma =_{\E} \sum_\mu \Gamma_\mu + \Err,
	\end{align}
	where every $\Gamma_\mu$ has 
	scaling order $\ge \ord(\Gamma)$, and $\Err$ is a sum of graphs of scaling size $\OO(W^{-D})$.
\end{lemma}

\subsection{Examples and nested property}\label{sec:graphs_ideas}

To illustrate the ideas underlying the proof of Lemma \ref{lem:LW_moment}, we examine the simplest case with $p=2$. 

\begin{example}\label{example:p=2}
	Consider the second moment of $f_{xy}(G)$: 
	\be\label{eq:p=2graph}\E|f_{xy}(G)|^{2}=\E\sum_{\al_1,\al_2\notin\{x,y\}}\sum_{\beta_1,\beta_2}S_{\al_1\beta_1}S_{\al_2\beta_2} \p{
		\Gc_{\beta_1\beta_1} G_{x\al_1}G_{\al_1 y}}
	\cdot \p{\Gc^*_{\beta_2\beta_2}\overline G_{x\al_2}\overline G_{\al_2 y}},\ee
	which can be represented by the following graph: 
	\begin{center}\scalebox{0.8}{
			\begin{tikzpicture}
				\begin{feynman}
					\vertex (a) at (0,0);
					\fill[black] (a) circle (1pt) node[left=3pt]{$x$}; 
					\vertex (b) at (2,1); 
					\fill[black] (b) circle (1pt) node[right=2pt]{$\al_1$};
					\vertex [dot, fill] (c) at (4,0);
					\fill[black] (c) circle (1pt) node[right=3pt]{$y$};
					\vertex [dot, fill](d) at (2,-1);
					\fill[black] (d) circle (1pt) node[left=2pt]{$\al_2$};
					
					\vertex [dot, fill] (top) [above=0.8cm of b];
					\fill[black] (top) circle (1pt) node[right=2pt]{$\beta_1$};
					\vertex [dot, fill] (bot) [below=0.8cm of d];
					\fill[black] (bot) circle (1pt) node[left=2pt]{$\beta_2$};
					
					\vertex (tadtop) [left=0.8cm of top];
					\vertex (tadbot) [right=0.8cm of bot];
					
					\path [draw=blue,postaction={on each segment={mid arrow=blue}}]
					(a) -- (b) -- (c);
					\path [draw=red,postaction={on each segment={mid arrow=red}}]
					(a) -- (d) -- (c);
					\diagram*{      
						(b) -- [photon] (top) -- [half left, looseness=1.5, draw=blue] (tadtop) -- [half left, looseness=1.5, draw=blue] (top),
						(d) -- [photon] (bot) -- [half left, looseness=1.5, color=red] (tadbot) -- [half left, looseness=1.5, color=red] (bot),
					};
				\end{feynman}
		\end{tikzpicture}}
	\end{center}
	We expand \eqref{eq:p=2graph} using the local expansions from \Cref{sec:graphs}. As shown in \Cref{lvl1 lemma}, these expansions decrease the scaling size of the graphs, while improving their structure for our purposes. 
	More precisely, to obtain locally standard graphs, during the expansion process, a $-$ charged edge from the path $(x\to \al_2\to y)$ or from the light-weight at $\beta_2$ must be “pulled” to the molecule containing $\al_1$,\footnote{Pulling an edge $e=(\al,\beta)$ to a molecule $\Mol$ means replacing it by two edges, one connecting $\al$ to $\Mol$ and one connecting $\beta$ to $\Mol$.} whereas a $+$ charged edge from the path $(x\to \al_1\to y)$ or from the light-weight at $\beta_1$ must be pulled to the molecule containing $\al_2$. 
	As an illustration, consider the left panel of \Cref{fig:p=2expansion}, which shows a graph $\cal G$ obtained by applying two light-weight expansions from \Cref{ssl} ($\cal G$ is not yet locally standard, but it is already sufficient for our proof). 
	The two short edges $G_{\beta_1\gamma_1}$ and $G_{\beta_2\gamma_2}$ contribute a $[\Psi_t(0)]^2$ factor. For the rest of the graph, we examine its molecular graph $\cal G_1$ shown in the middle panel of \Cref{fig:p=2expansion}, where $\cal M_1$ and $\cal M_2$ denote the molecules containing $\al_1$ and $\al_2$, respectively. 
	Without loss of generality, suppose the two edges $(x,\cal M_1)$ provide the factor $[\Psi_t(c|a-b|)]^2$. For the remaining graph, we first sum over $\cal M_1$ and then over $\cal M_2$, applying the Cauchy–Schwarz inequality and Ward’s identity to obtain a factor $\eta_t^{-2}$. Altogether, this gives the desired bound in \eqref{eq:LW_moment} for $p=2$.

	In the right panel of \Cref{fig:p=2expansion}, we illustrate another possible molecular graph generated from the expansion of \eqref{eq:p=2graph}. For this graph, there are choices of long edges such that, after removing them, each summation order in the resulting graphs yields a poor bound. Nevertheless, by applying an additional AM–GM inequality, the graph still leads to the correct estimate. 
	To demonstrate this, suppose $(y,\cal M_1)$ and $(y,\cal M_2)$ are taken as long edges. Then, in the resulting graph $\cal G_2'$ (with these edges removed), both $\cal M_1$ and $\cal M_2$ have degree $3$. 
	Summing over either one of them eliminates three incident edges and leaves a graph with only a single solid edge. The subsequent summation produces a poor bound of order \smash{$N^{1/2}\eta_t^{-1/2}$} via Cauchy-Schwarz and Ward’s identity. 
	To handle this case, we apply the AM–GM inequality $2z_1z_2\le |z_1|^2+|z_2|^2$ to the two edges $(x,\cal M_1)$ and $(x, M_2)$, which yield two new graphs, one with a pair of edges $(x,\cal M_1)$ and one with a pair of edges $(x, \cal M_2)$. 
	In each of these graphs, the summations over $\cal M_1$ and $\cal M_2$ can be carried out in a proper order, yielding the $\eta_t^{-2}$ factor by applying Ward’s identity twice.

	\begin{figure}
		\begin{center}
			\begin{tabular}{m{4cm}m{4cm}m{4cm}}
				\scalebox{0.9}{
					\begin{tikzpicture}
						\begin{feynman}
							\vertex (L) at (-2,0);        
							\fill[black] (L) circle (2pt) node[left=2pt]{$x$};
							\vertex (R) at (2,0);         
							\fill[black] (R) circle (2pt) node[right=2pt]{$y$};
							
							\node (G1) at (-0.8,0) {$\cal G$};
							
							\vertex [dot] (a) at (0,1);   
							\fill[black] (a) circle (2pt) node[left=2pt]{$\al_1$};
							\vertex [dot] (a1) at (0,1.7);
							\fill[black] (a1) circle (2pt) node[right=2pt]{$\beta_1$};
							\vertex [dot] (a2) at (-0.7,1.7);
							\fill[black] (a2) circle (2pt) node[left=2pt]{$\gamma_1$};
							\vertex [dot] (b) at (1,1);   
							\vertex [dot] (c) at (0,-1);  
							\fill[black] (c) circle (2pt) node[right=2pt]{$\al_2$};
							\vertex [dot] (c1) at (0,-1.7); 
							\fill[black] (c1) circle (2pt) node[left=2pt]{$\beta_2$};
							\vertex [dot] (c2) at (0.7,-1.7);
							\fill[black] (c2) circle (2pt) node[right=2pt]{$\gamma_2$};
							\vertex [dot] (d) at (1,-1);  
							
							\vertex (ta) at (0.8,2.5);
							\vertex (tc) at (0.8,-2.5);
							
							\path [draw=blue,postaction={on each segment={mid arrow=blue}}](L) -- (a);
							\path [draw=blue,postaction={on each segment={mid arrow=blue}}](a) to [bend right] (c1);
							\path [draw=blue,postaction={on each segment={mid arrow=blue}}](c2) to [bend right] (R);
							\path [draw=blue,postaction={on each segment={mid arrow=blue}}](a1) to [bend right] (a2);
							\path [draw=red,postaction={on each segment={mid arrow=red}}](L) -- (a2);
							\path [draw=red,postaction={on each segment={mid arrow=red}}](a1) to [bend left] (c); 
							\path [draw=red,postaction={on each segment={mid arrow=red}}] (c) -- (R);
							\path [draw=red,postaction={on each segment={mid arrow=red}}](c2) to [bend left] (c1);
							
							\diagram*{
								
								(a) -- [photon] (a1), (a1) -- [photon] (a2), 
								(c) -- [photon] (c1), (c1) -- [photon] (c2),
								
							};
						\end{feynman}
				\end{tikzpicture}}
				& 
				\scalebox{0.9}{
					\begin{tikzpicture}
						\begin{feynman}
							\vertex (L) at (-2,0);       
							\fill[black] (L) circle (3pt) node[left=2pt]{$x$};
							\vertex (R) at (2,0);        
							\fill[black] (R) circle (3pt) node[right=2pt]{$y$};
							
							\node (G1) at (-0.8,0) {$\cal G_1$};
							
							\vertex [dot] (T) at (0,1);  
							\fill[black] (T) circle (3pt) node[above=2pt]{$\cal M_{1}$};
							\vertex [dot] (B) at (0,-1); 
							\fill[black] (B) circle (3pt) node[below=2pt]{$\cal M_2$};
							
							\path [draw=blue,postaction={on each segment={mid arrow=blue}}]
							(L) -- (T) -- (B) -- (R);
							\path [draw=red,postaction={on each segment={mid arrow=red}}]
							(L) to [bend left] (T) to [bend left] (B) to [bend right] (R);
						\end{feynman}
				\end{tikzpicture}}
				&
				\scalebox{0.9}{
					\begin{tikzpicture}
						\begin{feynman}
							\vertex (L) at (-2,0);       
							\fill[black] (L) circle (3pt) node[left=2pt]{$x$};
							\vertex (R) at (2,0);        
							\fill[black] (R) circle (3pt) node[right=2pt]{$y$};
							
							\node (G1) at (-0.8,0) {$\cal G_2$};
							
							\vertex [dot] (T) at (0,1);  
							\fill[black] (T) circle (3pt) node[above=2pt]{$\cal M_{1}$};
							\vertex [dot] (B) at (0,-1); 
							\fill[black] (B) circle (3pt) node[below=2pt]{$\cal M_2$};
							
							\path [draw=blue,postaction={on each segment={mid arrow=blue}}]
							(L) -- (B) -- (T) -- (R);
							\path [draw=red,postaction={on each segment={mid arrow=red}}]
							(L) to (T) to [bend left] (B) to (R);
						\end{feynman}
				\end{tikzpicture}}
			\end{tabular}
		\end{center}
		\caption{Possible expansions of \eqref{eq:p=2graph} and the corresponding molecular graph. The $\times$-dotted edges in the first graph have been omitted.
		}
		\label{fig:p=2expansion}
	\end{figure}
	
\end{example}

To prove \Cref{lem:LW_moment}, we need to extend the analysis for the $p=2$ example to arbitrarily large $p\in 2\N$. In the original graph $|f_{xy}(G)|^p$, there are $p$ edge-disjoint paths from $x$ to $y$, comprising a total of $2p$ solid edges. Heuristically, to obtain locally standard graphs with local expansions, each path from $x$ to $y$ must be pulled at least once, thereby increasing its length by 1. 
Now, the key graphical property required for our proof consists of the following two components: (i) We can identify at least $p$ ``long edges"---that is, edges of length $\gtrsim |x-y|$---in the graph. (ii) In the remaining graph, after removing these long edges, there exists a specific order of summation over the internal vertices such that, at each step, a sufficient number of edges (at least two per summation) remain to apply the Cauchy-Schwarz inequality and Ward's identity to control the summation. 
If there is no requirement of property (i), this problem was addressed in  \cite{yang2021random} using a structure called  ``\emph{nested property}" of the graph. 
In \Cref{lem:Anp_key} below, we will extend this concept to inductively select both long solid edges and a summation order to guarantee the bound  \eqref{eq:LW_moment}.

\subsection{Proof of Lemma \ref{lem:LW_moment}}

	For the proof of \Cref{lem:LW_moment}, we begin by expanding $\big|f_{xy}(G)\big|^{p} = [f_{xy}(G)]^{p/2}[\overline{f_{xy}(G)}]^{p/2}$ into a sum of locally standard graphs, as in \Cref{lvl1 lemma}. By carefully tracking the local expansion process, we can show that each locally standard graph satisfies the graphical properties in the following lemma.

	\begin{lemma}\label{lem:localregular}
		Given any large constant $D>0$, we can expand $\E\big|f_{xy}(G)\big|^{p}$ into a linear combination of $\OO(1)$ many \emph{locally standard graphs} $\Gamma_{\mu,xy}$ in the sense of equal expectation:
		\be\label{eq:local_Gs}
		\big|f_{xy}(G)\big|^{p} =_{\E} \sum_{\mu}  \Gamma_{\mu,xy} + \Err,
		\ee
		where $\Err$ denotes a sum of $\OO(1)$ many graphs of scaling size $\OO(W^{-D})$.
		Moreover, each graph $\Gamma_{\mu,xy}$ on the RHS satisfies the following properties under the assumption \eqref{eq:far_ab} (recall that, up to an exponentially small error $\exp(-c(\log W)^{3/2})$, $x$ and $y$ do not belong to the same molecule by \eqref{scalemole}):
		\begin{enumerate}
			\item[(1)] $\Gamma_{\mu,xy}$ is a locally standard graph with external vertices $x$ and $y$, with the corresponding external molecules denoted by $\cal M_x$ and $\cal M_y$, respectively.
			
			\item[(2)] $\Gamma_{\mu,xy}$ contains $0\le q\le p$ many internal molecules, denoted by $\cal M_{i}$ with $i\in \qqq{q}$. Moreover, within each (internal or external) molecule $\Mol_i$, $i\in\{1,\ldots, q, x, y\}$, the number of waved edges (denoted by $n_W(\Mol_i)$) and the number of vertices (denoted by $n_V(\Mol_i)$) satisfy the relation 
			\be\label{eq:MolVW}
			n_V(\Mol_i) \le n_W(\Mol_i) + 1.
			\ee
			
			\item[(3)] There are at least $p$ many edge-disjoint paths $\fP_i$, $i\in\qqq{p}$, in the molecular graph that connect the molecules $\cal M_x$ and $\cal M_y$ via solid edges (and possibly passing through the internal molecules). 
			
			\item[(4)] For each internal molecule $\cal M_i$, $i\in \qqq{q}$, there are at least two paths $\fP_k\ne \fP_l$ passing through it on the molecular graph. As a consequence, each $\Mol_i$ has degree at least 4, i.e., 
			\be\label{eq:deg_mole}
			\deg(\Mol_i)\ge 4,\quad i\in \qqq{q}.\ee
			
			\item[(5)] For each $A\subset \qqq{q}$, the number of paths passing through 
			$\{\Mol_i: i\in A\}$ is at least $|A|$, that is,
			$$
			\left|\left\{j: \mathrm V(\fP_j) \cap  \{\Mol_i: i\in A\}\ne \emptyset\right\}\right|\ge |A| ,$$
			where $\mathrm V(\fP_j)$ denotes the subset of vertices (i.e., molecules) in the path $\fP_j$ on the molecular graph.
			
			\item[(6)] The scaling order of $\Gamma_{\mu,xy}$ satisfies that
			\be\label{eq:sizeGammamu}
			\ord(\Gamma_{\mu,xy}) \ge 2p .  
			\ee
		\end{enumerate}
		
	\end{lemma}

	The proof of this lemma is straightforward by applying the local expansions \Cref{ssl,Oe14,T eq0}, and we defer the details to \Cref{sec:pflocalregular}. 
	The path properties (3)--(5) in \Cref{lem:localregular} serve as the key structural ingredients for controlling the graphs $\Gamma_{\mu,xy}$ in \eqref{eq:local_Gs}.
	To see them, note that in the initial graph, there are $p$ edge-disjoint paths from $x$ to $y$, corresponding to the $p$ factors $f_{xy}(G)$ or \smash{$\overline{f_{xy}(G)}$}. Each path passes through a distinct internal molecule, and each internal molecule is attached to two solid edges of the same color. From the expansion rules in \Cref{ssl,Oe14,T eq0}, one can verify that these paths are preserved throughout the expansion process, which ensures properties (3) and (5).
	Furthermore, to obtain locally standard graphs via local expansions, each internal molecule must either (i) pull in an edge of a different color from another path or molecule, (ii) merge with another internal molecule connected by differently colored edges, or (iii) merge with an external molecule. In case (iii), the internal molecule disappears from $\Gamma_{\mu,xy}$, whereas in cases (i) and (ii), at least two paths will pass through it.

	To streamline the estimation process, we now simplify the structure of $\Gamma_{\mu,xy}$ through its underlying molecular graph.
	Specifically, we bound them by corresponding ``\emph{auxiliary graphs}", which are obtained as quotient graphs by collapsing each molecule into a single vertex. 
	More precisely, for each internal molecule $\Mol_i$, $i\in \qqq{q}$, we select a representative vertex $\al_i$ within the molecule, referred to as the ``center" of the molecule. Similarly, for the two external molecules, we set $x$ and $y$ as their respective centers. The auxiliary graph associated with $\Gamma_{\mu,xy}$ is then defined on the set of block-level vertices $[x]$, $[y]$, and $[\al_i]$ (for $i\in \qqq{q}$): 
	
	\begin{definition}[Block-level vertices] \label{def: BM2}
		Recall that $\ZL$ is partitioned into $n^d$ blocks as defined in \eqref{eq:blockIa}. For any $x\in \Z_L^d$, we use $[x]$ to denote the block that contains $x$. 
		With slight abuse of notation, we also view these blocks as vertices in the lattice $\Zn$. In other words, $[x]$ may be interpreted both as a vertex in $\Zn$ and as a subset of vertices in $\ZL$. 
	\end{definition}

	For clarity, throughout the following proof we use the notation $[\cdot]$ to indicate vertices in the auxiliary graph and the lattice $\Zn$, distinguishing them from vertices in the original graphs defined in \Cref{def_graph1} and from those in the lattice $\ZL$. 
	To define the auxiliary graph formally, we still need to define its edges. 
	For simplicity of notation, we will use ``$\al\sim_{\Mol} \beta$" to mean that ``vertices $\al$ and $\beta$ belong to the same molecule". 
	For any $\beta_i\sim_{\Mol}\al_i$, it suffices to assume that  
	\be\label{yixi}
	|\al_i-\beta_i|\le W(\log W)^{1+\e_0}
	\ee
	for a small enough constant $\e_0\in (0,1/10)$; otherwise the graph is smaller than $\exp(-\Omega((\log W)^{1+\e_0}))\le W^{-D}$ for any constant $D>0$. Combining \Cref{lem_GbEXP} with \eqref{yixi}, we can bound the solid edges between different molecules as follows for any large constant $D>0$:  
	\be\label{eq:Gbyxi}|(G-M)_{xy}|\prec \Psi_{t}(0)\delta_{xy}+\xi([x],[y])+W^{-D},\ee
	where the variables $\xi(\cdot,\cdot)$ are defined as 
	\be\label{eq:xia1a2} [\xi([a_1],[a_2])]^2:=\sum_{\bfb:\|\bfb-\ba\|_\infty\le (\log W)^{1+2\e_0}} \max_{\bsig\in\{(-,+),(+,-)\}}\cL^{(2)}_{t,\bsig,\bfb} + W^{-d}\mathbf 1(|[a_1]-[a_2]|\le (\log W)^{1+2\e_0}).\ee
	With Ward's identity and the properties of $\Psi_t$ given by \eqref{eq:Psi}, we readily see the following properties for the $\xi$ variables. 
	\begin{claim}
		In the setting of \Cref{lem:LWterm}, 
		for any $[a_1],[a_2]\in \Zn$, the following estimates hold:
		\be\label{eq:Gbyxi2}
		\xi([a_1],[a_2]) \prec \Psi_t(|[a_1]-[a_2]|), \quad \sum_{[a_2]}[\xi([a_1],[a_2])]^2 \prec (W^d\eta_t)^{-1}. 
		\ee
	\end{claim}
	
	Now, we define the auxiliary graphs formed with edges representing the $\xi$ variables. 
	
	\begin{definition}[Auxiliary graphs]\label{def_auxgraph}
		Let $\cal G_{xy}$ be an arbitrary normal graph with $q$ internal molecules $\Mol_i$, $i\in\qqq{q}$, and two external molecules $\Mol_x$ and $\Mol_y$ containing $x$ and $y$, respectively. We fix the centers $\al_i\in\Mol_i$ for $i\in \qqq{q}$. We then define the auxiliary graph $\cal G^{\aux}_{[x][y]}$ of $\cal G_{xy}$ as follows:
		\begin{itemize}
			\item The vertices are $[\al_i]$, $i\in \{1,\ldots,q,x,y\}$, where we adopt the convention $\al_x\equiv x$ and $\al_y\equiv y$.
			\item Corresponding to each solid edge between molecules $\Mol_i$ and $\Mol_j$ in $\cal G_{xy}$, we introduce a (black) solid edge between $[\al_i]$ and $[\al_j]$ in $\cal G^{\aux}_{[x][y]}$, representing the factor $\xi([\al_i],[\al_j])$. (Since $\xi$ is symmetric, these edges are not oriented.)
		\end{itemize}
		In this way, we obtain the auxiliary graph \smash{$\cal G^{\aux}_{[x][y]}$}. Its value is defined analogously to \Cref{ValG}, namely, as the product of all solid edges followed by summation over all internal vertices $[\al_i]$, $i\in \qqq{q}$. Finally, we define the scaling order of \smash{$\cal G^{\aux}_{[x][y]}$} in the same manner as in \eqref{eq:ordG}:
		\be\label{eq:ordGaux} \ord(\cal G^{\aux}_{[x][y]}):=\#\{\text{solid edges in } \cal G^{\aux}_{[x][y]}\}-2\#\{\text{internal vertices in } \cal G^{\aux}_{[x][y]}\}.\ee
	\end{definition}
	
	Now, the estimation of $\Gamma_{\mu, xy}$ can be reduced to bounding its auxiliary graph.
	
	\begin{lemma}\label{GtoAG}
		Given a locally standard graph $\Gamma_{\mu,xy}$ from \eqref{eq:local_Gs} that satisfies properties (1)--(6) in \Cref{lem:localregular}, we define its auxiliary graph $\Gamma^{\aux}_{\mu, [a][b]}$ (recall that $x\in[a]$ and $y\in[b]$) and the corresponding scaling order as in \Cref{def_auxgraph}. For any large constant $D>0$, we have 
		\be\label{G_by_auxG} 
		\Gamma_{\mu, xy} \prec  (\Psi_t)^{\ord(\Gamma_{\mu, xy}) - \ord(\Gamma^{\aux}_{\mu, [a][b]})}  \cdot W^{qd}\Gamma^{\aux}_{\mu, [a][b]} + W^{-D}\, .
		\ee
	\end{lemma}
	
	\begin{proof}
		As in \Cref{def_auxgraph}, the vertices of \smash{$\Gamma^{\aux}_{\mu,[a][b]}$} are denoted by $[\al_i]$, $i\in\{1,\ldots,q,x,y\}$, corresponding to the molecule centers $\al_i$ in $\Gamma_{\mu,xy}$. 
		First, by the definition of $S$ together with the estimate for $S^\pm$ in \eqref{eq:estSpm-W}, and recalling \eqref{yixi}, we may bound a waved edge between $\beta_i$ and $\gamma_i$ inside molecule $\Mol_i$ by
		$$ \OO_\prec\p{W^{-d}\mathbf 1_{|\beta_i-\al_i|\vee |\gamma_i-\al_i|\le W(\log W)^{1+\e_0}} + e^{-c(\log W)^{1+\e_0}}}.$$
		Second, by \Cref{lem_GbEXP}, solid edges within molecules can be bounded by $\Psi_t$, and every solid edge between vertices $\beta_i\in \Mol_i$ and $\beta_j\in \Mol_j$ can be bounded by
		\[ \Big(\sum_{|[a']-[\beta_i]|\le 1,|[b']-[\beta_j]|\le 1} \max_{\bsig\in \{(+,-),(-,+)\}]} \cL^{(2)}_{t,\bsig,([a'],[b'])}(z)+ W^{-d}\mathbf 1_{|[\beta_i]-[\beta_j]|\le 1}\Big)^{1/2} \lesssim \xi([\al_i],[\al_j]), \]
		where we again use \eqref{yixi} for $\beta_i\sim_{\Mol}\al_i$ and $\beta_j\sim_{\Mol}\al_j$.
		
		With these bounds, and after summing over all \emph{non-center} vertices inside the molecules $\Mol_i$, $i\in \qqq{q}$, we obtain (cf. the notations in \Cref{def scaling})
		\begin{align*}
			\Gamma_{\mu, xy} \prec (\Psi_t)^{n_S(\Gamma_{\mu, xy})-n_S(\Gamma^{\aux}_{\mu, [a][b]})}W^{-d(n_{W}(\Gamma_{\mu, xy}) - n_V(\Gamma_{\mu, xy})+q)} \cdot W^{qd}\Gamma^{\aux}_{\mu, [a][b]}.
		\end{align*}
		Here, we have also used \eqref{eq:MolVW} in obtaining the factor $W^{-d(n_{W}(\Gamma_{\mu, xy}) - n_V(\Gamma_{\mu, xy})+q)}$, while the compensating factor $W^{qd}$ reflects the difference between summing the molecule centers $\al_i$ (in $\ZL$ for $\Gamma_{\mu,xy}$) and summing the corresponding block representatives $[\al_i]$ (in $\Zn$ for \smash{$\Gamma^{\aux}_{\mu,[a][b]}$}). Finally, using $\Psi_t\ge W^{-d/2}$ together with the definitions \eqref{eq:ordG} and \eqref{eq:ordGaux}, we conclude \eqref{G_by_auxG}.
	\end{proof}
	
	To conclude \Cref{lem:LW_moment}, it remains to bound the auxiliary graph $\Gamma^{\aux}_{\mu, [a][b]}$. 

	\begin{lemma}\label{lem:Anp} 
		In the setting of \Cref{GtoAG}, there exists a constant $c>0$ (depending on $p$) such that 
		\be\label{eq:bddGamma_aux}
		\Gamma^{\aux}_{\mu, [a][b]} \prec \left({W^d\eta_t}\right)^{-q} \cdot [\Psi_t(c|[a]-[b]|)]^p \cdot (\Psi_t)^{\ord(\Gamma^{\aux}_{\mu, [a][b]})-p}.
		\ee
	\end{lemma}
	
	Before turning to the proof of \Cref{lem:Anp}, we first use it to complete the proof of \Cref{lem:LW_moment}.
	
	\begin{proof}[\bf Proof of \Cref{lem:LW_moment}]
		Using \Cref{lem:Anp,GtoAG}, we bound the locally standard graphs in \eqref{eq:local_Gs} as
		\begin{align*}
			\Gamma_{\mu, xy} \prec  \eta_t^{-q}\cdot (\Psi_t)^{\ord(\Gamma_{\mu, xy}) -p} \cdot [\Psi_t(c|[a]-[b]|)]^p  + W^{-D} \le \eta_t^{-p}\cdot \Psi_t^{p} \cdot [\Psi_t(c|[a]-[b]|)]^p  + W^{-D},
		\end{align*}
		where we use \eqref{eq:sizeGammamu} in the second step. This concludes \eqref{eq:LW_moment} by taking $\Psi_t=\Psi(0)$.
	\end{proof}

	\subsection{Proof of Lemma \ref{lem:Anp}}\label{sec:pf_Anp}

	This subsection is devoted to establishing the key technical result, \Cref{lem:Anp}. Its proof relies crucially on the path properties (3)--(5) of $\Gamma_{\mu, xy}$ stated in \Cref{lem:localregular}, which also hold for the auxiliary graph \smash{$\Gamma^{\aux}_{\mu, [a][b]}$}. We refer to any graph satisfying these path properties as a \emph{nested graph}. 
	By \Cref{def_auxgraph} of auxiliary graphs together with the bounds in \eqref{eq:Gbyxi2}, \Cref{lem:Anp} follows directly from the following combinatorial result concerning nested graphs.

	\begin{lemma}[Bounding nested graphs]\label{lem:Anp_key}  
		Let $\cal G_{\ba\bfb}$ be a graph with $2p$ external vertices, denoted by $\ba=(a_1,\ldots,a_p)\in (\Zn)^p$ and $\bfb=(b_1,\ldots,b_p)\in (\Zn)^p$, together with $0\le q\le p$ internal vertices. A solid edge between vertices $\al$ and $\beta$ represents a nonnegative random variable $\xi_{\al\beta}=\xi_{\beta\al}$ (not necessarily the specific random variables appearing in \eqref{eq:xia1a2}) that satisfies the bounds
		\be\label{eq:Gbyxi3}
		\xi_{\al\beta} \prec \Psi_t(|\al-\beta|), 
		\quad \sum_{\beta}|\xi_{\al\beta}|^2 \prec (W^d\eta_t)^{-1}, \quad \forall \al,\beta\in \Zn.
		\ee
		We assume that the graph contains no self-loops. Suppose $\cal G_{\ba\bfb}$ satisfies the following path properties:
		\begin{enumerate}
			\item[(1)] There are $p$ edge-disjoint paths $\fP_i$, $i\in\qqq{p}$, such that each path $\fP_i$ connects the pair $a_i$ and $b_i$. 
			
			\item[(2)] Each internal vertex $\al_i$, for $i\in \qqq{q}$, is traversed by at least two distinct paths $\fP_k\ne \fP_l$. As a consequence, each vertex $\al_i$ has degree at least 4: 
			\be\label{eq:deg_mole_aux}
			\deg(\al_i)\ge 4,\quad \forall i\in \qqq{q}.\ee
			
			\item[(3)] For any subset $A\subset \qqq{q}$, the number of paths passing through the internal vertices $\{\al_i: i\in A\}$ is at least $|A|$; that is,
			$\left|\left\{j: \mathrm V(\fP_j) \cap  \{\Mol_i: i\in A\}\ne \emptyset\right\}\right|\ge |A| ,$
			where $\mathrm V(\fP_j)$ denotes the subset of vertices in the path $\fP_j$.
		\end{enumerate}
		Then, there exists a constant $c>0$ such that 
		\begin{equation}\label{adsuu_orig}
			\cal G_{\ba\bfb} \prec (W^d\eta_t)^{-q}\cdot \Psi_t^{\ord(\cal G_{\ba\bfb})-p}\cdot \prod_{i=1}^p  \Psi_t(c|a_i-b_i|) ,
		\end{equation}
		where $\ord(\cal G_{\ba\bfb})$ is defined as in \eqref{eq:ordGaux}, namely \(\ord(\cal G_{\ba\bfb}):=\#\{\text{solid edges in } \cal G_{\ba\bfb}\}-2q.\)
	\end{lemma}
	
	\begin{proof}[\bf Proof of \Cref{lem:Anp}]
		By \Cref{def_auxgraph} of auxiliary graphs, the bounds in \eqref{eq:Gbyxi2}, and the path properties (3)--(5) from \Cref{lem:localregular}, the auxiliary graph \smash{$\Gamma^{\aux}_{\mu, [a][b]}$} satisfies the assumptions of \Cref{lem:Anp_key}. Consequently, \eqref{eq:bddGamma_aux} becomes a special case of \eqref{adsuu_orig}, with $a_1=\cdots=a_p=[a]$ and $b_1=\cdots=b_p=[b]$. 
	\end{proof}

In the proof of \Cref{lem:Anp_key}, we will select the ``long edges" and determine the nested summation order of internal vertices via an inductive approach. 
Specifically, we assume that \Cref{lem:Anp_key} holds for any graph with $(q-1)$ internal vertices. Then, given a graph with $q$ internal vertices that satisfies the assumptions of \Cref{lem:Anp_key}, we provide an algorithm to identify the long edges and determine the first vertex in a valid nested order, where the long edges have been removed.
We show that summing over this vertex yields the desired factors of $(W^d\eta_t)^{-1}$ and $\Psi_t(c|a_i-b_i|)$, and that the resulting graph still satisfies the assumptions of \Cref{lem:Anp_key}, now with $(q-1)$ internal vertices. (Note that in this reduced graph, the ending external vertices may change. This is precisely why we work with more general graphs allowing arbitrary external vertices, rather than restricting ourselves to the original external vertices $[a]$ and $[b]$.) Therefore, we can apply the inductive hypothesis to complete the argument.

Our induction argument relies critically on the path properties assumed in \Cref{lem:Anp_key}. However, these properties may be violated once some long edges are removed. To clarify the graphical structure and streamline the presentation of our proof, instead of deleting the chosen long edges entirely, we replace them with a new type of edges, called \emph{ghost edges}:
\begin{itemize}
	\item A {\bf ghost edge} is a black dashed edge between two vertices, representing a factor of $1$. 
\end{itemize}
Ghost edges do not contribute to the value of a graph (and hence do not affect its scaling order); rather, they are introduced solely to preserve the connectedness of paths in our graphs. As a consequence, they will be counted when we define the degrees of vertices, i.e., 
\[ \deg(\al)=\#\{\text{solid and ghost edges connected with }\al\}.\]
When we want to refer to the degree of solid edges for a vertex $\al$, we will use the notation $\deg_s(\al)$.  
The main advantage of ghost edges is that they allow us to maintain the essential path properties---namely properties (1)--(3) in \Cref{lem:Anp_key}---throughout the proof. In particular, they simplify the induction: without ghost edges, we would need to distinguish between paths depending on whether long edges had been removed from them. 
Now, \Cref{lem:Anp_key} is an easy corollary of the following more general result for graphs with ghost edges.

\begin{lemma}\label{lem:Anp_key_gh}  
	Suppose $\cal G_{\ba\bfb}$ is a nested graph that may contain ghost edges. Assume it satisfies the setting of \Cref{lem:Anp_key}, including properties (1)--(3) therein, with the convention that paths may also contain ghost edges. Furthermore, suppose that each path $\fP_i$, $i\in\qqq{p}$, contains at most one ghost edge, which---if present---must appear as an ending edge of the path. Here, the ending edges of a path $\fP_i$ are its first and last edges, necessarily incident to an external vertex $a_i$ or $b_i$. Then, there exists a constant $c>0$ such that
	\begin{equation}\label{adsuu22}
		\cal G_{\ba\bfb} \prec (W^d\eta_t)^{-q}\cdot \Psi_t^{\ord(\cal G_{\ba\bfb})-n_{\ngh}(\cal G_{\ba\bfb})}\cdot \prod_{i=1}^p  [\Psi_t(c|a_i-b_i|)]^{\mathbf 1(\fP_i \text{ has no ghost edge})} ,
	\end{equation}
	where $n_{\ngh}$ denotes the number of paths that have no ghost edge, i.e., 
	\(n_{\ngh}(\cal G_{\ba\bfb}):=\sum_{i=1}^p \mathbf 1(\fP_i \text{ has no ghost edge}) .\)
\end{lemma}

\begin{proof}[\bf Proof of \Cref{lem:Anp_key_gh}] 
We prove the estimate \eqref{adsuu22} by induction on $q$, the number of internal molecules. Suppose we aim to establish \eqref{adsuu22} for a target graph \smash{${\cal G}_{\ba\bfb}$} satisfying the assumptions of \Cref{lem:Anp_key_gh}. Moreover, assume that \smash{${\cal G}_{\ba\bfb}$} contains $K\in \N$ many solid edges. Our induction hypothesis is the following: the estimate \eqref{adsuu22} holds for all nested graphs satisfying the assumptions of \Cref{lem:Anp_key_gh}, with $k$ internal molecules and at most $K$ solid edges. (Note that we do not assume that the number of paths is smaller than $p$.)
In the following proof, we use the notation $(\alpha, \beta)$ to denote either a solid or a ghost edge between vertices $\alpha$ and $\beta$. A path from $\alpha$ to $\beta$ is denoted by $(\alpha \too \beta)$, where the double arrow indicates that intermediate vertices may be present along the path.
More generally, a path of the form
	$(\alpha_1 \too \alpha_2 \too \cdots \too \alpha_k)$ represents a sequence of edges from $\alpha_1$ to $\alpha_2$, then $\alpha_2$ to $\alpha_3$, and so on.
	
	First, the estimate \eqref{adsuu22} is trivial in the case $q=0$, by \eqref{eq:Gbyxi3}. Now, let $1\le q\le p$, and assume the induction hypothesis holds for all $0\le k \le q-1$. We prove that \eqref{adsuu22} holds for every graph satisfying the assumptions of \Cref{lem:Anp_key_gh}, with $q$ internal molecules and at most $K$ solid edges.
	Given such a graph $\cal G_{\ba\bfb}$, we partition the summation region over $\bal=(\al_1,\ldots,\al_q)$ into at most $2^{pq}$ subregions, according to whether each vertex $\al_i$ is closer to $a_j$ or to $b_j$ with $i\in\qqq{q}$ and $j\in \qqq{p}$:
	\[ \bD_{\bpi}\subset (\Zn)^q:=\left\{\bal: |\al_i-a_j|\le |\al_i-b_j| \ \text{if}\ \pi_{i,j}=0, \text{ and } |\al_i-a_j| > |\al_i-b_j| \ \text{if}\ \pi_{i,j}=1\right\},\]
	for $\bpi=(\pi_{1,1},\ldots, \pi_{1,p},\ldots, \pi_{q,1},\ldots, \pi_{q,p}) \in \{0,1\}^{pq}$. The union of these subregions covers the entire space \smash{$(\Zn)^q$}. Then, it suffices to prove that for each fixed $\bpi \in \{0,1\}^{pq}$,  
	\begin{align}\label{kwuyayw}
		\sum_{\bal\in \bD_{\bpi}} \cal G_{\ba\bfb}(\bal) \prec (W^d\eta_t)^{-q}\cdot \Psi_t^{\ord(\cal G_{\ba\bfb})-n_{\ngh}(\cal G_{\ba\bfb})}\cdot \prod_{i=1}^p  [\Psi_t(c|a_i-b_i|)]^{\chi(\fP_i)} ,
	\end{align} 
	where $\cal G_{\ba\bfb}(\bal)$ denotes the graph obtained by fixing the internal vertices in $\bal$ as external vertices, and we abbreviate $\chi(\fP_i)\equiv \mathbf{1}\left(\fP_i \text{ has no ghost edge}\right)$. 
	
	For the proof of \eqref{kwuyayw}, we look at the ending edges of the $p$ paths. Suppose $(a_j,\al_i)$ is an ending edge of a path $\fP_j$ (note that $\fP_j$ connects $a_j$ to $b_j$ through $\al_i$, but there may be no direct solid edge between $b_j$ and $\al_i$). On the region $\bD_{\bpi}$, the solid edge $(a_j,\al_i)$ falls into one of the following four categories: 
	\begin{enumerate}
		\item[{\bf A1}:] The path $\fP_j=(a_j\too b_j)$ has no ghost edge. Moreover, for $\bal\in \bD_{\bpi}$, the vertex $\al_i$ is closer to $a_j$ than to $b_j$, i.e., $|\al_i-a_j|\le  |\al_i-b_j|$. In particular, this implies $|\al_i-b_j|\ge |a_j-b_j|/2$. 
		
		\item[{\bf A2}:] The path $\fP_j=(a_j\too b_j)$ has no ghost edge. Moreover, for $\bal\in \bD_{\bpi}$, the vertex $\al_i$ is closer to $b_j$ than to $a_j$, i.e., $|\al_i-a_j| > |\al_i-b_j|$. In particular, this implies $|\al_i-a_j|\ge |a_j-b_j|/2$. 
		
		\item[{\bf B1}:] The path $\fP_j=( a_j \too b_j )$ contains a ghost edge, which is not $( a_j , \al_i )$. 
		
		\item[{\bf B2}:] The path $\fP_j=( a_j \too b_j )$ contains a ghost edge $( a_j , \al_i )$.
		
	\end{enumerate}
	By symmetry, we may also classify ending edges attached to $b_j$ by switching the roles of $a_j$ and $b_j$ in the definitions above. The only difference is in the classification of type-A1 (resp.~type-A2) edges: in this case, we require $|\al_i-b_j|< |\al_i-a_j|$ (resp.~$|\al_i-b_j|\ge |\al_i-a_j|$) instead.

	First, we notice that if an ending edge $\p{a_j,\al_i}$ is a type-A2 edge, then it contributes a factor $\Psi_t(|a_j-b_j|/2)$ by \eqref{eq:Gbyxi3}. In this case, we can bound the original graph $\cal G_{\ba\bfb}$ by the product of a new graph \smash{$\wt{\cal G}_{\ba\bfb}$}, obtained by replacing the edge $\p{a_j,\al_i}$ with a ghost edge, and the factor $\Psi_t(|a_j-b_j|/2)$. It is easy to see that if the bound \eqref{kwuyayw} holds for \smash{$\wt{\cal G}_{\ba\bfb}$}, then it also holds for the original graph. 
	We perform this replacement for every A2 edge that appears in the paths $\fP_i$, $i\in\qqq{p}$.  
	After performing all such replacements, it suffices to prove \eqref{kwuyayw} for the resulting graph.
	With a slight abuse of notation, we continue to denote this modified graph by $\cal G_{\ba\bfb}$, which now satisfies the condition: 
	\be\label{eq:noA2}
	\text{there is no A2 edge in }\cal G_{\ba\bfb}.
	\ee
	We now proceed to prove the estimate \eqref{kwuyayw}, assuming \eqref{eq:noA2} and the induction hypothesis. The proof is organized by classifying cases according to the number and type of ending edges attached to internal vertices.
	
	\medskip
	\paragraph{\bf (I) Two A1/B1 edges: the simple case}
	
	We first consider the case where an internal vertex is connected to two or more ending edges of type A1 or B1. That is, suppose there exists an internal vertex in the graph connected to at least two such edges---either two A1 edges, two B1 edges, or one of each. Without loss of generality, let this vertex be $\al_q$. To illustrate the idea, we begin with a simple scenario where $\deg(\al_q)=4$, so exactly two paths pass through $\al_q$. We will address the more general case $\deg(\al_q)\ge 4$ in Case (II). 
	
	By definition, the two ending edges connected to $\al_q$ must belong to distinct paths. 
	Without loss of generality, assume these edges are $(a_1,\al_q)$ and $(a_2,\al_q)$, which belong to paths $\fP_1$ and $\fP_2$, respectively. We now define a new graph \smash{$\cal G^{\mathrm{new}}_{\ba'\bfb}\equiv \cal G^{\mathrm{new}}_{\ba'(\al_q),\bfb}$} as follows: it has external vertices 
	\[\ba'=(a_1',a_2',a_3,\ldots,a_p,a_1)\quad \text{and}\quad \bfb=(b_1,\ldots,b_p,a_2), \quad \text{where}\ \ a_1'=\al_q, \ a_2'=\al_q,\] 
	and internal vertices $\bal'=(\al_1,\ldots,\al_{q-1})$, i.e., we remove $\al_q$ from the internal vertex set. 
	First, all edges not belonging to paths $\fP_1$ and $\fP_2$ remain unchanged, so the paths $\fP_i$ for $i\in\qqq{3,p}$ stay the same in $\cal G^{\mathrm{new}}_{\ba'\bfb}$. Second, we define new paths $\fP'_1=(\al_q\too b_1)$ and $\fP_2'=(\al_q\too b_2)$ obtained by removing the edges $(a_1,\al_q)$ and $(a_2,\al_q)$ from $\fP_1$ and $\fP_2$, respectively. 
	Finally, we introduce an auxiliary path $(a_1\too a_2)$ consisting of a single ghost edge between $a_1$ and $a_2$. The purpose of this auxiliary path is to ensure consistency with the induction hypothesis, which requires that every external vertex is an ending point of some path. (We cannot simply remove $a_1$ and $a_2$ from the graph, since some path $\fP_i$ may pass through them.)
	
	The new graph $\cal G^{\mathrm{new}}_{\ba'\bfb}$ produced by this construction also satisfies the assumptions of \Cref{lem:Anp_key_gh}, but it has one fewer internal vertex and two fewer solid edges than the original graph $\cal G_{\ba\bfb}$. Hence, by the induction hypothesis, we can bound it by 
	\begin{align}
		\sum_{\bal'} \cal G^{\mathrm{new}}_{\ba'\bfb}(\bal') \prec &~ \frac{ \Psi_t^{\ord(\cal G^{\mathrm{new}}_{\ba'\bfb})-n_{\ngh}(\cal G^{\mathrm{new}}_{\ba'\bfb})}}{(W^d\eta_t)^{q-1}}\cdot  \prod_{i=1}^2[\Psi_t(c|\al_q-b_i|)]^{\chi(\fP_i')} \cdot \prod_{i=3}^p  [\Psi_t(c|a_i-b_i|)]^{\chi(\fP_i)} \nonumber\\
		\prec &~ (W^d\eta_t)^{-(q-1)}\cdot \Psi_t^{\ord(\cal G_{\ba\bfb})-n_{\ngh}(\cal G_{\ba\bfb})}\cdot  \prod_{i=1}^p  [\Psi_t(c|a_i-b_i|/2)]^{\chi(\fP_i)} \label{eq:induc_Ggraph}
	\end{align}
	for a constant $c>0$. In the second step, we use that 
	\be\label{eq:change_of_order}\ord(\cal G^{\mathrm{new}}_{\ba'\bfb})=\ord(\cal G_{\ba\bfb}),\quad n_{\ngh}(\cal G_{\ba\bfb})=n_{\ngh}(\cal G^{\mathrm{new}}_{\ba'\bfb}),\ee 
	since $\al_q$ is now treated as an external vertex. Moreover, we have also used that 
	\be\label{eq:change_of_order2}\prod_{i=1}^2[\Psi_t(c|\al_q-b_i|)]^{\chi(\fP_i')}\le \prod_{i=1}^2[\Psi_t(c|a_i-b_i|/2)]^{\chi(\fP_i)} \ee
	by the definitions of A1 and B1 edges. With \eqref{eq:induc_Ggraph}, we can bound the LHS of \eqref{kwuyayw} by 
	\begin{align}\label{kwuyayw_case1}
		\sum_{\bal\in \bD_{\bpi}} \cal G_{\ba\bfb}(\bal) &\prec (W^d\eta_t)^{-(q-1)}\cdot \Psi_t^{\ord(\cal G_{\ba\bfb})-n_{\ngh}(\cal G_{\ba\bfb})}\cdot \prod_{i=1}^p  [\Psi_t(c|a_i-b_i|/2)]^{\chi(\fP_i)} \sum_{\al_q}\xi_{a_1\al_q}\xi_{a_2\al_q} \nonumber\\
		&\prec (W^d\eta_t)^{-q}\cdot \Psi_t^{\ord(\cal G_{\ba\bfb})-n_{\ngh}(\cal G_{\ba\bfb})}\cdot  \prod_{i=1}^p  [\Psi_t(c|a_i-b_i|/2)]^{\chi(\fP_i)} ,
	\end{align}
	where we use the Cauchy-Schwarz inequality and \eqref{eq:Gbyxi3} in the second step. This concludes \eqref{kwuyayw} in case (I).

	\medskip
	\paragraph{\bf (II) Two A1/B1 edges: general case} 
	
	We now consider a more general case than Case (I), where there is an internal vertex, say $\al_q$, connected to at least two ending edges of type A1 or B1, and where $\deg(\al_q)\ge 4$. Without loss of generality, assume that the two ending edges connected to $\al_q$ are $(a_1,\al_q)$ and $(a_2,\al_q)$, belonging to paths $\fP_1$ and $\fP_2$, respectively. As before, we define the new graph \smash{$\cal G^{\new}_{\ba'\bfb'}\equiv \cal G^{\new}_{\ba'(\al_q),\bfb'(\al_q)}$} obtained from \smash{$\cal G_{\ba\bfb}$} by removing the two edges $(a_1,\al_q)$ and $(a_2,\al_q)$. 
	
	The new graph has external vertices \(\al_q, a_1,a_2, a_3, \ldots, a_p,  b_1, \ldots,  b_p\) and internal vertices $\al_1,  \ldots,  \al_{q-1}$. To apply the induction hypothesis, we need to ensure that the structure of the paths in $\cal G^{\new}_{\ba' \bfb'}$ satisfies the assumptions in \Cref{lem:Anp_key_gh}.  
	For each path $\fP_j$, $j\in \qqq{3,p}$, in the original graph, suppose it takes the form 
	\be\label{eq:jthpath}( a_j \too  \al_q \too \al_q \too \cdots \too  \al_q \too b_j ),\ee 
	where each occurrence of $\alpha_q$ is explicitly listed. Here, $(a_j \too \alpha_q)$ denotes the initial segment of $\fP_j$ connecting $a_j$ to the first occurrence of $\alpha_q$, while $(\alpha_q \too b_j)$ denotes the final segment connecting the last occurrence of $\alpha_q$ to $b_j$. Each intermediate segment $(\alpha_q \too \alpha_q)$ represents a connection between two consecutive appearances of $\alpha_q$ along the path. Let $k_j$ denote the total number of times $\alpha_q$ appears in the path $\fP_j$. 
	Correspondingly, in $\cal G^{\mathrm{new}}_{\ba'\bfb'}$, we define the following $k_j+1$ paths for each $j$: for each $1 \leq r \leq k_j + 1$, the $(j,r)$-th path $\fP'_{j,r}$ is the $r$-th segment of $\fP_j$, connecting $a_{j,r}$ to $b_{j,r}$, where  
	\begin{equation}
		a_{j,r}:=  \begin{cases} a_j, & \text{if} \ j=1\\
			\al_q, & \text{if} \ j>1
		\end{cases},\quad  b_{j,r}:=  \begin{cases} \al_q,  &\text{if} \ j\le k_j\\
			b_j, &\text{if} \ j= k_j +1
		\end{cases}. 
	\end{equation}
	Next, for the first two paths $\fP_j$, $j \in \{1, 2\}$, assume they also take the form \eqref{eq:jthpath}, with $k_j$ appearances of $\al_q$. After removing the edges $(a_1,\al_q)$ and $(a_2,\al_q)$, we define $k_j$ paths in $\cal G^{\mathrm{new}}_{\ba'\bfb'}$ for each $j\in\{1,2\}$. Specifically, for each $1 \leq r \leq k_j$, the $(j,r)$-th path $\fP'_{j,r}$ is the $r$-th segment of $\fP_j$, connecting $a_{j,r}$ to $b_{j,r}$, where  
	\begin{equation}
		a_{j,r}:= \al_q,\quad  b_{j,r}:=  \begin{cases} \al_q,  &\text{if} \ j< k_j\\
			b_j, &\text{if} \ j= k_j
		\end{cases}. 
	\end{equation}
	Finally, we introduce an auxiliary path $(a_1\too a_2)$ consisting of a single ghost edge between $a_1$ and $a_2$. 
	
	
	Now, we let $\ba':=\{a_{j,r}:1\le j \le p, 1\le r \le k_j+\mathbf 1_{j\ge 3}\}\cup\{a_1\}$ and $\bfb':=\{b_{j,r}:1\le j \le p, 1\le r \le k_j+\mathbf 1_{j\ge 3}\}\cup\{a_2\}$. It is easy to see that the graph $\cal G^{\new}_{\ba'\bfb'}$ defined in this way satisfies the assumptions in \Cref{lem:Anp_key_gh}, but with one fewer internal vertex and two fewer solid edges than the original graph $\cal G_{\ba\bfb}$. 
	Therefore, using the induction hypothesis, we can bound it in a similar way as \eqref{eq:induc_Ggraph}, that is,
	\begin{align}\label{eq:bound_new_graph}
		\sum_{\bal'} \cal G^{\mathrm{new}}_{\ba'\bfb'}(\bal') \prec (W^d\eta_t)^{-(q-1)}\cdot \Psi_t^{\ord(\cal G_{\ba\bfb})-n_{\ngh}(\cal G_{\ba\bfb})}\cdot  \prod_{i=1}^p  [\Psi_t(c|a_i-b_i|/2)]^{\chi(\fP_i)} ,
	\end{align}
	where $\bal'=(\al_1,\ldots,\al_{q-1})$, and in the derivation of this bound, we have used \eqref{eq:change_of_order} and \eqref{eq:change_of_order2}, and that 
	$$
	\prod_{r=1}^{k_j+1} [\Psi_t(c|a_{j,r}- b_{j,r}|)]^{\chi(\fP'_{j,r})} \le [\Psi_t (c|a_{j}- b_{j}|/2)]^{\chi(\fP_{j})} [\Psi_t(0)]^{\sum_{r=1}^{k_j+1}\chi(\fP'_{j,r})- \chi(\fP_{j})},\quad \forall j\in\qqq{3,p}.
	$$
	Finally, using \eqref{eq:bound_new_graph}, we can complete the proof of \eqref{kwuyayw} for Case (II), following a similar argument to the one used in \eqref{kwuyayw_case1}.

	\medskip
	\paragraph{\bf (III) Two B2 edges}
	In this case, we assume there exists an internal vertex, say $\al_q$, that is connected to two B2 ending edges and two solid edges, i.e., $\deg_s(\al_q)=2$. By definition, these B2 edges do not belong to the same path. Without loss of generality, assume that these two B2 edges are $(a_1, \al_q)$ and $(a_2, \al_q)$, belonging to paths $\fP_1$ and $\fP_2$, respectively. 
	Additionally, suppose the two solid edges connected to $\al_q$ are $(\al_q, \beta_1)$ and $(\al_q, \beta_2)$, which belong to paths $\fP_1$ and $\fP_2$, respectively. Here, $\beta_1$ and $\beta_2$ may represent external vertices in $\ba, \bfb$ or internal vertices in $\bal' = (\al_1, \ldots, \al_{q-1})$.

	We now define a new graph $\cal G^{\new}_{\ba\bfb}$ as follows: it still has the external vertices $\ba$ and $\bfb$, and the internal vertices $\bal'$. The edges that do not belong to the paths $\fP_1$ and $\fP_2$ remain unchanged in the new graph $\cal G^{\mathrm{new}}_{\ba\bfb}$. For paths $\fP_1$ and $\fP_2$, we modify them as follows:
	\begin{itemize}
		\item We create a new path $\fP_1'=(a_1 \too b_1)$ obtained by removing the solid edge $(\al_q, \beta_1)$ from $\fP_1$ and adding a ghost edge $(a_1, \beta_1)$ to the path.
		
		\item Similarly, we create a new path $\fP_2'=(a_2 \too b_2)$ obtained by removing the solid edge $(\al_q, \beta_2)$ from $\fP_2$ and adding a ghost edge $(a_2, \beta_2)$ to the path.
	\end{itemize}
	The reader can refer to \Cref{Fig:construct2B2} for an illustration of this construction. 
	
	\begin{figure}[h]
		\centering
		\begin{tikzpicture}
			
			\coordinate (a1) at (0, 0.5);
			\coordinate (a2) at (0, -0.5);
			\coordinate (al) at (1, 0);
			\coordinate (b1) at (2, 0.5);
			\coordinate (b2) at (2, -0.5);
			\coordinate (c1) at (3, 1);
			\coordinate (c2) at (3, -1);
			\coordinate (d1) at (5, 0.5);
			\coordinate (d2) at (5, -0.5);
			\coordinate (L1) at (6, 0.5);
			\coordinate (L2) at (6, -0.5);
			
			\fill (a1) circle (1.5pt) node[above=2pt]{$a_1$};
			\fill (a2) circle (1.5pt) node[below=2pt]{$a_2$};
			\fill (b1) circle (1.5pt)  node[above=2pt]{$\beta_1$};
			\fill (b2) circle (1.5pt) node[below=2pt]{$\beta_2$};
			\fill (al) circle (1.5pt) node[above=2pt]{$\al_q$};
			\fill (c1) circle (1.5pt) node[right=16pt]{$\fP_1$};
			\fill (c2) circle (1.5pt) node[right=16pt]{$\fP_2$};
			\fill (d1) circle (1.5pt);
			\fill (d2) circle (1.5pt);
			\fill (L1) circle (1.5pt) node[above=2pt]{$b_1$};
			\fill (L2) circle (1.5pt) node[below=2pt]{$b_2$};

			\draw (a1) [dashed]-- (al); 
			\draw (al) -- (b1) -- (c1); 
			\draw (c1) [dotted]-- (d1); 
			\draw (d1) -- (L1); 
			\draw (a2) [dashed]-- (al);
			\draw (al) -- (b2) -- (c2); 
			\draw (c2) [dotted]-- (d2); 
			\draw (d2) -- (L2); 
		\end{tikzpicture}
		\hspace{10pt}
		\begin{tikzpicture}
			
			\coordinate (a1) at (0.5, 0.5);
			\coordinate (a2) at (0.5, -0.5);
			\coordinate (b1) at (2, 0.5);
			\coordinate (b2) at (2, -0.5);
			\coordinate (c1) at (3, 1);
			\coordinate (c2) at (3, -1);
			\coordinate (d1) at (5, 0.5);
			\coordinate (d2) at (5, -0.5);
			\coordinate (L1) at (6, 0.5);
			\coordinate (L2) at (6, -0.5);
			
			\fill (a1) circle (1.5pt) node[above=2pt]{$a_1$};
			\fill (a2) circle (1.5pt) node[below=2pt]{$a_2$};
			\fill (b1) circle (1.5pt)  node[above=2pt]{$\beta_1$};
			\fill (b2) circle (1.5pt) node[below=2pt]{$\beta_2$};
			
			\fill (c1) circle (1.5pt) node[right=16pt]{$\fP_1'$};
			\fill (c2) circle (1.5pt) node[right=16pt]{$\fP_2'$};
			\fill (d1) circle (1.5pt);
			\fill (d2) circle (1.5pt);
			\fill (L1) circle (1.5pt) node[above=2pt]{$b_1$};
			\fill (L2) circle (1.5pt) node[below=2pt]{$b_2$};
			
			\draw (a1) [dashed]-- (b1);
			\draw (b1) -- (c1); 
			\draw (c1) [dotted]-- (d1); 
			\draw (d1) -- (L1); 
			\draw (a2) [dashed]-- (b2);
			\draw (b2) -- (c2); 
			\draw (c2) [dotted]-- (d2); 
			\draw (d2) -- (L2); 
		\end{tikzpicture}
		\caption{Illustration of the construction of paths $\fP_1'$ and $\fP_2'$ (right panel) from $\fP_1$ and $\fP_2$ (left panel).}\label{Fig:construct2B2}
	\end{figure}

	Note that $\cal G^{\mathrm{new}}_{\ba\bfb}$ defined above is also a graph satisfying the assumptions in \Cref{lem:Anp_key_gh}, but with one fewer internal vertex and two fewer solid edges compared to the original graph $\cal G_{\ba\bfb}$. Therefore, by the induction hypothesis, we can bound it by 
	\begin{align*}
		\sum_{\bal'} \cal G^{\mathrm{new}}_{\ba\bfb}(\bal') \prec &~ (W^d\eta_t)^{-(q-1)}  \cdot \Psi_t^{\ord(\cal G^{\mathrm{new}}_{\ba\bfb})-n_{\ngh}(\cal G^{\mathrm{new}}_{\ba\bfb})} \cdot  \prod_{i=3}^p  [\Psi_t(c|a_i-b_i|)]^{\chi(\fP_i)} \nonumber\\
		= &~ (W^d\eta_t)^{-(q-1)}\cdot \Psi_t^{\ord(\cal G_{\ba\bfb})-n_{\ngh}(\cal G_{\ba\bfb})}\cdot  \prod_{i=1}^p  [\Psi_t(c|a_i-b_i|)]^{\chi(\fP_i)}, \end{align*}
	where, in the second step, we use that $\ord(\cal G^{\mathrm{new}}_{\ba\bfb})=\ord(\cal G_{\ba\bfb})$, $n_{\ngh}(\cal G^{\mathrm{new}}_{\ba\bfb})=n_{\ngh}(\cal G_{\ba\bfb})$, and $\chi(\fP_1)=\chi(\fP_2)=0$. 
	Thus, we can bound the LHS of \eqref{kwuyayw} as follows: 
	\begin{align}\label{kwuyayw_case3}
		\sum_{\bal\in \bD_{\bpi}} \cal G_{\ba\bfb}(\bal) &\prec \sum_{\bal'} \cal G^{\mathrm{new}}_{\ba\bfb}(\bal')\sum_{\al_q}\xi_{\al_q\beta_1}\xi_{\al_q\beta_2} \prec (W^d\eta_t)^{-1}\sum_{\bal'} \cal G^{\mathrm{new}}_{\ba\bfb}(\bal') \nonumber\\
		&\prec (W^d\eta_t)^{-q}\cdot \Psi_t^{\ord(\cal G_{\ba\bfb})-n_{\ngh}(\cal G_{\ba\bfb})}\cdot  \prod_{i=1}^p  [\Psi_t(c|a_i-b_i|/2)]^{\chi(\fP_i)} ,
	\end{align}
	where we use the Cauchy-Schwarz inequality and \eqref{eq:Gbyxi3} in the second step. This implies \eqref{kwuyayw} in Case (III).

	\paragraph{\bf (IV) B1 edge + B2 edge}

	Now, suppose Case (II) does not hold. In Case (II), we have already addressed all situations where an internal vertex is connected to at least two A1/B1 edges. Therefore, each internal vertex in the graph is now connected by at most one A1/B1 edge. However, note that there are a total of $2p$ ending edges, and the number of B2 edges is at most $p$ (since there are at most $p$ ghost edges across the $p$ paths). Thus, our graph must contain at least $p$ A1/B1 edges and at most $p$ internal vertices. By the pigeonhole principle, the graph $\cal G_{\ba\bfb}$ must satisfy the following properties:
	\begin{itemize}
		\item[(1)] It contains $q=p$ internal vertices, each of which is connected by exactly one A1/B1 edge. 
		
		\item[(2)] Each path $\fP_i$ for $i\in \qqq{p}$ contains exactly one ghost edge---specifically, a B2 ending edge---which in particular implies that the graph contains no A1 edges. 
		
	\end{itemize}
	Hence, the estimate \eqref{kwuyayw} now reduces to 
	\begin{align}\label{kwuyayw_ng}
		\sum_{\bal\in \bD_{\bpi}} \cal G_{\ba\bfb}(\bal) \prec (W^d\eta_t)^{-p}\cdot \Psi_t^{\ord(\cal G_{\ba\bfb})}.
	\end{align} 
	The key to proving \eqref{kwuyayw_ng} is to construct a \emph{nested order of summation} for the $p$ internal vertices, as in \cite{yang2021random}. 
	That is, at each step of the summation---when summing over a given internal vertex according to this order---there must be at least two solid edges connected to the vertex being summed over.
	
	Suppose there exists an internal vertex $\al_i$ with $\deg_s(\al_i)=2$. Then, by the property \eqref{eq:deg_mole_aux}, this vertex must be connected by exactly two ghost edges. This situation has already been covered in Case (III). It remains to consider the case where
	\be\label{eq:degali}
	\deg_{s} (\al_i)\ge 3,\quad \forall i \in \qqq{p}.
	\ee
	By property (1) above, each internal vertex $\al_i$ is connected by one B1 solid edge. Without loss of generality, suppose this B1 edge connects $\al_i$ to $a_i$.
	Aside from this B1 edge, suppose that $\al_i$ is connected to $0\le k_i\le \deg_{s} (\al_i)-1$ external vertices, denoted by $\{c_{i,j}:j\in \qqq{k_i}\}$ (allowing for repetitions), and to $s_i=\deg_{s} (\al_i)-1-k_i$ internal vertices.

	We remove all ghost edges and external solid edges (i.e., those connected to external vertices) and decompose the resulting graphs into distinct connected components. 
	In other words, in the resulting graph, two vertices belong to the same component if and only if they are connected by a path \emph{consisting solely of internal solid edges}. 
	Without loss of generality, assume $\cal G_{\bal_r}$ is one such component, with internal vertices $\bal_r=(\al_1,\ldots, \al_r)$ for some $1\le r \le p$. We claim that 
	\begin{align}\label{kwuyayw_ng_tree}
		\sum_{\bal_r} \cal G_{\bal_r} \prod_{i=1}^r \Big(\xi_{\al_i a_i} \prod_{j=1}^{k_i}\xi_{\al_i c_{i,j}}\Big)  \prec (W^d\eta_t)^{-r} \cdot \Psi_t^{\ord(\cal G_{\bal_r})+r+\sum_{i=1}^r k_i}.
	\end{align}
	Applying this estimate to each connected component of $\cal G_{\ba\bfb}$, taking the product of the resulting bounds, and using the relation $ \ord(\cal G_{\ba\bfb})= \sum \ord(\cal G_{\bal_r}) + k_e$ (where $k_e\ge p$ denotes the total number of external solid edges not contained in the connected components), we obtain the desired bound \eqref{kwuyayw_ng}. 
	
	To establish \eqref{kwuyayw_ng_tree}, we first treat the case where there exists an index $i$ with $k_i \ge 1$. Without loss of generality, assume $k_1\ge 1$, so that $\al_1$ is incident to at least two external solid edges. Then, we identify the nested order by fixing a spanning tree $\mathbb T$ of $\cal G_{\bal_r}$ rooted at $\al_1$. We sum over the internal vertices according to the structure of $\mathbb T$, starting from the leaves and moving toward the root. 
	At each step, if $\al_i$ is a leaf with $i\ne 1$, then it is attached to at least two solid edges—namely, the edge in the tree and the external edge $(\al_i,a_i)$.
	Summing over $\al_i$ yields a factor $(W^d \eta_t)^{-1}$ by Cauchy-Schwarz and \eqref{eq:Gbyxi3}, together with additional edges that contribute $\Psi_t$ factors. 
	Removing $\al_i$ from the graph produces a smaller spanning tree, to which we apply the same procedure. 
	Iterating this procedure until only the root $\al_1$ remains, we then sum over $\al_1$ and obtain one final factor $(W^{d}\eta_t)^{-1}$ (along with additional $\Psi_t$ factors), since $\al_1$ is also incident to at least two external solid edges. This completes the proof of \eqref{kwuyayw_ng_tree} in the case $\max_i k_i \ge 1$.

	It remains to consider the case where $k_i\equiv 0$ for all $i\in\qqq{r}$. Then, by condition \eqref{eq:degali}, each internal vertex is connected to other internal vertices by at least two solid edges. Assume without loss of generality that $\al_1$ and $\al_2$ are connected in $\cal G_{\bal_r}$. Then, we bound the LHS of \eqref{kwuyayw_ng_tree} as 
	\begin{align*}
		\sum_{\bal_r} \cal G_{\bal_r} \prod_{i=1}^r \xi_{\al_i a_i} \le \sum_{\bal_r} \cal G_{\bal_r} \p{\p{\xi_{\al_1 a_1}}^2+\p{\xi_{\al_2 a_2}}^2} \prod_{i=3}^r \xi_{\al_i a_i}.
	\end{align*}
	By symmetry, it suffices to bound the term involving $\p{\xi_{\al_1 a_1}}^2$. In this case, we again determine the nested order using a spanning tree $\mathbb T$ of $\cal G_{\bal_r}$ with root $\al_1$.\footnote{For the term involving $\p{\xi_{\al_2 a_2}}^2$, we use a spanning tree rooted at $\al_2$; the rest of the argument is identical.} 
At each step, when summing over a leaf of $\mathbb T$—say $\al_i$ with $i\notin\{1,2\}$—the vertex $\al_i$ is incident to at least two solid edges: one along the tree and one external edge $(\al_i,a_i)$. 
Summing over $\al_i$ yields a factor of $(W^d \eta_t)^{-1}$ by Cauchy-Schwarz and \eqref{eq:Gbyxi3}, together with additional edges that are bounded by $\Psi_t$ factors. Removing $\al_i$ leaves a smaller spanning tree, and we continue inductively. 

On the other hand, suppose that at some step, we need to sum over a leaf vertex $\al_2$. There are three cases to consider:
\begin{itemize}
	\item[(i)] If $\al_2$ is not the child vertex of $\al_1$, then $\al_2$ is connected by at least two solid edges: one edge in the tree and one external edge $(\al_1, \al_2)$. Summing over these edges gives a factor of $(W^d \eta_t)^{-1}$, by Cauchy-Schwarz and \eqref{eq:Gbyxi3}, along with additional edges that are bounded by $\Psi_t$ factors. 
	
	\item[(ii)] If $\al_2$ is the child vertex of $\al_1$ and there are at least two solid edges between them, we can again sum over $\al_2$ to get a factor of $(W^d\eta_t)^{-1}$, along with some additional $\Psi_t$ factors. 
	
	\item[(iii)] Finally, suppose $\al_2$ is the child vertex of $\al_1$ in the spanning tree $\mathbb T$, but there is only one solid edge between them. 
	In this case, by \eqref{eq:degali} and the fact that $k_1=0$, we know that $\al_1$ must be connected to another internal vertex, say $\al_j$ with $j \notin \{1, 2\}$. Therefore, we can find a new spanning tree, $\mathbb{T}'$, where $\al_1$ has $\al_j$ as its child. Under this configuration, we are effectively reduced to case (i) again: when summing over the vertices from the leaves of $\mathbb{T}'$ toward the root $\al_1$, the vertex $\al_2$ will eventually be summed over as a leaf that is not the child of $\al_1$ at some step.
	
\end{itemize}
The reader may refer to \Cref{Fig:construct_trees} for an illustration of cases (i) and (iii). In the left panel, $\al_2$ is a leaf of the black spanning tree and is \emph{not} the child of $\al_1$. In contrast, in the right panel, we switch the roles of $\al_2$ and $\al_3$, so that $\al_2$ becomes the child of $\al_1$ in the original tree. For this case, we select a different black spanning tree in the right graph, ensuring that $\al_2$ is no longer the child of $\al_1$ in the new tree. 
\begin{figure}[h]
	\centering
	\begin{tikzpicture}[scale=2, every node/.style={circle, fill=black, inner sep=1.5pt}]
		\node (A) at (0,0) {};       
		\node[anchor=south,draw=none,fill=none] at (A) {$\al_1$};
		\node (B) at (0,-0.5) {};  
		\node[anchor=north,draw=none,fill=none] at (B) {$\al_3$};
		\node (C) at (-0.7,-0.9) {};     
		\node (D) at (-1,-1.5) {};  
		\node (E) at (-0.5,-1.5) {};    
		\node (F) at (0.7,-0.9) {};  
		\node[anchor=north,draw=none,fill=none] at (F) {$\al_2$};
		
		\node (A1) [left of = A, distance=4pt] {};
		\node[anchor=south,draw=none,fill=none] at (A1) {$a_1$};
		\node (B1) [left of = B, distance=4pt] {};  
		\node (C1) [left of = C, distance=4pt] {};
		\node (D1) [left of = D, distance=4pt] {};
		\node (E1) [right of = E, distance=4pt] {};
		
		\draw[black, thick] (A) -- (B) -- (C) -- (D);
		\draw[black, thick] (C) -- (E);
		\draw[black, thick] (B) -- (F);
		\path [draw=blue, thick] (A) to [bend left] (A1); 
		\path [draw=blue, thick] (A) to [bend right] (A1); 
		\draw[blue, thick] (B) -- (B1);
		\draw[blue, thick] (C) -- (C1);
		\draw[blue, thick] (D) -- (D1);
		\draw[blue, thick] (E) -- (E1);
		\draw[purple, thick] (A) -- (F);
		
	\end{tikzpicture}
	\hspace{10pt}
	\begin{tikzpicture}[scale=2, every node/.style={circle, fill=black, inner sep=1.5pt}]
		\node (A) at (0,0) {};       
		\node[anchor=south,draw=none,fill=none] at (A) {$\al_1$};
		\node (B) at (0,-0.5) {};  
		\node[anchor=north,draw=none,fill=none] at (B) {$\al_2$};
		\node (C) at (-0.7,-0.9) {};     
		\node (D) at (-1,-1.5) {};  
		\node (E) at (-0.5,-1.5) {};    
		\node (F) at (0.7,-0.9) {};
		\node[anchor=north,draw=none,fill=none] at (F) {$\al_3$};
		
		\node (A1) [left of = A, distance=4pt] {};
		\node[anchor=south,draw=none,fill=none] at (A1) {$a_1$};
		\node (C1) [left of = C, distance=4pt] {};
		\node (D1) [left of = D, distance=4pt] {};
		\node (E1) [right of = E, distance=4pt] {};
		\node (F1) [right of = F, distance=4pt] {};

		\draw[purple, thick] (A) -- (B);
		\draw[black, thick] (A) -- (F);
		\draw[black, thick] (C) -- (D);
		\draw[black, thick] (C) -- (B);
		\draw[black, thick] (C) -- (E);
		\draw[black, thick] (B) -- (F);
		\path [draw=blue, thick] (A) to [bend left] (A1); 
		\path [draw=blue, thick] (A) to [bend right] (A1); 
		\draw[blue, thick] (C) -- (C1);
		\draw[blue, thick] (D) -- (D1);
		\draw[blue, thick] (E) -- (E1);
		\draw[blue, thick] (F) -- (F1);
		
	\end{tikzpicture}
	\caption{Illustration of the spanning trees in case (i) (left panel) and case (iii) (right panel). Blue edges represent external edges connected to external vertices (here blue is used solely for illustration and does not indicate the $+$ charge of $G$-edges as in \Cref{def_graph1}); purple edges represent the solid edges between $\al_1$ and $\al_2$; and black edges represent the spanning tree on the internal vertices.}\label{Fig:construct_trees}
\end{figure}

After summing over all the non-root vertices $\al_2, \dots, \al_r$, we are left with the vertex $\al_1$ connected to two solid edges, $(\al_1, a_1)$. Summing over this yields a factor of $(W^d \eta_t)^{-1}$, by \eqref{eq:Gbyxi3}. This completes the proof of \eqref{kwuyayw_ng_tree} in the case $\max_i k_i =0$. 
\end{proof}

\subsection{Proof of Lemma \ref{lem:LW_moment_exp}}\label{subsec_pf_LW_moment_exp}

Note that when $|a-b|\le (\log W)^{3/2}\ell_t$ or $\ell\le (\log W)^{3/2}\ell_t$, $\sT_{t}(|a-b|\wedge \ell)$ does not exhibit exponential decay. In this regime, \Cref{lem:LW_moment_exp} follows directly from \Cref{lem:LW_moment}. Therefore, for the remainder of the proof, it suffices to establish \eqref{eq:LW_moment_exp} under the conditions  
\be\label{eq:far_ab_K}
|a-b|> (\log W)^{3/2}\ell_t,\quad \text{and}\quad (\log W)^{3/2}\ell_t\le \ell \le (\log W)^{10}\ell_t.
\ee
For this purpose, we decompose $f_{xy}(G)$ into two parts:
\begin{align*}
f_{xy}^{>\ell}(G)&=W^{-d}\sum_{a_1:|a_1-a|\vee |a_1-b|>\ell}\sum_{a_2}S^{\LK}_{a_1a_2}\sum_{\substack{\al\in[a_2],\beta \in[a_1],\\ \al\notin\{x,y\}}}
\Gc_{\beta\beta} G_{x\al }G_{\al y},\\
f_{xy}^{\le \ell}(G)&=W^{-d}\sum_{a_1:|a_1-a|\vee |a_1-b|\le \ell}\sum_{a_2}S^{\LK}_{a_1a_2}\sum_{\substack{\al\in[a_2],\beta \in[a_1],\\ \al\notin\{x,y\}}}
\Gc_{\beta\beta} G_{x\al }G_{\al y}. 
\end{align*}
\Cref{lem:LW_moment_exp} follows immediately from the next two estimates on \(f_{xy}^{>\ell}(G)\) and \(f_{xy}^{\le \ell}(G)\) with $\Psi_t = (W^{-d}B_{t,0})^{1/2}$:

\begin{lemma}\label{lem:LW_moment_exp_far}
In the setting of \Cref{lem:LW_moment_exp}, for any fixed $p\in 2\N$, the following estimate holds for any large constant $D>0$:
\be\label{eq:LW_moment_exp_far}
\E\big|f_{xy}^{>\ell}(G)\big|^{p}\prec \frac{1}{\eta_t^p} \Psi_t^p \cdot \br{\sT_{t}(\ell) }^{p}+W^{-D}. 
\ee
\end{lemma}
\begin{lemma}\label{lem:LW_moment_exp_near}
In the setting of \Cref{lem:LW_moment_exp}, for any fixed $p\in 2\N$, the following estimate holds for any large constant $D>0$:
\be\label{eq:LW_moment_exp_near}
\E\big|f_{xy}^{\le \ell}(G)\big|^{p}\prec \frac{1}{\eta_t^p} \Psi_t^p \cdot \br{\sT_{t}(|a-b|\wedge \ell) }^{p}+W^{-D}. 
\ee
\end{lemma}

First, the proof of \Cref{lem:LW_moment_exp_far} is similar to that for \Cref{lem:LW_moment}.
\begin{proof}[\bf Proof of \Cref{lem:LW_moment_exp_far}]
First, similar to \Cref{lem:localregular}, we expand $\big|f_{xy}^{>\ell}(G)\big|^{p}$ into a sum of locally standard graphs satisfying properties (1)--(6) listed there. Next, we bound each locally standard graph using its associated auxiliary graph, constructed as in \Cref{def_auxgraph}.
We then estimate the auxiliary graphs using a result analogous to \Cref{lem:Anp_key}. More precisely, suppose \smash{$\cal G^{\aux}_{\ba\bfb}$} is a nested graph that satisfies the assumptions of \Cref{lem:Anp_key}. In addition, assume that all internal vertices lie in the domain  $\bD_{>\ell}:=\{c \in \Zn:|a_i-c|\vee |b_i-c|> \ell\}$. Then, we aim to show that 
\begin{equation}\label{adsuu33}
	\cal G^{\aux}_{\ba\bfb} \prec (W^d\eta_t)^{-q}\cdot \Psi_t^{\ord(\cal G^{\aux}_{\ba\bfb})-p}\cdot \br{\sT_{t}(\ell) }^{p}+W^{-D} .
\end{equation}
In the proof of \Cref{lem:Anp_key_gh}, each path contributed at least one factor of $\Psi_t(c|a_i-b_i|)$ because, roughly speaking, every path contains at least one long ending edge. In the current setting, we use the fact that each path $\fP_i$ must contain at least one ``long" ending edge of length $>\ell$, since all internal vertices are restricted to lie in the domain $\bD_{>\ell}$. 
Each such long edge provides a factor of $\sT_t(\ell)$, and we can treat it as a type-A2 edge. Replacing these A2 edges with ghost edges in all paths $\fP_i$, $i\in \qqq{p}$, yields a total factor of $\br{\sT_{t}(\ell)}^{p}$. For the remaining graph---denoted by \(\cal G^{\new}_{\ba\bfb}\)---we need to establish the following bound for any constant $D>0$:
\begin{equation}\label{adsuu44}
	\cal G^{\new}_{\ba\bfb} \prec (W^d\eta_t)^{-q}\cdot \Psi_t^{\ord(\cal G^{\new}_{\ba\bfb})} +W^{-D} .
\end{equation}
The proof of this estimate follows exactly the same argument as in \Cref{lem:Anp_key_gh}, and we omit the details.
\end{proof}

For the proof of \Cref{lem:LW_moment_exp_near}, we need a new argument that makes use of the exponential decay in $\sT_t$. 
\begin{proof}[\bf Proof of \Cref{lem:LW_moment_exp_near}]
Similar to \Cref{lem:localregular}, we expand $\big|f_{xy}^{\le \ell}(G)\big|^{p}$ into a sum of locally standard graphs satisfying properties (1)--(6) listed there. 
Next, we bound each locally standard graph using its associated auxiliary graph, constructed as in \Cref{def_auxgraph}. This yields a class of auxiliary graphs \smash{$\cal G^{\aux}_{[a][b]}$} satisfying the assumptions of \Cref{lem:Anp_key} with $[a_i]=[a]$ and $[b_i]=[b]$ for $i\in \qqq{p}$. Moreover, all internal vertices of the graph lie in the domain $\bD_{\le \ell}:=\{c \in \Zn:|a-c|\vee |b-c|\le \ell\}$. It remains to bound such graphs as follows: 
\begin{equation}\label{adsuu_exp}
	\cal G^{\aux}_{[a][b]} \prec (W^d\eta_t)^{-q}\cdot \Psi_t^{\ord(\cal G^{\aux}_{[a][b]})-p}\cdot \br{\sT_{t}(|[a]-[b]|\wedge \ell)}^{p}+W^{-D} .
\end{equation}

We first control \smash{$\cal G^{\aux}_{[a][b]}$} by bounding each solid edge, say $\xi([\al],[\beta])$, by its upper bound $\sT_t(|[\al]-[\beta]|\wedge \ell)+W^{-D}$. 
Discarding the negligible error term containing $W^{-D}$ factors, we obtain a new graph, denoted by \smash{${\mathscr G}_{[a][b]}$}, which has the same graphical structure as \smash{$\cal G^{\aux}_{[a][b]}$} but with each solid edge representing a $\sT_t$ factor instead. Hence, to prove \eqref{adsuu_exp}, it suffices to establish that
\begin{equation}\label{adsuu_exp2}
	\mathscr G_{[a][b]} = \sum_{\bal=([\al_1],\ldots,[\al_q])\in (\bD_{\le \ell})^q}\mathscr G_{[a][b]}(\bal) \prec (W^d\eta_t)^{-q}\cdot \Psi_t^{\ord(\mathscr G_{[a][b]})-p}\cdot \br{\sT_{t}(|[a]-[b]|\wedge \ell)}^{p} ,
\end{equation}
where, recall, \smash{$\mathscr G_{[a][b]}(\bal)$} denotes the graph obtained by fixing the external vertices to $\bal$, and $\ord(\mathscr G_{[a][b]})$ is defined in \eqref{eq:ordGaux}.
To show \eqref{adsuu_exp2}, we sum over the internal vertices in $\bal$ one by one. Unlike in the proof of \Cref{lem:Anp_key}, the order of summation here is arbitrary; for definiteness, we follow the order $[\al_1],\ldots,[\al_q]$. At each step, we apply the following key estimate.
\begin{claim}\label{claim:TTk}
	For any $k \ge 2$, the following bound holds with $\Psi_t=(W^{-d}B_{t,0})^{1/2}$:\footnote{This estimate can be viewed as an extension of \Cref{lem:propT} in the regime $1-t \ge \ilambda^2/L^2$, except that here we obtain a ``$\prec$'' bound, rather than the ``$\lesssim$'' bound in \eqref{TTT2}, due to the presence of additional logarithmic factors.}
	\begin{align}\label{eq:key_T_reudce}
		\sum_{[\al]\in \bD_{\le \ell}} \prod_{i=1}^k \br{\sT_{t}(|[x_i]-[\al]|\wedge \ell) \cdot \sT_{t}(|[y_i]-[\al]|\wedge \ell) } \prec (W^{d}\eta_t)^{-1} \cdot {\Psi_t^{k-2}} \cdot   \prod_{i=1}^k \sT_{t}(|[x_i]-[y_i]|\wedge \ell) .
	\end{align}
\end{claim}

We defer the proof of \Cref{claim:TTk} to \Cref{pf:claim_TTk}, where it is derived from elementary calculus estimates. 
The estimate \eqref{eq:key_T_reudce} shows that when summing over an internal vertex $[\al]$, each path of the form $([x_i]\to [\al]\to [y_i])$ consisting of two solid edges through $[\al]$ is effectively replaced by a single solid edge $[x_i]\to [y_i]$ in the resulting graph. We refer to this as the ``path-preserving phenomenon". In addition, the summation over $[\al]\in \bD_{\le \ell}$ contributes a factor $(W^d\eta_t)^{-1}$ from each such pair of solid edges, along with the factor $\Psi_t^{k-2}$ reflecting the reduction in scaling order. Note that the case $k=2$ is critical---the estimate \eqref{eq:key_T_reudce} fails when $k=1$. Figure~\ref{Fig:construct_keepT} illustrates the path-preserving phenomenon in the critical case $k=2$.

\begin{figure}[h]
	\centering
	\begin{tikzpicture}
		\coordinate (L) at (-2, 0);
		\coordinate (x1) at (-1.3, 0.3);
		\coordinate (x2) at (-1.3,- 0.3);
		\coordinate (a1) at (0, 0.5);
		\coordinate (a2) at (0, -0.5);
		\coordinate (al) at (1, 0);
		\coordinate (b1) at (2, 0.5);
		\coordinate (b2) at (2, -0.5);
		\coordinate (d1) at (3.5, 0.4);
		\coordinate (d2) at (3.5, -0.4);
		\coordinate (R) at (4.8, 0);
		
		\fill (L) circle (1.5pt) node[left=2pt]{$[a]$};
		\fill (x1) circle (1.5pt) node[above=2pt]{};
		\fill (x2) circle (1.5pt) node[below=2pt]{};
		\fill (a1) circle (1.5pt) node[above=2pt]{$[x_1]$};
		\fill (a2) circle (1.5pt) node[below=2pt]{$[x_2]$};
		\fill (b1) circle (1.5pt)  node[above=2pt]{$[y_1]$};
		\fill (b2) circle (1.5pt) node[below=2pt]{$[y_2]$};
		\fill (al) circle (1.5pt) node[above=2pt]{$[\al]$};
		\fill (d1) circle (1.5pt);
		\fill (d2) circle (1.5pt);
		\fill (R) circle (1.5pt) node[right=2pt]{$[b]$};
		
		\draw (L) -- (x1); 
		\draw (L) -- (x2); 
		\draw (a1) [dotted]-- (x1); 
		\draw (a2) [dotted]-- (x2); 
		\draw (a1) -- (al); 
		\draw (a2) -- (al); 
		\draw (al) -- (b1); 
		\draw (b1) [dotted]-- (d1); 
		\draw (d1) -- (R); 
		\draw (a2) [dashed]-- (al);
		\draw (al) -- (b2);
		\draw (b2) [dotted]-- (d2); 
		\draw (d2) -- (R); 
	\end{tikzpicture}
	\hspace{10pt}
	\begin{tikzpicture}
		\coordinate (L) at (-2, 0);
		\coordinate (x1) at (-1.3, 0.3);
		\coordinate (x2) at (-1.3,- 0.3);
		\coordinate (a1) at (0, 0.5);
		\coordinate (a2) at (0, -0.5);
		\coordinate (b1) at (1.5, 0.5);
		\coordinate (b2) at (1.5, -0.5);
		\coordinate (d1) at (3, 0.4);
		\coordinate (d2) at (3, -0.4);
		\coordinate (R) at (4.3, 0);
		
		\fill (L) circle (1.5pt) node[left=2pt]{$[a]$};
		\fill (x1) circle (1.5pt) node[above=2pt]{};
		\fill (x2) circle (1.5pt) node[below=2pt]{};
		\fill (a1) circle (1.5pt) node[above=2pt]{$[x_1]$};
		\fill (a2) circle (1.5pt) node[below=2pt]{$[x_2]$};
		\fill (b1) circle (1.5pt)  node[above=2pt]{$[y_1]$};
		\fill (b2) circle (1.5pt) node[below=2pt]{$[y_2]$};
		\fill (d1) circle (1.5pt);
		\fill (d2) circle (1.5pt);
		\fill (R) circle (1.5pt) node[right=2pt]{$[b]$};
		
		\draw (L) -- (x1); 
		\draw (L) -- (x2); 
		\draw (a1) [dotted]-- (x1); 
		\draw (a2) [dotted]-- (x2); 
		\draw (a1) -- (b1); 
		\draw (a2) -- (b2); 
		\draw (b1) [dotted]-- (d1); 
		\draw (d1) -- (R); 
		\draw (b2) [dotted]-- (d2); 
		\draw (d2) -- (R); 
	\end{tikzpicture}
	\caption{Illustration of the path-preserving phenomenon in applying \eqref{eq:key_T_reudce}.}\label{Fig:construct_keepT}
\end{figure}


We perform the summations of $\mathscr G_{[a][b]}$ over the internal vertices by generating a sequence of new graphs that consistently satisfy conditions (1)--(3) in \Cref{lem:Anp_key} with $[a_i]\equiv [a]$ and $[b_i]\equiv [b]$. To illustrate this procedure, consider the summation over $[\al_1]$. Suppose there are $k_1$ two-edge paths passing through $[\al]$ in $\mathscr G_{[a][b]}$. More precisely, assume that each path $\fP_i$ can be decomposed as
\[([a_i]\too [x_{1,1}]\to [\al_1] \to [y_{1,1}] \too [x_{1,2}]\to[\al_1]\to [y_{1,2}]\too \cdots \too [x_{1,r_{i}}]\to[\al_1]\to [y_{1,r_{i}}]\too [b_i]),\]
where we omit intermediate vertices on the path and only list those vertices $[x_{1,i}]$ and $[y_{1,i}]$ that are directly connected to $[\al_1]$, and $r_i\ge 0$ is a non-negative integer. Then, we apply \eqref{eq:key_T_reudce} to the summation 
\begin{align}\nonumber
	\sum_{[\al_1]\in \bD_{\le \ell}} \prod_{i=1}^p\prod_{j=1}^{r_i}  \br{\sT_{t}(|[x_{i,j}]-[\al_1]|\wedge \ell) \cdot \sT_{t}(|[y_{i,j}]-[\al_1]|\wedge \ell)} ,
\end{align}
which yields the bound
\[\sum_{[\al_1]}\mathscr G_{[a][b]}(\bal) \prec (W^d\eta_t)^{-1} \Psi_t^{k_1-2} \cdot \mathscr G^{(1)}_{[a][b]}([\al_2],\ldots,[\al_q]) ,\] 
where $k_1$ is defined as $k_1=\sum_{i=1}^p r_i$, 
and \smash{$\mathscr G^{(1)}_{[a][b]}$} is a new graph obtained by removing the vertex $[\al_1]$ from $\mathscr G_{[a][b]}$ and replacing each pair of solid edges $([x_{i,j}]\to[\al_1]\to [y_{i,j}])$ by $([x_{i,j}], [y_{i,j}])$.  
It is easy to see that the new graph \smash{$\mathscr G^{(1)}_{[a][b]}$} also satisfies conditions (1)--(3) in \Cref{lem:Anp_key}. Next, summing over the vertex $[\al_2]$ in \smash{$\mathscr G^{(1)}_{[a][b]}$} gives another graph \smash{$\mathscr G^{(2)}_{[a][b]}([\al_3],\ldots,[\al_q])$} that again satisfies conditions (1)--(3) in \Cref{lem:Anp_key}, along with a factor $(W^d\eta_t)^{-1} \Psi_t^{k_2-2}$ for some $k_2\ge 2$. 
Continuing this procedure, after summing over all internal vertices, we can bound the LHS of \eqref{adsuu_exp2} as  
\[\sum_{\bal\in (\bD_{\le \ell})^q}\mathscr G_{[a][b]}(\bal)\prec (W^d\eta_t)^{-q}\cdot \Psi_t^{\ord(\mathscr G_{[a][b]})-p}\cdot \mathscr G^{(q)}_{[a][b]},\]
where \smash{$\mathscr G^{(q)}_{[a][b]}$} is a graph consisting solely of $p$ solid edges between $[a]$ and $[b]$, with no internal vertices. 
Since each solid edge is bounded by $\sT_{t}(|[a]-[b]|\wedge \ell)$, this yields \eqref{adsuu_exp2}. Combining this with \eqref{G_by_auxG} completes the proof of \Cref{lem:LW_moment_exp_near}. 
\end{proof}

\section{Extension to the block Anderson model}\label{sec:ext-to-BA}

The proof of \Cref{MR:decol_BA} for the block Anderson model is based on the following flow framework.

\begin{lemma}[Lemma 3.3 of \cite{RBSO1D}]\label{zztE_BA}
For the block Anderson model, fix any $\ilambda>0$ and $z\in \mathbb C_+$ with $\im z\in (0, 1]$ and $|\re z|\le  2-\kappa$. We choose 
\be\label{eq:t0E0_BA}t_0=\frac{\im m(z,\ilambda)}{\im m(z,\ilambda)+ \im z},\quad E= \frac{t_0 \re z -(1-t_0)\re m(z,\ilambda)}{\sqrt{t_0}},\quad \ilambda_0=\sqrt{t_0}\ilambda \, .\ee
Then, we have that 
\begin{equation}\label{eq:zztE_BA}
	\sqrt{t_0}m(E,\ilambda_0)=m(z,\ilambda), \ \  z_{t_0}(E,\ilambda_0) = \sqrt{t_0}z, \ \  \sqrt{t_0}M(E,\ilambda_0)=M(z,\ilambda), \ \ G(z,\ilambda) \stackrel{d}{=} \sqrt{t_0} G_{t_0;E,\ilambda_0},
\end{equation} 
where recall that $G(z,\ilambda)=(H-z)^{-1}=(V+\ilambda\Psi-z)^{-1}$. 
\end{lemma}

In the following proof, we fix a target spectral parameter $z=\hat{E}+\ii \eta\in \mathbf D_{\kappa,\e}$ for an arbitrarily small constant $\e>0$, where the spectral domain $\mathbf D_{\kappa,\e}$ is now defined as 
\be\label{eq:spectral_domainBA}
\mathbf D_{\kappa,\e}:=\{z=\hat{E}+\ii\eta \in \C_+: |\hat{E}|\le e_\ilambda -\kappa, N^{-1+\e}\le \eta\le 1\}.
\ee
Accordingly, we choose the parameters $t_0$, $E$, and $\ilambda_0$ as specified in \eqref{eq:t0E0_BA}. 
Again, for simplicity of presentation, unless we want to emphasize the dependence on the flow parameters $E$ and $g_0$, we will omit them from various notations, such as $z_t(E,\ilambda_0)$, $E_t(E,\ilambda_0)$, $\eta_t(E,\ilambda_0)$, $m(E,\ilambda_0)$, $M(E,\ilambda_0)$, and $G_{t;E,\ilambda_0}\equiv G_t$. 
In the flow framework of \Cref{zztE_BA}, we can establish an analogue of \Cref{lem:main_ind} for the block Anderson model.

\begin{theorem}\label{lem:main_ind_BA} 
In the setting of \Cref{MR:decol_BA}, fix any $z=\hat{E}+\ii \eta\in \mathbf D_{\kappa,\e}$ and consider the flow framework in \Cref{zztE_BA}. Suppose the estimates \eqref{Eq:L-KGt+IND}--\eqref{Eq:Gtlp_exp+IND} hold at some fixed $s\in [0,t_0]$. Then, there exists a constant $0<\fc_d \le 10^{-2}$ such that for any $s< t < 1$ satisfying \eqref{con_st_ind}, the estimates \eqref{Eq:L-KGt}--\eqref{Gt_bound} hold. In addition, if $1-t\ge \ilambda^2$, then the estimate \eqref{Eq:Gdecay+s<g} holds.
\end{theorem}

With \Cref{lem:main_ind_BA} in hand, we can establish \Cref{MR:decol_BA} by induction on $t$.  

\begin{proof}[\bf Proof of \Cref{MR:decol_BA}]
Fix $z=\hat{E}+\ii \eta\in \mathbf D_{\kappa,\e}$, and choose the flow as in \Cref{zztE_BA}. By \Cref{lem:main_ind_BA}, applying induction on $t$ from $t=0$ to $t=t_0$ yields the estimates \eqref{Eq:L-KGt}--\eqref{Gt_bound} at $t=t_0$. Using \eqref{eq:zztE_BA}, \eqref{eq:BtBt}, and \eqref{Kn2sol}, we see that these estimates together imply the entrywise local law \eqref{G_bound}, the averaged local law \eqref{G_bound_ave}, and the quantum diffusion estimates \eqref{eq:diffu1}--\eqref{Meq:QdS2} for each fixed $z$. To extend these estimates uniformly to all $z\in \mathbf D_{\kappa,\e}$, we apply a standard $N^{-C}$-net and perturbation argument. Finally, the delocalization estimate \eqref{eq:psikLinfty}, the QUE estimates \eqref{Meq:QUE} and \eqref{Meq:QUE2}, and the bulk universality \eqref{eq:universality} then follow as consequences, as shown in \Cref{subsec:main}.
\end{proof}

The proof of \Cref{lem:main_ind_BA} follows the same six-step strategy outlined below \Cref{lem:main_ind}. We now explain how the arguments in \Cref{Sec:Steps12}--\ref{Sec:graph} for the random band matrix model extend to the block Anderson model. Our proof relies on the properties of the $\Theta$-propagators stated in \Cref{lem_propTH}, together with the following properties of $m(z)$ (defined in \eqref{self_m}) and the matrix $M^{\LK}$ (defined in \eqref{def_G0}).

\begin{lemma}
\label{lem:propM}
For the block Anderson model, under the condition \eqref{eq:WO}, $m$ and $M^{\LK}$ satisfy the following properties for any $|E|\le e_g - \kappa$ (recall that $E$ is the spectral parameter in the flow \eqref{eq:t0E0_BA}):

\begin{enumerate}
	
	\item[(1)] {\bf Translation invariance}: For any $a,b,r\in \Zn$, we have $M^{\LK}_{a+r,b+r}= M^{\LK}_{ab}$ and $M^{\LK}_{aa}\equiv m$.

	\item[(2)] {\bf Ward's identity}: We have \(|m|\le 1\), 
	\(\im m\gtrsim 1\), and 
	\be\label{eq:WardM}
	\sum_{b}|M^{\LK}_{ab}|^2=1,\quad \forall a\in \Zn \, .
	\ee
		
	\item[(3)] {\bf Combes–Thomas bound}: There exists a constant $C>0$ (depending only on $d$ and $\kappa$) such that the following estimate holds for $\ilambda<(2C)^{-1}$: 
	\begin{align}		\label{Mbound_AO}
		C^{-1}\ilambda \mathbf 1(a\sim b)\le |M^{\LK}_{ab}|\le (C\ilambda)^{|a-b|}. 
	\end{align}
	Furthermore, for $\ilambda\ge (2C)^{-1}$, there exists a constant $c>0$ such that the following  bound holds:
	\begin{align}\label{Mbound_AO2}
		|M^{\LK}_{ab}|\le c^{-1}\exp\p{-c|a-b|}. 
	\end{align}
\end{enumerate}
\end{lemma}
\begin{proof}
Property (1) follows directly from the translation invariance of the matrix $\Psi^{\LK}$ in \eqref{eq:Psi3D}. The bound $\im m \gtrsim 1$ is a consequence of \cite[Lemma 3.5]{LeeSchSteYau2015}, while Ward’s identity \eqref{eq:WardM} follows by taking the imaginary part of the equation \smash{$m=M^{\LK}_{aa}=(\ilambda \Psi^{\LK} - E - m)^{-1}_{aa}$}. The identity \eqref{eq:WardM} gives directly $|m|\le 1$. 
For sufficiently small $\ilambda$, the estimates in \eqref{Mbound_AO} are obtained from the Taylor expansion 
\[ M^{\LK} = -\sum_{k=0}^\infty (E+m)^{-k-1}(\ilambda \Psi^{\LK})^k.\]
For $\ilambda$ of order 1, \eqref{Mbound_AO2} is given by the classical Combes–Thomas estimate (see, e.g., \cite[Theorem 10.5]{Aizenman_book}).
\end{proof}

Another key ingredient in the proof of \Cref{lem:main_ind_BA} is the following analogue of \Cref{lem_GbEXP}. 

\begin{lemma}\label{lem_GbEXP_BA}
In the setting of \Cref{lem:main_ind_BA}, suppose the estimates in \eqref{initialGT2} hold for a deterministic control parameter $W^{-d/2}\le \Psi_t \le W^{-\e_0}$. 
Furthermore, suppose there exist deterministic control parameters $0<\Phi_t(a,b)\le W^{-\e_0}$ such that  
\be\label{eq:def_Psit}
{\cal L}^{(2)}_{t, (-,+), (a,b)}\prec [\Phi_t(a,b)]^2,\quad \forall a,b\in \Zn .
\ee
\begin{itemize}
	\item[(a)] {\bf Local laws}: The following entrywise and averaged local laws hold: 
	\begin{align}\label{GiiGEX_BA} 
		\|G_t-M\|_{\max} \prec \Psi_t, \quad \max_a\left|\tr \p{ \left(G_t - M\right)E_{a}}\right| \prec \Psi_t^2  . 
	\end{align}
		
	\item[(b)] {\bf Entrywise decay estimate}:
	There exists a constant $c_{\ilambda}>0$ such that, for any large constant $D>0$ and all \(a,b\in \Zn \), the following estimate holds:
	\begin{align}\label{GijGEX_BA}
		\max_{x\in [a], y \in [b]} |(G_t-M)_{xy}|  \prec & \sum_{a', b' \in \Zn}  \Phi_t(a',b') e^{-c_\ilambda(|a'-a|+|b'-b|)}  + \Psi_t e^{-c_{\ilambda}|a-b|}+W^{-D} .
	\end{align}
\end{itemize}
\end{lemma}
\begin{proof}
This lemma was proved as Lemma 6.1 in \cite{RBSO1D} under the condition $\ilambda\le W^{-\e}$ for a small constant $\e>0$. The same arguments, however, apply verbatim to our setting with the bounds in \eqref{Mbound_AO} and \eqref{Mbound_AO2}.
\end{proof}

\noindent{\bf Proof of Step 1 for \Cref{lem:main_ind_BA}.}
The proof of Step 1 depends on the following continuity estimates in \Cref{lem_ConArg_BA}. To make the dependence on the spectral parameter $z$ and the coupling parameter $\ilambda$ explicit, we denote the resolvent by \(G_t(z_t, \ilambda)=(V_t+\ilambda\Psi-z_t)^{-1}\) and the corresponding $G$-loop by \smash{${\cal L}^{(\fn)}_{t, \boldsymbol{\sigma}, \ba}(z_t, \ilambda)$}, where $z_t\equiv z_t(E,\ilambda)$ is defined in \eqref{eq:zt}. 

\begin{lemma}\label{lem_ConArg_BA}
Fix any \( \e\le s \leq t \leq 1\) for a constant \(\e > 0\). Given any $\ilambda$ satisfying the condition \eqref{eq:WO}, let $\ilambda_s:=\ilambda\cdot \sqrt{s/t}$. Then, in the flow setting given by \Cref{def_flow} and \Cref{zztE_BA}, we have the following continuity estimates for the $G$-loops and (generalized) resolvent entries. 
\begin{enumerate}
	
	\item Assume that the bound \eqref{55} holds at time \(s\) for the loops ${\cal L}^{(\fn)}_{s, \boldsymbol{\sigma}, \ba}\p{z_s,\ilambda_s}$ for each fixed \(\fn \in \mathbb{N}\).  
	Then, for any $\fn\ge 2$, we have 
	\begin{align}\label{res_lo_bo_eta_BA}
		\max_{\boldsymbol{\sigma}, \ba} 
		\left|{\cal L}_{t, \boldsymbol{\sigma}, \ba}^{(\fn)}\p{z_t,\ilambda}\right| \prec
		\left(\frac{\eta_{s}}{\eta_{t}}\cdot W^{-d}B_{s,0}\right)^{\fn-1} {\max_{a}\tr\p{ \im G_t(z_t,\ilambda) E_{a}}} .
	\end{align}
	
	\item Given any deterministic unit vectors $\bv,\bw\in \C^N$, suppose 
	\be\label{eq:ImGs}\im (G_s)_{\bv\bv}(z_s,\ilambda_s)\lesssim 1,\quad \im (G_s)_{\bw\bw}(z_s,\ilambda_s)\lesssim 1, \quad \left|(G_s)_{\bv\bw}(z_s,\ilambda_s)\right|\lesssim 1,
	\ee
	with high probability, where we adopt the simplified notation of generalized matrix entries: given a matrix $\cal A$ and any vectors $\bv,\bw$, we denote $ \cal A_{\bv\bw}:= \bv^* \cal A\bw$. Then, the following estimates hold with high probability at time $t$: 
	\be\label{eq:ImGt}
	\im (G_t)_{\bv\bv}(z_t,\ilambda)\lesssim {\eta_s}/{\eta_t},\quad (G_t)_{\bv\bw}(z_t,\ilambda)\lesssim {\eta_s}/{\eta_t}.
	\ee
	
\end{enumerate}

\end{lemma}
\begin{proof}
The proof of this lemma follows exactly the same argument as that of Lemma 7.1 in \cite{RBSO1D}, except that the factor $W^{-d}B_{s,0}$ in \eqref{res_lo_bo_eta_BA} corresponds to $(W^d\ell_s^d\eta_s)^{-1}$ therein.
\end{proof}

With \Cref{lem_GbEXP_BA,lem_ConArg_BA}, Step 1 of the proof of \Cref{lem:main_ind_BA} for the block Anderson model (i.e., the proof of \eqref{lRB1} and \eqref{Gtmwc}) is the same as that in \cite[Section 7.1]{RBSO1D}. Hence, we omit the details. 

\medskip
\noindent{\bf Proof of Step 2 for \Cref{lem:main_ind_BA}.} 
In Step 2, the technical lemmas—\Cref{lem:newKLK,lem:LWterm,lem: EWGn2_N,lem: EMn2_N}—remain valid in the setting of the block Anderson model, except that the estimate \eqref{eq:MG_conclusion3} is modified as follows. Let ${\cal J}^{\ell}_{t,D}\ge W^{-d}$ be a \emph{deterministic} control parameter such that \smash{\(\wh{\cal J}^{\ell}_{t,D} \prec {\cal J}^{\ell}_{t,D}.\)}
Then, \eqref{eq:MG_conclusion3} continues to hold with \smash{$\wh{\mathcal{J}}_{t,D}^{\ell}$} replaced by \smash{${\mathcal{J}}_{t,D}^{\ell}$}; that is, 
\be\label{eq:MG_conclusion3_BA}
(\mathcal{E} \otimes \mathcal{E})^{M}_{t,\bsig, \ba,\ba} \prec \frac{1}{\eta_t} \left[\left(W^{-d}B_{t,0}\right)^{1/2} +  \left({\mathcal{J}}_{t,D}^{\ell}\right)^3\right]\cdot \left(W^{-d} \wT^{\ell}_{t,D}(|a-b|)\right)^2.
\ee
Assume that \eqref{Eq:Gdecay_w} holds. We combine \eqref{Kn2sol}, \eqref{Mbound_AO} (or \eqref{Mbound_AO2}), and \eqref{prop:ThfadC} to obtain \eqref{eq:kn2sol_decay}. Together with \eqref{Eq:Gdecay_w}, this yields \eqref{eq:L2_decay}. Then, applying \Cref{lem_GbEXP_BA} gives the desired estimates \eqref{Gt_bound_flow} and \eqref{Gt_avgbound_flow}.
Finally, we prove \eqref{Eq:Gdecay_w}. Using the same argument as between \eqref{lokis2} and \eqref{eq:L-K2max}, we obtain that for all $u\in[s,t]$,
\be\label{eq:L-K2max_BA} \max_{\bsig,\ba}\left|\left({\cal L}-{\cal K}\right)^{(2)}_{u,\bsig,\ba}\right|\prec 
\p{\frac{1-s}{1-u}}^{C_0+\frac54}(W^{-d}B_{u,0})^{\frac 5 4} \le (W^{-d}B_{u,0})^{\frac1 5}\cdot \p{W^{-d} \widetilde{\cal {T}}^{\ell^{(0)}}_{u,D}\p{|a-b|}},
\ee
where $\ell^{(0)}=0$ as chosen in \eqref{ksjjuw}. Next, we implement a similar inductive argument as that below \eqref{ksjjuw}: assume that \eqref{kwr3juw} holds for length scales $K_u$ satisfying \eqref{eq:monotone_Ku}.
Moreover, suppose we have the initial estimate
\be\label{eq:Gronwall_dervJuD_BA}
\wh{\mathcal{J}}_{u,D}^{K_u} \prec \mathcal{J}_{u,D}^{K_u}, \ \ \ \forall u\in [s,t], \quad \text{where}\quad \mathcal{J}_{u,D}^{K_u}\equiv (W^{-d}{B}_{t,0})^{\frac{1}{50}}.    
\ee
We then define the stopping time 
\be\label{eq:def2_stopping_BA} 
\tau :=t\wedge T,\quad \text{with}\quad T:=\inf\left\{u\ge s: \wh{\cal J}_{u,D}^{K_u} \ge (W^{-d}B_{u,0})^{\frac1{100}} \right\}.\ee
By \Cref{lem: EWGn2_N}, the estimate \eqref{eq:boundGcterm} remains valid.
On the other hand, using \eqref{eq:MG_conclusion3_BA} together with \Cref{lem:DIfREP}, the estimate \eqref{eq:boundmgterm} becomes
\be\label{eq:boundmgterm_BA}
\begin{split}
\int_s^\tau \dd \mathcal{E}^{M}_{u, \boldsymbol{\sigma}, {\ba}} &\prec  \Big[\left(W^{-d}B_{\tau,0}\right)^{1/4}  + \sup_{u\in[s,\tau]}\p{{\cal J}_{u,D}^{K_u}}^{3/2}\Big]\cdot \left(W^{-d} \wT^{K_\tau}_{\tau,D}(|a-b|)\right) .\end{split}
\ee
Applying the same reasoning as below \eqref{eq:boundmgterm} and invoking Grönwall’s inequality, we obtain
\be\label{eq:Gronwall_dervJuD_BA2}
\wh{\mathcal{J}}_{u,D}^{K_u} \prec\left(\frac{1-s}{1-u}\right)^{C_0}\br{(W^{-d} {B} _{t,0})^{1/5} + \p{{\cal J}_{t,D}^{K_t}}^{3/2} },\quad \forall u\in [s,\tau].  
\ee
This implies that $T\ge t$ with high probability, so \eqref{eq:Gronwall_dervJuD_BA2} holds for all $u\in [s,t]$. We then take the RHS of \eqref{eq:Gronwall_dervJuD_BA2} as the new control parameter $\mathcal{J}_{u,D}^{K_u}$ and repeat the above argument. Iterating this procedure $\OO(1)$ times yields the improved bound 
\be\label{eq:Gronwall_dervJuD_BA3}
\wh{\mathcal{J}}_{u,D}^{K_u} \prec\left(|1-s|/|1-u|\right)^{C_0}(W^{-d} {B} _{t,0})^{1/5},\quad \forall u\in [s,t].  
\ee
Next, we redefine the scale $K_u'$ as in \eqref{eq:def_ell1}. Under this choice, using \eqref{eq:Gronwall_dervJuD_BA3}, we can re-establish the initial estimate \eqref{eq:Gronwall_dervJuD_BA} with $K_u$ replaced by $K_u'$, provided the constant $\fc_d$ in \eqref{con_st_ind} is chosen sufficiently small.
Finally, repeating the above procedure $\OO(1)$ more times yields \eqref{Eq:Gdecay_w} at $u=t$. The same argument clearly applies to all $u\in [s,t]$. 

It remains to justify the validity of \Cref{lem:newKLK,lem:LWterm,lem: EWGn2_N,lem: EMn2_N} for the block Anderson model. First, the proof of \Cref{lem:newKLK} in \Cref{subsec:pf_lem:newKLK} carries over with only a minor modification: in the derivation of \eqref{uu2j2ois}, we additionally make use of \eqref{Mbound_AO} or \eqref{Mbound_AO2}. 
The proofs of \Cref{lem:LWterm,lem: EWGn2_N} require more substantial changes in the graphical tools, which will be detailed in \Cref{subsec:LWchange-to-BA} below. 
Finally, the proof of \Cref{lem: EMn2_N} (with \eqref{eq:MG_conclusion3} replaced by \eqref{eq:MG_conclusion3_BA}) follows the same argument as in \Cref{subsec:pf_lem: EMn2_N}, with the following adjustments:

\begin{proof}[\bf Proof of \Cref{lem: EMn2_N} for the block Anderson model]
Most of the arguments in \Cref{subsec:pf_lem: EMn2_N} carry over directly to the block Anderson model, provided we can establish the following resolvent estimate under the assumption \eqref{eq:LW_assm}: for any $a,b\in \Zn$ and $x\in[a]$, $y\in [b]$, 
\begin{equation}\label{eq_resolventunderpoly}
	|(G_t-M)_{xy}| \prec \Psi_t(|a-b|).
\end{equation}
To prove \eqref{eq_resolventunderpoly}, we apply \eqref{GijGEX_BA} to obtain
\begin{align*}
	|(G_t-M)_{xy}|   \prec \sum_{|a'-a|\vee |b'-b| \le (\log W)^{3/2}}\Psi_t(|a'-b'|)e^{-c_\ilambda(|a'-a|+|b'-b|)} + \Psi_t(0)e^{-c_{\ilambda}|a-b|} + W^{-D} \prec \Psi_t(|a-b|) ,  
\end{align*}
where $D>0$ is an arbitrarily large constant. Here, in the second step, we have used the conditions in \eqref{eq:Psi} to obtain that 
\be\label{eq:Psirelations}\begin{aligned}
	&\Psi_t(|a'-b'|)\prec \Psi_t(|a-b|) ,\quad \Psi_t(|a-b|)\gtrsim C_1^{-1}\p{|a-b|+1}^{-C_2}\Psi_t(0)\ge W^{-D},\\  
	&\Psi_t(0)e^{-c_{\ilambda}|a-b|} \lesssim C_1^{-1}(|a-b|+1)^{-C_2}\Psi_t(0) \lesssim \Psi_t(|a-b|).
\end{aligned}
\ee
Under the stronger assumption \eqref{eq:LW_assm_exp}, a parallel argument yields the following analogue of \eqref{eq_resolventunderpoly}:
\begin{equation}\label{eq_resolventunderexp}
	|(G_t-M)_{xy}| \prec \big(W^{-d} \wT_{t,D}^{\ell}(|a-b|) \big)^{1/2},\quad \forall x\in [a], y\in [b].
\end{equation}

We now prove \eqref{eq:MG_conclusion} using the estimate \eqref{eq_resolventunderpoly}. As in \eqref{eq:S1+S2}, we need to bound 
\[\cal S_1:= W^d \sum_{|c-b|> |c-a|}  \mathcal{L}^{(6)}_{t, (\boldsymbol{\sigma}\otimes \bsig)^{(1)}, \ba(c,c)},\quad \cal S_2:= W^d \sum_{|c-b|\le |c-a|} \mathcal{L}^{(6)}_{t, (\boldsymbol{\sigma}\otimes \bsig)^{(1)}, \ba(c,c)}. \]
By symmetry, it suffices to establish \eqref{eq:MG_conclusion} for $\mathcal{S}_1$. 
In this case, we apply \eqref{eq_sym_loop_bound} with $c'=c$ to get that 
\begin{equation}\label{eq_S1_bound_BA}
	\mathcal{S}_1 \lesssim \frac{1}{\eta_t}\max_{|c-b| \ge |a-b|/2} \left(\mathcal{L}^{(4)}_{t,\boldsymbol{\sigma}^{(\alt)}, (c,b,c,b)}\right)^{1/2} \cdot  \max_{\sig\in\{+,-\}} \big|\mathcal{L}^{(3)}_{t, (\sigma,\sig_2,-\sig_2), (a,b,a)} \big|. 
\end{equation}
We write the $(c,b,c,b)$-loop as 
\[\mathcal{L}^{(4)}_{u,\bsig^{(\alt)}, (c,b,c,b)}=W^{-4d} \sum_{z',z''\in[c]}\sum_{y_1,y_2\in [b]} G_{y_1 z'}(\sig_1) G_{z'y_2}(-\sig_1) G_{y_2 z''} (\sig_1)G_{z''y_1}(-\sig_1).\]
Expanding each resolvent entry as 
\be\label{eq:G+G-M}
G_{\al \beta}(\sig)=M_{\al \beta}(\sig)+(G-M)_{\al\beta}(\sig),\quad \forall \al,\beta\in \{y_1,y_2,z',z''\}, \ \sig\in\{+,-\},\ee 
we bound every $(G-M)_{\al\beta}$ using \eqref{eq_resolventunderpoly} and each $M_{\al \beta}$ using \eqref{Mbound_AO} or \eqref{Mbound_AO2}. This yields (with $\e>0$ a small constant): 
\begin{align*}
	\mathcal{L}^{(4)}_{u,\bsig^{(\alt)}, (c,b,c,b)} &\prec \Psi_t^4(|b-c|)+\br{W^{-d}\Psi_t^3(|b-c|)+ W^{-2d}\Psi_t^2(|b-c|) + W^{-3d}\Psi_t(|b-c|)+W^{-3d}} e^{-\e|b-c|} \\
	& \lesssim \Psi_t^4(|b-c|) \lesssim \Psi_t^4(|a-b|),
\end{align*}
where the second step uses a bound analogous to \eqref{eq:Psirelations}, and the last step follows from $|c-b|\ge |a-b|/2$ together with the condition \eqref{eq:Psi}. 
Inserting this bound into \eqref{eq_S1_bound_BA}, we obtain the estimate \eqref{eq:reduce4to3} again. To handle the 3-loop on the RHS of \eqref{eq:reduce4to3}, we again expand each resolvent entry as in \eqref{eq:G+G-M}, bound every $(G-M)_{\al\beta}$ with \eqref{eq_resolventunderpoly}, and each $M_{\al \beta}$ with \eqref{Mbound_AO} or \eqref{Mbound_AO2}. This gives the same bound as in \eqref{eq:reduce4_bdd3}, where the factor $ \Psi_t^2(|a-b|)$ comes from the entries $G_{x'y}$ and $G_{yx}$ between blocks $a$ and $b$, while the short leg $G_{xx'}$ contributes a factor $\Psi_t(0)+W^{-d}\lesssim\Psi_t(0)$ by \eqref{GiiGEX_BA}. (Specifically, $(G-M)_{xx'}$ contributes $\Psi_t(0)$, while the deterministic part $M_{xx'}$ yields a factor $W^{-d}$.) 
Substituting \eqref{eq:reduce4_bdd3} into \eqref{eq:reduce4to3} concludes  \eqref{eq:MG_conclusion}.

Next, the exponential decay bound \eqref{eq:MG_conclusion3_BA} can be established by an argument analogous to that used for \eqref{eq:MG_conclusion3} below \eqref{eq:cutoff_scales}. 
First, it follows directly from the polynomial decay bound \eqref{eq:MG_conclusion} whenever \smash{$|a-b| \le \ell_t^{\dag}$}. It therefore remains to consider the case \eqref{eq:largea-b}, where we must control the three terms \smash{\(\widetilde{\mathcal{S}}_i\)}, $i\in \{1,2,3\}$, appearing in \eqref{eq;S123} with $c'=c$. The terms \smash{$\widetilde{\mathcal{S}}_1$ and $\widetilde{\mathcal{S}}_2$} can be bounded by the same argument used between \eqref{eq:pointwise_loop2} and \eqref{eq:boundwtS_1}, except that here we apply \eqref{eq_resolventunderexp} in place of the estimates \eqref{GijGEX} and \eqref{eq:LW_assm_exp} employed there.
For the term \smash{$\wt {\cal S}_3$}, we bound all legs of the original $6$-$G$-loop directly using \eqref{GijGEX_BA} together with the control parameter \smash{$\cal J_{t,D}^{\ell}$}. In this setting, \eqref{eq_L2-J} remains valid with \smash{$\wh{\cal J}_{t,D}$} replaced by \smash{$\cal J_{t,D}^{\ell}$}, where the $2$-$\cK$-loop \smash{$\mathcal{K}^{(2)}_{t,(-,+),(a,b)}$} defined in \eqref{Kn2sol} is of order $\OO(W^{-D})$ thanks to the exponential decay of the $\Theta$-propagators in \eqref{prop:ThfadC} and the decay of the $2$-$M$-loop due to \eqref{Mbound_AO} or \eqref{Mbound_AO2}. Using \eqref{eq_L2-J} (with \smash{$\wh{\cal J}_{t,D}^{\ell}$} replaced by $\cal J_{t,D}^{\ell}$) and an argument analogous to that below \eqref{eq_resolventunderpoly}, we obtain that for any $x\in [a]$ and $y\in [b]$ with $|a-b|\gtrsim\ell_t^*$,
\begin{equation}\label{eq_resolventunderexp2}
|(G_t-M)_{xy}| \prec \big(W^{-d} \wT_{t,D}^{\ell}(|a-b|) \big)^{1/2} \implies |(G_t)_{xy}| \prec \big(W^{-d} \wT_{t,D}^{\ell}(|a-b|) \big)^{1/2},
\end{equation}
where the implication uses the bound \eqref{Mbound_AO} or \eqref{Mbound_AO2}.
Finally, combining \eqref{eq_resolventunderexp2} with the same reasoning as below \eqref{eq_L2-J}, we conclude the exponential decay bound \eqref{eq:MG_conclusion3_BA}.
\end{proof}

\noindent{\bf Proof of Steps 3--6 for \Cref{lem:main_ind_BA}.} 
The proofs in \Cref{Sec:Steps34,Sec:Step5,subsection:step6} extend verbatim to the block Anderson model, except for the arguments in \Cref{sec:Step5_larget} and the proof of \Cref{lem:LWterm_EXP}. In \Cref{sec:Step5_larget}, the random band matrix case relies on \Cref{lem_GbEXP} and follows the approach of \cite[Section 5.3]{YY_25}, whereas the corresponding arguments for the block Anderson model are based on \Cref{lem_GbEXP_BA} and parallel those in \cite[Section 7.3]{RBSO1D}. We therefore omit the details. 
The proof of \Cref{lem:LWterm_EXP} is deferred to \Cref{subsec:pf-LWterm_EXP} below.

\appendix

\section{Proofs of auxiliary graphical lemmas}

\subsection{Proof of \texorpdfstring{\Cref{lem:LWterm_EXP}}{Lemma ExpLW}}\label{subsec:pf-LWterm_EXP}

Recall the light-weight term in \eqref{eq:EGC}. To prove \Cref{lem:LWterm_EXP}, it suffices to establish 
\begin{align}
	W^{-2d}\sum_{a_1, a_2}S^{\LK}_{a_1a_2} \sum_{x\in [a],y\in [b],\al \in [a_2]}\E       
	\tr\big( \Gc_t E_{a_1} \big)  (G_t(\sig))_{yx} (G_t)_{x\al }(G_t)_{\al y} \prec \eta_t^{-1}(W^{-d}B_{t,0})^{5/2}\label{eq:ELW_term}
\end{align} 
for any $\sig\in\{+,-\}$. For brevity, write $G_t\equiv G$, and assume $\sig=-$ in the following proof. The case $\sig=+$ is analogous. 
We first focus on the random band matrix model. Performing the $GG$ expansion in \eqref{Oe2x} with respect to $G_{x\al}G_{\al y}$, we can expand the LHS of \eqref{eq:ELW_term} as 
\begin{align*}
	&~mW^{-2d}\sum_{a_1}S^{\LK}_{a_1 b} \sum_{x\in [a],y\in [b]}\E  \tr \big(\Gc E_{a_1} \big) |G_{xy}|^2 \tag{$I_1$} \\
	+&~m^3 W^{-2d}\sum_{a_1, a_2}S^{\LK}_{a_1a_2} \sum_{x\in [a],y\in [b],\al \in [a_2]} S^+_{\al y} \E       
	\tr \big(\Gc E_{a_1} \big) |G_{xy}|^2 \tag{$J_1$}\\
	+&~m W^{-2d}\sum_{a_1, a_2}S^{\LK}_{a_1a_2} \sum_{x\in [a],y\in [b],\al \in [a_2]}\sum_{\beta}\E       
	\tr \big(\Gc E_{a_1} \big) (S_{\al\beta} \Gc_{\beta\beta})  G_{x\al }G_{\al y} \overline G_{xy} \tag{$I_2$}\\
	+&~m^3 W^{-2d}\sum_{a_1, a_2}S^{\LK}_{a_1a_2} \sum_{x\in [a],y\in [b],\al \in [a_2]}\sum_{\beta,\gamma}S^{+}_{\al\beta}\E       
	\tr \big(\Gc E_{a_1} \big) (S_{\beta\gamma} \Gc_{\gamma\gamma})  G_{x\beta }G_{\beta y} \overline G_{xy} \tag{$J_2$}\\
	+&~m W^{-2d}\sum_{a_1, a_2}S^{\LK}_{a_1a_2} \sum_{x\in [a],y\in [b],\al \in [a_2]}\sum_{\beta}\E       
	\tr \big(\Gc E_{a_1} \big) (\Gc_{\al\al} S_{\al\beta})G_{x\beta }G_{\beta y} \overline G_{xy} \tag{$I_3$}\\
	+&~m^3 W^{-2d}\sum_{a_1, a_2}S^{\LK}_{a_1a_2} \sum_{x\in [a],y\in [b],\al \in [a_2]}\sum_{\beta,\gamma}S^{+}_{\al\beta} \E       
	\tr \big(\Gc E_{a_1} \big) (\Gc_{\beta\beta} S_{\beta\gamma})G_{x\gamma }G_{\gamma y} \overline G_{xy} \tag{$J_3$}\\
	-&~m W^{-2d}\sum_{a_1, a_2}S^{\LK}_{a_1a_2} \sum_{x\in [a],y\in [b],\al \in [a_2]}\sum_{\beta}S_{\al\beta} \E       
	G_{x\al }G_{\beta y} \partial_{\beta\al}\left(\tr \big(\Gc E_{a_1} \big) \overline G_{xy}\right) \tag{$I_4$}\\
	-&~m^3 W^{-2d}\sum_{a_1, a_2}S^{\LK}_{a_1a_2} \sum_{x\in [a],y\in [b],\al \in [a_2]}\sum_{\beta,\gamma}S^+_{\al\beta}S_{\beta\gamma} \E      
	G_{x\beta}G_{\gamma y} \partial_{\gamma\beta}\left(\tr \big(\Gc E_{a_1} \big) \overline G_{xy}\right) \tag{$J_4$},
\end{align*} 
where we recall that $\partial_{\beta\al}\equiv \partial_{(H_t)_{\beta\al}}$. We will only estimate the terms $I_i$ for $i\in\qqq{4}$; the terms $J_i$ can be treated in exactly the same way by using \eqref{eq:def-Spm}, together with the bound \eqref{prop:ThfadC_short}.

For the term $I_1$, we have
\begin{align}
	I_1&= m \sum_{a_1}S^{\LK}_{a_1 b} \E \tr \big(\Gc E_{a_1} \big)\cal L^{(2)}_{t,(-,+),(a,b)} \nonumber\\
	&=m \sum_{a_1}S^{\LK}_{a_1 b} \E \tr \big(\Gc E_{a_1} \big)\cal K^{(2)}_{t,(-,+),(a,b)} +  m \sum_{a_1}S^{\LK}_{a_1 b} \E \tr \big(\Gc E_{a_1} \big) (\cal L-\cK)^{(2)}_{t,(-,+),(a,b)} \prec (W^{-d}B_{t,0})^3, \label{eq:termI1}
\end{align}
where in the last step, we use \eqref{eq:bcal_k} and \eqref{res_ELK_n=1} to bound the first term, and \eqref{Gt_avgbound_flow} together with \eqref{Eq:L-KGt-flow} to bound the second.
For $I_2$, applying the averaged local law \eqref{Gt_avgbound_flow} to both $\tr( \Gc E_{a_1} )$ and $\sum_\beta S_{\al\beta} \Gc_{\beta\beta}$ yields
\begin{align}
	I_2&\prec (W^{-d}B_{t,0})^2 \cdot W^{-2d}\sum_{a_1, a_2}S^{\LK}_{a_1a_2} \sum_{x\in [a],y\in [b],\al \in [a_2]}\E  |G_{x\al}||G_{\al y}| |G_{xy}| \nonumber\\
	&\lesssim (W^{-d}B_{t,0})^2 \cdot W^{-2d}\sum_{x\in [a],y\in [b]} \sum_\al \E  |G_{x\al}||G_{\al y}| |G_{xy}| \nonumber\\
	&\prec \eta_t^{-1}(W^{-d}B_{t,0})^2 \cdot W^{-2d}\sum_{x\in [a],y\in [b]}\E|G_{xy}| \prec \eta_t^{-1}(W^{-d}B_{t,0})^{5/2},\label{eq:termI2}
\end{align}
where in the third step we use Cauchy–Schwarz together with Ward’s identity, and in the last step the local law \eqref{Gt_bound_flow} to control the averages over $x\in[a]$ and $y\in[b]$. The term $I_3$ can be handled in exactly the same way. It remains to estimate $I_4$. Splitting according to which factor the partial derivative acts on, we obtain
\begin{align*}
	I_4&=m W^{-3d}\sum_{a_1, a_2,a_3}S^{\LK}_{a_1a_2}S^{\LK}_{a_2a_3} \sum_{x\in [a],y\in [b],\al \in [a_2],\beta\in[a_3]}  \E  \tr\big( \Gc E_{a_1} \big) |G_{x\al}|^2 |G_{\beta y}|^2 \\
	&\quad +m W^{-4d}\sum_{a_1, a_2, a_3}S^{\LK}_{a_1a_2}S^{\LK}_{a_2a_3} \sum_{x\in [a],y\in [b],\gamma\in[a_1],\al \in [a_2],\beta\in[a_3]} \E       
	G_{x\al }G_{\al\gamma}G_{\gamma\beta}G_{\beta y}  \overline G_{xy} \\
	&= mW^d\sum_{a_1, a_2,a_3}S^{\LK}_{a_1a_2}S^{\LK}_{a_2a_3} \E  \tr\big( \Gc E_{a_1} \big) \cL^{(2)}_{t,(-,+),(a,a_2)}\cL^{(2)}_{t,(-,+),(a_3,b)} +mW^d \sum_{a_1, a_2, a_3}S^{\LK}_{a_1a_2}S^{\LK}_{a_2a_3} \E \cL^{(5)}_{t,\bsig_5,\ba_5} \\
	&=:I_{41}+I_{42},
\end{align*}
where $\bsig_5=(+,+,+,+,-)$ and $\ba_5=(a_2,a_1,a_3,b,a)$. For $I_{41}$, we decompose each 2-$G$-loop as $(\cL-\cK)^{(2)}+\cK^{(2)}$, giving 
\begin{align}
	I_{41}&= mW^d\sum_{a_1, a_2,a_3}S^{\LK}_{a_1a_2}S^{\LK}_{a_2a_3} \E  \tr\big( \Gc E_{a_1} \big)  \cL^{(2)}_{t,(-,+),(a,a_2)}\cK^{(2)}_{t,(-,+),(a_3,b)} \nonumber\\
	&\quad +mW^d\sum_{a_1, a_2,a_3}S^{\LK}_{a_1a_2}S^{\LK}_{a_2a_3} \E  \tr\big( \Gc E_{a_1} \big)  \cL^{(2)}_{t,(-,+),(a,a_2)}(\cL-\cK)^{(2)}_{t,(-,+),(a_3,b)}\nonumber\\ 
	&= mW^d\sum_{a_1, a_2,a_3}S^{\LK}_{a_1a_2}S^{\LK}_{a_2a_3} \E  \tr\big( \Gc E_{a_1} \big) \cL^{(2)}_{t,(-,+),(a,a_2)}\cK^{(2)}_{t,(-,+),(a_3,b)}+ \OO_\prec\p{(W^{-d}B_{t,0})^3}\cdot W^d\sum_{ a_2}\E \cL^{(2)}_{t,(-,+),(a,a_2)} \nonumber\\
	&= mW^d\sum_{a_1, a_2,a_3}S^{\LK}_{a_1a_2}S^{\LK}_{a_2a_3} \E  \tr\big( \Gc E_{a_1} \big) \cK^{(2)}_{t,(-,+),(a,a_2)}\cK^{(2)}_{t,(-,+),(a_3,b)}\nonumber\\
	&\quad +mW^d\sum_{a_1, a_2,a_3}S^{\LK}_{a_1a_2}S^{\LK}_{a_2a_3} \E  \tr\big( \Gc E_{a_1} \big) (\cL-\cK)^{(2)}_{t,(-,+),(a,a_2)}\cK^{(2)}_{t,(-,+),(a_3,b)}+ \OO_\prec\p{\eta_t^{-1}(W^{-d}B_{t,0})^3} \nonumber\\
	&\prec (W^{-d}B_{t,0})^3\cdot W^d\sum_{a_3}\cK^{(2)}_{t,(-,+),(a_3,b)} + \eta_t^{-1}(W^{-d}B_{t,0})^3 \prec \eta_t^{-1}(W^{-d}B_{t,0})^3.\label{eq:termI41}
\end{align}
Here, we use \eqref{Eq:L-KGt-flow} and \eqref{Gt_avgbound_flow} in the second step; Ward’s identity \eqref{WI_calL} for \smash{$\sum_{a_2} \cL^{(2)}_{t,(-,+),(a,a_2)}$} and the local law \eqref{Gt_bound_flow} in the third; \eqref{Eq:L-KGt-flow}, \eqref{eq:bcal_k}, and \eqref{res_ELK_n=1} in the fourth; and Ward’s identity \eqref{WI_calK} for \smash{$\sum_{a_3}\cK^{(2)}_{t,(-,+),(a_3,b)}$} in the last. 
Finally, for the term $I_{42}$, we rewrite it as an average over the graphs $\cal G_{xy}$: 
\begin{align}\label{eq;I42inG}
	I_{42}=mW^{-2d} \sum_{x\in [a],y\in [b]} \E\cal G_{xy},\quad \text{where}\quad \cal G_{xy}:=\sum_{\gamma,\al,\beta}S_{\gamma\al}S_{\al\beta} \p{G_{x\al }G_{\al\gamma}G_{\gamma\beta}G_{\beta y}  \overline G_{xy}}.
\end{align}
We claim the following bound for $\E\cal G_{xy}$: 
\begin{align}\label{eq;EGxy:x=y}
	\E\cal G_{xy} \prec \mathbf 1_{x = y}\cdot  \eta_t^{-1} (W^{-d}B_{t,0})^{2}+\mathbf 1_{x\ne y}\cdot \eta_t^{-1}(W^{-d}B_{t,0})^{5/2}. 
\end{align}
Substituting this estimate into \eqref{eq;I42inG} gives
\[ I_{42} \prec \eta_t^{-1} (W^{-d}B_{t,0})^{5/2}.\]
Together with \eqref{eq:termI1}--\eqref{eq:termI41}, this completes the proof of \eqref{eq:ELW_term}.

Finally, we prove the estimate \eqref{eq;EGxy:x=y}. To this end, we apply the $GG$-expansion from \Cref{T eq0} at the non-standard neutral vertices $\al,\gamma,\beta$, which expands $\cal G_{xy}$ into a collection of new graphs:
\begin{align*}
	\cal G_{xy}=_{\E} \sum_\mu \Gamma_{\mu,xy} \; .
\end{align*}
By \Cref{GtoAG}, each $\Gamma_{\mu,xy}$ can be bounded through its auxiliary graph \smash{$\Gamma^{\aux}_{\mu,[a][b]}$} as in \eqref{G_by_auxG}, with $q\in\{0,1\}$ internal molecules and $\Psi_t=(W^{-d}B_{t,0})^{1/2}$. 
If an auxiliary graph contains an internal molecule, then that molecule is attached to at least two solid edges corresponding to the $\xi$-variables defined in \eqref{eq:xia1a2}. Applying Cauchy–Schwarz together with \eqref{eq:Gbyxi2} therefore gives
\[ \Gamma^{\aux}_{\mu, [a][b]} \prec (W^d\eta_t)^{-q}(\Psi_t)^{\ord(\Gamma^{\aux}_{\mu, [a][b]})}. \]
Combining this with \eqref{G_by_auxG} implies that 
\begin{align}\label{Gammamuxy}
	\Gamma_{\mu, xy} \prec \eta_t^{-1} (W^{-d}B_{t,0})^{\frac12\ord(\Gamma_{\mu, xy})} .
\end{align}
With this estimate in hand, to conclude \eqref{eq;EGxy:x=y} it remains to show that the scaling order of each $\Gamma_{\mu,xy}$ satisfies
\be\label{eq:sizeGammamu_E}
\ord(\Gamma_{\mu,xy}) \ge 4\cdot \mathbf 1_{x= y}+5\cdot \mathbf 1_{x\ne y} , 
\ee
where $x$ and $y$ are regarded as external vertices.

To obtain \eqref{eq:sizeGammamu_E}, we decompose $\cal G_{xy}$ into four parts $\cal G^{(i)}_{xy}$, $i\in\{1,2,3,4\}$, corresponding to the cases: (1) $\al=\gamma=\beta$, (2) $\al=\gamma\ne \beta$, (3) $\al\ne \gamma = \beta$, and (4) $\al\ne \gamma$ and $\gamma\ne \beta$. It is easy to see that \smash{$\ord(\cal G^{(1)}_{xy})$} already satisfies the lower bound in \eqref{eq:sizeGammamu_E}. Next, in case (2), we write 
\[ \cal G^{(2)}_{xy}= m\sum_{\al\ne\beta}S_{\al\al}S_{\al\beta} \p{G_{x\al }G_{\al\beta}G_{\beta y}  \overline G_{xy}}+ \sum_{\al\ne\beta}S_{\al\al}S_{\al\beta} \Gc_{\al\al} \p{G_{x\al }G_{\al\beta}G_{\beta y}  \overline G_{xy}}=:\cal G^{(21)}_{xy}+\cal G^{(22)}_{xy}.\]
The second graph \smash{$\cal G^{(22)}_{xy}$} already meets the lower bound in \eqref{eq:sizeGammamu_E}, while the first graph \smash{$\cal G^{(21)}_{xy}$} has scaling order $3\cdot \mathbf 1_{x= y}+4\cdot \mathbf 1_{x\ne y} .$ To raise its scaling order, we expand \smash{$\cal G^{(21)}_{xy}$} using the $GG$ expansion \eqref{Oe2x} at the vertex $\al$. A direct inspection shows that every graph produced by this expansion has strictly larger scaling order than \smash{$\cal G^{(21)}_{xy}$}, and thus satisfies \eqref{eq:sizeGammamu_E}.
Since the check is straightforward, we omit the details.
Case (3) is handled in exactly the same manner as case (2).
Finally, in case (4), the graph \smash{$\cal G^{(4)}_{xy}$} has scaling order $2\cdot \mathbf 1_{x= y}+3\cdot \mathbf 1_{x\ne y}.$ To reach \eqref{eq:sizeGammamu_E}, we apply the $GG$-expansion \eqref{Oe2x} at both vertices $\al$ and $\beta$. Each application of \eqref{Oe2x} strictly increases the scaling order, so the two expansions raise it by at least 2 in total.
This yields \eqref{eq:sizeGammamu_E} and thus completes the proof of \eqref{eq;EGxy:x=y}. Again, as the verification is routine, we omit the details.

The proof of \Cref{lem:LWterm_EXP} for the block Anderson model is analogous to the argument above, except that the $GG$-expansion in \eqref{GGGamma} below (for the block Anderson model) is used in place of \eqref{Oe2x} (for random band matrices). Hence, we omit the details.

\subsection{Proof of \Cref{lem:localregular}}\label{sec:pflocalregular}

Our local expansion strategy leading to the proof of \Cref{lem:localregular} proceeds as follows. Given a graph $\Gamma$ that is not locally standard, we first find all weights in it, and apply the weight expansion \eqref{Owx} to remove them one by one. 
Once the graph is free of weights, find all vertices connected to more than two solid edges, and apply the edge expansion \eqref{Oe1x} iteratively to reduce their degrees.  
After removing weights and ensuring each vertex is connected to at most two solid edges, identify all vertices that are not standard neutral---that is, vertices connected to two $G$ edges or two $\overline G$ edges---and apply the $GG$ expansion \eqref{Oe2x} to correct them.
It is important to note that after each expansion, the resulting graphs may require earlier steps to be re-applied. For instance, an edge expansion might introduce new weights, which must be removed via weight expansions before proceeding with further edge or $GG$ expansions. 
We now formally state the local expansion rules. Throughout the following proof, when we refer to the degree of a vertex $\al$, we mean \emph{the number of solid edges connected to it}, excluding any solid self-loops (i.e., weights).

\begin{strategy}[Local expansion strategy]\label{strat_local}
	Given an arbitrarily large constant $D>0$, we apply the following local expansion strategy. 
	\begin{itemize}
		\item[Step 1:] Given an input graph, we apply the \emph{weight expansion} $\cal O_{\mathrm{weight}}$ as follows. If the graph contains no weights (neither regular nor light), then $\cal O_{\mathrm{weight}}$ is a null operation, and the graph is passed to the next step. Otherwise, we select one of its weights, say on vertex $\al$. If this is a light-weight, we apply the expansion \eqref{Owx} directly. If it is a regular weight, i.e., $G_{\al\al}$ or \smash{$\overline G_{\al\al}$}, we decompose it as a sum of a light-weight and a factor of $m$ or $\overline m$: \smash{$(G_{\al\al}-m)+m$} or \smash{$(\overline G_{\al\al}-\overline m)+\overline m$}, and then apply the expansion \eqref{Owx} to the resulting light-weight term.  
		For each new graph generated by the weight expansion, we apply the operations in \Cref{dot-def} to write it as a sum of normal graphs. For each normal graph $\cal G$, if its scaling size is sufficiently small:
		\be\label{eq:smallsize}
		\size(\cal G)\le W^{-D},
		\ee
		then we send it directly to the output; if it contains no weights, we pass it to Step 2; if it still contains weights, we restart Step 1.

		\item[Step 2:] At this stage, the input graph has no weights. We apply the edge expansion $\cal O_{\mathrm{edge}}$ as follows. If there exists a vertex $\al$ whose degree is $\notin\{0,2\}$, or whose charge is not neutral (recall the definition in \eqref{eq:neutralcharge}),\footnote{If $\al$ is attached to two $+$ solid edges and has non-neutral charge, then the two edges incident to it must be of the form $G_{xy}G_{xy'}$ or $G_{yx}G_{y'x}$. Such configurations cannot be expanded using the $GG$ expansion \eqref{Oe2x}, and therefore must first be removed via the edge expansion \eqref{Oe1x}.} then we apply the edge expansion \eqref{Oe1x} at $\al$. 
		For each resulting graph, we again apply the operations in \Cref{dot-def} to express it as a sum of normal graphs. Those satisfying the small-size condition \eqref{eq:smallsize} are sent to the output, while the rest are sent back to Step 1.
		If all internal vertices have degree $\in \{0,2\}$ and are neutral in charge, then $\cal O_{\mathrm{edge}}$ is a null operation, and the graph is passed to Step 3.  
		
		\item[Step 3:] Now, the input graph has no weights, and all internal vertices have degree $\in\{0,2\}$ and neutral charge. We apply the $GG$ expansion $\cal O_{GG}$ as follows. 
		If there exists a vertex $\al$ connected to two solid edges of the same charge (i.e., two $G$ edges or two $\overline G$ edges), we apply the expansion \eqref{Oe2x} at $\al$.
		Each resulting graph is expanded into a sum of normal graphs using the operations in \Cref{dot-def}. Those satisfying \eqref{eq:smallsize} are sent to the output, while the rest are sent back to Step 1.
		If all internal vertices are standard neutral, then $\cal O_{GG}$ is a null operation, and the graph is sent to the output. 
		
	\end{itemize}
\end{strategy}

We now apply \Cref{strat_local} to complete the proof of \Cref{lem:localregular}.

\begin{proof}[\bf Proof of \Cref{lem:localregular}]
	By \Cref{lvl1 lemma}, we obtain the expansion of $\big|f_{xy}(G)\big|^{p}$ as in \eqref{eq:local_Gs}, where the graphs $\Gamma_{\mu,xy}$ trivially satisfy property (1). For property (2), note that the original graph 
	\be\label{eq:originGamma}\Gamma_{xy}=\big|f_{xy}(G)\big|^{p}= \sum_{\bal,\bbeta}\prod_{i=1}^p \left[S_{\al_i\beta_i} \mathbf 1_{x\ne \al_i}\mathbf 1_{y\ne \al_i} \cdot 
	\Gc_{\beta_i\beta_i} G_{x\al_i}G_{\al_i y}\right]\ee
	contains $p$ internal molecules, corresponding to the vertices $\bal=(\al_1,\ldots, \al_p)$ and $\bbeta=(\beta_1,\ldots, \beta_p)$. During the expansions, the number of internal molecules never increases; it may decrease when two molecules---internal or external---are merged due to new dotted or waved edges created in the process. Consequently, $\Gamma_{\mu,xy}$ may contain strictly fewer molecules than $\Gamma_{xy}$. In particular, if all internal molecules merge with external ones, then $\Gamma_{\mu,xy}$ contains no internal molecules, in which case $q=0$. Finally, \eqref{eq:MolVW} follows directly from the definition of a molecule, since all vertices within a molecule are connected by paths of waved edges. 
	
	For properties (3)--(5), note that the original graph $\Gamma_{xy}$ contains $p$ edge-disjoint paths between $\Mol_x$ and $\Mol_y$, each passing through an internal molecule $\Mol_i$ that contains the vertex $\al_i$, $i\in \qqq{p}$. We denote the path through $\Mol_i$ by $\fP_i$.
	During the expansions, molecules may merge, but for notational convenience, we retain their original labels: if molecules $\Mol_i$ and $\Mol_j$ merge, we continue to refer to the resulting molecule as both  $\Mol_i$ and $\Mol_j$. Similarly, although the paths $\fP_i$ may change during the expansion, we keep their names and regard each $\fP_i$ as the path associated with $\Mol_i$. 
	Checking \Cref{ssl,Oe14,T eq0}, we observe that these paths never disappear under local expansions. More precisely, if $e$ is a solid edge between two distinct molecules $\cal M$ and $\cal M'$, then in any local expansion either: 
	\begin{itemize}
		\item  the edge $e$ remains unaffected on the molecular graph; or  
		
		\item $e$ is replaced by two new edges $e_1$ and $e_2$, which still form a connected path between $\Mol$ and $\Mol'$ (though this connectedness may fail at the vertex level); or  
		
		\item the molecules $\Mol$ and $\Mol'$ merge, in which case $e$ disappears on the molecular graph, but the connectedness between $\Mol$ and $\Mol'$ holds trivially. 
		
	\end{itemize}
	From this observation, properties (3) and (5) follow.
	
	For property (4), let $\Mol$ be an internal molecule in $\Gamma_{\mu,xy}$. If $\Mol$ arises from merging at least two internal molecules, say $\Mol_i$ and $\Mol_j$, then both associated paths $\fP_i$ and $\fP_j$ pass through $\Mol$.
	If the previous scenario does not occur, then we are in the case $\Mol=\Mol_i$, which corresponds to the following path in the original graph $\Gamma_{xy}$ (assuming, without loss of generality, that the path carries a $+$ charge):
	\[ \sum_{\al_i,\beta_i}S_{\al_i\beta_i} \mathbf 1_{x\ne \al_i}\mathbf 1_{y\ne \al_i}\cdot \Gc_{\beta_i\beta_i} G_{x\al_i}G_{\al_i y}.\] 
	To ensure that $\Gamma_{\mu,xy}$ is locally standard, the molecule $\Mol_i$ must pull in some $\overline G$ edges from other paths or molecules during the expansions. 
	If a $\overline G$-edge from path $\fP_j$ is pulled to $\Mol_i$, then $\fP_j$ also passes through $\Mol_i$. Similarly, if a $\overline G$-edge inside $\Mol_j$ is pulled to $\Mol_i$, then the two new edges between $\Mol_i$ and $\Mol_j$ can be incorporated into path $\fP_j$, so that $\fP_j$ again passes $\Mol_i$. In either case, property (4) holds.

	Finally, we examine the scaling order of $\Gamma_{\mu,xy}$. For the original graph $\Gamma_{xy}$, the scaling order is  
	\be\label{eq:initial_scaling} \ord(\Gamma_{xy}) =
	p, \ee 
	which follows from the presence of $2p$ off-diagonal solid edges, $p$ light-weights, $p$ waved edges, and $2p$ internal vertices. We now track how the scaling order evolves during the expansions. 
	For clarity, we refer to the light-weights \smash{$\Gc_{\beta_i\beta_i}$} in the original graph \eqref{eq:originGamma} as \emph{distinguished light-weights}, and the vertices $\al_i$ in \eqref{eq:originGamma} as \emph{distinguished vertices}, each incident to two solid edges of the same charge. 
	To obtain locally standard graphs, these distinguished light-weights and vertices must be removed one by one through local expansions.\footnote{A distinguished vertex is removed once its local structure---two incident solid edges of the same charge---is broken. This can occur either by merging it with other vertices or by applying the $GG$-expansion \eqref{Oe2x} at the vertex.}
	We will show that removing any distinguished light-weight or vertex increases the scaling order by at least 1/2. Let $\cal G_0$ be a normal graph before a weight expansion, where \emph{no} two distinguished vertices are connected to each other through solid edges. Let $\cal G_1$ denote a new normal graph obtained after a local expansion.

	We first consider the weight expansions. Suppose we apply the expansion \eqref{Owx} to a distinguished light-weight \smash{$\Gc_{ww}$} in $\cal G_0$. In the first two terms on the RHS of \eqref{Owx}, the number of light-weights increases by 1, which in turn increases the scaling order of the new graph by one: $\ord(\cal G_1)\ge \ord(\cal G_0)+1$. It remains to analyze the graphs generated by the last two terms on the RHS of \eqref{Owx}.
	Consider the third term as an example. Without loss of generality, assume that the derivative $\partial_{\al w}$ acts on either a solid edge $G_{\beta_1\beta_2}$ or a light-weight \smash{$\Gc_{\beta_1\beta_2}$} (with $\beta_1=\beta_2$) of positive charge, pulling it into the molecule containing $w$. (The case of a negative-charge edge/light-weight is treated analogously.) This yields the graph
	\[\cal G'_1=m \sum_{\beta_1,\beta_2} \sum_{ \al} S_{w\al} G_{\al w} G_{\beta_1\al}G_{w\beta_2} \cal G'(\beta_1,\beta_2),\]
	where $\cal G'(\beta_1,\beta_2)$ denotes the graph obtained from $\cal G_0$ by removing the light-weight $\Gc_{ww}$ together with the solid edge $G_{\beta_1\beta_2}$ or light-weight \smash{$\Gc_{\beta_1\beta_2}$}, and by setting the vertices $\beta_1$ and $\beta_2$ as external (note that either $\beta_1$ or $\beta_2$ may already be an external vertex $x$ or $y$). Assigning the dotted edge partition to $\cal G'_1$, we get the graph $\cal G_1$, in which each of the three new solid edges $G_{\al w},$ $ G_{\beta_1\al},$ and $G_{w\beta_2}$ may be either off-diagonal or diagonal (in the latter case, certain vertices are merged). More precisely:
	\begin{itemize}
		\item[(i)] If all edges $G_{\al w}$, $G_{\beta_1\al}$, and $G_{w\beta_2}$ are off-diagonal, then $\ord(\cal G_1)\ge \ord(\cal G_0)+1$. If \smash{$\Gc_{\beta_1\beta_2}$} is a distinguished light-weight (necessarily with $\beta_1=\beta_2$), then in $\cal G_1$, we designate $\beta_1$ as a \emph{distinguished vertex}, which is incident to two solid edges of the same charge. In this case, $n_{\rm{lw}}(\cal G_1)\ge n_{\rm{lw}}(\cal G_0)-2$ and $n_{\rm{dv}}(\cal G_1)\ge n_{\rm{dv}}(\cal G_0)+1$, where $n_{\rm{lw}}$ and $n_{\rm{dv}}$ denote the number of distinguished light-weights and distinguished vertices in the graph, respectively. Otherwise, if \smash{$\Gc_{\beta_1\beta_2}$} is not a distinguished light-weight, then $n_{\rm{lw}}(\cal G_1)\ge n_{\rm{lw}}(\cal G_0)-1$ and $n_{\rm{dv}}(\cal G_1)\ge n_{\rm{dv}}(\cal G_0)$. In either case, we have the relation
		\be\label{eq:relateG1G0} \ord(\cal G_1)+ n_{\rm{dv}}(\cal G_1)+ n_{\rm{lw}}(\cal G_1) \ge \ord(\cal G_0) + n_{\rm{dv}}(\cal G_0)+ n_{\rm{lw}}(\cal G_0).\ee

		\item[(ii)] If $\beta_1\ne \beta_2$, and both $G_{\beta_1\al}$ and $G_{w\beta_2}$ are diagonal in $\cal G_1$, then necessarily $\beta_1=\al\ne w=\beta_2$. In this case, the parameters satisfy 
		\[\qquad \quad 
		n_W(\cal G_1) = n_W(\cal G_0)+1, \quad  n_V(\cal G_1) \le   n_V(\cal G_0)-1 , \quad n_S(\cal G_1) \ge n_S(\cal G_0)-1,\quad n_{\rm{lw}}(\cal G_1) = n_{\rm{lw}}(\cal G_0)-1. \]
		By definition \eqref{eq:ordG}, this implies \(\ord(\cal G_1)\ge \ord(\cal G_0) + 3\). Moreover, \(n_{\rm{dv}}(\cal G_1)\ge n_{\rm{dv}}(\cal G_0)-1\), with equality if $\beta_1$ or $\beta_2$ is a distinguished vertex. Thus, the relation \eqref{eq:relateG1G0} still holds.

		\item[(iii)] If $\beta_1=\beta_2$ and both $G_{\beta_1\al}$ and $G_{w\beta_2}$ are diagonal in $\cal G_1$, then compared to $\cal G_0$, the graph $\cal G_1$ loses two light-weights, namely $\Gc_{ww}$ and $\Gc_{\beta_1\beta_1}$. In this case, the parameters satisfy
		\[\qquad \quad
		n_W(\cal G_1) = n_W(\cal G_0)+1,\quad n_V(\cal G_1) \le n_V(\cal G_0)-1, \quad n_S(\cal G_1) \ge n_S(\cal G_0)-2, \quad n_{\rm{dv}}(\cal G_1)= n_{\rm{dv}}(\cal G_0). \]
		By definition \eqref{eq:ordG}, this implies \(\ord(\cal G_1)\ge \ord(\cal G_0) +2\), and hence the relation \eqref{eq:relateG1G0} still holds.

		\item[(iv)] Suppose only one of $G_{\beta_1\al}$ and $G_{w\beta_2}$ is diagonal. If $G_{\al w}$ is off-diagonal in $\cal G_1$, then  
		\[n_W(\cal G_1) = n_W(\cal G_0)+1,\quad n_V(\cal G_1) \le   n_V(\cal G_0) , \quad n_S(\cal G_1) \ge n_S(\cal G_0). \]
		By definition \eqref{eq:ordG}, this implies \(\ord(\cal G_1)\ge \ord(\cal G_0) +2\). If neither $\beta_1$ nor $\beta_2$ is distinguished, this immediately yields the relation \eqref{eq:relateG1G0}.
		Otherwise, if one of them is a distinguished vertex (which necessarily means $\beta_1\ne\beta_2$), then $n_{\rm{lw}}(\cal G_1)\ge n_{\rm{lw}}(\cal G_0)-1$ and $n_{\rm{dv}}(\cal G_1)\ge n_{\rm{dv}}(\cal G_0)-1$, which still gives the relation \eqref{eq:relateG1G0}.

		\item[(v)] Suppose only one of $G_{\beta_1\al}$ and $G_{w\beta_2}$ is diagonal. If $G_{\al w}$ is diagonal in $\cal G_1$, then  
		\[\qquad \quad 
		n_W(\cal G_1) = n_W(\cal G_0)+1,\quad n_V(\cal G_1) \le   n_V(\cal G_0) - 1 , \quad n_S(\cal G_1) \ge n_S(\cal G_0) - 1,\quad n_{\rm{lw}}(\cal G_1) \ge n_{\rm{lw}}(\cal G_0)-2. \]
		By definition \eqref{eq:ordG}, this implies \(\ord(\cal G_1)\ge \ord(\cal G_0) +3.\) Moreover, \(n_{\rm{dv}}(\cal G_1)\ge n_{\rm{dv}}(\cal G_0)-1\), with equality if $\beta_1$ or $\beta_2$ is distinguished. Thus, the relation \eqref{eq:relateG1G0} still holds.
		
	\end{itemize}
	It is not hard to see that the same reasoning applies to graphs arising from the fourth term on the RHS of \eqref{Owx}. 
	Hence, after removing all $p$ distinguished light-weights, the relation \eqref{eq:relateG1G0} holds between any graph $\cal G$ arising in the expansions and the original graph $\Gamma_{xy}$. Together with \eqref{eq:initial_scaling}, this implies
	\be\label{eq:ordGinter} \ord(\cal G) \ge 3p - n_{\rm{dv}}(\cal G),\ee
	where $n_{\rm{dv}}(\cal G) \le 3p/2$, because $n_{\rm{dv}}$ can increase only in case (i), namely when the expansion of a distinguished light-weight pulls in another distinguished light-weight. 
	
	It remains to expand such graphs $\cal G$ using the edge expansion \eqref{multi setting} and the $GG$ expansion \eqref{Oe2x}. A direct check shows that the removal of each distinguished vertex increases the scaling order by at least 1/2. The worst case occurs in a $GG$ expansion involving two distinguished vertices: in this case, two distinguished vertices may disappear, but the scaling order increases by 1. Since this is a straightforward case-by-case counting argument as above, we omit the details. Thus, for a locally standard graph $\Gamma_{\mu,xy}$ obtained from the local expansions of $\cal G$, we have 
	\[ \ord(\Gamma_{\mu,xy})\ge \ord(\cal G) + n_{\rm{dv}}(\cal G)/2 \ge 3p - n_{\rm{dv}}(\cal G)/2 > 2p,\]
	where the second inequality follows from \eqref{eq:ordGinter} and the last strict inequality uses $n_{\rm{dv}}(\cal G) \le 3p/2$. This concludes \eqref{eq:sizeGammamu}. 
\end{proof}

\begin{remark}
	With a more careful analysis, one can show that the removal of each distinguished vertex increases the scaling order by at least 1. This would yield the stronger bound $\ord(\Gamma_{\mu,xy}) \ge 3p$, which in turn improves \eqref{eq:LW_moment} by replacing the factor $[\Psi_t(0)]^p$ with $[\Psi_t(0)]^{2p}$. We do not pursue this refinement, however, since the bottleneck of the main proof lies instead in the martingale estimate of \Cref{lem: EMn2_N} (see \Cref{rmk_bottleneck}).
\end{remark}

\subsection{Proof of \texorpdfstring{\Cref{lem:LWterm,lem: EWGn2_N}}{LemGraph1} for the block Anderson model}\label{subsec:LWchange-to-BA}

Analogous to the random band matrix model, the proofs of \Cref{lem:LWterm,lem: EWGn2_N} for the block Anderson model also require establishing \Cref{lem:LW_moment,lem:LW_moment_exp}. To this end, we introduce new graphical notations and corresponding local expansion rules. 
In the setting of the block Anderson model, we continue to define the matrices $S^\pm$ as in \eqref{eq:def-Spm}, with \smash{$\Theta^{(+,+)}(z)$} defined in \eqref{def:Theta_BA} and $S^{\LK}=I_{L^d}$. 
In addition to the components introduced in \Cref{def_graph1}, we introduce additional types of edges that represent the matrix entries of $\Psi$ and $M$, specific to this model: 
\begin{itemize}			
	
	\item {\bf $\Psi$-dotted edge}: A black dotted edge labeled $\Psi$ between vertices $x$ and $y$ represents a factor $\ilambda\Psi_{xy}$.
	
	\item {\bf $M$-dotted edge}: A blue (resp.~red) dotted edge  labeled $M$ between vertices $x$ and $y$ represents a factor $M_{xy}$ (resp.~$\bar M_{xy}$).
	
\end{itemize}
With these new types of dotted edges, the definitions of molecules and molecular graphs remain the same as in \Cref{def_poly}. However, graphs in the block Anderson model carry an additional level of structure, which we refer to as \emph{atoms}. More precisely, our graphs exhibit microscopic structures within atoms, which are equivalence classes of vertices connected via dotted edges. These atoms form mesoscopic structures within each molecule, and the global (macroscopic) structure is represented by the molecular graph.

\begin{definition}[Atoms and atomic graphs]\label{def_atom}
	We partition the set of all vertices in a graph into disjoint subsets called atoms. Two vertices belong to the same atom if and only if they are connected by a path consisting entirely of dotted edges---that is, any combination of dotted edges, $\Psi$-dotted edges, or $M$-dotted edges.
	An atom is called external if it contains at least one external vertex; otherwise, it is called internal. 
	
	Given a graph $\cal G$, we define its atomic graph as follows:
	\begin{itemize}
		\item Merge all vertices belonging to the same atom into a single vertex. 
		
		\item Retain all solid and waved edges that connect different atoms. 
		
		\item Discard all other components within $\cal G$, including $\times$-dotted edges, edges between vertices within the same atom, and coefficients.
	\end{itemize}
\end{definition}

By the definition of (regular) dotted edges, the form of the $\Psi$-dotted edges in \eqref{eq:Psi3D}, and the bounds for the $M$-dotted edges in \eqref{Mbound_AO} or \eqref{Mbound_AO2}, we deduce that, up to an error of order $e^{-c(\log W)^{1+\e_0}}$ for any $\e_0>0$, 
\be\label{scaleatom}
\hbox{$x$, $y$ belong to the same atom} \implies  x-W[x] = y-W[y],\ \text{and} \ |[x]-[y]|\le (\log W)^{1+\e_0}.
\ee
Here, $[x],[y]\in \Zn$ are the block-level vertices as defined in \Cref{def: BM2}. 
The concept of atoms is introduced for two main purposes: first, to define the scaling size and scaling order of our graphs; and second, to enable a structural comparison between the block Anderson model and the random band matrix model. 
In particular, under the atomic graph formalism, the atomic graphs in the block Anderson model correspond directly to the vertex-level graphs in the random band matrix setting, where the $M$ matrices reduce to scalars. As a result, all statements and arguments concerning molecular graphs from the previous proofs for the random band matrix model carry over verbatim to the block Anderson model.

\begin{definition}[Normal graphs]  \label{defn_normalBA}  
	We say a graph is \emph{normal} if it satisfies the following properties:
	\begin{itemize}
		\item[(i)] It contains at most $\OO(1)$ many vertices and edges.
		
		\item[(ii)] There are no regular dotted edges between vertices (note that $\Psi$- or $M$-dotted edges are allowed).
		
		\item[(iii)] Every solid edge carries a $\circ$ (i.e., it represents an entry of \smash{$\Gc$ or $\Gc^*$}). In particular, all weights are light-weights.
		
	\end{itemize}
\end{definition}

Given an arbitrary graph with $\OO(1)$ many vertices and edges, we can decompose it into a linear combination of normal graphs by expanding each $G_{xy}$ (resp.~$G^*_{xy}$) edge into a \smash{$\Gc_{xy}$} edge plus an $M_{xy}$ edge (resp.~a \smash{$\Gc^*_{xy}$} edge plus an $M^*_{xy}$ edge), and by merging any vertices connected by regular dotted edges. For a normal graph, we define its scaling size in the same way as in \Cref{def scaling}, except that we replace the number of internal vertices with the number of internal atoms.

\begin{definition}[Scaling size and scaling order]\label{def scalingBA} 
We define the scaling size of a normal graph $\Gamma$ as 
	\begin{align}
		\size(\Gamma): = (L^d)^{n_M(\Gamma)}\cdot(\Psi_t)^{n_S(\Gamma)} \cdot W^{-d\left(n_W(\Gamma) - n_A(\Gamma) \right)},
		\label{eq_defsize_BA}
	\end{align} 
	where $n_S(\Gamma)$, $n_W(\Gamma)$, $n_A(\Gamma)$, and $n_M(\Gamma)$ denote the numbers of solid edges (including light-weights), waved edges, internal atoms, and internal molecules, respectively. The scaling order of $\Gamma$ is then defined as
	\be\label{eq:ordG_BA} \ord(\Gamma):=n_S(\Gamma)+ 2 \left(n_W(\Gamma)  -  n_A(\Gamma)\right) . \ee
	If $\Gamma$ can be expressed as a sum of $\OO(1)$ many normal graphs $\Gamma_k$, i.e., $\Gamma=\sum_k \Gamma_k$, we define its scaling size as in \eqref{eq_defsizemax}, and its scaling order by
	\(\ord(\Gamma)=\min_k \ord(\Gamma_k).\)
\end{definition}
Next, we state the local expansion rules for the block Anderson model, as given in \cite{yang2024Del}.  

\begin{lemma}[Basic expansion, Lemma B.9 of \cite{yang2024Del}]\label{lanlw}
	In the setting of the block Anderson model, let $f$ be a differentiable function of $G$. Then, we have the expansion
	\begin{align} \label{eq:BE}
		\Gc_{xy} f(G) =_{\E}  &
		\sum_{\al,\beta}  M_{x\al} S_{\al \beta}
		\Gc_{\beta\beta} G _{\al y} f(G)
		- \sum_{\al,\beta}M_{x\al} S_{\al \beta} G _{\beta y} 
		\partial_{h_{ \beta\al}} f(G) .
	\end{align}
\end{lemma}

The purpose of the basic expansion is to expand any graph as a sum of graphs in which every vertex has a solid-edge degree $\in \{0,2\}$. If a vertex still carries self-loops (i.e., weights), we then apply the following weight expansion.

\begin{lemma} [Weight expansion, Lemma B.10 of \cite{yang2024Del}]\label{lem_lweight}
	In the setting of the block Anderson model, let $f$ be a differentiable function of $G$. Then, we have the expansion
	\begin{align} \label{eq:LW}
		\Gc _{xx} f(G)  
		=_{\E} &\sum_{y}(1+M^+S^+)_{xy}
		\Big(   \sum_{\al,\beta}  M_{y\al} S_{\al \beta} \Gc_{ \al y} 
		\Gc_{\beta\beta} f(G)
		- \sum_{\al,\beta}M_{y\al} S_{\al \beta} G _{\beta y} 
		\partial_{h_{ \beta\al}} f(G)\Big) ,
	\end{align}
	where $M^+$ is the $N\times N$ matrix with entries $M^+_{xy}:=M_{xy}M_{yx}$. 
\end{lemma}

If there are vertices incident to two solid edges of the same charge, we use the following $GG$ expansion.

\begin{lemma}[$GG$ expansion, Lemma B.11 of \cite{yang2024Del}]
	In the setting of the block Anderson model, let $f$ be a differentiable function of $G$. Then, we have the expansion
	\begin{align} 
		&\Gc _{y'x}\Gc _{xy} f(G) 
		=_\E \sum_\beta S^+ _{x\beta } M _{\beta y}M_{y'\beta}f(G)
		+ \sum_{ \beta}S^+_{x\beta}\left(\Gc _{\beta y}M_{y'\beta}+M_{\beta y}\Gc _{y'\beta}\right)f(G)\label{GGGamma}
		\\ & +\sum_{w,\al,\beta}\left(1+M^+S^+\right)_{xw}M_{w\al} S_{\al \beta}\left(   
		\Gc_{\beta\beta} G _{\al y} \Gc _{y'w}f(G)
		+ \Gc _{\al w}G _{\beta y}G_{y'\beta}
		\bigGamma -  G _{\beta y} \Gc _{y'w}
		\partial_{h_{ \beta\al}}   f(G)\right) .\nonumber
	\end{align}
\end{lemma}

As explained in \cite[Appendix B]{yang2024Del}, repeated applications of the above local expansions to an arbitrary normal graph yield a sum of $\OO(1)$ \emph{locally standard graphs} (see \Cref{deflvl1}), subject to the additional requirement that all solid edges are \smash{$\Gc$} edges. This allows us to establish an analogue of \Cref{lvl1 lemma} for the block Anderson model. 
With the above preparations, we can now proceed to the proof of \Cref{lem:LW_moment,lem:LW_moment_exp} for the block Anderson model.
Since the argument is very similar to that in \Cref{Sec:graph}, we only sketch the proof and omit the repetitive details.

\begin{proof}[\bf Proof of \Cref{lem:LW_moment,lem:LW_moment_exp}]
	We begin by applying the local expansion to the graph $|f_{xy}(G)|^p$. This allows us to establish the same result as in \Cref{lem:localregular} for the block Anderson model, including all properties (1)--(6). In fact, due to our earlier discussion, the molecular graphs in the block Anderson setting share the exact same structure as those in the random band matrix model. As a result, all path properties stated in \Cref{lem:localregular} continue to hold in the current context.

	Next, we define the auxiliary graph in a similar manner as in \Cref{def_auxgraph}, with minor modifications that we now describe. Let $\cal G_{xy}$ be a locally standard graph obtained from the local expansions, with $q$ internal molecules $\Mol_i$, $i\in\qqq{q}$, and two external molecules $\Mol_x$ and $\Mol_y$, as in \Cref{def_auxgraph}. For each $i\in \{1,\ldots, q,x,y\}$, we fix a center $\al_i\in \Mol_i$, and we also choose centers $x_{i,j}$ for the atoms $\cal A_{i,j}\subset\Mol_i$, where $j\in\qqq{k_i}$ and $k_i$ denotes the number of atoms in $\Mol_i$. Without loss of generality, we set $\al_i\equiv x_{i,1}$ and assume that $\al_i$ is connected to other molecules via solid edges.
	(By the definition of locally standard graphs, the solid-edge degree of each $\al_i$ is at least 2.)
	If $\p{\beta\to \beta'}$ is a solid edge connecting the atoms $\cal A_{i,j}$ and $\cal A_{i',j'}$, then by \eqref{scaleatom}, the form of the $\Psi$-dotted edges in \eqref{eq:Psi3D}, and the exponential decay estimates \eqref{Mbound_AO} or \eqref{Mbound_AO2} for the $M$-dotted edges, we obtain
	\be\label{Gyiyjatom}
	\big|(\Gc_t)_{\beta\beta'}\big|\prec \zeta(x_{i,j},x_{i',j'}):=\sum_{|[a]|+|[b]|\le (\log W)^{1+\e_0}}\big|(\Gc_{t})_{x_{i,j}+W[a],x_{i',j'}+W[b]}\big| + W^{-D}
	\ee
	for any large constant $D>0$. This quantity satisfies bounds analogous to \eqref{eq:Gbyxi2}:
	\be\label{eq:Gbyxi2_BA}
	\zeta(\al,\beta) \prec \Psi_t(|[\al]-[\beta]|),\quad \sum_{\al\in \ZL}(|\zeta(\al,\beta)|^2+|\zeta(\beta,\al)|^2) \prec \eta_t^{-1},
	\ee
	where the first estimate follows from \Cref{lem_GbEXP_BA}, and the second from Ward’s identity for $G$ together with \eqref{Mbound_AO} or \eqref{Mbound_AO2} for $M$. If $\p{\beta,\beta'}$ is a waved edge connecting the atoms $\cal A_{i,j}$ and $\cal A_{i',j'}$, then by \eqref{scaleatom}, the definition of $S$ in \eqref{bandcwV}, and the estimate \eqref{prop:ThfadC_short}, we obtain
	\be\label{Gyiyjatom2}
	\big|S_{\beta\beta'}\big|+ \big|S^+_{\beta\beta'}\big|\prec W^{-d}\mathbf 1\p{|x_{i,j}-x_{i',j'}|\le W(\log W)^{1+\e_0}} + W^{-D}.\ee

	We first define a graph $\cal G^{a}_{xy}$ as an ``atomic reduction" of $\cal G_{xy}$:
	\begin{itemize}
		\item Its vertices are the atom centers $x_{i,j}$ for $i\in\{1,\ldots,q,x,y\}$ and $j\in\qqq{k_i}$.
		
		\item For each solid edge in $\cal G_{xy}$ from the atom $\cal A_{i,j}$ to the atom $\cal A_{i',j'}$, we introduce an oriented solid edge from $x_{i,j}$ to $x_{i',j'}$ in $\cal G^{a}_{xy}$, representing the factor $\zeta(x_{i,j},x_{i',j'})$.
		
		\item For each waved edge in $\cal G_{xy}$ connecting the atoms $\cal A_{i,j}$ and $\cal A_{i',j'}$, we introduce a waved edge between $x_{i,j}$ and $x_{i',j'}$ in $\cal G^{a}_{xy}$, representing the factor $W^{-d}\mathbf 1 \left(|x_{i,j}-x_{i',j'}|\le W(\log W)^{1+\e_0}\right)$.
		
	\end{itemize}
	Next, we further simplify the structure of $\cal G^{a}_{xy}$ and define the auxiliary graph $\cal G^{\aux}_{xy}$ as follows.
	First, we remove all solid and waved edges that lie entirely within individual molecules of $\cal G^{a}_{xy}$.
	Second, for each molecule $\Mol_i$, we retain only those vertices (including the center $\al_i$) that connect to other molecules through solid edges; all remaining vertices inside $\Mol_i$ are discarded. 
	Third, for each retained vertex $x_{i,j}\ne \al_i$ inside $\Mol_i$, we add a new waved edge connecting $x_{i,j}$ to $\al_i$, representing the factor $W^{-d}\mathbf 1(|x_{i,j}-\al_i|\le W(\log W)^{1+2\e_0})$.
	In other words, the molecular structure is reduced to the following simplified form:
	\be  \nonumber
	\parbox[c]{0.8\linewidth}{
		\begin{center}
			\scalebox{0.9}{
				\begin{tikzpicture}
					\begin{feynman}
						
						\vertex [dot] (a1) at (0,0.7);
						\fill[black] (a1) circle (2pt) node at (0.2,0.5) {$\al_i$};
						\vertex [dot] (a) at (0,0);   
						\fill[black] (a) circle (2pt) node[left=2pt]{};
						\vertex [dot] (a2) at (-0.7,0.7);
						\fill[black] (a2) circle (2pt) node[left=2pt]{$x_{i,2}$};
						\vertex [dot] (a3) at (0,1.4);   
						\fill[black] (a3) circle (2pt) node[left=2pt]{$x_{i,3}$};
						\vertex [dot] (a4) at (0.7,0.7);   
						\fill[black] (a4) circle (2pt) node[right=2pt]{$x_{i,4}$};
						\vertex (a5) at (0.5,0.2) {$\cdots$};
						
						\diagram*{
							
							(a) -- [photon] (a1), (a1) -- [photon] (a2), (a1) -- [photon] (a3), (a1) -- [photon] (a4)
							
						};
					\end{feynman}
			\end{tikzpicture}}
	\end{center}}\ee
	For the auxiliary graph $\cal G^{\aux}_{xy}$ constructed above, we define its scaling order by
	\be\label{eq:ordGaux_BAM} \ord(\cal G^{\aux}_{xy}):=\#\{\text{solid edges in } \cal G^{\aux}_{xy}\}+2\#\{\text{waved edges in } \cal G^{\aux}_{xy}\}-2\#\{\text{internal vertices in } \cal G^{\aux}_{xy}\}.
	\ee
	From the construction, together with \eqref{eq:Gbyxi2_BA} and by repeating the argument used in the proof of \Cref{GtoAG}, we can control $\cal G_{xy}$ via its auxiliary graph as
	\be\label{G_by_auxG_BA} 
	\cal G_{xy} \prec  (\Psi_t)^{\ord(\cal G_{xy}) - \ord(\cal G_{xy}^{\aux})}  \cdot \cal G^{\aux}_{xy} + W^{-D}\, .
	\ee
	Finally, using \eqref{eq:Gbyxi2_BA}, the auxiliary graph $\cal G^{\aux}_{xy}$ can be bounded by an argument parallel to that in \Cref{sec:pf_Anp}.
	In particular, the argument there, when applied to the molecular graph of $\cal G^{\aux}_{xy}$, provides a systematic procedure for selecting the long solid edges between molecules and for establishing a nested summation order over the internal molecules.
	The only difference from the random band matrix case is that the ``summation over a molecule" here also includes summing over certain non-center vertices within the molecule. However, only two solid edges incident to $\Mol_i$ are used in the summation step, while every other solid edge contributes either a long-edge factor or a $\Psi_t$-factor.
	Applying the Cauchy-Schwarz inequality together with the bound \eqref{eq:Gbyxi2_BA} to these two solid edges, the summation over all vertices in the molecule $\Mol_i$ produces the desired factor of $\eta_t^{-1}$.
	With this modification, the argument of \Cref{sec:pf_Anp} carries over verbatim, completing the proof of \Cref{lem:LW_moment} for the block Anderson model.

	For the proof of \Cref{lem:LW_moment_exp}, the bounds in \eqref{eq:Gbyxi2_BA} remain valid with $\Psi_t(|[\al]-[\beta]|)=\sT_t(|[\al]-[\beta]|\wedge\ell)+W^{-D}$. Then, by following the same reasoning as in \Cref{subsec_pf_LW_moment_exp}, we complete the proof of \Cref{lem:LW_moment_exp} for the block Anderson model.
\end{proof}

\section{Proofs of some deterministic estimates}

This appendix is devoted to proving several deterministic estimates used in the main proof.

\subsection{Proof of \texorpdfstring{\Cref{lem_propTH}}{DeterLemma}}\label{sec:pf_propTH}


Property 1 follows from the underlying symmetry of $S^{\LK}$ in the random band matrix model, where $M^{(\sig_1,\sig_2)}$ is a scalar matrix, or from the symmetry of $M^{(\sig_1,\sig_2)}$ in the block Anderson model, where \smash{$S^{\LK}$} is the identity matrix. Property 2 is a consequence of the translation invariance of the matrix $M^{(\sig_1,\sig_2)}S^{\LK}$. Property 3 follows from the fact that \smash{$\Theta^{(\sig_1,\sig_2)}_t$} is a rational function of $S^{\LK}$ in the random band matrix model, and from that $S^{\LK}=I_{L^d}$ in the block Anderson model. 
To prove property 4, we expand \smash{$\Theta^{(\sig_1,\sig_2)}_t$} using the Taylor series
\be\label{eq;Taylor}
\Theta^{(\sig_1,\sig_2)}_t= \sum_{k=0}^\infty   t^k \big(M^{(\sig_1,\sig_2)}S^{(\sB)}\big)^k.\ee
Since $|M_{ab}^{(\sig_1,\sig_2)}|\le M_{ab}^{(+,-)}$ for both models, it follows that for any $\sig_1,\sig_2\in\{+,-\}$,
\[ |\Theta_{t,ab}^{(\sig_1,\sig_2)}| \le  \sum_{k=0}^\infty t^k \big(M^{(+,-)}S^{(\sB)}\big)_{ab}^k  = \Theta_{t,ab}^{(+,-)}.\]
The estimate \eqref{eq:THETAinftinf} then follows directly from this inequality and the identity $\sum_{b}\Theta_{t,ab}^{(+,-)}\equiv (1-t)^{-1}$, which holds because \smash{$M^{(+,-)}S^{(\sB)}$} is doubly stochastic. This uses $|m|=1$ in the random band matrix model, and Ward's identity \eqref{eq:WardM} in the block Anderson model.

To prove \eqref{prop:ThfadC_short} for the random band matrix model, we use the following shifted Taylor expansion:
\be\label{eq:shiftTaylor} \Theta_t^{(+,+)} = \frac{1}{1+\e}\sum_{k=0}^\infty \p{\frac{t m^2 S^{(\sB)} + \e}{1+\e}}^k, \ee
where $\e>0$ is a positive constant. As shown in \cite[Lemma 4.2]{bourgade2019random}, one has $\|(tm^2 S^{(\sB)} + \e)/(1+\e)\|_{\infty\to \infty} \le 1-c$ for some constant $c>0$ depending on $\e$ and $\kappa$. Applying this estimate to the expansion \eqref{eq:shiftTaylor} and noticing that the off-diagonal entries of $\p{t m^2 S^{(\sB)} + \e}/\p{1+\e}$ contain a $\ilambda^2$ factor (by the definition \eqref{eq:variancematrix}), we derive the bound \eqref{prop:ThfadC_short}. For the block Anderson model, we instead use the Taylor expansion 
\begin{align}
	\left(1-tM^{(+,+)}\right)^{-1}_{0a} &=\left[\left(1-tm^2\right)I - t M'\right]^{-1}_{0a} = \sum_{k=0}^{\infty}\left(1-tm^2\right)^{-(k+1)}
	\left(t M'\right)^{k}_{0a},\label{eq:expMLn}
\end{align}
where $M':=M^{(+,+)}- m^2I$ is obtained from $M^{(+,+)}$ by setting its diagonal entries to zero.
By property (2) of \Cref{lem:propM}, there exists a constant $\e>0$ such that $|1-tm^2|\ge \e$ and
\be\label{eq:off_diagM}
\|M'\|_{\infty\to \infty}= \sum_{a\ne 0}|M'_{0a}| = 1 - |m|^2 \le (1-\e)|1-tm^2| ,\quad \forall t\in[0,1].
\ee
Applying this bound to \eqref{eq:expMLn}, we deduce the following estimate: there exists a constant $C>0$ such that for any $\delta>0$,
\begin{align}
	\Big|\big(1-tM^{(+,+)}\big)^{-1}_{0a}\Big| &\le  \sum_{0\le k\le \delta |a|}\left(1-tm^2\right)^{-(k+1)}
	\left(t M'\right)^{k}_{0a} + \e^{-1}\sum_{k> \delta |a|}(1-\e)^k \nonumber\\
	&\le \sum_{0\le k\le \delta |a|} C^k \exp(-c|a|) +\e^{-2}(1-\e)^{\delta|a|} \le C' e^{-c' |a|}.\label{eq:expMLn2}
\end{align}
In the second inequality, we use the exponential decay of $M'$ by \eqref{Mbound_AO} or \eqref{Mbound_AO2}, together with the convolution bound $\sum_{b}e^{-c|a_1-b|-c|a_2-b|}\lesssim e^{-c|a_1-a_2|}$ for all $a_1,a_2\in \Zn$, and in the third inequality we choose $\delta>0$ sufficiently small so that $C^{\delta}\le e^{-c/2}$. The constants $c', C'$ depend only on $C,\ c,\ \e$, and $\delta$. The bound \eqref{eq:expMLn2} yields \eqref{prop:ThfadC_short} in the regime $\ilambda\gtrsim 1$. When $\ilambda\ll 1$, we return to \eqref{eq:expMLn} and note that $\|M'\|_{\infty\to\infty}\lesssim \lambda^2$ by \eqref{Mbound_AO}. In this case, for $a\ne 0$, the same argument as in \eqref{eq:expMLn2} applies, leading to the bound \eqref{prop:ThfadC_short}.

The bounds \eqref{prop:ThfadC} and \eqref{prop:BD1}–\eqref{prop:ThfadC0} follow directly from \eqref{prop:ThfadC_short} in the case $\sigma_1=\sigma_2$. It therefore remains to consider the case $\sigma_1\neq \sigma_2$. In this case, the estimates \eqref{prop:ThfadC0} and \eqref{prop:BD2} were proved in \cite{yang2024Del} by analyzing the Fourier series of \smash{$\zTheta_t^{(+,-)}$} through a standard summation-by-parts argument.\footnote{In \cite{yang2024Del}, the assumption $\ilambda\ll 1$ was imposed; however, the same argument (in fact, slightly simpler) applies when $\ilambda\gtrsim 1$.} 
Specifically, \eqref{prop:ThfadC0} is proved in Lemma 3.1 of \cite{yang2024Del}, while \eqref{prop:BD2} appears as equation (E.19) therein. Although the bound \eqref{prop:BD1} is not stated explicitly in \cite{yang2024Del}, its proof proceeds analogously to those of \eqref{prop:ThfadC0} and \eqref{prop:BD2}, by applying the same summation-by-parts technique to the corresponding Fourier expansion. We therefore omit the details.
These bounds \eqref{prop:BD1}--\eqref{prop:ThfadC0} have also been derived for dimension $d=2$ in \cite[Lemma 3.10]{RBSO1D}, where the summation-by-parts argument is explained in Appendix~B. The same reasoning extends directly to dimensions $d\ge 3$.

Finally, it remains to prove the bound \eqref{prop:ThfadC} for the case $\sig_1\ne \sig_2$. Its proof is similar to that of \cite[Lemma 2.14]{DYYY25}, using the Taylor expansion, along with the random walk representation of \smash{$(S^{(\sB)})^k$} (for the random band matrix model,) or \smash{$(M^{(+,-)})^k$} (for the block Anderson model). 
More precisely, we consider the random walk $\{X_k:k\ge 0\}$ on $\Z^d$, with transition probabilities \smash{$\big\{p(0,a)=S_{0a}^{(\sB)}:a\in \Z^d\big\}$} for the random band matrix model, or \smash{$\big\{p(0,a)=M^{(+,-)}_{0a}:a\in \Z^d\big\}$} for the block Anderson model.  
First, given any $a\in \Z^d\setminus\{0\}$, by applying the Bernstein inequality to $X_k\cdot \wh a$ with $\wh a$ denoting the unit vector $a/\|a\|_2$, we can derive the following large deviation estimate: there exist constants $c, C>0$ (which does not depend on $a$) such that 
\[ \P_0(X_k=a) \le C\exp\p{-c \p{\frac{|a|^2}{\ilambda^2k}\wedge |a|}},\quad \forall a\in \Z^d,\ k\ge 1 . \]
On the other hand, using the local CLT for $X_k$ (see e.g., \cite[Section 2]{Lawler_book}), we obtain that 
\[ \P_0(X_k=a) \le 1\wedge \p{C (\ilambda^2k)^{-d/2}},\quad \forall a\in \Z^d,\ k\ge 1 . \]
Combining the above two bounds and using the argument below equation (8.3) of \cite{DYYY25}, we obtain the large deviation bound
\[ \P_0(X_k=a) \le C\p{ (\ilambda^2k)^{-d/2}\wedge 1}\exp\p{-c \p{\frac{|a|^2}{\ilambda^2k}\wedge |a|}},\quad \forall a\in \Z^d,\ k\ge 1,  \]
for some constants $c,C>0$. This estimate allows us to control $(M^{(+,-)}S^{(\sB)})_{0a}^k$ by projecting the random walk onto the torus $\Zn$. Applying these estimates to the expansion \eqref{eq;Taylor} and summing over the resulting terms, we can derive the desired bound \eqref{prop:ThfadC}.
Since the argument closely follows that in \cite[Section 8]{DYYY25}, we omit the details. In fact, the proof here is somewhat simpler than in \cite{DYYY25}, because $d=2$ is the critical dimension, where logarithmic corrections (e.g., $\log L$) appear. In contrast, for $d\ge 3$, all relevant series (or integrals) are summable, resulting in a dimension-dependent constant $C_d$ in \eqref{prop:ThfadC}.

\subsection{Proofs of evolution kernel estimates}\label{ks}

In this subsection, we present the proofs of \Cref{lem:sum_Ndecay,lem:sum_decay,lem:sum_decay_nonzero}. 
Parts of the proofs parallel those in \cite{YY_25,DYYY25,RBSO1D} for random band matrices and for the block Anderson model in dimensions 1 and 2. However, certain key arguments must be adapted to handle the higher-dimensional setting $d \ge 3$. We therefore outline the proofs of the evolution kernel estimates, emphasizing the modifications needed compared to the arguments in \cite{YY_25,DYYY25,RBSO1D}. 

\begin{proof}[\bf Proof of \Cref{lem:sum_Ndecay}]
	Notice the following decomposition:
	\be\label{eq:decompUalt}\frac{1 - s \cdot M^{(\sig_i,\sig_{i+1})}S^{(\sB)}}{1 - t \cdot M^{(\sig_i,\sig_{i+1})}S^{(\sB)}} =1+ \Xi^{(i)},\quad \text{where}\quad \Xi^{(i)}:=(t-s)\cdot M^{(\sig_i,\sig_{i+1})}S^{\LK} \Theta_{t}^{(\sig_i,\sig_{i+1})} .\ee
	Using \eqref{eq:THETAinftinf}, along with \eqref{eq:WardM} in the case of block Anderson model, we obtain that
	\be\label{Xi_infint} \|\Xi^{(i)}\|_{\infty\to \infty} = \max_{a}\sum_b |\Xi^{(i)}_{ab}| \le \p{t-s}\|\Theta^{(\sig_i,\sig_{i+1})}_{t}\|_{\infty\to \infty} \le \frac{t-s}{1-t}. \ee
	Together with \eqref{eq:decompUalt}, this implies that 
	$$\left\|\frac{1 - s \cdot M^{(\sig_i,\sig_{i+1})}S^{(\sB)}}{1 - t \cdot M^{(\sig_i,\sig_{i+1})}S^{(\sB)}}\right\|_{\infty\to \infty}\le \frac{1-s}{1-t} .$$
	With this estimate, we conclude \eqref{sum_res_Ndecay} immediately using the definition \eqref{def_Ustz}. 
\end{proof}

\begin{proof}[\bf Proof of \Cref{lem:sum_decay}]
	With the decomposition \eqref{eq:decompUalt}, we can express \smash{$\mathcal{U}^{(\fn)}_{s, t, \boldsymbol{\sigma}} \circ \mathcal{A}$} as
	\begin{align}\label{eq:decomp_U2}
		\left(\mathcal{U}^{(\fn)}_{s, t, \boldsymbol{\sigma}} \circ \mathcal{A}\right)_{\ba}&=\sum_{\mathbf b\in (\Zn)^\fn} \prod_{i=1}^\fn \left( \delta_{a_ib_i}+\Xi^{(i)}_{a_ib_i} \right)  \cdot \mathcal{A}_{\mathbf b} =\sum_{\mathbf b\in (\Zn)^\fn}\sum_{A\subset \qqq{\fn}} \prod_{i\in A} \delta_{a_ib_i} \cdot \prod_{i\in A^c} \Xi^{(i)}_{a_ib_i}  \cdot \mathcal{A}_{\mathbf{b}}\, .
	\end{align}
	By the estimate \eqref{prop:ThfadC} (along with \eqref{Mbound_AO} or \eqref{Mbound_AO2} in the case of block Anderson model), we have that
	\be\label{eq:decayXi}
	\Xi^{(i)}_{a_ib_i}\lesssim (1-s)\frac{(\ilambda^2+|1-t|)^{-1}}{(|a_i-b_i|+1)^{d-2}} e^{-c |a_i-b_i|/ {\ell}_t} 
	\ee
	for a constant $c>0$. We claim that for any subset $A$ with $|A|=k\in \qqq{1,\fn}$, 
	\be\label{sum_res_1_red0}
	\sum_{\mathbf b\in (\Zn)^\fn} \prod_{i\in A} \delta_{a_ib_i} \cdot \prod_{i\in A^c} \Xi^{(i)}_{a_ib_i}  \cdot \mathcal{A}_{\boldsymbol{b}} \le W^{C\e}\p{\frac{\ilambda^2+|1-s|}{\ilambda^2+|1-t|}}^{\fn-k} \|{\cal A}\|_{\infty}+ W^{-D+C },
	\ee
	for a constant $C$ that does not depend on $\e$ or $D$, while for $A= \emptyset$, we claim that 
	\be\label{sum_res_1_red}
	\sum_{\mathbf b\in (\Zn)^\fn} \prod_{i=1}^\fn \Xi^{(i)}_{a_ib_i}  \cdot \mathcal{A}_{\boldsymbol{b}} \le W^{C\e} \frac{\ell_t^2}{\ell_s^2}\p{\frac{\ilambda^2+|1-s|}{\ilambda^2+|1-t|}}^{\fn} \|{\cal A}\|_{\infty}+ W^{-D+C}.
	\ee
	Note that combining \eqref{sum_res_1_red0} and \eqref{sum_res_1_red} concludes the proof of \eqref{sum_res_1}.   
	
	To show the estimate \eqref{sum_res_1_red0}, we assume that $A=\qqq{1,k}$ without loss of generality. With the notations $\mathbf a'=(a_{1},\ldots,a_{k})$ and $\mathbf b'=(b_{k+1},\ldots,b_{\fn})$, we can bound the LHS of \eqref{sum_res_1_red0} as 
	\begin{align*}
		\sum_{\mathbf b'\in (\Zn)^{\fn-k}} \prod_{i=k+1}^\fn  \Xi^{(i)}_{a_ib_i}  \cdot \mathcal{A}_{\ba',\mathbf{b}'} &\lesssim  \frac{\left(1-s\right)^{\fn-k}}{(\ilambda^2+|1-t|)^{n-k}} \|\cal A\|_\infty \sum_{\mathbf b'} \prod_{i=k+1}^n\frac{\mathbf 1\left(|b_i-a_1|\le W^\e\ell_s \right)}{|a_i-b_i|^{d-2}+1} +  W^{-D+(\fn-k)}\\
		&\lesssim (W^{2\e})^{\fn-k}\left( \frac{\ell_s^2|1-s|}{\ilambda^2+|1-t|} \right)^{\fn-k} \|\cal A\|_\infty  + W^{-D+(\fn-k)} \\
		& \lesssim \p{W^{2\e}\frac{\ilambda^2+|1-t|}{\ilambda^2+|1-s|}}^{\fn-k} \|\cal A\|_\infty  + W^{-D+(\fn-k)} , 
	\end{align*}
	where we have used \eqref{eq:decayXi}, the decay property \eqref{deccA0} for $\cal A_{\mathbf b}$, and 
	\be\label{eq:1-sells2}(1-s)\ell_s^2\asymp \ilambda^2+|1-s|,\quad \text{for}\quad s\le 1-\ilambda^2/L^{2},\ee 
	along with the condition $(1-t)/(1-s)\ge W^{-1}$ and the bound \eqref{Xi_infint} in getting the $W^{-D+(\fn-k)}$ term. This concludes \eqref{sum_res_1_red0} for any constant $C>2(\fn-k)$.  
	For \eqref{sum_res_1_red}, with the notation $\mathbf b'=(b_{2},\ldots,b_{\fn})$, we get that 
	\begin{align}
		\sum_{\mathbf b}\prod_{i=1}^{\fn} \Xi^{(i)}_{a_ib_i} \cdot \mathcal{A}_{\mathbf b} &= \sum_{\mathbf b}\prod_{i=1}^{\fn} \Xi^{(i)}_{a_ib_i} \cdot \mathcal{A}_{\mathbf b}\cdot  \mathbf 1\left(\max_{i\ne j}|b_i-b_j|\le W^\e\ell_s \right) + \OO\left( W^{-D+\fn}\right) \nonumber\\
		&\lesssim \frac{(1-s)^{\fn-1}}{(\ilambda^2+|1-t|)^{n-1}} \|\cal A\|_\infty \sum_{b_1} \left|\Xi^{(1)}_{a_1b_1}\right| \sum_{\mathbf b'} \prod_{i=2}^n \frac{\mathbf 1\left(|b_i-b_1|\le W^\e\ell_s \right)}{|a_i-b_i|^{d-2}+1} +  W^{-D+\fn} \nonumber\\
		&\lesssim W^{2(\fn-1) \e}\frac{1-s}{1-t}\left(\frac{\ell_s^2|1-s|}{\ilambda^2+|1-t|}\right)^{\fn-1} \|\cal A\|_\infty + W^{-D+\fn} \nonumber\\
		&\lesssim  W^{2(\fn-1) \e}\frac{\ell_t^2}{\ell_s^2}\left(\frac{\ilambda^2+|1-s|}{\ilambda^2+|1-t|}\right)^{\fn} \|\cal A\|_\infty + W^{-D+\fn}, \label{sum_res_deriv_red}
	\end{align} 
	where in the first step, we use the decay property \eqref{deccA0} for $\cal A_{\mathbf b}$ and the bound \eqref{Xi_infint} (in getting the $W^{-D+n}$ term); in the second step, we apply \eqref{eq:decayXi} for $i\in\qqq{2,\fn}$; in the third step, we take the summation over $\mathbf b'$ and apply \eqref{Xi_infint}; in the last step, we use \eqref{eq:1-sells2} again.  This gives \eqref{sum_res_1_red} for any constant $C>2(\fn-1)$.

	For the estimates \eqref{sum_res_2_NAL} and \eqref{sum_res_2}, 
	we notice that when $|A| \ge 1$, \eqref{sum_res_1_red0} already gives a good enough bound. 
	Hence, we only need to focus on the case where $A=\emptyset$. 
	First, for Case I, due to \eqref{prop:ThfadC_short}, we can use the better bound \smash{$\sum_{b_1}|\Xi^{(1)}_{a_1b_1}|=\OO(1)$} in the third step of \eqref{sum_res_deriv_red}, which leads to \eqref{sum_res_2_NAL}. 
	
	Next, for the estimate \eqref{sum_res_2} in Case II, we need to show that
	\be\label{sum_res_2_red}
	\sum_{\mathbf b\in (\Zn)^\fn} \prod_{i=1}^\fn \Xi^{(i)}_{a_ib_i}  \cdot \mathcal{A}_{\boldsymbol{b}} \le W^{C_\fn\e}\p{\frac{\ilambda^2+|1-s|}{\ilambda^2+|1-t|}}^{\fn}  \|{\cal A}\|_{\infty} +W^{-D+C_\fn}.
	\ee
	It suffices to assume that $\sigma_i\ne \sigma_{i+1}$ for all $i\in \qqq{\fn}$. 
	We decompose $\Zn$ into the following two regions:
	\[S_{\mathrm{far}}:=\Big\{b\in \Zn:\min_{i=1}^n|b-a_i|>  W^{2\e}\ell_s\Big\},\quad S_{\mathrm{near}}:=\Big\{b\in \Zn:\min_{i=1}^n|b-a_i|\le W^{2\e}\ell_s\Big\}.\]
	Following a similar argument as in \eqref{sum_res_deriv_red}, and using that \smash{$\sum_{b_1\in S_{\mathrm{near}}} |\Xi^{(1)}_{a_1b_1}|\lesssim W^{4\e}(\ilambda^2+|1-s|)/(\ilambda^2+|1-t|)$} by \eqref{eq:decayXi} and \eqref{eq:1-sells2}, we obtain
	\begin{align}
		\sum_{b_1\in S_{\mathrm{near}}} \sum_{\bfb'}\prod_{i=1}^{\fn} \Xi^{(i)}_{a_ib_i} \cdot \mathcal{A}_{\mathbf b} \lesssim W^{2(\fn+1) \e} \left(\frac{\ilambda^2+|1-s|}{\ilambda^2+|1-t|}\right)^{\fn}\|\cal A\|_\infty + W^{-D+\fn}. \label{sum_res_deriv_red2}
	\end{align} 
	It remains to control the sum over $b_1\in S_{\mathrm{far}}$. In this case, we decompose $\Xi^{(i)}$ as 
	\be\nonumber \Xi^{(i)}_{a_ib_i}=\Xi^{(i)}_{a_ib_1} + \Delta \Xi^{(i)}_{a_i;b_1b_i},\quad\text{with}\quad \Delta \Xi^{(i)}_{a_i;b_1b_i}:=\Xi^{(i)}_{a_ib_i}-\Xi^{(i)}_{a_ib_1}. \ee
	Then, we expand the LHS of \eqref{sum_res_2_red} as 
	\begin{align} \label{eq:decompXii}
		\sum_{A: A\subset \qqq{2,\fn}}f(A),\quad \text{with}\quad 
		f(A):=\sum_{b_1\in S_{\mathrm{far}}}\sum_{\bfb'\in (\Zn)^{\fn-1}}\prod_{i\in A^c}  \Xi^{(i)}_{a_ib_1} \cdot \prod_{i\in A} \Delta \Xi^{(i)}_{a_i;b_1b_i}  \cdot \mathcal{A}_{\mathbf{b}}. 
	\end{align}
	By the sum-zero property \eqref{sumAzero}, the leading term with $A=\emptyset$ vanishes. For the remaining terms, we will use  \eqref{eq:decayXi} to control the factors \smash{$\Xi^{(i)}_{a_ib_1}$}, and apply \eqref{eq:decayXi} and \eqref{prop:BD1} (along with \eqref{Mbound_AO} or \eqref{Mbound_AO2} in the case of block Anderson model) to control the factors 
	\smash{$\Delta \Xi^{(i)}_{a_i;b_1b_i}$} for $|a_i-b_1|\ge W^{2\e}\ell_s$ and $|b_i-b_1|\le W^{\e}\ell_s$:
	\be\label{eq:Xibb}
	\Xi^{(i)}_{a_ib_1}\lesssim \frac{W^{-(d-1)\e}}{\ell_s^d}\left(\frac{\ilambda^2+|1-s|}{\ilambda^2+|1-t|}\right),\   \Delta \Xi^{(i)}_{a_i;b_1b_i} \lesssim  \frac{1-s}{\ilambda^2+|1-t|}\frac{|b_i-b_1|}{|a_i-b_1|^{d-1}}  \prec \frac{W^{-d\e}}{\ell_s^d}\left(\frac{\ilambda^2+|1-s|}{\ilambda^2+|1-t|}\right),
	\ee
	where we have also used $2(d-2)\ge d-1$ and \eqref{eq:1-sells2} in the derivation. 
	Without loss of generality, suppose $2\notin A$. Then, using the estimates in \eqref{eq:decayXi} and \eqref{eq:Xibb}, along with the decay property \eqref{deccA0} for $\cal A$, by a similar argument as in \eqref{sum_res_deriv_red}, we can bound $f(A)$ by 
	\begin{align}
		f(A)&\lesssim W^{\e (n-1)} \frac{(1-s)(\ilambda^2+|1-s|)^{n-2}}{(\ilambda^2+|1-t|)^{n-1}}  \|\cal A\|_\infty 
		\sum_{b_1\in S_{\mathrm{far}}} \left|\Xi^{(1)}_{a_1b_1}\right| \frac{ (W^\e\ell_s)^{d+1}}{|a_2-b_1|^{d-1}} + W^{-D+\fn}\nonumber\\
		&\lesssim  W^{(\fn+d+1) \e }\frac{(1-s)^2\ell_s^{d+1}(\ilambda^2+|1-s|)^{n-2}}{(\ilambda^2+|1-t|)^{n}}  \|\cal A\|_\infty  
		\sum_{b_1\in S_{\mathrm{far}}} \frac{\exp(-c|a_1-b_1]|/\ell_t)}{|a_1-b_1|^{d-2} \cdot |a_2-b_1|^{d-1}} + W^{-D+\fn} \nonumber\\
		&\lesssim  W^{(\fn+d+1)\e } \|\cal A\|_\infty \cdot  \frac{(1-s)^2\ell_s^{d+1}}{(W^{2\e}\ell_s)^{d-3}}\frac{(\ilambda^2+|1-s|)^{n-2}}{(\ilambda^2+|1-t|)^{n}} + W^{-D+\fn} \nonumber\\
		&\lesssim W^{(\fn+4)\e }  \p{\frac{\ilambda^2+|1-s|}{\ilambda^2+|1-t|}}^{\fn} \|\cal A\|_\infty+ W^{-D+\fn}. 
		\label{eq:bddfA}
	\end{align}
	This concludes the estimate \eqref{sum_res_2_red}, which further concludes \eqref{sum_res_2}. 
\end{proof}


\begin{proof}[\bf Proof of \Cref{lem:sum_decay_nonzero}]
	Let $\mathbf e\in \C^{L^d}$ denote the unit vector with $\mathbf e(a)\equiv L^{-d/2}$, and define the projection matrix ${\mathsf{Proj}}_{\mathbf e^{\perp}}$ on to the orthogonal complement of $\mathbf e$:
	\( {\mathsf{Proj}}_{\mathbf e^{\perp}} =I-\mathbf e\mathbf e^\top.\) 
	When $\sig_i=\sig_{i+1}$, we have
	\be\label{eq:samecolor} \|\Xi^{(i)}\|_{\infty\to\infty} \lesssim t-s,\quad \text{and}\quad \|{\mathsf{Proj}}_{\mathbf e^{\perp}}\cdot  \Xi^{(i)}\|_{\infty\to\infty} \lesssim t-s,\ee
	by using the definition of $\Xi^{(i)}$ in \eqref{eq:decompUalt} and the estimate \eqref{prop:ThfadC_short}. 
	On the other hand, when $\sig_i\ne \sig_{i+1}$, we have  \[{\mathsf{Proj}}_{\mathbf e^{\perp}}\cdot  \Xi^{(i)}= (t-s) {\mathsf{Proj}}_{\mathbf e^{\perp}}\cdot M^{(+,-)}S^{\LK}\Theta_t^{(+,-)} = (t-s)M^{(+,-)}S^{\LK} \zTheta_t^{(+,-)}.\] 
Then, using \eqref{prop:ThfadC0}  (along with \eqref{eq:WardM} in the case of block Anderson model), we get 
	\be\label{eq:diffcolor}
	\|{\mathsf{Proj}}_{\mathbf e^{\perp}}\cdot  \Xi^{(i)}\|_{\infty\to\infty} \lesssim (t-s) \max_{a}\sum_{b}\big|\zTheta_{t,ab}^{(+,-)}\big| \prec (1-s)\cdot \ilambda^{-2}L^2 \le 1,
	\ee
	where we use $1-s\le \ilambda^2/L^{2}$ in the last step.
	Plugging \eqref{eq:samecolor} and \eqref{eq:diffcolor} into \eqref{eq:decompUalt}, we obtain that 
		\[ \left\|{\mathsf{Proj}}_{\mathbf e^{\perp}}\cdot  \frac{1 - s \cdot M^{(\sig_i,\sig_{i+1})}S^{(\sB)}}{1 - t \cdot M^{(\sig_i,\sig_{i+1})}S^{(\sB)}}\right\|_{\infty\to\infty}\lesssim 1\]
	for all $i\in \qqq{n}$. With this bound, we readily conclude the proof of \eqref{sum_res_Ndecay_nonzero}. 
\end{proof}

\subsection{Proof of \Cref{lem:propT}}\label{sec:pfpropT}
First, we consider the case $1-t \le  1-u \le \ilambda^2/L^{2}$, where we have $\ell_t=\ell_u=L$. In this case, the exponential factor is of order 1, and we have 
\begin{align}
	{\cal T}_{u}(|a-c|)\cdot {\cal T}_{t}(|c-b|) &\lesssim \sum_{c} \left(\frac{\ilambda^{-2}}{|a-c|^{d-2}+1}+\frac{1}{L^d|1-u|}\right)\left(\frac{\ilambda^{-2}}{|c-b|^{d-2}+1}+\frac{1}{L^{d}|1-t|}\right) \nonumber
	\\\nonumber
	&\lesssim  \frac{\ilambda^{-4}L^2}{|a-b|^{d-2}+1} + \frac{\ilambda^{-2}}{L^{d-2}|1-t|}+\frac{1}{1-u}\frac{1}{L^d |1-t|} 
	\\
	&\lesssim   \frac{1}{1-u}\left(
	\frac{\ilambda^{-2}}{|a-b|^{d-2}+1}+\frac{1}{L^d |1-t|}\right)
	\lesssim   \frac{1}{1-u} {\cal T}_{t}(|a-b|),\label{eq:smalletacase}
\end{align}   
where, in the third and fourth steps, we use that $1-u\le \ilambda^2/L^2$. Next, we consider the case $1-u\ge 1-t \ge \ilambda^2/L^{2}$, where we have $\ell_u \le \ell_t=\max(\ilambda(1-t)^{-1/2},1)\le L$. In this case, the term $(L^d|1-t|)^{-1}$ can always be neglected in the function $\cT_t$. Then, we get that
\begin{align}
	&~	 {{\cal T}_{u}(|a-c|)\cdot {\cal T}_{t}(|c-b|)}\big/{{\cal T}_{t}(|a-b|)} \nonumber\\
	\lesssim &~ \frac{1}{\ilambda^2+|1-u|}\Bigg\{1 + \sum_{c\notin \{a,b\}}  \frac{(|a-b|+1)^{d-2}}{|a-c|^{d-2}|c-b|^{d-2}  }\exp \left(-\frac{\sqrt{|a-c|}+
		(\ell_u/\ell_t)^{1/2}(\sqrt{|c-b|}-\sqrt{|a-b|})}{\ell_u^{1/2}} \right)\Bigg\} \nonumber\\
	\lesssim &~ {\ell_u^2}/\p{\ilambda^2+|1-u|} \lesssim |1-u|^{-1},\label{eq:largeetacase}
\end{align}
where, in the second step, we use the following basic calculus fact for $d\ge 3$ and any $0\le \e \le 1$:
$$
\max_{a\in \mathbb R^d} \int_{x\in \mathbb R^d}\frac{|a|^{d-2}}{|a-x|^{d-2}\cdot |x|^{d-2}}\exp\left(-\sqrt{|a-x|}-\e \p{\sqrt{|x|}-\sqrt{|a|}} \right)\dd x\le C_d$$
for a constant $C_d>0$. Combining the two cases \eqref{eq:smalletacase} and \eqref{eq:largeetacase} completes the proof of \eqref{TTT2}. 

\subsection{Proof of \Cref{claim:TTk}}\label{pf:claim_TTk}
To prove the estimate \eqref{eq:key_T_reudce}, we begin by partitioning the summation region over $[\al]$ into $2^{2k}$ subregions according to whether each $|[x_i]-[\al]|$ or $|[y_i]-[\al]|$ is larger than $\ell$ or not. Namely, we define
\[ \bD_{\le \ell,\bsig}:=\left\{[\al]:  |[w_{i}]-[a]|\le \ell \ \text{if}\ \sig_i=0, \text{ and } |[w_{i}]-[a]|>\ell \ \text{if}\ \sig_i=1, \ \forall i\in\qqq{2k} \right\},\]
where $\bsig=(\sig_1,\ldots, \sig_{2k}) \in \{0,1\}^{2k}$, and $[w_{2i-1}]=[x_i]$ and $[w_{2i}]=[y_i]$ for $i\in\qqq{k}$. It therefore suffices to show that, for each fixed $\bsig \in \{0,1\}^{2k}$,
\begin{align}\label{eq:key_T_reudce_pf}
\sum_{[\al]\in \bD_{\le \ell,\bsig}} \prod_{i=1}^k \br{\sT_{t}(|[x_i]-[\al]|\wedge \ell) \cdot \sT_{t}(|[y_i]-[\al]|\wedge \ell) } \prec (W^d\eta_t)^{-1} \Psi_t^{k-2}  \prod_{i=1}^k \sT_{t}(|[x_i]-[y_i]|\wedge \ell) .
\end{align} 
We now analyze three cases, depending on how many paths consist entirely of “short edges”, i.e., edges of length $\le \ell$.
 
\medskip
\noindent{\bf 1.} Assume that there are at least two indices $i$ such that $|[x_i]-[\al]|\vee|[y_i]-[\al]|\le \ell$. Without loss of generality, suppose this condition holds for $1\le i \le r$, and $|[x_i]-[\al]|\vee|[y_i]-[\al]|> \ell$ for $r+1\le i \le k$, where $2\le r\le k$. For $1\le i\le r$, we have  
\be\label{eq:TtTt}\sT_{t}(|[x_i]-[\al]|)\sT_{t}(|[y_i]-[\al]|)\lesssim \sT_{t}(|[x_i]-[y_i]|) 
\cdot 
\frac{(W^{-d}B_{t,0})^{1/2}}{(|[x_i]-[\al]| \wedge |[y_i]-[\al]|+1)^{(d-2)/{2}}}, \ee
and for $r+1\le i \le k$, we have
\be\label{eq:KtKt}\sT_{t}(|[x_i]-[\al]|)\sT_{t}(|[y_i]-[\al]|)\lesssim  \sT_{t}(\ell) \cdot (W^{-d}B_{t,0})^{1/2} . \ee
Using these two estimates, we can bound the LHS of \eqref{eq:key_T_reudce_pf} by
\begin{align*}
&~(W^{-d}B_{t,0})^{k/2} \prod_{i=1}^{k} \sT_{t}(|[x_i]-[y_i]|\wedge \ell) \cdot \sum_{[\al]\in \bD_{\le \ell,\bsig}} \prod_{i=1}^r (|[x_i]-[\al]| \wedge |[y_i]-[\al]|+1)^{-\frac{d-2}{2}} \\
\lesssim &~\ell^2 (W^{-d}B_{t,0})^{k/2} \prod_{i=1}^{k} \sT_{t}(|[x_i]-[y_i]|\wedge \ell) \prec (W^d\eta_t)^{-1} \Psi_t^{k-2}  \prod_{i=1}^{k} \sT_{t}(|[x_i]-[y_i]|\wedge \ell) ,
\end{align*}
where in the first step, we use
\[ \sum_{[\al]\in \bD_{\le \ell,\bsig}} \prod_{i=1}^r (|[x_i]-[\al]| \wedge |[y_i]-[\al]|+1)^{-\frac{d-2}{2}} \lesssim  \sum_{i=1}^r \sum_{[\al]\in \bD_{\le \ell}}(|[x_i]-[\al]| \wedge |[y_i]-[\al]|+1)^{-(d-2)}\lesssim \ell^2,  \]
and in the second step, we use the facts that $\ell\le (\log W)^{10}\ell_t$ and $\ell_t^2B_{t,0}\lesssim |1-t|^{-1} \lesssim\eta_t^{-1}$ when $1-t\ge \ilambda^2/L^{2}$.

\medskip
\noindent{\bf 2.} Suppose there is exactly one index $i$ such that $|[x_i]-[\al]|\vee|[y_i]-[\al]|\le \ell$.  
Without loss of generality, assume that $|[x_1]-[\al]|\vee|[y_1]-[\al]|\le \ell$, and $|[y_2]-[\al]|\ge \ell$. Then, applying \eqref{eq:TtTt} for $i=1$, and \eqref{eq:KtKt} for $3\le i \le k$, we can bound the LHS of \eqref{eq:key_T_reudce_pf} as
\begin{align*}
&~(W^{-d}B_{t,0})^{k/2} \prod_{i=1}^{k} \sT_{t}(|[x_i]-[y_i]|\wedge \ell) \cdot \sum_{[\al]\in \bD_{\le \ell,\bsig}}  (|[x_1]-[\al]| \wedge |[y_1]-[\al]|+1)^{-\frac{d-2}{2}}(|[x_2]-[\al]| \wedge \ell +1)^{-\frac{d-2}{2}} \\
\lesssim &~\ell^2 (W^{-d}B_{t,0})^{k/2} \prod_{i=1}^{k} \sT_{t}(|[x_i]-[y_i]|\wedge \ell) \prec (W^{d}\eta_t)^{-1} \Psi_t^{k-2} \prod_{i=1}^{k} \sT_{t}(|[x_i]-[y_i]|\wedge \ell) .
\end{align*}

\medskip
\noindent{\bf 3.} Finally, assume that for all $i\in\qqq{k},$ we have $|[x_i]-[\al]|\vee|[y_i]-[\al]|> \ell$. Without loss of generality, suppose $|[y_1]-[\al]|\ge \ell$ and $|[y_2]-[\al]|\ge \ell$. Then, applying \eqref{eq:KtKt} for $3\le i \le k$, we can bound the LHS of \eqref{eq:key_T_reudce_pf} by 
\begin{align*}
&~(W^{-d}B_{t,0})^{k/2} \prod_{i=1}^{k} \sT_{t}(|[x_i]-[y_i]|\wedge \ell) \cdot \sum_{[\al]\in \bD_{\le \ell,\bsig}}  (|[x_1]-[\al]| \wedge \ell+1)^{-\frac{d-2}{2}}(|[x_2]-[\al]| \wedge \ell +1)^{-\frac{d-2}{2}} \\
\lesssim &~\ell^2 (W^{-d}B_{t,0})^{k/2} \prod_{i=1}^{k} \sT_{t}(|[x_i]-[y_i]|\wedge \ell) \prec (W^{d}\eta_t)^{-1} \Psi_t^{k-2} \prod_{i=1}^{k} \sT_{t}(|[x_i]-[y_i]|\wedge \ell) .
\end{align*}

\medskip

By combining all three cases, we conclude that \eqref{eq:key_T_reudce_pf} holds, which, in turn, implies \eqref{eq:key_T_reudce}.

\subsection{Basic properties of  \texorpdfstring{$\cal K$}{K}-loops}\label{Sec:CalK}

In this subsection, we collect several basic properties of the ${\cal K}$-loops used in the analysis of the loop hierarchy and apply them to establish the $\cK$-loop bounds stated in \Cref{ML:Kbound,lem_wardineq_K}. The results presented here are higher-dimensional analogues (in dimensions $d \ge 3$) of those in \cite[Section 3]{YY_25} and \cite[Section 4]{RBSO1D}. We begin by introducing a dimension-independent \emph{tree representation formula} for $\cal K$-loops, first discovered in \cite{YY_25} for random band matrices and later extended to the block Anderson model in \cite{RBSO1D}. This tree representation is constructed using the notion of \emph{canonical partitions of polygons}. Roughly speaking, a canonical partition of an oriented polygon $\mathcal{P}_{\ba}$ is a partition in which each edge of the polygon is in one-to-one correspondence with each region in the partition. 

\begin{definition}[Canonical partitions]\label{def:canpnical_part}
Fix $n\ge 3$ and let $\cal P_{\ba}$ be an oriented polygon with vertices $\ba=(a_1,a_2, \ldots ,a_\fn)$ arranged in a (counterclockwise) cyclic order,  where we adopt the cyclic convention that $a_i=a_j$ if and only if $i=j\mod n$. Let $(a_{k-1},a_k)$ denote the $k$-th side of $\cal P_{\ba}$. A planar partition of the polygonal domain enclosed by $\cal P_{\ba}$ is called {\bf canonical} if the following properties hold:
\begin{itemize}
	\item Every sub-region in the partition is also a polygonal domain. 
	\item There is a one-to-one correspondence between the edges of the polygon and the sub-regions, where every side $(a_{k-1},a_k)$ belongs to exactly one sub-region, denoted by $R_k$, and each sub-region contains exactly one side of $\cal P_{\ba}$. 
	
	\item Every vertex $a_k$ of $\cal P_{\ba}$ belongs to exactly two regions, $R_k$ and $R_{k+1}$ (with the convention $R_{\fn+1}=R_1$). 
	
\end{itemize}
Note that given a canonical partition, by removing the $\fn$ sides of the polygon ${\cal P}_{\ba}$, the remaining interior edges form a tree, with the leaves being the vertices of ${\cal P}_{\ba}$. Following the definitions in \cite{YY_25}, we define the equivalence classes of all such trees under graph isomorphism, and denote the collection of equivalence classes by $\TSP({\cal P}_{\ba})$. 
We will consider each element of $\TSP({\cal P}_{\ba})$ as an abstract tree structure rather than as an equivalence class, and call it a \emph{canonical tree partition}.
\end{definition}

In a canonical tree partition $\Gamma \in \TSP({\cal P}_{\ba})$, we call an edge that contains exactly one external vertex $a_k$ an \emph{external edge}, and an edge connecting two internal vertices an \emph{internal edge}. Two regions $R_k$ and $R_l$ are said to be \emph{neighbors} if they share a common side, which may be either an external or an internal edge. In the case of an external edge, we necessarily have $k-l = \pm 1 \pmod{\fn}$, and we refer to $R_k$ and $R_l$ as \emph{trivial neighbors}; otherwise, they are called \emph{nontrivial neighbors}.
Given $\bsig \in \{+,-\}^n$, we assign charges to the subregions as follows: each subregion $R_k$ carries the charge of the edge $(a_{k-1},a_k)$, which is given by $\sig_k$. An illustration is provided in the left panel of \Cref{example}, which shows a canonical tree partition $\Gamma \in \TSP({\cal P}_{\ba})$ of a polygon with six vertices, where $R_4$ and $R_6$ form a pair of nontrivial neighbors. We note that the figures in this section are all reproduced from \cite{RBSO1D}.

In the context of random band matrices, we assign a value to $\Gamma$ according to the following rule.

\begin{definition}
Given any $t\in [0,1)$ and $\bsig\in\{+,-\}^\fn$, we define the values of the edges in $\Gamma$ as follows:
\begin{enumerate}
	\item If $e = \p{a_k, b}$ is an external edge lying between regions $R_k$ and $R_{k+1}$, then we define 
	\begin{equation}\label{f-external}
		f_{t,\bsig}\p{e} \coloneqq \Theta^{(\sig_k,\sig_{k+1})}_{t}(a_k,b).
	\end{equation}
	\item If $e = \p{b_1, b_2}$ is an internal edge lying between regions $R_k$ and $R_{l}$, then 
	\begin{equation}\label{f-internal}
		\begin{aligned}
			f_{t,\bsig}\p{e}
			&\coloneqq \big(\Theta^{(\sig_k,\sig_l)}_{t}-I\big) \p{b_1, b_2} = m(\sig_k)m(\sig_{l})\cdot \big(tS^{\LK} \Theta^{(\sig_k,\sig_l)}_{t}\big)\p{b_1,b_2}.
		\end{aligned}
	\end{equation}
\end{enumerate}
Then, we assign a value $\Gamma^{(\fn)}_{t,\bsig,\ba}$ to  $\Gamma$ as:
\begin{equation}\label{M-graph-value-unsummed}
	\Gamma^{(\fn)}_{t,\bsig,\ba} \coloneqq \p{\prod_{i=1}^\fn m(\sig_i)} \cdot \sum_{\mathbf b} \prod_{e} f_{t,\bsig}\p{e}    \, ,
\end{equation}
where $\mathbf b=(b_1,\ldots,b_{r})$ denotes the internal vertices in $\Gamma$ and $e$ denotes all the edges in $\Gamma$. 
\end{definition}

With these definitions, we recall the tree representation formula of the $\mathcal{K}$-loops in Lemma 3.4 of \cite{YY_25}. Recall that the formulas for the $\cal K$-loops of length 2 and 3 have been given in \eqref{Kn2sol} and \eqref{Kn3sol}.

\begin{lemma}[Lemma 3.4 of \cite{YY_25}]\label{tree-representation}
In the setting of random band matrices, for any $\fn\ge 4$, $t\in[0,1)$, $\bsig\in \{+,-\}^\fn$, and $\ba\in (\Zn)^\fn$, we have the following representation formula for $\cal K$-loops:
\begin{equation}\label{eq_Ktree}
	\cK_{t,\bsig,\ba}^{(\fn)}
	=W^{-d(\fn-1)} \sum_{\Gamma \in \TSP\p{\mathcal{P}_{\ba}}} \Gamma^{(\fn)}_{t,\bsig,\ba}.
\end{equation}
\end{lemma}

When extending \eqref{eq_Ktree} to the block Anderson model, certain factors of $m$ must be replaced by entries of the matrix $M$. A canonical procedure for this replacement is described in \cite[Section 4]{RBSO1D}. To formulate it, we first extend \Cref{def:canpnical_part} to include loops that contain \emph{$M$-edges}. As the name suggests, these edges correspond to entries of $M$, while the remaining edges in our graphs are left \emph{unlabeled} (i.e., without label $M$): external edges represent entries of $\Theta_t$, and internal edges represent entries of $tS^{\LK}\Theta_t$.

\begin{definition}[Canonical partitions with $M$-loops]\label{m-loop-tsp}
Let $\Gamma \in \TSP\p{\mathcal{P}_{\ba}}$ be a canonical tree partition of the oriented polygon $\mathcal{P}_{\ba}$, and denote the internal vertices of $\Gamma$ by $\mathbf b = \p{b_1, \ldots, b_r}$. We define the graph {$\Gamma_{M}$} by replacing each $b_i$ with an $M$-loop in the following way.
\begin{enumerate}
	\item Consider the subgraph $(V^{\p{b_i}}, E^{\p{b_i}})$ of all vertices in $\Gamma$ connected to $b_i$. More precisely, we let
	\begin{equation}\label{b_i-subgraph}
		V^{\p{b_i}} \coloneqq \{b_i,c_1, \ldots, c_{k_i}\},
		\quad \text{and}\quad E^{\p{b_i}} = \{\p{c_1, b_i}, \ldots, \p{c_{k_i}, b_i}\}
	\end{equation}
	denote the subsets of all vertices (including $b_i$) and edges connected to $b_i$. Reordering the $c_k$'s if necessary, we can ensure that $\p{c_1, \ldots, c_{k_i}}$ form a loop without any crossing edges and with vertices arranged in counterclockwise order.
	
	\item Next, we construct a new graph $(\widetilde{V}^{\p{b_i}}, \widetilde{E}^{\p{b_i}})$ with vertices and edges
	\begin{align}\label{eq:M-loopexample}
		\qquad      \widetilde{V}^{\p{b_i}} &= \{b_{i,1}, \ldots, b_{i,k_i}, c_1, \ldots, c_{k_i}\},\ \
		\widetilde{E}^{\p{b_i}}
		=\{\p{c_j, b_{i,j}} :  j \in \qqq{k_i}\} \cup \{\p{b_{i,j}, b_{i,j+1}; M}: j \in \qqq{k_i}\}, 
	\end{align}
	where $e = \p{e_i, e_f; M}$ refers to an edge labeled with $M$, and we adopt the cyclic convention that $b_{i,k_i+1}=b_{i,1}$. We will call $\p{e_i, e_f; M}$ as an \emph{$M$-edge}. 
	
	\item Lastly, we replace the subgraph $(V^{\p{b_i}}, E^{\p{b_i}})$ with $(\widetilde{V}^{\p{b_i}}, \widetilde{E}^{\p{i}})$.
\end{enumerate}

\begin{figure}[h]\label{b_i-loop-replacement}
	\centering
	\scalebox{0.9}{
		\begin{tikzpicture}
			\coordinate (bi) at (0, 0);
			
			\fill (bi) circle (1pt) node[left=4pt]{$b_i$};
			\foreach \i/\t in {1/72, 2/144, 3/216, 4/288, 5/360} {
				\fill (\t:2) circle (1pt);
				\draw (bi) -- (\t:2);
				\node at (\t+72:2.4) {$c_{\i}$};
			}
		\end{tikzpicture}
		$\quad$
		\begin{tikzpicture}
			\node at (-3, 0) {$\to$};
			\node at (0, 0) {$M$};
			\foreach \i/\t in {1/72, 2/144, 3/216, 4/288, 5/360} {
				\fill (\t:2) circle (1pt);
				\fill (\t:1) circle (1pt);
				\draw (\t:1) -- (\t:2);
				\draw (\t:1) -- (\t+72:1);
				\node at (\t+72:2.4) {$c_{\i}$};
				\node at (\t+92:1.3) {$b_{i,\i}$};
			}
	\end{tikzpicture}}
	\caption{Replacing $b_i$ with an $M$-loop with 5 sides.}\label{Fig:Mloop}
\end{figure}

In \Cref{Fig:Mloop}, we illustrate the above procedure for an example with $k_i=5$. Repeating these steps for each internal vertex $b_i$, $i\in\qqq{r}$, we get a graph \smash{$\Gamma_{M}$}, which we will refer to as the \emph{$M$-graph corresponding to $\Gamma$}. Note that the order in which we replace $b_i$'s by $M$-loops does not matter, and every boundary edge $\p{a_{k-1}, a_k}$ still belongs to exactly one polygonal region in the $M$-graph \smash{$\Gamma_{M}$}. 
With a slight abuse of notation, we still use $R_k$ to denote the sub-region containing $\p{a_{k-1}, a_k}$ in \smash{$\Gamma_{M}$}, and assign the charge $\sig_k$ of $\p{a_{k-1}, a_k}$ to $R_k$. 

\end{definition}

We refer readers to \Cref{example} for an example of a canonical tree partition $\Gamma$ and its $M$-graph $\Gamma_{M}$. 
\begin{figure}[h]
\centering
\scalebox{0.9}{
	\begin{tikzpicture}
		\coordinate (b1) at (-1, 0);
		\coordinate (b2) at (1, 0);
		
		\fill (b1) circle (1pt);
		\fill (b2) circle (1pt);
		
		\foreach \i/\t in {1/60, 2/120, 3/180, 4/240, 5/300, 6/360} {
			\coordinate (a\i) at (\t+90:2.5);
			\fill (a\i) circle (1pt);
			\draw (a\i) [dashed]-- (\t+150:2.5);
			\node at (\t+90:2.85) {$a_{\i}$};
			\node at (\t+60:2.5) {$\sigma_{\i}$};
		}
		
		\draw (a6) -- (b1);
		\draw (a1) -- (b1);
		\draw (a2) -- (b1);
		\draw (a3) -- (b1);
		
		\draw (a4) -- (b2);
		\draw (a5) -- (b2);
		
		\draw (b1) -- (b2);
		
		\node at (0, -3.5) {$\Gamma$};
	\end{tikzpicture}
	\hspace{1in}
	\begin{tikzpicture}
		\foreach \i/\t in {1/60, 2/120, 3/180, 4/240, 5/300, 6/360} {
			\coordinate (a\i) at (\t+90:2.5);
			\fill (a\i) circle (1pt);
			\draw (a\i) [dashed]-- (\t+150:2.5);
			\node at (\t+90:2.85) {$a_{\i}$};
			\node at (\t+60:2.5) {$\sigma_{\i}$};
		}
		
		\coordinate (c11) at ($(b1)!0.45!(a6)$);
		\coordinate (c12) at ($(b1)!0.45!(a1)$);
		\coordinate (c13) at ($(b1)!0.45!(a2)$);
		\coordinate (c14) at ($(b1)!0.45!(a3)$);
		\coordinate (c15) at ($(b1)!0.45!(b2)$);
		
		\coordinate (c21) at ($(b2)!0.25!(b1)$);
		\coordinate (c22) at ($(b2)!0.45!(a4)$);
		\coordinate (c23) at ($(b2)!0.45!(a5)$);
		
		\foreach \i/\nc in {1/5, 2/3} {
			\foreach \j in {1, 2, ..., \nc} {
				\fill (c\i\j) circle (1pt);
			}
		}
		
		\foreach \i/\nc in {1/5, 2/3} {
			\draw (c\i1)
			\foreach \j in {1, ..., \nc} {
				-- (c\i\j)
			}
			-- cycle;
		}
		
		\draw (c11) -- (a6);
		\draw (c12) -- (a1);
		\draw (c13) -- (a2);
		\draw (c14) -- (a3);
		
		\draw (c21) -- (c15);
		
		\draw (c22) -- (a4);
		\draw (c23) -- (a5);
		
		\node at ($0.2*(c11) + 0.2*(c12) + 0.2*(c13) + 0.2*(c14) + 0.2*(c15)$) {$M$};
		\node at ($0.333*(c21) + 0.333*(c22) + 0.333*(c23)$) {$M$};
		
		\node at (0, -3.5) {$\Gamma_{{M}}$};
\end{tikzpicture}}
\caption{Example of $\Gamma \in \TSP\protect\p{\mathcal{P}_{\ba}}$ and its corresponding $M$-graph $\Gamma_{M}$.}\label{example}
\end{figure}

\begin{definition}\label{M-graph-value-definition}
Given an $M$-graph \smash{$\Gamma_{M}$}, we assign a value to $\Gamma$ according to the following rule. For any $t\in[0,1)$ and $\bsig\in\{+,-\}^\fn$, we define the values of the edges in $\Gamma_{M}$ as follows. 
\begin{enumerate}
	\item If $e = \p{a_k, b}$ is an external edge lying between regions $R_k$ and $R_{k+1}$, then we define
	\begin{equation}\label{f-external2}
		g_{t,\bsig}\p{e} \coloneqq \Theta^{\p{\sig_{k}, \sig_{k+1}}}_{t}(a_k,b) \, .
	\end{equation}
	
	\item If $e = \p{b_1, b_2}$ is an internal unlabeled edge lying between regions $R_k$ and $R_{l}$, then we define
	\begin{equation}\label{f-internal2}
		g_{t,\bsig}\p{e}
		\coloneqq \big(tS^{\LK}\Theta^{(\sig_k,\sig_{l})}_{t}\big)\p{b_1,b_2} .
	\end{equation}
	
	\item Corresponding to each loop of $M$-edges as in \eqref{eq:M-loopexample}, we its value as 
	\begin{equation}\label{eq:Mloop_defgen}
		\cal F\p{b_i} \equiv W^{(k_i-1)d} \cal M_{\bsig(b_i),\ba(b_i)}^{(k_i)}, 
	\end{equation}
	where we recall the $M$-loop defined in \eqref{eq:KMloop}. 
	Here, the charges $\sig_j(b_i)$ and vertices $a_j(b_i)$ are fixed according to the following rule: for each $j \in \qqq{k_i}$, the edge $\p{b_{i,j-1} \to b_{i,j}}$ (with the cyclic convention that $b_{i,0}=b_{i,k_i}$) belongs to exactly one region $R_{k_j}$. Then, we set ${\sigma}_j{\p{b_i}} = {\sigma}_{k_j}$ and $a_j{\p{b_i}} = b_{i,j}$. 
	
\end{enumerate}
Now, we assign a value $\Gamma^{(\fn)}_{M;t,\bsig,\ba}(z)$ to $\Gamma_M$ as:
\begin{equation}\label{M-graph-value-unsummed2}
	\Gamma^{(\fn)}_{M;t,\bsig,\ba}(z) \coloneqq \sum_{i=1}^r \sum_{j=1}^{k_i} \sum_{b_{i,j} \in \Zn}  \prod_{\text{unlabeled }e} g_{t,\bsig}(e) \cdot \prod_{i=1}^r \cal F(b_i) \,  ,
\end{equation}
where $b_{i,j}$'s denote the internal vertices in $\Gamma_M$ and $e$ denotes all unlabeled (i.e., non-$M$) edges.  
\end{definition}

It is straightforward to verify that the definition in \eqref{M-graph-value-unsummed2} is consistent with \eqref{M-graph-value-unsummed} in the case of random band matrices, where $M(\sig) = m(\sig) I$. Furthermore, in the block Anderson model, the tree representation formula \eqref{eq_Ktree} extends as follows.

\begin{lemma}[Lemma 4.16 of \cite{RBSO1D}]\label{tree-representation_BA}
In the setting of the block Anderson model, for any $\fn\ge 4$, $t\in[0,1)$, $\bsig\in \{+,-\}^\fn$, and $\ba\in (\Zn)^\fn$, we have the following representation formula for $\cal K$-loops:
\begin{equation}\label{eq:tree_rep2}
	\cK_{t,\bsig,\ba}^{(\fn)}
	=W^{-d(\fn-1)} \sum_{\Gamma \in \TSP\p{\mathcal{P}_{\ba}}} \Gamma^{(\fn)}_{M;t,\bsig,\ba}.
\end{equation}
\end{lemma}

Next, we recall the molecule structure of $\mathcal{K}^{(\fn)}_{t,\boldsymbol{\sigma}, \mathbf{a}}$ given in Section 3.3 of \cite{YY_25}. We split the $M$-graphs $\Gamma_M$ with $\Gamma\in \TSP(\mathcal{P}_{\mathbf{a}})$ according to which edges of the tree are ``long''---we refer to a \smash{$\Theta_t^{(\sig_1,\sig_2)}$}-edge as a \emph{long edge} if $\sig_1\ne \sig_2$, and a \emph{short edge} otherwise. We adopt this terminology because, according to \eqref{prop:ThfadC} and \eqref{prop:ThfadC_short}, the short edge decays on a scale of order 1, which is much shorter than the decay scale $\ell_t$ for a long edge. Given any $\bsig \in \{+, -\}^{n}$, define the subset of long internal edges (i.e., the boundary between two nontrivial neighbors of different charges) as 
\[
{\cal F}_{\mathrm{long}}(\Gamma, \bsig) := \left\{\{k,l\} \in \Z^{\mathrm{off}}_\fn: R_k\cap R_l\neq \emptyset ,\ \sigma_k\ne \sigma_l \right\},
\]
where we define the subset 
\[
\Z^{\mathrm{off}}_\fn :=\left\{\{k, \ell\} | 1 \leq k < \ell \leq \fn,\, k - \ell \pmod \fn \notin \{1,-1\}\right\}.
\]
Given any subset \(\pi \subset \mathbb{Z}_\fn^{\mathrm{off}}\), we use \(\TSP(\mathcal{P}_{\ba}, \boldsymbol{\sigma}, \pi) := \big\{\Gamma \in \TSP(\mathcal{P}_{\ba}) : {\cal F}_{\text{long}}(\Gamma, \boldsymbol{\sigma}) = \pi\big\}\) to represent the subset of $\Gamma$ such that \(\pi\) labels the pairs of all non-trivial neighbors in $\Gamma$ (note $\pi$ can be $\emptyset$). Then, we define 
\begin{align}\label{eq:defKpi}
\cK^{\p{\pi}}\p{t,\bsig,\ba}
\coloneqq W^{-d(\fn-1)}\sum_{\Gamma \in \TSP\p{\mathcal{P}_{\ba}, \bsig, \pi}} \Gamma^{\p{\fn}}_{M;t,\bsig,\ba}  .
\end{align}
Note that $\mathcal{K}^{(n)}_{t,\boldsymbol{\sigma}, \mathbf{a}}$ can be decomposed as 
\begin{equation}\label{eq_K-Kpi}
\mathcal{K}^{(n)}_{t, \boldsymbol{\sigma}, \mathbf{a}} = W^{-d(n-1)}\sum_{\pi\subset \mathbb{Z}_n^{\rm{off}}} \mathcal{K}^{(\pi)}(t, \boldsymbol{\sigma}, \mathbf{a}).
\end{equation}
Moreover, we have the following molecule decomposition of $\mathcal{K}^{(\pi)}$ given in equation (3.53) of \cite{YY_25}:
\begin{equation}\label{eq:molecule-Kpi}
\mathcal{K}^{(\pi)}(t, \boldsymbol{\sigma}, \mathbf{a}) = \sum_{\mathbf{b},\mathbf{c}} \prod_{k=1}^n \Theta^{(\sig_k,\sig_{k+1})}_{t,a_k b_k} \cdot \prod_{k=1}^{r-1} \left(tS^{(\sB)} \Theta_t^{(+,-)}\right)_{c_{2k-1} c_{2k}} \cdot \prod_{k=1}^r \Sigma^{(\pi)}(t, \boldsymbol{\sigma}^{(k)}, \mathbf{b}^{(k)}).
\end{equation}
Here, each term $\Sigma^{(\pi)}(t, \boldsymbol{\sigma}^{(k)}, \mathbf{b}^{(k)})$ represents a \emph{molecule}, defined as a maximal subgraph consisting solely of $M$-loops and \emph{short (unlabeled) internal edges}.\footnote{
With a slight abuse of notation, we again use the term ``molecule” here. This definition is in the same spirit as that in \Cref{def_poly}, although the graph considered here is different from the one in \Cref{def_graph1}.
} 
In other words, if each molecule is collapsed into a single vertex, the resulting quotient graph contains only long internal edges connecting different molecules, along with external edges, which may be either short or long. 
In \eqref{eq:molecule-Kpi}, we assume there are $r$ molecules in total. The terms \smash{$tS^{(\sB)} \Theta_t^{(+,-)}$} correspond to the long internal edges $(c_{2k-1},c_{2k})$ between molecules, while the terms \smash{$\Theta^{(\sig_k,\sig_{k+1})}_{t,a_k b_k}$} correspond to the external edges $(a_k,b_k)$. The vectors $\mathbf b=(b_1,\ldots, b_n)$ and $\mathbf c=(c_1, c_2,\ldots, c_{2r-1},c_{2r})$ denote the internal vertices that serve as endpoints of these external and long internal edges, respectively. Furthermore, $\boldsymbol{\sigma}^{(k)}$ and $ \mathbf{b}^{(k)}$ specify the charges and distinguished vertices associated with the $k$-th molecule. All other internal vertices within the molecules are implicitly summed over.

Using the estimate (\ref{prop:ThfadC_short}), we can derive the following exponential decay estimate on ``pure loops" where all charges are identical.

\begin{lemma}\label{lem_pureloop} 
For $\bsig \in \{+, -\}^n$ with $\sigma_1 = \sigma_2 = \cdots = \sigma_n$, there exist constants $c_n,C_n>0$ such that 
\begin{equation}\label{res_pureKes}
	\left|\mathcal{K}^{(n)}_{t, \boldsymbol{\sigma}, \mathbf{a}}\right|\le C_n W^{-d(n-1)} \exp\left(-c_n\max_{i,j}|a_i-a_j|\right).
\end{equation} 
\end{lemma}
\begin{proof}
In the tree representation \eqref{eq:tree_rep2}, each term $\Gamma^{(n)}_{t,\boldsymbol{\sigma},\ba}$ on the RHS contains only $M$-edges and short unlabeled edges. Using the exponential decay from \eqref{prop:ThfadC_short}, along with the bound \eqref{Mbound_AO} or \eqref{Mbound_AO2} for $M$-edges in the setting of the block Anderson model, we obtain the desired estimate. 
\end{proof}

We are now ready to prove the key bounds on $\cK$-loops, stated in \Cref{ML:Kbound}. The proof is analogous to that of Lemma 3.11 in \cite{YY_25}, but requires additional modifications to handle the higher-dimensional setting $d \ge 3$. For the reader’s convenience, we provide the proof below.

\begin{proof}[\bf Proof of \Cref{ML:Kbound}]
By \eqref{eq_K-Kpi}, it suffices to show that for any $\pi \subset \Z^{\rm{off}}_n$, we have
\begin{equation}\label{eq:K-pi-bound}
	\left|\mathcal{K}^{(\pi)}(t, \boldsymbol{\sigma}, \mathbf{a})\right| \prec B_{t,0}^{n-1}.
\end{equation}
We prove the above bound by induction on the number of molecules $r$ in $\pi$. In order to carry out the induction step, we will prove \eqref{eq:K-pi-bound} for a slightly generalized quantity $\tilde{\mathcal{K}}^{(\pi)}$, defined by
\begin{equation}\label{eq:wtKpi}
	\tilde{\mathcal{K}}^{(\pi)}(t, \boldsymbol{\sigma}, \mathbf{a}) = \sum_{\mathbf{b},\mathbf{c}} \prod_{k=1}^n \tilde{\Theta}_{t,a_k b_k} \cdot \prod_{k=1}^{r-1} \left(tS^{(\sB)}\Theta^{(+,-)}_{t}\right)_{c_{2k-1} c_{2k}} \cdot \prod_{k=1}^r \Sigma^{(\pi)}(t, \boldsymbol{}{\sigma}^{(k)}, \mathbf{b}^{(k)}),
\end{equation}
where each $\tilde{\Theta}_{t}$ represents either $\Theta^{(\sig,\sig')}_{t}$ for some $\sig,\sig'\in\{+,-\}$ or $tS^{\LK}\Theta^{(+,-)}_{t}$.  

We start with the single molecule case with $\pi=\emptyset$. In this case, we can write 
\begin{equation*}
	\tilde{\mathcal{K}}^{(\emptyset)}(t, \boldsymbol{\sigma}, \mathbf{a}) = \sum_{\mathbf{b}} \prod_{i=1}^n \tilde{\Theta}_{t,a_i b_i} \cdot \Sigma^{(\emptyset)}(t, \boldsymbol{\sigma}, \mathbf{b}).
\end{equation*}
First, the $M$-edges and short unlabeled edges in $\tilde{\mathcal{K}}$, including all edges contained within the molecule, have constant size and fast exponential decay outside their constant range by \eqref{prop:ThfadC_short}, together with \eqref{Mbound_AO} or \eqref{Mbound_AO2} in the block Anderson model. Thus, the molecule $\Sigma^{(\emptyset)}$ has constant size and fast exponential decay as well:
\begin{equation}\label{eq:molecule-decay}
	\left|\Sigma^{(\emptyset)}(t, \boldsymbol{\sigma}, \mathbf{b})\right| \le C\exp\left\{-c\max_{i,j} |b_i-b_j|\right\}
\end{equation}
for some constants $c,C>0$. Additionally, if we have a pure loop (i.e., $\boldsymbol{\sigma}$ consists entirely of the same charge), then every edge in $\tilde{\mathcal{K}}^{(\emptyset)}$ is short, so we have 
\begin{equation*}
	\left|\tilde{\mathcal{K}}^{(\emptyset)}(t, \boldsymbol{\sigma}, \mathbf{a})\right| \le C\exp\left\{-c\max_{i,j} |a_i-a_j|\right\} \lesssim B_{t,0}^{n-1}.
\end{equation*}
If $\boldsymbol{\sigma}$ is not a pure loop, then without loss of generality, we can assume that $\sigma_1 \neq \sigma_2$. In this case, we have 
\begin{equation*}
	\tilde{\mathcal{K}}^{(\emptyset)}(t, \boldsymbol{\sigma}, \mathbf{a}) = \sum_{b_1} \tilde{\Theta}_{t,a_1b_1} \sum_{\mathbf{b}\setminus \{b_1\}} \Sigma^{(\emptyset)}(t, \boldsymbol{\sigma}, \mathbf{b}) \prod_{i=2}^n  \tilde{\Theta}_{t,a_i b_i}.
\end{equation*}
We now claim that 
\begin{equation}\label{eq:ind-step-bound}
	\sum_{b_1} \bigg|\sum_{\mathbf{b}\setminus \{b_1\}} \Sigma^{(\emptyset)}(t, \boldsymbol{\sigma}, \mathbf{b}) \prod_{i=2}^n \tilde{\Theta}_{t,a_i b_i}\bigg| \prec B_{t,0}^{n-2}.
\end{equation}
Combining this estimate with the bound \eqref{prop:ThfadC} for the $\Theta$-propagators, we obtain that 
\begin{align*}
	\left|\tilde{\mathcal{K}}^{(\emptyset)}(t, \boldsymbol{\sigma}, \mathbf{a})\right| 
	&\lesssim B_{t,0}\sum_{b_1} \bigg|\sum_{\mathbf{b}\setminus \{b_1\}} \Sigma^{(\emptyset)}(t, \boldsymbol{\sigma}, \mathbf{b}) \prod_{i=2}^n \tilde{\Theta}_{t,a_i b_i}\bigg| \prec B_{t,0}^{n-1}.
\end{align*}
This concludes \eqref{eq:K-pi-bound} for the case $\pi=\emptyset$.

To prove the bound \eqref{eq:ind-step-bound}, we consider two cases depending on whether (i) $\boldsymbol{\sigma}$ is non-alternating, or (ii) $\boldsymbol{\sigma}$ is alternating, where $\sig_j\ne \sig_{j+1}$ for all $j\in \qqq{n}$.  

\medskip
{\noindent\bf (i)} If $\boldsymbol{\sigma}$ in non-alternating, then there exists $j\in\llbracket 2,n\rrbracket$ such that $\sigma_j = \sigma_{j+1}$. The corresponding short edge has fast exponential decay on the constant scale by \eqref{prop:ThfadC_short}: 
\begin{equation}\label{eq:shortexternal}
	\left|\tilde{\Theta}_{t,a_jb_j}\right| \le C\exp\{-c|a_i - b_i|\}.
\end{equation}
For the remaining $(n-2)$ factors \smash{$\tilde{\Theta}_{t,a_i b_i}$}, we can bound them pointwise using \eqref{prop:ThfadC} as
\be\label{eq:pointwise_Theta}
\left|\tilde{\Theta}_{t,a_ib_i}\right| \lesssim B_{t,0}, \quad \forall i \in\llbracket 2,n \rrbracket \setminus \{j\}.
\ee
The above two bounds, together with the exponential decay for the molecule weight in \eqref{eq:molecule-decay}, allow us to bound the summation on the LHS of \eqref{eq:ind-step-bound} as $B_{t,0}^{n-2}$.

\medskip
{\noindent\bf (ii)} It remains to handle the challenging case of alternating $\boldsymbol{\sigma}$. In this setting, in addition to the constant-range exponential decay, we must also use the sum-zero property of molecule weights:
\begin{equation}\label{eq:Sigma-empty-sum-zero}
	\sum_{\mathbf{b}\setminus \{b_1\}} \Sigma^{(\emptyset)}(t, \boldsymbol{\sigma}, \mathbf{b}) = \OO(|1-t|),\quad \sum_{\mathbf{b}\setminus \{b_1\}} \big|\Sigma^{(\emptyset)}(t, \boldsymbol{\sigma}, \mathbf{b})\big| = \OO(\ilambda^2+|1-t|).
\end{equation}
The first estimate was established in \cite[Lemma 3.10]{YY_25} for 1D random band matrices and in \cite[Lemma 4.29]{RBSO1D} for 1D and 2D block Anderson models, while the second estimate was proved in \cite[Claim 4.30]{RBSO1D}. Importantly, these proofs are dimension-independent: they rely only on the pure loop estimate (\Cref{lem_pureloop}), the short-edge bound \eqref{prop:ThfadC_short}, the $M$-edge bound \eqref{Mbound_AO} or \eqref{Mbound_AO2}, and Ward’s identity \eqref{WI_calK} for $\cal K$-loops.

We now decompose the long external edges into three parts as follows:
\begin{align*}
	\tilde{\Theta}^{(+,-)}_{t,a_jb_j} = f_0(a_j, s_j) + f_1(a_j, s_j) + f_2(a_j, s_j), \quad \tilde{\Theta}_{t}^{(+,-)}\in \big\{\Theta^{(+,-)}_{t},\ tS^{\LK}\Theta^{(+,-)}_{t}\big\}  ,
\end{align*}
where we denote $f(a_j, s_j)= \tilde{\Theta}^{(+,-)}_{t}\p{a_j,b_1+s_j}$ with $s_j = b_j - b_1$, and 
\begin{align*}
	f_0(a_j, s_j) &= f(a_j, 0),\quad
	f_1(a_j, s_j) = \frac12 f(a_j, s_j) - \frac12 f(a_j, -s_j),\\
	f_2(a_j, s_j) &= \frac12 f(a_j, s_j) + \frac12 f(a_j, -s_j) - f(a_j, 0).
\end{align*}
By Lemma \ref{lem_propTH}, we have the following bounds under the condition $|s_j|\prec 1$ for $j\in\llbracket2,n\rrbracket$:
\begin{align}\label{eq:f12}
	\left|f_0(a_j, s_j)\right| &\lesssim B_{t,|a_j - b_1|} \le B_{t, 0},\quad \left|f_1(a_j, s_j)\right| \prec \frac{(\ilambda^2+|1-t|)^{-1}}{|a_j - b_1|^{d-1}+1},\quad 
	\left|f_2(a_j, s_j)\right| \prec \frac{(\ilambda^2+|1-t|)^{-1}}{|a_j - b_1|^{d}+1}.
\end{align}
We view $\Sigma^{(\emptyset)}(t, \boldsymbol{\sigma}, \mathbf{b}) = g(\mathbf{s})$ as a function of the shifts $\mathbf{s} = (s_2, \ldots, s_n)$. Then, we can write 
\begin{equation*}
	\sum_{\mathbf{b}\setminus \{b_1\}} \Sigma^{(\emptyset)}(t, \boldsymbol{\sigma}, \mathbf{b}) \prod_{j=2}^n \tilde{\Theta}_{t,a_j b_j} = \sum_{\mathbf{s}} g(\mathbf{s}) \prod_{j=2}^n \sum_{\xi_j\in\{0,1,2\}} f_{\xi_j}(a_j, s_j).
\end{equation*}
We split the sum above into several parts.
\begin{enumerate}
	
	\item Consider the terms where $\xi_j = 0$ for all $j\in\llbracket 2,n\rrbracket$. Then, we use the sum-zero property \eqref{eq:Sigma-empty-sum-zero} to get 
	\begin{align*}
		\sum_{b_1}\bigg|\sum_{\mathbf{s}} g(\mathbf{s}) \prod_{j=2}^n f_{0}(a_j, s_j)\bigg| &\lesssim (1-t)\sum_{b_1}\prod_{j=2}^n \left|\tilde{\Theta}_{t,a_jb_1}\right|\lesssim (1-t) B_{t, 0}^{n-2} \sum_{b_1} \left|\tilde{\Theta}_{t,a_2b_1}\right| \lesssim B_{t, 0}^{n-2},
	\end{align*}
	where in the second step, we use \eqref{eq:pointwise_Theta} for all but one $\tilde{\Theta}_{t}$ factor, and in the last step, we use \eqref{eq:THETAinftinf}. 
	
	\item Consider the case where exactly one $\xi_j$ equals $1$ while all other $\xi_j$’s are zero. In this situation, the sum vanishes by the skew symmetry of $f_1$ (namely, $f_1(a_j, s_j) = -f_1(a_j, -s_j)$) together with the symmetry of $\Sigma^{(\emptyset)}$ (i.e., $g(\mathbf{s}) = g(-\mathbf{s})$). The latter symmetry follows from the translation invariance and symmetry of the $\Theta$-propagators (\Cref{lem_propTH}) and of the $M$-loops (\Cref{lem:propM}).
	
	\item Suppose there is at least one $i\in\llbracket2,n\rrbracket$ such that $\xi_i = 2$. Then, using \eqref{eq:f12} and \eqref{eq:Sigma-empty-sum-zero}, we get 
	\begin{equation*}
		\qquad   \sum_{b_1}\bigg|\sum_{\mathbf{s}} g(\mathbf{s}) f_2(a_i, s_i) \prod_{j\notin \{1,i\}} f_{\xi_j}(a_j, s_j)\bigg| \prec \frac{B_{t, 0 }^{n-2}}{\ilambda^2+|1-t|} \sum_{b_1} \frac{1}{|a_i-b_1|^{d}+1} \sum_{\mathbf{b}\setminus \{b_1\}} \big|\Sigma^{(\emptyset)}(t, \boldsymbol{\sigma}, \mathbf{b})\big| \prec B_{t, 0 }^{n-2}.
	\end{equation*}
	
	\item Suppose there are at least two $i,k\in \llbracket2,n\rrbracket$ such that $i\neq k$ and $\xi_i=\xi_k = 1$. In this case, using \eqref{eq:f12} and \eqref{eq:Sigma-empty-sum-zero}, we get 
	\begin{align*}
		\quad   \sum_{b_1}\bigg|\sum_{\mathbf{s}} g(\mathbf{s}) f_1(a_i, s_i) f_1(a_k, s_k)\prod_{j\notin \{1,i,k\}} f_{\xi_j}(a_j, s_j)\bigg| \prec \frac{B_{t, 0 }^{n-3}}{\ilambda^2+|1-t|} \sum_{b_1} \frac{1}{|a_i-b_1|^{d-1}+1}\frac{1}{|a_k-b_1|^{d-1}+1} \\
		\lesssim  B_{t, 0 }^{n-2}  \sum_{b_1} \left[\frac{1}{|a_i-b_1|^{d}+1}\frac{1}{|a_k-b_1|^{d-2}+1}+\frac{1}{|a_i-b_1|^{d-2}+1}\frac{1}{|a_k-b_1|^{d}+1}\right] \lesssim B_{t, 0 }^{n-2},
	\end{align*}
	where, in the second step, we also use $(\ilambda^2+|1-t|)^{-1}\lesssim B_{t,0}$. 
	
\end{enumerate}
This concludes the proof of the claim \eqref{eq:ind-step-bound}.

Finally, we establish \eqref{eq:K-pi-bound} by induction on the number of molecules $r$. 
Assume that \eqref{eq:K-pi-bound} holds for every $\pi$ with at most $n$ external vertices and at most $(r-1)$ molecules. 
Now, consider a configuration $\pi$ consisting of $r$ molecules. In the molecule-level quotient graph---where each molecule is represented as a single vertex---the molecules, together with the long internal edges, form a tree. Therefore, there exists a leaf in this tree. Without loss of generality, assume that the molecule indexed by $k=1$ in \eqref{eq:molecule-Kpi} is such a leaf, and that it is connected to $ a_1, \ldots, a_{l}$ by external edges. 
Then, we can decompose \smash{$\tilde{\mathcal{K}}^{(\pi)}$} in \eqref{eq:wtKpi} as  
\begin{align}
	\left|\tilde{\mathcal{K}}^{(\pi)}(t, \boldsymbol{\sigma}, \mathbf{a})\right| &= \bigg|\sum_{\mathbf{b}^{(1)}} \prod_{i=1}^l \tilde{\Theta}_{t,a_i b_i} \cdot \Sigma^{(\emptyset)}(t, \boldsymbol{\sigma}^{(1)}, \mathbf{b}^{(1)}) \cdot\tilde{\mathcal{K}}^{(\pi')}(t, \boldsymbol{\sigma}', \mathbf{a}')\bigg| \nonumber\\
	&\le \sum_{c_1}\bigg|\sum_{\mathbf{b}^{(1)}\setminus \{c_1\}} \prod_{i=1}^l \tilde{\Theta}_{t,a_i b_i} \cdot \Sigma^{(\emptyset)}(t, \boldsymbol{\sigma}^{(1)}, \mathbf{b}^{(1)})\bigg| \cdot\left|\tilde{\mathcal{K}}^{(\pi')}(t, \boldsymbol{\sigma}', \mathbf{a}')\right|, \label{eq:Kpipi}
\end{align}
where $\pi'$ consists of the remaining $(r-1)$ molecules, $\mathbf{b}^{(1)} = (b_1,\ldots, b_l,c_1)$ with $c_1$ denoting the vertex in the first molecule that connects to other molecules through a long internal edge, $\boldsymbol{\sigma}^{(1)}=(\sig_1,\ldots,\sig_l,\sig_{l+1})$, $\mathbf{a}' = (c_1, a_{l+1}, \ldots, a_{n})$, and $\boldsymbol{\sigma}' = (\sigma_{l+1}, \ldots, \sigma_n,\sig_1)$.  
Applying the induction hypothesis to bound \smash{$\tilde{\mathcal{K}}^{(\pi')}$}, and using \eqref{eq:ind-step-bound} to estimate the remaining part of the product (noting that \eqref{eq:ind-step-bound} is applicable because $\sigma_1 \neq \sigma_{l+1}$ by the definition of a molecule), we obtain that 
\begin{equation*}
	\left|\tilde{\mathcal{K}}^{(\pi)}(t, \boldsymbol{\sigma}, \mathbf{a})\right| \prec B_{t, 0}^{l-1} \cdot B_{t, 0}^{n-l} = B_{t, 0}^{n-1}.
\end{equation*}
This concludes the proof of Lemma \ref{ML:Kbound}.
\end{proof}

Finally, \Cref{lem_wardineq_K} follows from the bound \eqref{eq:ind-step-bound}, together with an induction argument analogous to the one used above in the proof of Lemma~\ref{ML:Kbound}.

\begin{proof}[\bf Proof of \Cref{lem_wardineq_K}]
By \eqref{eq_K-Kpi}, it suffices to prove that for any $\pi \subset \Z^{\rm{off}}_n$, 
\begin{equation}\label{eq:K-pi-bound_partial}
	\sum_{a_n}   \left|\mathcal{K}^{(\pi)}(t, \boldsymbol{\sigma}, \mathbf{a})\right| \prec \eta_t^{-1}B_{t,0}^{n-2}.
\end{equation}
In the case $\pi=\emptyset$, using \eqref{eq:THETAinftinf} and \eqref{eq:ind-step-bound}, we can bound the LHS of \eqref{eq:K-pi-bound_partial} by 
\begin{equation}\nonumber
	\sum_{a_n} \Theta^{(+,-)}_{t,a_n b_n} \sum_{b_n} \bigg|\sum_{\mathbf{b}\setminus \{b_n\}} \Sigma^{(\pi)}(t, \boldsymbol{\sigma}, \mathbf{b}) \prod_{i=1}^{n-1} \tilde{\Theta}_{t,a_i b_i}\bigg| \le \frac{1}{1-t}\sum_{b_n} \bigg|\sum_{\mathbf{b}\setminus \{b_n\}} \Sigma^{(\pi)}(t, \boldsymbol{\sigma}, \mathbf{b}) \prod_{i=1}^{n-1} \tilde{\Theta}_{t,a_i b_i}\bigg| \prec \eta_t^{-1}B_{t,0}^{n-2}.
\end{equation} 
For $\pi\ne \emptyset$, we argue by induction as in the proof of Lemma~\ref{ML:Kbound}.
Assume that \eqref{eq:K-pi-bound_partial} holds for any $\pi$ with at most $n$ external vertices and at most $(r-1)$ molecules. Using the notation of \eqref{eq:Kpipi}, and applying \eqref{eq:ind-step-bound} together with the induction hypothesis, we obtain
\begin{align*}
	\sum_{a_n} \left|\tilde{\mathcal{K}}^{(\pi)}(t, \boldsymbol{\sigma}, \mathbf{a})\right| &\le \sum_{c_1}\bigg|\sum_{\mathbf{b}^{(1)}\setminus \{c_1\}} \prod_{i=1}^l \tilde{\Theta}_{t,a_i b_i} \cdot \Sigma^{(\emptyset)}(t, \boldsymbol{\sigma}^{(1)}, \mathbf{b}^{(1)})\bigg| \cdot \sum_{a_n}\left|\tilde{\mathcal{K}}^{(\pi')}(t, \boldsymbol{\sigma}', \mathbf{a}')\right| \\
	&\prec B_{t, 0}^{l-1} \cdot \eta_t^{-1}B_{t, 0}^{n-l-1} =\eta_t^{-1} B_{t, 0}^{n-2}. 
\end{align*}
This concludes the proof of \eqref{eq:K-pi-bound_partial}.
\end{proof}


\begin{thebibliography}{10}
	
	\bibitem{Bethe-Anderson}
	A.~Aggarwal and P.~Lopatto.
	\newblock Mobility edge for the {A}nderson model on the {B}ethe lattice.
	\newblock {\em arXiv:2503.08949}, 2025.
	
	\bibitem{Aizenman1994}
	M.~Aizenman.
	\newblock Localization at weak disorder: Some elementary bounds.
	\newblock {\em Reviews in Mathematical Physics}, 06(05a):1163--1182, 1994.
	
	\bibitem{Aizenman1993}
	M.~Aizenman and S.~Molchanov.
	\newblock {Localization at large disorder and at extreme energies: an
		elementary derivation}.
	\newblock {\em Communications in Mathematical Physics}, 157(2):245--278, 1993.
	
	\bibitem{ASFH2001}
	M.~Aizenman, J.~H. Schenker, R.~M. Friedrich, and D.~Hundertmark.
	\newblock Finite-volume fractional-moment criteria for {Anderson} localization.
	\newblock {\em Communications in Mathematical Physics}, 224(1):219--253, 2001.
	
	\bibitem{Bethe_PRL}
	M.~Aizenman and S.~Warzel.
	\newblock Extended states in a {L}ifshitz tail regime for random
	{S}chr\"odinger operators on trees.
	\newblock {\em Physical Review Letters}, 106:136804, 2011.
	
	\bibitem{aizenman2013resonant}
	M.~Aizenman and S.~Warzel.
	\newblock Resonant delocalization for random {S}chr{\"o}dinger operators on
	tree graphs.
	\newblock {\em Journal of the European Mathematical Society}, 15(4):1167--1222,
	2013.
	
	\bibitem{Aizenman_book}
	M.~Aizenman and S.~Warzel.
	\newblock {\em Random operators: disorder effects on quantum spectra and
		dynamics}, volume 168 of {\em Graduate Studies in Mathematics}.
	\newblock American Mathematical Society, Providence, 2015.
	
	\bibitem{AEK_PTRF}
	O.~H. Ajanki, L.~Erd{\H o}s, and T.~Kr{\"u}ger.
	\newblock Stability of the matrix {Dyson} equation and random matrices with
	correlations.
	\newblock {\em Probability Theory and Related Fields}, 173(1):293--373, 2019.
	
	\bibitem{Anderson}
	P.~W. Anderson.
	\newblock Absence of diffusion in certain random lattices.
	\newblock {\em Physical Review}, 109:1492--1505, 1958.
	
	\bibitem{BaoErd2015}
	Z.~Bao and L.~Erd{\H{o}}s.
	\newblock Delocalization for a class of random block band matrices.
	\newblock {\em Probability Theory and Related Fields}, 167(3):673--776, 2017.
	
	\bibitem{Biane}
	P.~Biane.
	\newblock On the free convolution with a semi-circular distribution.
	\newblock {\em Indiana University Mathematics Journal}, 46(3):705--718, 1997.
	
	\bibitem{PB_review}
	P.~Bourgade.
	\newblock Random band matrices.
	\newblock In {\em Proceedings of the International Congress of Mathematicians:
		Rio de Janeiro 2018}, pages 2759--2783. World Scientific, 2018.
	
	\bibitem{bourgade2017universality}
	P.~Bourgade, L.~Erdos, H.-T. Yau, and J.~Yin.
	\newblock Universality for a class of random band matrices.
	\newblock {\em Advances in Theoretical and Mathematical Physics},
	21(3):739--800, 2017.
	
	\bibitem{bourgade2019random}
	P.~Bourgade, F.~Yang, H.-T. Yau, and J.~Yin.
	\newblock Random band matrices in the delocalized phase, {II}: Generalized
	resolvent estimates.
	\newblock {\em Journal of Statistical Physics}, 174(6):1189--1221, 2019.
	
	\bibitem{bourgade2020random}
	P.~Bourgade, H.-T. Yau, and J.~Yin.
	\newblock Random band matrices in the delocalized phase, {I}: Quantum unique
	ergodicity and universality.
	\newblock {\em Communications on Pure and Applied Mathematics},
	73(7):1526--1596, 2020.
	
	\bibitem{Bourgain2005}
	J.~Bourgain and C.~Kenig.
	\newblock On localization in the continuous {A}nderson-{B}ernoulli model in
	higher dimension.
	\newblock {\em Inventiones mathematicae}, 161(2):389--426, 2005.
	
	\bibitem{Carmona1982_Duke}
	R.~Carmona.
	\newblock {Exponential localization in one dimensional disordered systems}.
	\newblock {\em Duke Mathematical Journal}, 49(1):191--213, 1982.
	
	\bibitem{Carmona1987}
	R.~Carmona, A.~Klein, and F.~Martinelli.
	\newblock Anderson localization for {B}ernoulli and other singular potentials.
	\newblock {\em Communications in Mathematical Physics}, 108(1):41--66, 1987.
	
	\bibitem{PhysRevLett.64.5}
	G.~Casati, I.~Guarneri, F.~Izrailev, and R.~Scharf.
	\newblock Scaling behavior of localization in quantum chaos.
	\newblock {\em Physical Review Letters}, 64:5--8, 1990.
	
	\bibitem{PhysRevLett.64.1851}
	G.~Casati, L.~Molinari, and F.~Izrailev.
	\newblock Scaling properties of band random matrices.
	\newblock {\em Physical Review Letters}, 64:1851--1854, 1990.
	
	\bibitem{Chen2022}
	N.~Chen and C.~K. Smart.
	\newblock Random band matrix localization by scalar fluctuations.
	\newblock {\em arXiv:2206.06439}, 2022.
	
	\bibitem{Cipolloni2024}
	G.~Cipolloni, R.~Peled, J.~Schenker, and J.~Shapiro.
	\newblock Dynamical localization for random band matrices up to {$W\ll
		N^{1/4}$}.
	\newblock {\em Communications in Mathematical Physics}, 405(3):82, 2024.
	
	\bibitem{Damanik2002}
	D.~Damanik, R.~Sims, and G.~Stolz.
	\newblock {Localization for one-dimensional, continuum, {B}ernoulli-{A}nderson
		models}.
	\newblock {\em Duke Mathematical Journal}, 114(1):59--100, 2002.
	
	\bibitem{DingSmart2020}
	J.~Ding and C.~Smart.
	\newblock Localization near the edge for the {A}nderson {B}ernoulli model on
	the two dimensional lattice.
	\newblock {\em Inventiones mathematicae}, 219(2):467--506, 2020.
	
	\bibitem{DisPinSpe2002}
	M.~Disertori, L.~Pinson, and T.~Spencer.
	\newblock Density of states for random band matrices.
	\newblock {\em Communications in Mathematical Physics}, 232:83--124, 2002.
	
	\bibitem{Localization1_2}
	R.~Drogin.
	\newblock Localization of one-dimensional random band matrices.
	\newblock {\em arXiv:2508.05802}, 2025.
	
	\bibitem{DY}
	S.~Dubova and K.~Yang.
	\newblock {Quantum diffusion and delocalization in one-dimensional band
		matrices via the flow method}.
	\newblock {\em arXiv:2412.15207}, 2024.
	
	\bibitem{DYYY25}
	S.~Dubova, K.~Yang, H.-T. Yau, and J.~Yin.
	\newblock Delocalization of two-dimensional random band matrices.
	\newblock {\em arXiv:2503.07606}, 2025.
	
	\bibitem{EK_band1}
	L.~Erd{\H{o}}s and A.~Knowles.
	\newblock Quantum diffusion and delocalization for band matrices with general
	distribution.
	\newblock {\em Annales Henri Poincar{\'e}}, 12(7):1227, 2011.
	
	\bibitem{ErdKno2011}
	L.~Erd{\H{o}}s and A.~Knowles.
	\newblock Quantum diffusion and eigenfunction delocalization in a random band
	matrix model.
	\newblock {\em Communications in Mathematical Physics}, 303(2):509--554, 2011.
	
	\bibitem{Average_fluc}
	L.~Erd{\H o}s, A.~Knowles, and H.-T. Yau.
	\newblock Averaging fluctuations in resolvents of random band matrices.
	\newblock {\em Annales Henri Poincar{\'e}}, 14:1837--1926, 2013.
	
	\bibitem{erdos2013delocalization}
	L.~Erdos, A.~Knowles, H.-T. Yau, and J.~Yin.
	\newblock Delocalization and diffusion profile for random band matrices.
	\newblock {\em Communications in Mathematical Physics}, 323(1):367--416, 2013.
	
	\bibitem{erdHos2013local}
	L.~Erd{\H{o}}s, A.~Knowles, H.-T. Yau, and J.~Yin.
	\newblock The local semicircle law for a general class of random matrices.
	\newblock {\em Electronic Journal of Probability}, 18:1--58, 2013.
	
	\bibitem{EKS_Forum}
	L.~Erd{\H o}s, T.~Kr{\"u}ger, and D.~Schr{\"o}der.
	\newblock Random matrices with slow correlation decay.
	\newblock {\em Forum of Mathematics, Sigma}, 7:e8, 2019.
	
	\bibitem{erdHos2025zigzag}
	L.~Erd{\H{o}}s and V.~Riabov.
	\newblock The zigzag strategy for random band matrices.
	\newblock {\em arXiv:2506.06441}, 2025.
	
	\bibitem{erdHos2012rigidity}
	L.~Erd{\H{o}}s, H.-T. Yau, and J.~Yin.
	\newblock Rigidity of eigenvalues of generalized wigner matrices.
	\newblock {\em Advances in Mathematics}, 229(3):1435--1515, 2012.
	
	\bibitem{Block_reduction}
	J.~Fan, F.~Yang, and J.~Yin.
	\newblock A block reduction method for random band matrices with general
	variance profiles.
	\newblock {\em arXiv:2507.11945}, 2025.
	
	\bibitem{PhysRevLett.66.986}
	M.~Feingold, D.~M. Leitner, and M.~Wilkinson.
	\newblock Spectral statistics in semiclassical random-matrix ensembles.
	\newblock {\em Physical Review Letters}, 66:986--989, 1991.
	
	\bibitem{FroSpen_1985}
	J.~Fr{\"o}hlich, F.~Martinelli, E.~Scoppola, and T.~Spencer.
	\newblock {Constructive proof of localization in the Anderson tight binding
		model}.
	\newblock {\em Communications in Mathematical Physics}, 101(1):21--46, 1985.
	
	\bibitem{FroSpen_1983}
	J.~Fr{\"o}hlich and T.~Spencer.
	\newblock {Absence of diffusion in the Anderson tight binding model for large
		disorder or low energy}.
	\newblock {\em Communications in Mathematical Physics}, 88(2):151--184, 1983.
	
	\bibitem{PhysRevLett.67.2405}
	Y.~V. Fyodorov and A.~D. Mirlin.
	\newblock Scaling properties of localization in random band matrices: A
	\ensuremath{\sigma}-model approach.
	\newblock {\em Physical Review Letters}, 67:2405--2409, 1991.
	
	\bibitem{Germinet2013}
	F.~Germinet and A.~Klein.
	\newblock {A comprehensive proof of localization for continuous Anderson models
		with singular random potentials}.
	\newblock {\em Journal of the European Mathematical Society}, 15(1):53--143,
	2013.
	
	\bibitem{GMP}
	I.~Y. Gol'dshtein, S.~A. Molchanov, and L.~A. Pastur.
	\newblock A pure point spectrum of the stochastic one-dimensional
	{S}chr{\"o}dinger operator.
	\newblock {\em Functional Analysis and Its Applications}, 11(1):1--8, 1977.
	
	\bibitem{He2018}
	Y.~He, A.~Knowles, and R.~Rosenthal.
	\newblock Isotropic self-consistent equations for mean-field random matrices.
	\newblock {\em Probability Theory and Related Fields}, 171(1):203--249, 2018.
	
	\bibitem{HeMa2018}
	Y.~He and M.~Marcozzi.
	\newblock Diffusion profile for random band matrices: A short proof.
	\newblock {\em Journal of Statistical Physics}, 177(4):666--716, 2019.
	
	\bibitem{knowles2017anisotropic}
	A.~Knowles and J.~Yin.
	\newblock Anisotropic local laws for random matrices.
	\newblock {\em Probability Theory and Related Fields}, 169(1-2):pp--257, 2017.
	
	\bibitem{KunzSou}
	H.~Kunz and B.~Souillard.
	\newblock Sur le spectre des op{\'e}rateurs aux diff{\'e}rences finies
	al{\'e}atoires.
	\newblock {\em Communications in Mathematical Physics}, 78(2):201--246, 1980.
	
	\bibitem{Lawler_book}
	G.~F. Lawler and V.~Limic.
	\newblock {\em Random Walk: A Modern Introduction}.
	\newblock Cambridge Studies in Advanced Mathematics. Cambridge University
	Press, 2010.
	
	\bibitem{LeeSchSteYau2015}
	J.-O. Lee, K.~Schnelli, B.~Stetler, and H.-T. Yau.
	\newblock Bulk universality for deformed {Wigner} matrices.
	\newblock {\em Annals of Probability}, 44(3):2349--2425, 2016.
	
	\bibitem{LiZhang2019}
	L.~Li and L.~Zhang.
	\newblock {Anderson–Bernoulli localization on the three-dimensional lattice
		and discrete unique continuation principle}.
	\newblock {\em Duke Mathematical Journal}, 171(2):327--415, 2022.
	
	\bibitem{Wegner3}
	R.~Oppermann and F.~Wegner.
	\newblock Disordered system withn orbitals per site: 1/n expansion.
	\newblock {\em Zeitschrift f{\"u}r Physik B Condensed Matter}, 34(4):327--348,
	1979.
	
	\bibitem{PelSchShaSod}
	R.~Peled, J.~Schenker, M.~Shamis, and S.~Sodin.
	\newblock On the {Wegner} orbital model.
	\newblock {\em International Mathematics Research Notices}, 2019(4):1030--1058,
	2017.
	
	\bibitem{Wegner2}
	L.~Sch{\"a}fer and F.~J. Wegner.
	\newblock Disordered system with $n$ orbitals per site: Lagrange formulation,
	hyperbolic symmetry, and goldstone modes.
	\newblock {\em Zeitschrift f{\"u}r Physik B Condensed Matter}, 38:113--126,
	1980.
	
	\bibitem{Sch2009}
	J.~Schenker.
	\newblock Eigenvector localization for random band matrices with power law band
	width.
	\newblock {\em Communications in Mathematical Physics}, 290:1065--1097, 2009.
	
	\bibitem{SimonWolff}
	B.~Simon and T.~Wolff.
	\newblock Singular continuous spectrum under rank one perturbations and
	localization for random hamiltonians.
	\newblock {\em Communications on Pure and Applied Mathematics}, 39(1):75--90,
	1986.
	
	\bibitem{Spencer2}
	T.~Spencer.
	\newblock Random banded and sparse matrices.
	\newblock In G.~Akemann, J.~Baik, and P.~D. Francesco, editors, {\em Oxford
		Handbook of Random Matrix Theory}, chapter~23. Oxford University Press, New
	York, 2011.
	
	\bibitem{Spencer3}
	T.~Spencer.
	\newblock Duality, statistical mechanics and random matrices.
	\newblock {\em Current Developments in Mathematics}, 2012:229--260, 2012.
	
	\bibitem{Spencer1}
	T.~Spencer.
	\newblock {SUSY} statistical mechanics and random band matrices.
	\newblock In {\em Quantum Many Body Systems}, Lecture Notes in Mathematics, vol
	2051. Springer, Berlin, Heidelberg, 2012.
	
	\bibitem{RBSO1D}
	S.~K. Truong, F.~Yang, and J.~Yin.
	\newblock On the localization length of finite-volume random block
	{S}chr\"odinger operators.
	\newblock {\em arxiv:2503.11382}, 2025.
	
	\bibitem{Sooster2019}
	P.~von Soosten and S.~Warzel.
	\newblock Non-ergodic delocalization in the {R}osenzweig--{P}orter model.
	\newblock {\em Letters in Mathematical Physics}, 109(4):905--922, 2019.
	
	\bibitem{10.1214/19-ECP278}
	P.~von Soosten and S.~Warzel.
	\newblock {Random characteristics for Wigner matrices}.
	\newblock {\em Electronic Communications in Probability}, 24(none):1--12, 2019.
	
	\bibitem{Wegner1}
	F.~J. Wegner.
	\newblock Disordered system with $n$ orbitals per site: $n=\ensuremath{\infty}$
	limit.
	\newblock {\em Physical Review B}, 19:783--792, Jan 1979.
	
	\bibitem{Wigner}
	E.~P. Wigner.
	\newblock Characteristic vectors of bordered matrices with infinite dimensions.
	\newblock {\em Annals of Mathematics}, 62(3):548--564, 1955.
	
	\bibitem{Xu:2024aa}
	C.~Xu, F.~Yang, H.-T. Yau, and J.~Yin.
	\newblock Bulk universality and quantum unique ergodicity for random band
	matrices in high dimensions.
	\newblock {\em Annals of Probability}, 52(3):765--837, 5 2024.
	
	\bibitem{yang2021delocalization}
	F.~Yang, H.-T. Yau, and J.~Yin.
	\newblock Delocalization and quantum diffusion of random band matrices in high
	dimensions {I}: Self-energy renormalization.
	\newblock {\em arXiv:2104.12048}, 2021.
	
	\bibitem{yang2022delocalization}
	F.~Yang, H.-T. Yau, and J.~Yin.
	\newblock Delocalization and quantum diffusion of random band matrices in high
	dimensions {II}: {$T$}-expansion.
	\newblock {\em Communications in Mathematical Physics}, pages 1--96, 2022.
	
	\bibitem{yang2021random}
	F.~Yang and J.~Yin.
	\newblock Random band matrices in the delocalized phase, {III}: Averaging
	fluctuations.
	\newblock {\em Probability Theory and Related Fields}, 179(1):451--540, 2021.
	
	\bibitem{yang2024Del}
	F.~Yang and J.~Yin.
	\newblock Delocalization of a general class of random block {S}chr{\"o}dinger
	operators.
	\newblock {\em arXiv:2501.08608}, 2025.
	
	\bibitem{YY_25}
	H.-T. Yau and J.~Yin.
	\newblock Delocalization of one-dimensional random band matrices.
	\newblock {\em arXiv:2501.01718}, 2025.
	
\end{thebibliography}
\end{document}